\documentclass[]{article}

\usepackage{hyperref}

\renewcommand{\v}[1]{\ensuremath{\mathbf{#1}}}

\newcommand{\mc}{\mathcal}

\def\st{{\rm s.t.}}

\usepackage{algorithmic}
\usepackage{algorithm}
\usepackage{booktabs}
\usepackage{caption,subcaption}
\usepackage[usenames,dvipsnames]{color}
\usepackage{fullpage,graphicx,amssymb,amsmath}
\usepackage{enumerate}
\usepackage{empheq}
\usepackage[authoryear]{natbib}
\usepackage{lscape}
\usepackage[normalem]{ulem}
\usepackage{url}
\usepackage{verbatim} 
\usepackage{siunitx} 

\usepackage{morefloats} 

\usepackage[colorinlistoftodos,prependcaption,textsize=tiny]{todonotes}

\newcommand{\be}{\begin{enumerate}}
\newcommand{\ee}{\end{enumerate}}
\newtheorem{theorem}{Theorem}

\newtheorem{definition}{Definition}

\newcommand{\qed}{\hfill $\Box$}

\renewcommand{\Re}{\mathbb{R}} 
\newcommand{\Z}{\mathbb{Z}} 
\newcommand{\diag}{\mathop{\mathrm{diag}}}

\newcommand{\vbeta}{\boldsymbol{\beta}}

\newcommand{\sortedList}{\texttt{sortedList}} %
\newcommand{\newPoints}{\mc{I}^{\textrm{new}}} %
\newcommand{\sampled}{\mc{I}^{\textrm{sampled}}} %
\newcommand{\nontabuPoints}{\mc{I}^{\textrm{nontabu}}} %
\newcommand{\tabuPointsCurrent}{\mc{I}^{\textrm{tabu,cur}}} %
\newcommand{\tabuPointsNew}{\mc{I}^{\textrm{tabu,new}}} %
\newcommand{\aspirationPoints}{\mc{I}^{\textrm{aspire}}} %
\newcommand{\ftrue}{f^{\textrm{true}}} %
\newcommand{\fbar}{\bar{f}} %
\newcommand{\fhat}{\hat{f}} %
\newcommand{\vftrue}{\v{f}^{\textrm{true}}} %
\newcommand{\vfhat}{\hat{\v{f}}} %

\newcommand{\tabuTenureShort}{\tau^{\textrm{short}}} %
\newcommand{\neighborhoodShort}{\mc{N}^{\textrm{short}}} %
\newcommand{\neighborhoodLong}{\mc{N}^{\textrm{long}}} %
\newcommand{\tabuGridDistanceThresholdShort}{G^{\textrm{short}}} %
\newcommand{\tabuGridDistanceThresholdLong}{G^{\textrm{long}}} %
\newcommand{\iterFound}{C^{\textrm{itrFound}}} %

\newcommand\filledcirc{{\color{red}\bullet}\mathllap{\circ}} 

\newcommand{\newcodecolor}[1]{\textcolor{teal}{#1}}

\title{LineWalker: Line Search for Black Box Derivative-Free Optimization and Surrogate Model Construction}
\author{Dimitri J. Papageorgiou$^1$, Jan Kronqvist$^2$, Krishnan Kumaran$^1$  \\
{\small $^1$ExxonMobil Technology and Engineering Company}\\
{\small 1545 Route 22 East, Annandale, NJ 08801 USA}\\
{\small \{dimitri.j.papageorgiou,krishnan.kumaran\}@exxonmobil.com} \\
\\
{\small $^2$Department of Mathematics, KTH Royal Institute of Technology}\\
{\small Lindstedtsvagen 25, 114 28 Stockholm, Sweden}\\
{\small jankr@kth.se} \\
}

\begin{document}

\maketitle

\begin{abstract}
This paper describes a simple, but effective sampling method for optimizing and learning a discrete approximation (or surrogate) of a multi-dimensional function along a one-dimensional line segment of interest.
The method does not rely on derivative information and the function to be learned can be a computationally-expensive ``black box'' function that must be queried via simulation or other means. 
It is assumed that the underlying function is noise-free and smooth, although the algorithm can still be effective when the underlying function is piecewise smooth.
The method constructs a smooth surrogate on a set of equally-spaced grid points by evaluating the true function at a sparse set of judiciously chosen grid points.  At each iteration, the surrogate's non-tabu local minima and maxima are identified as candidates for sampling. Tabu search constructs are also used to promote diversification. If no non-tabu extrema are identified, a simple exploration step is taken by sampling the midpoint of the largest unexplored interval. The algorithm continues until a user-defined function evaluation limit is reached. Numerous examples are shown to illustrate the algorithm's efficacy and superiority relative to state-of-the-art methods, including Bayesian optimization and NOMAD, on primarily nonconvex test functions.

\vspace{2mm}
\noindent \textbf{keywords:} active learning, black-box optimization, derivative-free optimization, Gaussian process regression, surrogate model, tabu search.
\end{abstract}

\begin{quote}
\centering
\textit{``I keep a close watch on this heart of mine.\\
I keep my eyes wide open all the time.\\
I keep the ends out for the tie that binds\\
Because you're mine, I walk the line.''\\
\indent~~ --Johnny Cash ``I walk the line'' (1956)}
\end{quote}

\newpage
\section*{Nomenclature}

All sets are denoted in calligraphic font, e.g., $\mc{S}$ as opposed to $S$.
All vectors and matrices are written in bold font, e.g., $\v{s} = (s_1,\dots,s_N)$.
$\Re_+$ denotes the set of non-negative reals.
$\Z_+$ denotes the set of non-negative integers.
$D$ denotes the number of dimensions of the underlying function $\ftrue$, while $N$ denotes the number of grid points for the discretized approximation function $\fhat$, i.e., the dimension of the vector $\vfhat$. 

\begin{tabular}{ll}
\toprule
 & Definition \\
\hline
\textbf{Sets} & \\
$i,j \in \mc{I}$ & set of equally-spaced grid point indices; $\mc{I} = \{1,\dots,N\}$ \\
$i \in \sampled$ & set of already sampled grid points where the true function $\ftrue$ has been evaluated \\
$i \in \newPoints$ & set of newly identified grid points to be sampled \\
$\neighborhoodShort_i \subset \mc{I}$ & set of short-term neighboring indices to $i \in \mc{I}$; $\neighborhoodShort_i = \left\{j \in \mc{I} \cap \{ i - \tabuGridDistanceThresholdShort_i, \dots, i + \tabuGridDistanceThresholdShort_i \} \right\} $ \\
$\neighborhoodLong_i \subset \mc{I}$ & set of long-term neighboring indices to $i \in \mc{I}$; $\neighborhoodLong_i = \left\{j \in \mc{I} \cap \{ i - \tabuGridDistanceThresholdLong_i, \dots, i + \tabuGridDistanceThresholdLong_i \} \right\} $ \\
$\mc{S}^{\max}/\mc{S}^{\min}$ & set of maximizers/minimizers (grid point indices) of the approximation function $\hat{\v{f}}$ \\
\\
\multicolumn{2}{l}{\textbf{User-Defined Parameters}} \\
$\v{A} \in \Re^{N \times N}$ & smoothing matrix \\
$\iterFound_i \in \Z_+$ & iteration in which index $i \in \sampled$ was sampled/evaluated \\ 
$e^{\min}$ & user-defined optimality tolerance; maximum error between current and previous fits \\
$E^{\max,\textrm{total}} \in \Z_+$ & maximum number of total function evaluations allowed \\
$E^{\max,\textrm{itr}} \in \Z_+$ & maximum number of function evaluations allowed per major iteration \\
$\tabuGridDistanceThresholdShort_i \in \Z_+$ & short-term tabu grid distance threshold for grid index $i \in \mc{I}$ \\
$\tabuGridDistanceThresholdLong_i \in \Z_+$ & long-term tabu grid distance threshold for grid index $i \in \mc{I}$ \\
$N \in \Z_+$ & number of equally-spaced grid points \\
$N^{\max,\textrm{nbrs}} \in \Z_+$ & maximum number of sampled neighbors \\
$\alpha \in \Re_+$ & first-derivative smoothing parameter \\
$\mu \in \Re_+$ & second-derivative smoothing parameter \\
$\delta^{\min} \in \Re_+$ & objective function tolerance for local minima \\
$\delta^{\max} \in \Re_+$ & objective function tolerance for local maxima \\
$\nu^{\min} \in \Re_+$ & Minimum grid point separation multiplier  \\
$\nu^{\max} \in \Re_+$ & Maximum grid point separation multiplier  \\
$\tabuTenureShort \in \Z_+$ & short-term tabu tenure \\
\\
\multicolumn{2}{l}{\textbf{General Parameters}} \\
$s_i \in \{0,1\}$ & takes value 1 if true function $\ftrue$ has been evaluated at index $i$; 0 otherwise \\
$\diag(\v{s})$ & $\diag(\v{s}) \in \{0,1\}^{N \times N}$ such that $\diag(\v{s})_{ii} = 1$ if $s_i = 1$; $\diag(\v{s})_{ij} = 0$ otherwise \\
$\v{x}_i \in \Re^D$ & sample at grid index $i$ \\
\\
\textbf{Functions} & \\
$\ftrue:\Re^D \mapsto \Re$ & true $D$-dimensional function that we are trying to learn/optimize along a single dimension \\
$\ftrue_i$ & true discretized function value evaluated at grid point $i$ (i.e., $\ftrue_i = \ftrue(\v{x}_i)$) \\
$\fhat_i$ & approximate value of $\ftrue_i$ \\
$\vfhat \in \Re^N$ & ``approximation function'' (i.e., vector) of $\ftrue$ along a line segment; $\vfhat=(\fhat_1,\dots,\fhat_N)$ \\
\bottomrule
\end{tabular}

\section{Introduction}

Across engineering and scientific disciplines, one is often faced with the tasks of learning and optimizing a function whose analytical form is not known beforehand.
Learning a high-dimensional function is, in general, extremely challenging.  However, learning low-dimensional subspaces of this function may serve as a practical compromise that still reveals useful information.  With particular focus on computationally expensive black box functions, this paper describes an approach for learning a one-dimensional deterministic smooth function on a bounded interval, although the ideas can be extended to 2-, 3-, and other low-dimensional subspaces. While the method is not guaranteed to learn the one-dimensional function over the domain of interest, we show that pursuing the extrema of this function often produces a high-resolution approximation. 

In addition to learning, this algorithm can be used for optimizing a function along a line segment, a step commonly referred to as ``line search'' in continuous optimization. 
Unlike traditional gradient/Hessian-based methods, which seek to find the nearest local optima and then stop, this method seeks to approximate the function along the entirety of a given line segment using a small number of function evaluations.  Since local information for a truly nonconvex function tells you nothing about the function's behavior far from the current point, the hope is that the line search will efficiently uncover more information about the function than a traditional method. 
The algorithms presented in this paper do not describe how to find a search direction; they assume that one is given. 

From a mathematical vantage point, our motivations and goals can be described as follows:
Assume we are given a one-dimensional deterministic continuous (ideally, smooth) function $f:\Re \mapsto \Re$ on the domain $[x^L,x^U] \subset \Re$. Let $\epsilon > 0$ be a parameter used to define local optimality.  Then, our goals are to
\begin{eqnarray}
\text{find} && x^* \in \arg\min \Big\{ f(x) : x \in [x^L,x^U] \subset \Re \Big\} \label{eq:goal1_find_argmin}, \\
\text{find all} && x \in [x^L,x^U] : f(x) \leq f(y)~\text{or}~f(x) \geq f(y)~\forall y : |x-y|<\epsilon \label{eq:goal2_find_all_extrema}, \\
\text{find} && \fhat(\cdot) \in \arg\min_{\fbar(\cdot)} \max_{x \in [x^L,x^U]} || \fbar(x)-f(x) || \label{eq:goal3_find_best_surrogate}.
\end{eqnarray}
Goal~\eqref{eq:goal1_find_argmin} encapsulates the standard derivative-free optimization (DFO) goal of finding a global minimum of a black box function.
Goal~\eqref{eq:goal2_find_all_extrema} captures our less common and more challenging goal of finding all local extrema of this same black box function.
Goal~\eqref{eq:goal3_find_best_surrogate} is to find the ``best'' surrogate function $\fhat$ that minimizes the maximum error between it and the underlying function $f$, typically subject to a limit on the number of function evaluations that may be used to construct the surrogate.

Given the vast literature on DFO and surrogate modeling, we approached this research with a high degree of skepticism that improvements were possible.  Indeed, many state-of-the-art methods purport to find global optimal solutions to challenging DFO instances in 10, 20, and even 50 dimensions.  Certainly the one-dimensional setting has been solved, we thought. This hypothesis turns out to be false. We find that our proposed method is competitive with and sometimes superior to leading methods for solving one-dimensional deterministic DFO problems and/or producing a high-quality (low error) surrogate.

\subsection{Literature review}

Since our motivation is to learn and optimize a one-dimensional function, we briefly discuss relevant literature in the areas of line search, derivative-free optimization, and surrogate modeling.

\subsubsection{What is a ``line search'' searching for?}

The importance of line search in classical deterministic continuous optimization is unequivocal as captured by the assertion ``One-dimensional search is the backbone of many algorithms for solving a nonlinear programming problem'' \cite[Chapter 8, p.344]{bazaraa2006nonlinear}. Moreover, interest in line search has experienced a resurgence since ``Choosing appropriate step sizes is critical for reducing the computational cost of training large-scale neural network models'' \citet{chae2019empirical}. 
But what exactly is a line search method searching for?  In both the deterministic and stochastic optimization communities, the term ``line search'' is essentially synonymous with the task of finding an optimal, near-optimal, or sufficiently good step size (also known as the ``learning rate'' in the machine learning community \citep{ruder2016overview}) in which to move, after a search direction has been selected.
Indeed, as a cornerstone of numerous direct search algorithms for continuous optimization, line search methods attempt to answer the basic question: ``How far should I move from my current point along a direction of interest to a new point to improve my objective function value?''  
More formally, the prototypical line search algorithm (see, for example, \citet[Chapter 3]{nocedal2006numerical}) for minimizing a function $\phi: \Re^n \mapsto \Re$ assumes that a point $\v{x}_k \in \Re^n$ and direction $\v{d}_k \in \Re^n$ are available at iteration $k$ and that one then seeks to solve the following univariate minimization problem for the optimal step length $\gamma_k \in \Re$:
\begin{equation} \label{model:prototypical_line_search}
\min_{\gamma_k \in \Re} \phi(\v{x}_k + \gamma_k \v{d}_k).
\end{equation}
The method used to solve the minimization problem \eqref{model:prototypical_line_search} is referred to as a line search method.

Classic exact and inexact line search procedures are discussed in \cite[Chapter 8]{bazaraa2006nonlinear} and \citet[Chapter 3]{nocedal2006numerical}. 
Exact methods seek a global optimum of \eqref{model:prototypical_line_search}, while inexact methods attempt to find a ``good enough'' step size to guarantee descent at lower computational expense and are thus more commonly used in practice. 
Relatively few recent works have investigated line search methods.
For deterministic problems, \citet{neumaier2019line} present a line search method for optimizing continuously differentiable functions with Lipschitz continuous gradient.
\citet{bergou2018line} propose an adaptive regularized framework using cubics, which behaves like a line search procedure along the quasi-Newton direction with a special backtracking strategy for smooth nonconvex optimization. Meanwhile, for stochastic problems,
\citet{mahsereci2015probabilistic} pursue a probabilistic line search by constructing a Gaussian process surrogate of the univariate optimization objective, and using a probabilistic belief over the Wolfe conditions to monitor the descent.
\citet{bergou2022subsampling} assume a twice-continuously differentiable objective function and investigate a stochastic algorithm with subsampling to solve it. 
\citet{paquette2020stochastic} adapt a classical backtracking Armijo line search to the stochastic optimization setting.

But is the search for a scalar the only item that one could search for?  Certainly not. 
The quintessential line search algorithm used to solve \eqref{model:prototypical_line_search} is driven by the goal for \textit{iterative descent} whereby an algorithm is designed to successively improve the solution until convergence to a local optimum is achieved \citep{bertsekas1999nonlinear}.
In some ways, this classic approach can be viewed as an exploitation step since one is most often searching along a descent direction and therefore attempting to exploit this knowledge in hopes of making guaranteed improvement, however small or large that improvement may be. This classic approach also reveals that the primary goal of line search is optimality, not learning, an important theme addressed below.

\subsubsection{Derivative-free optimization}

Given the immense volume of DFO literature (also known as black box optimization), we highlight only the most relevant themes here, while pointing the interested reader to the surveys by \cite{conn2009introduction}, \cite{larson2019derivative}, \cite{rios2013derivative}, and the references therein.
\cite{larson2019derivative} categorize DFO methods along three main dimensions: 1) direct-search vs. model-based; 2) local vs. global; and 3) deterministic vs. randomized.
Direct-search methods progress by comparing function values to directly determine candidate points and include popular methods like the Nelder-Mead simplex method \citep{nelder1965simplex} and mesh adaptive search algorithms (NOMAD) \citep{Le2011a}. In contrast, model-based methods (discussed below) rely on an approximate model, also known as a surrogate or response surface, whose predictions guide the selection of candidate points. 
The ``local vs. global'' categorization distinguishes DFO methods that seek convergence to local optima from ``global'' ones that involve some degree of exploration. Unlike in deterministic global optimization, the qualifier ``global'' here does not typically mean that such a method is able to provably optimize a black box function.
Finally, the ``deterministic vs. randomized'' categorization differentiates methods that do not possess any probabilistic components with those that do.
Not surprisingly, there are also many hybrid methods attempting to combine the best attributes of the aforementioned methods.

Since we incorporate some tabu search concepts into our enhanced LineWalker algorithm, we note that \cite[p.6]{conn2009introduction} caution that simulated annealing, evolutionary algorithms, artificial neural networks, tabu search, and population-based methods should only be used for DFO in ``extreme cases'' and as a ``last resort.''  This is due to empirical evidence that such general-purpose heuristics typically require many function evaluations and provide no convergence guarantees.  There is also a body of work on surrogate-assisted heuristics \citep{ong2005surrogate}.  However, we would not describe our approach as a ``surrogate-assisted tabu search'' since tabu search components play a subservient role in LineWalker.

\subsubsection{Surrogate Modeling}

As described in \cite{bhosekar2018advances}, surrogate models play a critical role in three common problem classes: (1) prediction and modeling; (2) derivative-free optimization; and (3) feasibility analysis, where one must also satisfy design constraints. They also point out that key differences emerge when using surrogates for each of these three problem classes. In this work, we are primarily focused on the first two.

Popular surrogate models include Gaussian process regression in Bayesian optimization \citep{brochu2010tutorial,shahriari2015taking}, radial basis functions \citep{gutmann2001radial,muller2016miso,costa2018rbfopt}, and a mixture of basis functions \cite{cozad2014learning}.  Basis function-guided approaches share a common thread: They presuppose a set of basis functions, which transform the input data (the $\v{x}$ values) by operating on the raw feature space, and then create a surrogate by determining the weights to assign to each basis function.  Rather than try to map input data into a potentially higher-dimensional feature space, our approach emphasizes the objective function values and attempts to constrain how much these values are allowed to vary. 

Surrogate-based methods for DFO generally follow the same steps. First, an initial set of function evaluations (samples) are made. A surrogate model is then constructed and an ``acquisition'' function is used to select the next sample. After the new function value has been obtained, the surrogate model is updated and a new sample is chosen. This process repeats until some termination criteria are met, e.g., a maximum number of function evaluations has been reached. The acquisition function governs the tradeoff between exploration and exploitation.

\subsection{Contributions}

The contributions of this paper are:
\begin{enumerate}
\item With a particular focus on moderate to complicated noise-free smooth functions, we introduce a simple, but effective sampling method for optimizing and learning a discrete surrogate of a multi-dimensional function along a one-dimensional line segment of interest. 
\item We provide theoretical underpinnings that connect our approach to constrained nonlinear fitting and Gaussian Process Regression. 
\item Numerous examples are shown to illustrate the algorithm's efficacy and superiority relative to state-of-the-art methods, including Bayesian optimization and NOMAD, in terms of the number of function evaluations needed for optimality and overall surrogate quality.
\end{enumerate}

It is worth mentioning what is not considered in this paper. 
First, we deliberately avoid discussion of how to find a direction in which to search as it is a research topic in and of itself.
Second, we do not consider noisy (i.e., stochastic) function evaluations.
Third, we assume that estimating partial derivatives by finite differences or automatic differentiation is impractical or impossible, consistent with our assumption that a computationally expensive simulator is the main bottleneck.

The remainder of this paper is organized as follows:
Section~\ref{sec:LineWalker_algorithms} first describes a na\"{i}ve, but surprisingly effective extrema hunting algorithm, which relies solely on exploitation.  This algorithm lays the foundation for our main LineWalker algorithms, which incorporate various exploration steps and tabu search constructs to improve overall performance.  
Section~\ref{sec:theory} outlines the theoretical underpinnings of our approach as well as connections with Bayesian optimization.  
Section~\ref{sec:Numerical_experiments} describes our numerical experiments and showcases the performance of our LineWalker algorithms against state-of-the-art methods.
Conclusions and future research directions are offered in Section~\ref{sec:conclusions}.
The Appendix provides a detailed visual comparison of our \texttt{LineWalker-full} algorithm with its closest competitor - Bayesian optimization.

\section{Main results: LineWalker algorithms} \label{sec:LineWalker_algorithms}

\subsection{A line search algorithm for learning extrema of a function}
We now describe a sampling algorithm to learn the extrema of (and consequently optimize) a multi-dimensional continuous (ideally, smooth) function $\ftrue : \Re^D \mapsto \Re$ along a single dimension.
Note that this dimension does need to align with the axes of the original function.  In other words, given any two points $\v{x}_1$ and $\v{x}_2$ in $\Re^D$, the algorithm attempts to approximate the true function on the line segment that connects them.
First, we construct a grid of $N$ equally-spaced ``grid points'' along the single dimension of interest.
This grid of indices is denoted by the set $\mc{I}$.
Suppose we have function evaluations $\ftrue_i = \ftrue(\v{x}_i)$ at a subset of grid points $i \in \sampled$ and let $s_i$ be a binary parameter taking value 1 if $i \in \sampled$; 0 otherwise.     
Second, using all samples (function evaluations) obtained thus far, we construct a function approximation $\hat{\v{f}} = (\fhat_1,\dots,\fhat_N)$ by solving the following unconstrained least-squares optimization problem
\begin{equation}\label{model:unconstrained_opt_for_approximate_f}
\min_{\hat{\v{f}}}~~ \sum_{i=1}^N s_i (\fhat_i - \ftrue_i)^2 + \alpha \sum_{i=1}^{N-1} (\fhat_{i+1} - \fhat_i)^2 + \mu \sum_{i=2}^{N-1} (\fhat_{i+1} + \fhat_{i-1} - 2 \fhat_i)^2~.
\end{equation}
The first summation denotes the error in the function approximation $\fhat_i$ and the true function $\ftrue_i$ at the grid points $i \in \sampled$ where function evaluations have been made.  The second and third summations denote the squared first and second derivatives, respectively, of the function approximation $\hat{\v{f}}$. Thus, $\alpha$ and $\mu$ can be viewed as weights, smoothing parameters, or regularizers to encourage the minimization to choose function approximations that do not vary widely.
Third, given the function approximation $\hat{\v{f}}$, we identify new grid points to sample by detecting the extrema (i.e., the local maxima and minima, although saddle points could also be considered) of $\hat{\v{f}}$.

Using a standard calculus derivation for least-squares minimization, one can show that an optimal approximation $\vfhat^* \in \Re^N$ occurs by solving the linear system 
\begin{equation} \label{eq:least_squares_linear_system}
(\v{A}+\diag(\v{s}))\vfhat^*=\diag(\v{s})\vftrue
\end{equation}
where $\diag(\v{s}) \in \{0,1\}^{N \times N}$ is a sparse binary diagonal matrix whose positive diagonal entries correspond to the grid points $i \in \sampled$ where the true function has been evaluated; $\vftrue = (\ftrue_1,\dots,\ftrue_N)$; and $\v{A} \in \Re^{N \times N}$ is a sparse, symmetric, pentadiagonal, positive semidefinite matrix (see Theorem~\ref{thm:surrogate_matrix_is_positive_semidefinite}) given by
\begin{equation} \label{eq:smoothing_matrix_A}
\v{A} = 
\begin{pmatrix*}[c]
\mu	&	-\alpha-2\mu	&	\mu	&	0	&	\dots	&		&	 &	 &	0	\\
-\alpha-2\mu	&	2\alpha+5\mu	&	-\alpha-4\mu	&	\mu	&	0 	&	\dots	&	 &	&		\\
\mu	&	-\alpha-4\mu	&	2\alpha+6\mu	&	-\alpha-4\mu	&	\mu	&	0	&	\dots &	&		\\
0	&	\mu	&	-\alpha-4\mu	&	2\alpha+6\mu	&	-\alpha-4\mu	&	\mu	&	0&	&		\\
\vdots	&	\ddots	&	\ddots	&	\ddots	&	\ddots	&	\ddots	&	\ddots&	\ddots&	\vdots	\\
	&	&  0 & \mu	&	-\alpha-4\mu	&	2\alpha+6\mu	&	-\alpha-4\mu	&	\mu	&	0	\\
	&   &	&  0 & \mu	&	-\alpha-4\mu	&	2\alpha+6\mu	&	-\alpha-4\mu	&	\mu		\\
	&		&		&	& 0 &	\mu	&	-\alpha-2\mu	&	2\alpha+5\mu	&	-\alpha-4\mu		\\
0	&		&		&		&	\dots	&	0		&	\mu	&	-\alpha-2\mu	&	\mu
\end{pmatrix*}
\end{equation}

The algorithm is outlined in pseudocode in Algorithm~\ref{algo:segment_search_1D}.
In Step~\ref{step:initialize_samples}, an initial set of samples (indices) is selected where the function should be evaluated. 
As stated above and conveyed in the \textbf{while} loop beginning in Step~\ref{step:main_while_loop}, the algorithm continues to sample strict extrema of the approximate function $\vfhat$ until the approximation does not change.  In Step~\ref{step:solve_linear_system}, a linear system of equations, i.e., Equation~\eqref{eq:least_squares_linear_system}, is solved to obtain a least-squares fit relative to the samples obtained thus far. In Steps~\ref{step:maxima_hunt} and \ref{step:minima_hunt}, strict extrema of the approximate function $\vfhat$ are identified.  In Step~\ref{step:error_between_successive_fits}, the error between successive fits is computed to determine if the fit has materially changed. The algorithm terminates once the error between successive fits has fallen below the user-defined tolerance $e^{\min}$ or no new unsampled extrema are found (Step~\ref{step:no_new_extrema_found}).
\begin{algorithm} 
\caption{\texttt{extremaHunter()}: One-dimensional approximate global line segment search}
\label{algo:segment_search_1D}
\begin{algorithmic}[1]
\REQUIRE Optimality tolerance $e^{\min} > 0$; Number of grid points $N$; Smoothing matrix $\v{A}$; End points $\v{x}_1$ and $\v{x}_N$ 
\STATE Define the set $\{\v{x}_i\}_{i \in \mc{I}}$ of points on the line segment connecting $\v{x}_1$ and $\v{x}_N$
\STATE Set $s_i = 0~\forall i \in \mc{I}$; $\vftrue = \v{0}$; $\hat{\v{f}}^{\textrm{prev}} = \v{0}$; $e = e^{\min} + 1$
\STATE Evaluate the true function at an initial set $\sampled$ of grid points; Set $s_i = 1~\forall i \in \sampled$ \label{step:initialize_samples}
\WHILE{$(e > e^{\min})$} \label{step:main_while_loop}
	\STATE $\vfhat = (\v{A}+\diag(\v{s}))\backslash \diag(\v{s})\vftrue$, i.e., $\vfhat$ solves the linear system $(\v{A}+\diag(\v{s}))\vfhat=\diag(\v{s})\vftrue$ \label{step:solve_linear_system}
	\STATE $\mc{S}^{\max} = \{ i \in \{2,\dots,N-1\} : \fhat_i > \max\{\fhat_{i-1},\fhat_{i+1}\} + \delta^{\max}\}$ \label{step:maxima_hunt} 
	\STATE $\mc{S}^{\min} =\{ i \in \{2,\dots,N-1\} : \fhat_i < \min\{\fhat_{i-1},\fhat_{i+1}\} - \delta^{\min}\}$ \label{step:minima_hunt}
	\STATE $\newPoints= (\mc{S}^{\max} \cup \mc{S}^{\min}) \backslash \sampled$; $s_i = 1~\forall i \in \newPoints$
	\STATE \textbf{if} $|\newPoints|=\emptyset$ \textbf{then return} $\hat{\v{f}}$ \textbf{end if} \label{step:no_new_extrema_found}
	\FOR{$i \in \newPoints$}
		\STATE $\ftrue_i = \ftrue(\v{x}_i)$ \label{step:evaluate_function}
	\ENDFOR
	\STATE $\sampled = \sampled \cup \newPoints$
	\STATE $e = N^{-1} \sum_{i=1}^N |\hat{f}^{\textrm{prev}}_i - \fhat_i|$ \label{step:error_between_successive_fits}
	\STATE $\hat{\v{f}}^{\textrm{prev}} = \vfhat$
\ENDWHILE
\RETURN $\hat{\v{f}}$
\end{algorithmic}
\end{algorithm}

It is important to note that there are only two potentially time-consuming steps in the entire algorithm.
First, Step~\ref{step:solve_linear_system} requires the solution of a linear system of equations, whose computational complexity $O(N^3)$ depends on the number of grid points $N$ used in the discretization. 
Second, Step~\ref{step:evaluate_function} requires the true function $\ftrue$ to be evaluated, which may require a call to a computationally-expensive simulation or oracle.  

The algorithm is easily explained by way of example as shown in the following subsection.   
An \textit{iteration} refers to a single pass through all steps in the \textbf{while} loop beginning in Step~\ref{step:main_while_loop}.

\subsection{Illustrative example of \texttt{extremaHunter()}}

The $D$-dimension Rastrigin function \url{https://en.wikipedia.org/wiki/Rastrigin_function} is 
\begin{equation*}
f(\v{x} )=10D + \sum_{i=1}^{D}\left[x_{i}^{2}-10\cos(2\pi x_{i})\right].
\end{equation*}
It is typically defined on the domain $x_i \in [-5.12,5.12]$ for $i=1,\dots,D$, and has a global minimizer at $\v{x}^*=\v{0}$ with a function value of $f(\v{0})=0$. 
Figure~\ref{fig:rastrigin_demo_iter1} shows the ground truth of the 1-dimensional Rastrigin function on the interval $[-3,3]$ as well as the resulting function approximation obtained from sampling 11 initial uniformly-spaced points, including the endpoints, on a grid of size $N=1000$. Smoothing parameters are set to $\alpha=0$ and $\mu=0.01$.
Although the algorithm was ``lucky'' to compute an initial sample point near the true global minimum of zero, we are not privy to this fact.  Moreover, the ``Current Fit'' is a rather poor approximation of the ``Ground Truth'' function as it misses several local minima and maxima.

\begin{figure}[h!] 
\centering
\includegraphics[width=9cm]{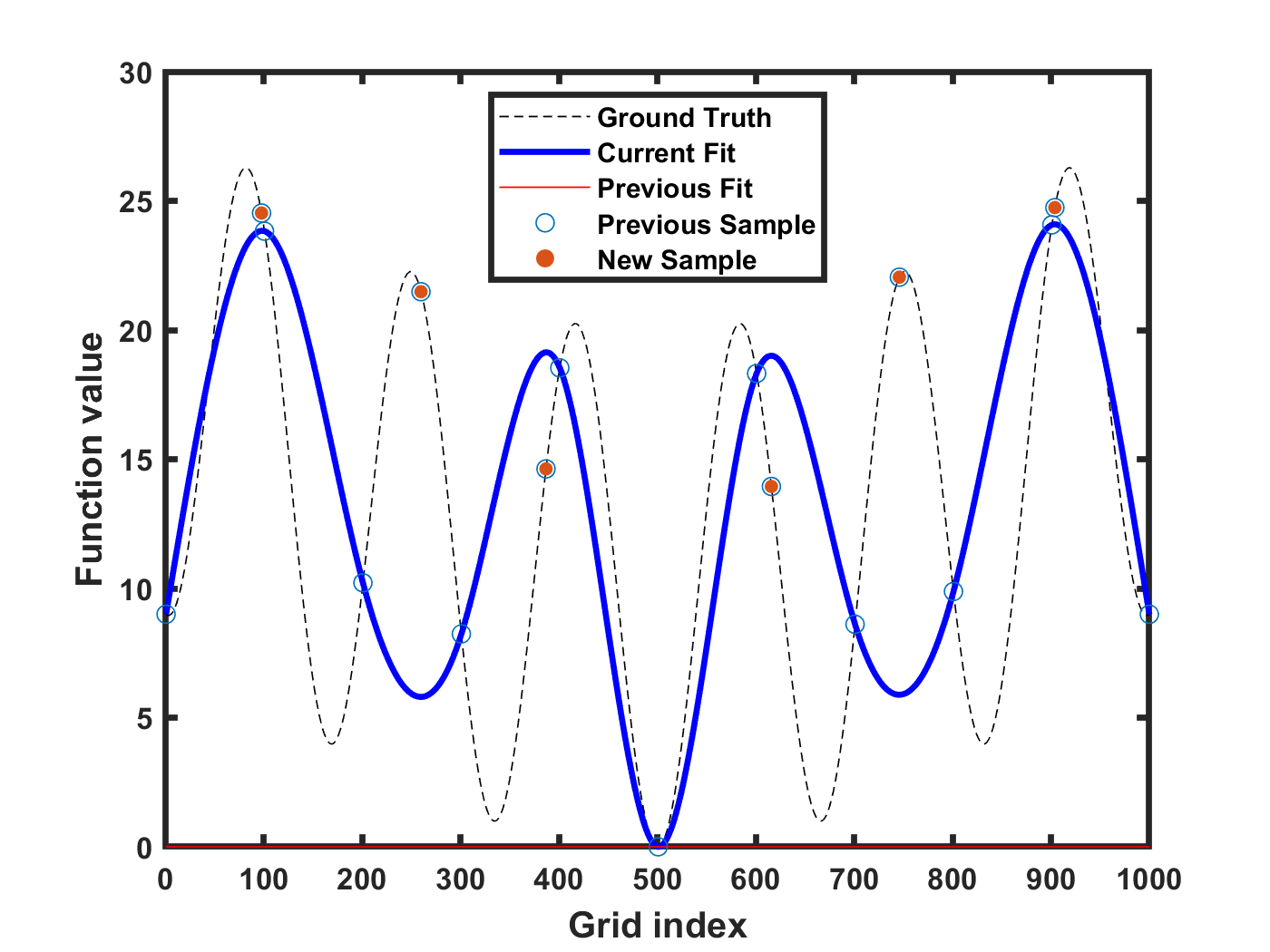}
\caption{Iteration 1 results for the Rastrigin function.  The ``Ground Truth'' function that we are trying to learn/optimize is shown with a dashed line.  An initial function approximation ``Current Fit'' is made using only 11 initial samples (function evaluations) labeled $\circ$.  The ``Previous Fit'' is initialized to the zero vector. 
The algorithm recommends ``New Sample'' function evaluations at critical points, shown with a $\filledcirc$, of the ``Current Fit.'' Parameter settings: $N=1000$, $e^{\min} = 0.001$, $\alpha=0$, $\mu=0.01$.}
\label{fig:rastrigin_demo_iter1}
\end{figure} 

The function approximations from iterations 2, 4, and 6 of Algorithm~\ref{algo:segment_search_1D} are shown in Figure~\ref{fig:rastrigin_iterations}. In six iterations of the main while loop in Step~\ref{step:main_while_loop}, the algorithm made a total of 52 function evaluations, i.e., only 5.2\% of the $N=1000$ grid points were sampled.  The algorithm terminated with a minimizer at $x^{\min} = 0.003$ with a function value of $f^{\min}=0.0018$, slightly off from the true minimum function value of 0. Moreover, the algorithm identified all local extrema (with small error) as it is designed to do. 
\begin{figure} [h!]	
  \begin{subfigure}[b]{0.333\textwidth}
    \includegraphics[width=\textwidth]{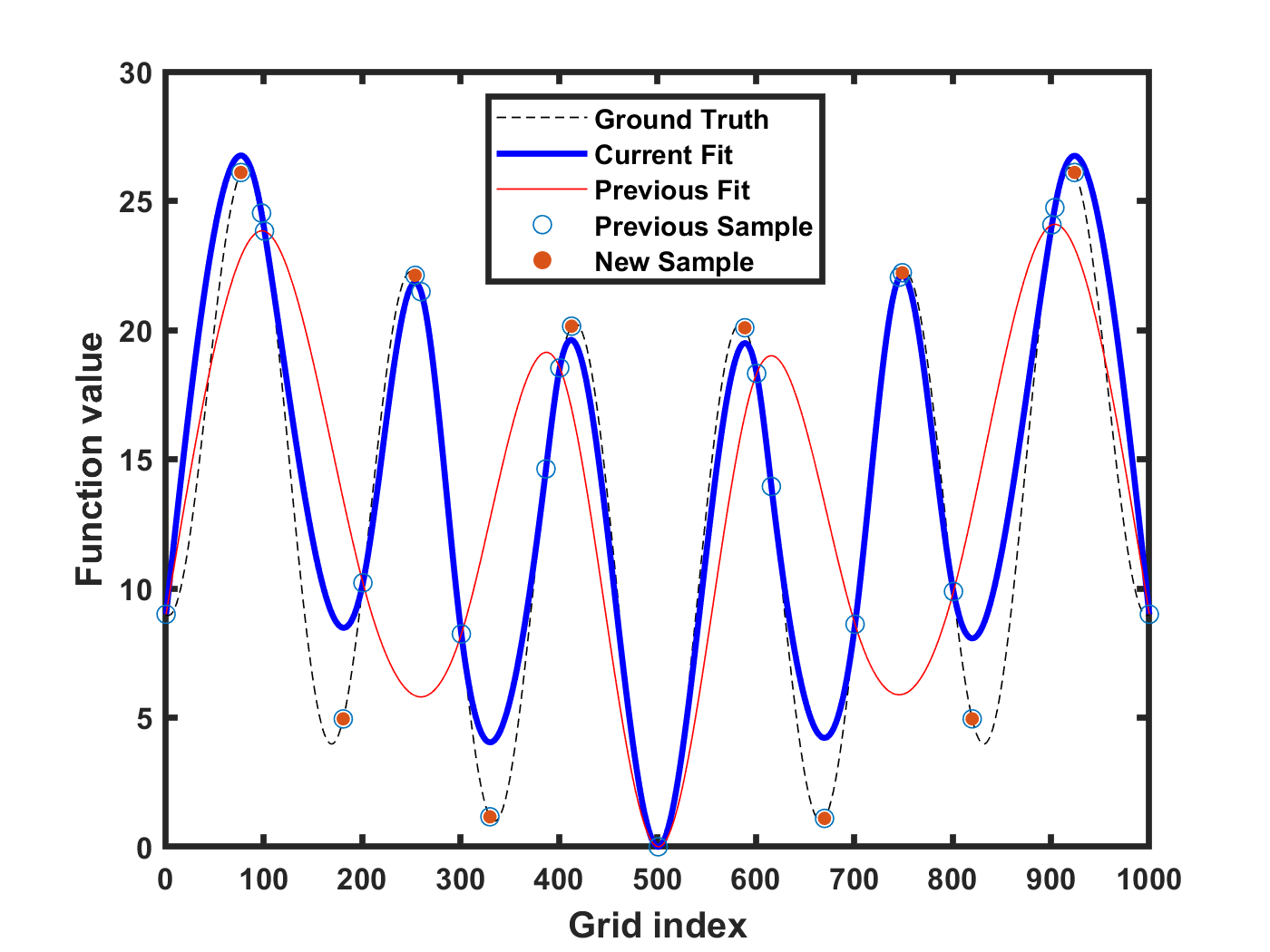}
    \caption{Iteration 2}
    \label{fig:rastrigin_demo_iter2}
  \end{subfigure}
    \begin{subfigure}[b]{0.333\textwidth}
    \includegraphics[width=\textwidth]{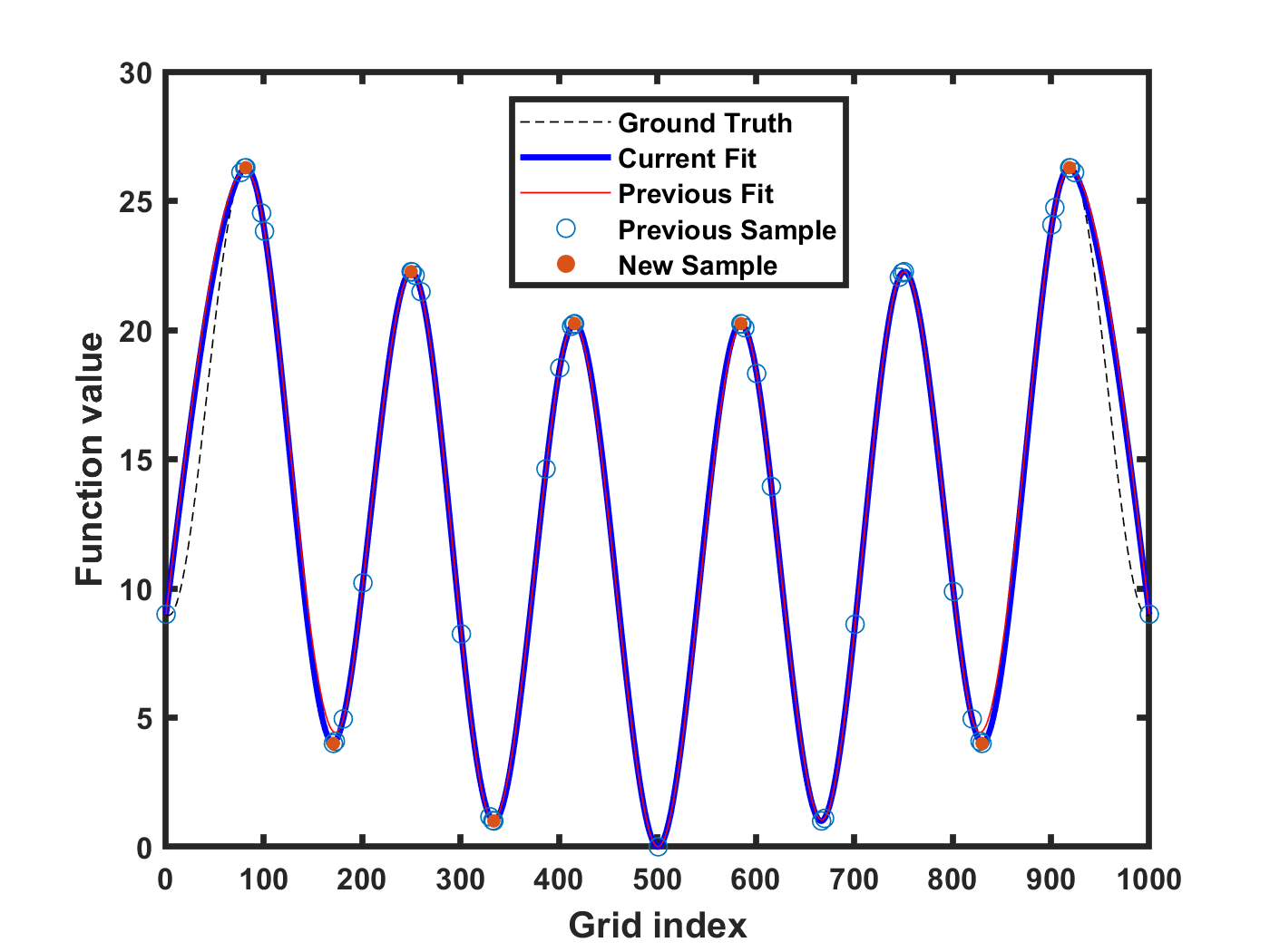}
    \caption{Iteration 4}
    \label{fig:rastrigin_demo_iter4}
  \end{subfigure}
      \begin{subfigure}[b]{0.333\textwidth}
      \includegraphics[width=\textwidth]{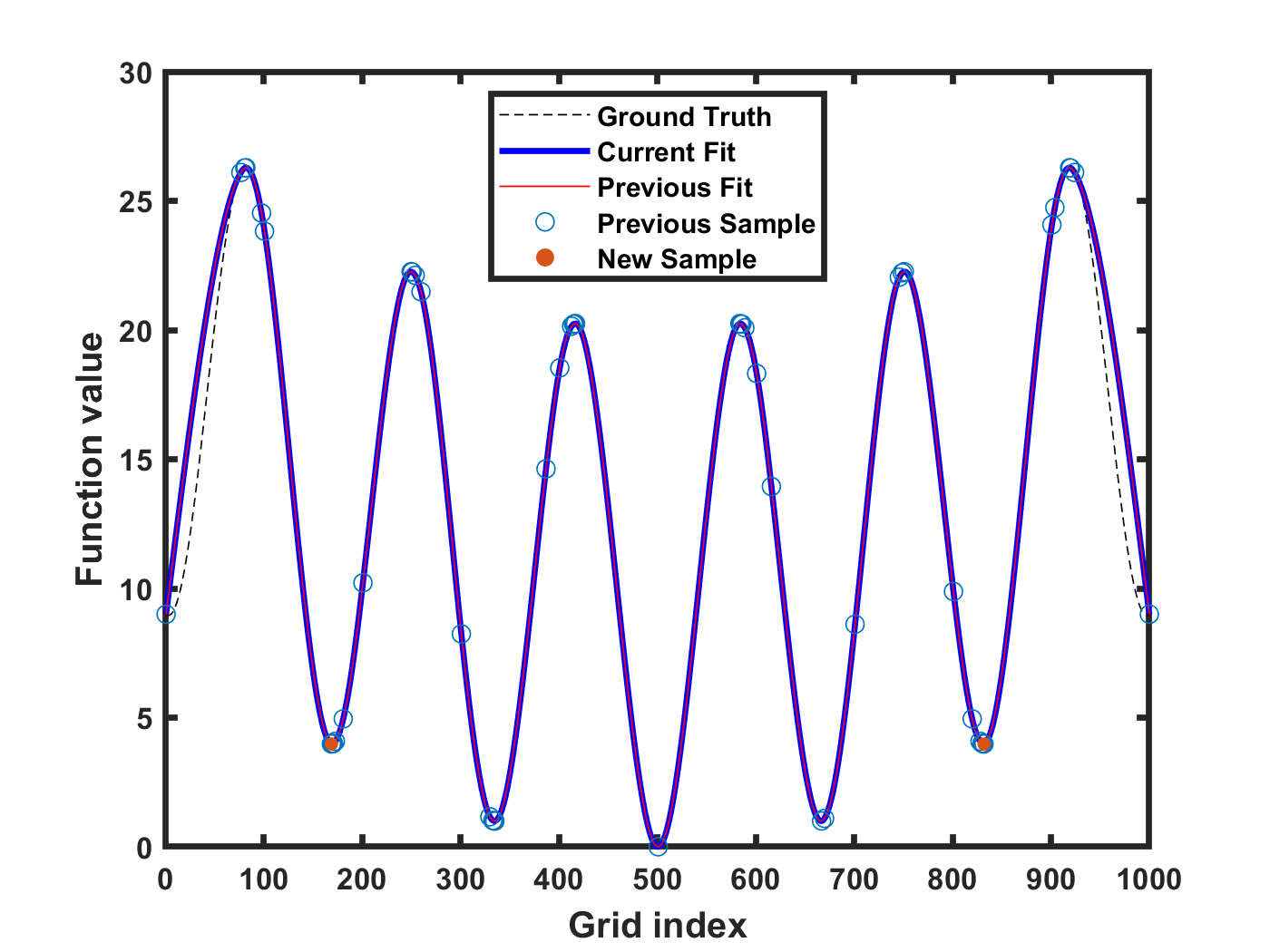}
    \caption{Iteration 6}
      \label{fig:rastrigin_demo_iter6}
    \end{subfigure}
  \caption{By iteration 4, the sampling algorithm has achieved a relatively close approximation of the 1-dimensional Rastrigin function.  By iteration 6,  the algorithm identifies the point $x^{\min} = 0.003$ with a function value of $f^{\min}=0.0018$ as a global minimizer.}
  \label{fig:rastrigin_iterations}
\end{figure}

This example shows that the \texttt{extremaHunter()} algorithm is capable of finding a near global minimum along with all other extrema, while producing a high-quality surrogate.  At the same time, several potential shortcomings emerge.  First, as far as the ``exploration vs. exploitation'' tradeoff is concerned, Algorithm~\ref{algo:segment_search_1D} only exploits.  Specifically, it exploits all extrema of the current fit and otherwise takes no exploratory samples, which could hinder the approximation quality from improving. Second, it may indiscriminately re-sample very close to an existing sample. While at times this may be a wise strategy (e.g., when an existing sample is near a true global minimum), it may also lead to inefficient sampling. Third, it samples all newly-identified local extrema in each iteration. Many state-of-the-art algorithms take one sample per iteration because function evaluations are computationally expensive.

\subsection{Enhanced LineWalker algorithms} \label{sec:enhanced_linewalker_algorithms}

Algorithm~\ref{algo:segment_search_1D} was conceived to terminate as soon as the fit ceases to materially change or no new unsampled extrema are identified.
It is also possible, and perhaps more common, for a user to prefer a termination criterion based on a limited budget of function evaluations. Algorithm~\ref{algo:segment_search_1D_budgetLimited} sketches such a variant and highlights all new steps in a different font color. It also attempts
to overcome the shortcomings of the basic approach described above, while preserving Algorithm~\ref{algo:segment_search_1D}'s extrema hunting nature. Our enhancements come in three flavors: (1) Tabu search structures forbidding neighborhoods of sampled points from being visited too frequently. (2) A simple exploration (a.k.a. diversification) strategy to sample in sparsely-sampled regions when no non-tabu peaks and valleys are available. (3) Sampling near, but not directly at, an extremum of the current approximation.
Each of these three components is described in the subsections below.
Component (2) serves as a mechanism for global exploration, while components (1) and (3) inject some local exploration to balance Algorithm~\ref{algo:segment_search_1D}'s purely exploitative nature.

\begin{algorithm} [h]
\caption{ \texttt{LineWalker-full()}: Budget-limited one-dimensional approximate global line segment search}
\label{algo:segment_search_1D_budgetLimited}
\begin{algorithmic}[1]
\REQUIRE \newcodecolor{Maximum number of function evaluations $E^{\max,\textrm{total}} > 0$; 
Maximum number of function evaluations per major iteration $E^{\max,\textrm{itr}} > 0$; }
Number of grid points $N$; Smoothing matrix $\v{A}$; End points $\v{x}_1$ and $\v{x}_N$ 
\STATE Define the set $\{\v{x}_i\}_{i \in \mc{I}}$ of points on the line segment connecting $\v{x}_1$ and $\v{x}_N$
\STATE Set $s_i = 0~\forall i \in \mc{I}$; $\vftrue = \v{0}$; 
\STATE Evaluate the true function at an initial set $\sampled$ of grid points; Set $s_i = 1~\forall i \in \sampled$; \label{step:initialize_samples_budgetLimited} 
\STATE \newcodecolor{Set $\texttt{itr}=0$; $\iterFound_i = 0~\forall i \in \sampled$} \label{step:initialize_iter_found}
\WHILE{\newcodecolor{$(|\sampled| < E^{\max,\textrm{total}})$}} \label{step:main_while_loop_budgetLimited}
	\STATE \newcodecolor{$\texttt{itr}=\texttt{itr}+1$ } \label{step:increment_iteration_counter}
	\STATE $\vfhat = (\v{A}+\diag(\v{s}))\backslash \diag(\v{s})\vftrue$, i.e., $\vfhat$ solves the linear system $(\v{A}+\diag(\v{s}))\vfhat=\diag(\v{s})\vftrue$ \label{step:solve_linear_system_budgetLimited}
	\STATE $\mc{S}^{\max} = \{ i \in \{2,\dots,N-1\} : \fhat_i > \max\{\fhat_{i-1},\fhat_{i+1}\} + \delta^{\max}\}$ \label{step:maxima_hunt_budgetLimited} 
	\STATE $\mc{S}^{\min} =\{ i \in \{2,\dots,N-1\} : \fhat_i < \min\{\fhat_{i-1},\fhat_{i+1}\} - \delta^{\min}\}$ \label{step:minima_hunt_budgetLimited}
	\STATE $\newPoints= (\mc{S}^{\max} \cup \mc{S}^{\min}) \backslash \sampled$; $s_i = 1~\forall i \in \newPoints$
	\STATE \newcodecolor{\texttt{manageTabuStruct($\sampled$)} \#See Algorithm~\ref{algo:manage_tabu_struct} } \label{step:manage_tabu_struct}
	\STATE \newcodecolor{$\nontabuPoints =$ \texttt{findNonTabuPoints($\newPoints$)} \#See Algorithm~\ref{algo:find_nontabu_points} } \label{step:findNonTabuPoints}
	\STATE \newcodecolor{$\newPoints = \nontabuPoints$ } \label{step:set_newPoints_equal_to_nontabuPoints}
	\STATE \newcodecolor{\textbf{if}($\newPoints=\emptyset$) \textbf{then} $\newPoints=$\texttt{findLargestUnexploredInterval($\sampled$)} \textbf{end if} \#See Alg~\ref{algo:find_largest_unexplored_interval} } \label{step:findLargestUnexploredInterval}
	\STATE \newcodecolor{$\sortedList = \texttt{sort}(\newPoints)$ \quad \#Sort in ascending order according to $\fhat$ } \label{step:sort_points_in_ascending_order}
	\FOR{\newcodecolor{ $j=1:\min\{\texttt{length}(\sortedList),E^{\max,\textrm{itr}},E^{\max,\textrm{total}}-|\sampled| \}$}}
		\STATE \newcodecolor{$i=\sortedList[j]$ }
		\STATE \newcodecolor{\textbf{if}($i \in \nontabuPoints$) \textbf{then} $i=$\texttt{sampleAroundTheBend($i$)} \textbf{end if} \#See Algorithm~\ref{algo:sample_around_the_bend} } \label{step:sample_around_the_bend}
		\STATE \newcodecolor{$\iterFound_i = \texttt{itr}$ } \label{step:mark_iter_found}
		\STATE $\ftrue_i = \ftrue(\v{x}_i)$ \label{step:evaluate_function_budgetLimited}
		\STATE $\sampled = \sampled \cup \{i\}$ 
	\ENDFOR
\ENDWHILE
\RETURN $\hat{\v{f}}$
\end{algorithmic}
\end{algorithm}

Unlike Algorithm~\ref{algo:segment_search_1D} 
,
Algorithm~\ref{algo:segment_search_1D_budgetLimited} chooses at most $E^{\max,\textrm{itr}}$ per iteration and is ``budget-limited'' in that it collects $E^{\max,\textrm{total}}$ total samples (Step~\ref{step:main_while_loop_budgetLimited}). 
Specifically, Algorithm~\ref{algo:segment_search_1D_budgetLimited} judiciously selects new samples in Step~\ref{step:findNonTabuPoints} by finding all non-tabu points amongst the newly-identified local extrema. If no such points are available, it explores in Step~\ref{step:findLargestUnexploredInterval} by finding the largest unexplored interval.
Rather than immediately accept all candidate samples, in Step~\ref{step:sort_points_in_ascending_order}, it sorts the newly-identified extrema in ascending order according to their approximate value $\fhat_i$. After which, samples are taken in sorted order until the maximum number $E^{\max,\textrm{itr}}$ of function evaluations per iteration is reached or some other criterion is met. Note that one could employ a more sophisticated \texttt{sort} function that uses other available information. In Step~\ref{step:sample_around_the_bend}, we sample near, but not directly at, a non-tabu candidate extremum of the current approximation.

For completeness and ease of reference, we present in Algorithm~\ref{algo:LineWalker_pure} a one-line pseudocode of what we refer to as \texttt{LineWalker-pure()}. It is essentially a budget-limited version of the \texttt{extremaHunter()} method in Algorithm~\ref{algo:segment_search_1D} and excludes all tabu search-related components and the sample ``around the bend'' strategy. It includes the same exploration strategy as Algorithm~\ref{algo:segment_search_1D_budgetLimited} in the event that no new extrema are identified.  
This simplified algorithm serves as a useful reference point in our computational experiments.

\begin{algorithm} 
\caption{ \texttt{LineWalker-pure()}: \texttt{LineWalker-full()} excluding all tabu-related enhancements and the sample ``around the bend'' strategy}
\label{algo:LineWalker_pure}
\begin{algorithmic}[1]
\STATE Identical to Algorithm~\ref{algo:segment_search_1D_budgetLimited}, except it excludes Steps 
\ref{step:initialize_iter_found},
\ref{step:increment_iteration_counter},
\ref{step:manage_tabu_struct},
\ref{step:findNonTabuPoints},
\ref{step:set_newPoints_equal_to_nontabuPoints},
 \ref{step:sample_around_the_bend}, 
 \ref{step:mark_iter_found} 
\end{algorithmic}
\end{algorithm}

\subsubsection{Tabu search structures}

We incorporate a simple tabu search heuristic in which neighborhoods of sampled points (grid indices) are forbidden from being sampled for a certain number of major iterations (one pass of the main \textbf{while} loop in Algorithm~\ref{algo:segment_search_1D_budgetLimited}). Tabu search has been one of the most successful heuristics for finding high-quality solutions to a variety of nonconvex and combinatorial optimization problems over the past few decades and is masterfully presented in \cite{glover1998tabu}.  As described by the algorithm's inventor, ``tabu search is based on the premise that problem solving, in order to qualify as intelligent, must incorporate adaptive memory and responsive exploration'' \cite[p.4]{glover1998tabu}. Below we describe the essential tabu search ingredients that we borrow to ``incorporate adaptive memory and responsive exploration'' into our algorithm.  

\textbf{Tabu lists and neighborhoods}. We maintain two tabu lists: a short-term and long-term list of sampled grid indices. Whereas the short-term tabu list strives to prevent revisiting an interval around a recently sampled index $i \in \sampled$ 
(regardless of that sample's approximate value $\fhat_i$ or any other information),
the long-term tabu list aims to deter re-sampling ``too close'' to any existing sample, where ``closeness'' is defined dynamically and depends on several factors. These high-level concepts are described rigorously below.
  
We explicitly store the long-term list via the set $\sampled$. That is, all sampled indices are deemed ``long-term tabu'' since we have no reason to re-sample these points when the black box function is deterministic, which is our assumption.  In contrast, the short-term list is maintained implicitly.  Following common practice, 
we store the iteration number $\iterFound_i$ in which grid index $i$ was sampled.
$\iterFound_i = 0$ for each sample $i$ in the initial sample (see Step~\ref{step:initialize_iter_found}).
An index remains on the (implicit) short-term tabu list as long as $(\texttt{itr}-\iterFound_i) \leq \tabuTenureShort$ iterations, where \texttt{itr} is the current major iteration (see Step~\ref{step:increment_iteration_counter}) and $\tabuTenureShort$ is a nonnegative integer defining the dynamic short-term tabu tenure parameter.

Associated with each tabu index $i$ is a short- and long-term neighborhood ($\neighborhoodShort_i$ and $\neighborhoodLong_i$, respectively) of forbidden neighboring indices. 
Moreover, each neighborhood of tabu index $i$ is governed by a grid distance threshold, a nonnegative integer parameter denoted by $\tabuGridDistanceThresholdShort_i$ and $\tabuGridDistanceThresholdLong_i$, respectively.
The short-term neighborhood of $i$ is defined as $\neighborhoodShort_i = \left\{j \in \mc{I} \cap \{ i - \tabuGridDistanceThresholdShort_i, \dots, i + \tabuGridDistanceThresholdShort_i \} \right\}$. $\neighborhoodLong_i$ is defined similarly.
The two neighborhoods clearly intersect but can be related in several ways: $\neighborhoodShort_i \subset \neighborhoodLong_i$, $\neighborhoodLong_i \subset \neighborhoodShort_i$, or $\neighborhoodShort_i = \neighborhoodLong_i$ for a given $i \in \sampled$.

The short-term grid distance threshold $\tabuGridDistanceThresholdShort_i$ is held constant throughout the algorithm with  $\tabuGridDistanceThresholdShort_i = N/(2 E^{\max,\textrm{total}})$.  The logic behind the value $N/(2 E^{\max,\textrm{total}})$ is simple: If there are $N$ grid points and we allow a maximum of $E^{\max,\textrm{total}}$ function evaluations, then, assuming equidistant samples, samples would occur every $N/E^{\max,\textrm{total}}$ grid indices when the algorithm terminates.  For example, if $N=5000$ and $E^{\max,\textrm{total}}=50$, then equidistant samples would occur every $N/E^{\max,\textrm{total}}=100$ grid indices.  Dividing by 2, we obtain $N/(2 E^{\max,\textrm{total}})$ (or $5000/100 = 50$, in our example) grid indices, which we use as our short-term tabu grid distance threshold $\tabuGridDistanceThresholdShort_i$.  

Unlike the static management of $\tabuGridDistanceThresholdShort_i$, the long-term tabu grid distance threshold parameter is defined as $\tabuGridDistanceThresholdLong_i = \nu_i N/|\sampled|$ for all $i \in \sampled$ and changes dynamically as a function of (1) the distance between $\fhat_i$ and $\hat{F}^{\textrm{min}}=\min\{\fhat_j : j \in \mc{I} \}$ and $\hat{F}^{\textrm{max}}=\max\{\fhat_j : j \in \mc{I}\}$ (captured via the multiplier $\nu_i \in [\nu^{\min},\nu^{\max}] \subseteq [0,1]$) and (2) the number $|\sampled|$ of current samples. 
Being more complex than $\tabuGridDistanceThresholdShort_i$, the long-term threshold $\tabuGridDistanceThresholdLong_i$ definition deserves explanation. 
When there are relatively few samples, i.e., $|\sampled|$ is small and hence $N/|\sampled|$ is large, a larger long-term tabu neighborhood is desired to encourage more exploration of new extrema or unexplored intervals. When $|\sampled|$ is large, i.e., close to $E^{\max,\textrm{total}}$, then a relatively smaller long-term tabu neighborhood is desired. While the term $(N/|\sampled|)$ diminishes as more samples are collected, the multiplier $\nu_i$ aims to scale $(N/|\sampled|)$ even further using the following logic:    
If sample $i$'s approximate value $\fhat_i$ is near the current approximation $\vfhat$'s minimum $\hat{F}^{\textrm{min}}$ or maximum $\hat{F}^{\textrm{max}}$ objective function value, then we would like the multiplier $\nu_i$ to be small (i.e., close to $\nu^{\min}$) to permit re-sampling near, but not too close to $i$.  If $\fhat_i$ is near the midpoint of all $\fhat$ values, then we would like $\nu_i$ to be large (i.e., close to $\nu^{\max}$) to create a larger tabu neighborhood and avoid re-sampling near $i$. That is, we wish to avoid sampling too frequently at local extrema whose $\fhat$ value is close to the ``middle'' of the fit and instead prioritize sampling points near a global minimum or maximum of $\vfhat$. 

To implement this logic for computing $\tabuGridDistanceThresholdLong_i$, we set $\nu_i = \nu^{\min} + \kappa_i(\nu^{\max}-\nu^{\min})$, where
$\kappa_i = \min\{ \hat{F}^{\textrm{max}}-\fhat_i,\fhat_i-\hat{F}^{\textrm{min}} \}/\hat{F}^{\textrm{mid}}$, $\hat{F}^{\textrm{range}}=(\hat{F}^{\textrm{max}}-\hat{F}^{\textrm{min}})$, and 
$\hat{F}^{\textrm{mid}} = \tfrac{1}{2}\hat{F}^{\textrm{range}}$.
Note that if $\fhat_i=\hat{F}^{\textrm{min}}+\hat{F}^{\textrm{mid}}$, then $\kappa_i = 1$. 
If $\fhat_i=\hat{F}^{\textrm{min}}$ or $\fhat_i=\hat{F}^{\textrm{max}}$, then $\kappa_i = 0$. 
Since $N/|\sampled| \geq N/E^{\max,\textrm{total}} > N/(2 E^{\max,\textrm{total}}) = \tabuGridDistanceThresholdShort_i ~\forall i \in \sampled$, it should be clear that if $\nu_i =1$, then $\tabuGridDistanceThresholdLong_i = \nu_i N/|\sampled| > \tabuGridDistanceThresholdShort_i$ for any $i \in \sampled$, rendering $\tabuGridDistanceThresholdShort_i$ redundant. Instead, we set $\nu^{\min} = 0.10$ and $\nu^{\max} = 0.25$ so that once we have collected $\tfrac{1}{2}E^{\max,\textrm{total}}$ samples (i.e., half of our total sample budget), $\tabuGridDistanceThresholdLong_i \leq \tabuGridDistanceThresholdShort_i$ for all $i \in \sampled$. 

In traditional tabu search fashion, each index $i$ and its neighbors should not be revisited unless some other criteria are satisfied: the tabu tenure has been reached or aspiration criteria are met.  Both are discussed below.

\textbf{Tabu tenure}.  While the short-term tabu grid distance threshold $\tabuGridDistanceThresholdShort_i$ (and hence $\neighborhoodShort_i$) for each $i \in \sampled$ is held constant throughout the algorithm, the short-term tabu tenure $\tabuTenureShort$ is dynamic and is governed by the number of non-boundary local extrema in the current fit.  The rationale behind using this metric is that if the number of local extrema is larger than the current short-term tabu tenure, the algorithm may be inclined to revisit an existing sample's neighborhood before it has explored another extremum.  

\begin{algorithm} 
\caption{ \texttt{manageTabuStruct($\sampled$)} } 
\label{algo:manage_tabu_struct}
\begin{algorithmic}[1]
\REQUIRE Current samples $\sampled$
\STATE $\eta=$ number of non-boundary extrema in the current fit 
\IF{$\eta > \tabuTenureShort$}
	\STATE $\tabuTenureShort=\tabuTenureShort+1$ 
\ELSIF{$\eta < \tabuTenureShort -1$}
	\STATE $\tabuTenureShort=\tabuTenureShort-1$
\ENDIF 
\STATE $\hat{F}^{\textrm{range}}=(\hat{F}^{\textrm{max}}-\hat{F}^{\textrm{min}})$; $\hat{F}^{\textrm{mid}} = \tfrac{1}{2}\hat{F}^{\textrm{range}}$ 
\STATE $\kappa_i = \min\{ \hat{F}^{\textrm{max}}-\fhat_i,\fhat_i-\hat{F}^{\textrm{min}} \}/\hat{F}^{\textrm{mid}}$ 
\STATE $\nu_i = \nu^{\min} + \kappa_i(\nu^{\max}-\nu^{\min})$
\STATE $\tabuGridDistanceThresholdLong_i = \nu_i (N/|\sampled|)$
\end{algorithmic}
\end{algorithm}

Algorithm~\ref{algo:manage_tabu_struct} describes how we dynamically update our tabu structures.
Let $\eta$ be the number of non-boundary (non-endpoint) local extrema in the current fit.  Intuitively, if there are few local extrema ($\eta$ is small), then we may wish to revisit previously sampled peaks and valleys more frequently. In contrast, if there are many local extrema, then we may wish to retain a longer tabu tenure so that each extremum is explored. Using this rationale, Algorithm~\ref{algo:manage_tabu_struct} does the following: If $\eta$ is greater than or equal to the current short-term tabu tenure $\tabuTenureShort$, then we increment $\tabuTenureShort$ by one. If $\eta$ is less than $\tabuTenureShort-1$ and $\tabuTenureShort>1$, then we decrement $\tabuTenureShort$ by one.  Else, we leave $\tabuTenureShort$ as is.
Meanwhile, all sample points $i \in \sampled$ are deemed long-term tabu. Hence, we do not explicitly keep a long-term tabu tenure parameter because it is always infinite. At the same time, we dynamically adjust the long-term tabu neighborhood size.

\textbf{Aspiration criteria}. ``Aspiration criteria are introduced in tabu search to determine when tabu activation rules can be overridden, thus removing a tabu classification otherwise applied to a move'' \cite[p.26]{glover1998tabu}.  We introduce two aspiration criteria.  Aspiration criterion 1 is a standard approach and overrides both short- and long-term tabu tenure values. It simply requires a candidate solution to be a potential minimizer and have at most $N^{\max,\textrm{nbrs}}$ sampled neighbors around it. When 30 or fewer samples have been collected, we deem a candidate a potential minimizer if its objective function value is within 1\% of the best known objective function value and $N^{\max,\textrm{nbrs}}=1$. When more than 30 samples have been collected, we deem a candidate a potential minimizer if its objective function value is within 10\% of the best known objective function value and $N^{\max,\textrm{nbrs}}=2$. Thus, with fewer than 30 samples, we are much more selective about our samples and only sample at most two points in a very small interval around an approximate minimum (compared to \texttt{bayesopt}, which may sample many points in the neighborhood of a minimum).  With greater than 30 samples, we allow for at most three samples near a minimum.  The reason for increasing the potential minimizer criterion from 1\% to 10\% is because certain functions (e.g., dejong, SawtoothD, and Easom-Schaffer2A) have multiple valleys with similar objective function values that may need to be revisited.  This helps our LineWalker Algorithm~\ref{algo:segment_search_1D_budgetLimited} balance exploitation and exploration.

Aspiration criterion 2 only overrides a short-term tabu tenure and only does so if the objective value improved by a minimum amount in the previous iteration. Specifically, if a candidate solution is ``near, but not too close to,'' a newly-found minimum (i.e., one discovered in the previous iteration), and, in the previous iteration $\texttt{itr}-1$, the minimum objective function value decreased by at least 1\% of $\hat{F}^{\textrm{range}}$ (the range of the entire fit in the \textit{current} iteration $\texttt{itr}$), then we may wish to visit this candidate solution as it may suggest that a new valley has been discovered and is worthy of immediate investigation.  Here ``near, but not too close to'' means that, given the newly-found minimum $\hat{\v{x}}_i$ and the candidate solution $\tilde{\v{x}}_i$, the following conditions are met: (``near'') $\hat{\v{x}}_i$ is the nearest evaluated sample to the right or left of $\tilde{\v{x}}_i$, and (``but not too close to'') $\tilde{\v{x}}_i \notin \neighborhoodLong(\hat{\v{x}}_i)$.  Without this aspiration criterion, the algorithm may find a new minimizer in a promising valley, but then not revisit its neighborhood for many iterations because of a high short-term tabu tenure. This aspiration proved to be helpful on a number of benchmark instances with very steep drops (e.g., our dis-continuous Grimacy \& Lee variant, de Jong, Michal, and Easom-Schaffer2A).

\begin{algorithm} 
\caption{ \texttt{findNonTabuPoints($\newPoints$)} }
\label{algo:find_nontabu_points}
\begin{algorithmic}[1]
\REQUIRE Current samples $\sampled$; Candidate new samples $\newPoints$
\STATE $\tabuPointsCurrent = \{i \in \sampled : \iterFound_i \leq \tabuTenureShort \}$ \label{step:get_current_short_term_tabu_points}
\STATE $\tabuPointsNew = \{ i \in \newPoints : (\exists j \in \tabuPointsCurrent : i \in \neighborhoodShort_j) \cup (\exists j \in \sampled : i \in \neighborhoodLong_j) \}$ \label{step:get_new_tabu_points}
\STATE $\aspirationPoints = \{ i \in \tabuPointsNew : \textrm{aspiration criteria are satisfied} \}$ \#See main text \label{step:get_aspiration_points} 
\STATE $\nontabuPoints = (\newPoints \setminus \tabuPointsNew) \cup \aspirationPoints$
\RETURN $\nontabuPoints$
\end{algorithmic}
\end{algorithm}

Given the set $\newPoints$ of newly-identified local extrema and the set $\sampled$ of current samples, Algorithm~\ref{algo:find_nontabu_points} outlines how non-tabu points are identified.
After identifying the set $\tabuPointsCurrent$ of all current short-term tabu samples in Step~\ref{step:get_current_short_term_tabu_points}, we determine which of the new points are deemed tabu (according to the updated neighborhood definitions) and place them in the set $\tabuPointsNew$ in Step~\ref{step:get_new_tabu_points}. We then check which of these latter points satisfy the aforementioned aspiration criteria in Step~\ref{step:get_aspiration_points} before returning the set $\nontabuPoints$.

\subsubsection{Exploration/Diversification} 

Algorithm~\ref{algo:segment_search_1D} has no exploration components; it only pursues exploitation of newly-identified extrema. 
There are many potential options for exploration.  We adopt a very simple diversification mechanism in which, if there are no non-tabu candidate points to sample, we find the largest interval, i.e., the one with the largest number of unexplored grid points, and select the grid index that bisects it. This point becomes the next sampled point.  If there are multiple intervals with the same number of unexplored grid points, then we break ties by finding the one possessing the smallest objective function value according to our current approximation $\vfhat$. This approach does not use any information (e.g., slope, curvature, proximity to other extrema) of the current approximation, only the number of unexplored grid points and function values at existing sample points. 
Hence, our exploration step is different from what is done in Bayesian optimization where the acquisition function attempts to use prior knowledge of the blackbox function to estimate the uncertainty at any given unexplored point.  We do not assume that such prior information is available.
Intuitive and straightforward, the subroutine is presented in Algorithm~\ref{algo:find_largest_unexplored_interval} for completeness. Note that the \texttt{First} function in Step~\ref{step:tie_breaker_first} selects the first element in a set and is used as the final tie-breaker if there is more than one unexplored interval with the same minimum approximate objective function value.     


\begin{algorithm} 
\caption{ \texttt{findLargestUnexploredInterval($\sampled$)} } 
\label{algo:find_largest_unexplored_interval}
\begin{algorithmic}[1]
\REQUIRE Current samples $\sampled$
\STATE $R_i = \min\{ j \in \sampled: j > i \} \quad \forall i \in \sampled \setminus \{N\}$ 
\STATE $\mc{U} = \arg\max \left\{ R_i - i : i \in \sampled \setminus \{N\} \right\}$ 
\STATE$\hat{F}_i^{\min} = \min \left\{ \fhat_j : j \in \{i,\dots,R_i\} \right\} \quad \forall i \in \mc{U}$
\STATE $k = \texttt{First} \big( \arg\min \{ \hat{F}_i^{\min} : i \in \mc{U} \} \big)$ \label{step:tie_breaker_first}
\RETURN $k+ \Big\lfloor \tfrac{R_k - k}{2} \Big\rfloor $
\end{algorithmic}
\end{algorithm}

\subsubsection{Sampling ``around the bend''} Early in the search, when there are very few samples and the function approximation is relatively inaccurate, sampling extrema as done in Algorithm~\ref{algo:segment_search_1D} can lead to closely-spaced samples, which in turn leaves many large intervals unexplored.  To account for these potential early misfits and to mitigate over-sampling in a narrow interval, we select samples near, but not necessarily on, the grid point deemed to be an extremum. Consequently, this subroutine can be viewed as a means to induce partial exploration near a non-tabu candidate extremum.

Our \texttt{sampleAroundTheBend()} procedure is described in Algorithm~\ref{algo:sample_around_the_bend}. For each non-tabu candidate index $j \in \nontabuPoints$, we determine a ``left-middle-right'' index triple $(L,M,R)$ where $L \in \sampled$ and $R \in \sampled$ correspond to the nearest sampled index to the left and right of $j$, respectively, and $M \in \mc{I}$ is the index at the midpoint between $L$ and $R$.  
We then determine which interval is larger (i.e., has fewer samples): $\{j,\dots,R\}$ (to the right of $j$) or $\{L,\dots,j\}$ (to the left of $j$). If it is the former (Step~\ref{step:sample_to_the_right}), we attempt to find a sample to right of $j$, in the interval $\{j,\dots,M\}$, whose approximate objective function value $\fhat_i$ is close to $\fhat_j$, but perhaps slightly less optimal. Else (Step \ref{step:sample_to_the_left}), we attempt to sample to the left of $j$ using symmetric logic.   
The parameter $\theta$, set to 0.01 or 1\% in our experiments, governs the degree of local optimality that can be sacrificed when sampling ``around the bend.'' As for the role of the index $M$, it should be clear that sampling at the midpoint $M$ itself would be tantamount to bisection search and result in a more uniformly-spaced sampling strategy.  Thus, we use $M$ as a threshold past which we will not sample; otherwise, we could inadvertantly sample very close to $L$ or $R$, which would fail to promote exploration.

An example of sampling ``around the bend'' is shown in the inset of Subfigure~\ref{fig:shekel_tabu_demo_iter24}. In iteration 24, grid index 2197 corresponds to the current approximation's minimizer.  Since a previous sample was just taken ``to the right'' of grid index 2197 in iteration 23, Algorithm~\ref{algo:sample_around_the_bend} chooses not to sample directly at index 2197, but instead to sample ``to the left'' at grid index 2179 where there is a larger unexplored interval and hence fewer samples. This sample ultimately improves the approximation quality, relative to taking a sample directly at the minimizer (grid index 2197), as the updated fit better aligns with the curvature at the minimizer and ``around the bend'' in the vicinity of grid index 2179.      

\begin{algorithm} 
\caption{\texttt{sampleAroundTheBend($j$)}}
\label{algo:sample_around_the_bend}
\begin{algorithmic}[1]
\REQUIRE Current fit $\vfhat$; Current samples $\sampled$; Maximum fractional deviation from optimum $\theta \in [0,1]$ 
\STATE $\hat{F}^{\textrm{range}} =\max\{\fhat_i : i \in \mc{I} \} - \min\{\fhat_i : i \in \mc{I} \}$ 
	\STATE Determine the ``left-middle-right'' index triple $(L,M,R)$ where $L = \max\{ i \in \sampled: i < j \}$; $R = \min\{ i \in \sampled: i > j \}$; $M = L + \texttt{round}((R-L)/2)$
	\IF{$(R-j \geq j-L)$}
		\STATE $k=\max\{ i \in \mc{I} \cap \{j,\dots,M\}: |\fhat_i - \fhat_j| \leq \theta \hat{F}^{\textrm{range}} \}$ \label{step:sample_to_the_right}
	\ELSE
		\STATE $k=\min\{ i \in \mc{I} \cap \{M,\dots,j\}: |\fhat_i - \fhat_j| \leq \theta \hat{F}^{\textrm{range}} \}$ \label{step:sample_to_the_left}
	\ENDIF
\RETURN $k \in \mc{I}$
\end{algorithmic}
\end{algorithm}

\subsection{Illustrative example of \texttt{LineWalker-full()} enhancements} 

Having described the enhancements above,  
Figure~\ref{fig:shekel_tabu_demo} showcases three successive iterations of Algorithm~\ref{algo:segment_search_1D_budgetLimited}, which nicely demonstrate tabu indices, aspiration criteria, exploration, and our \texttt{sampleAroundTheBend()} strategy. 
We investigate a 1-dimensional version of the Shekel function \url{http://www.sfu.ca/~ssurjano/shekel.html} defined as  
\begin{equation*}
\sum_{i=1}^{10} \Bigg( \sum_{j=1}^4 (x-C_{ji})^2 + \beta_i \Bigg)^{-1}
\end{equation*}
where $\vbeta$ and $\v{C}$ are defined in Table~\ref{tab:test_fnc_gallery_table}. We investigate it on the domain $x \in [0,9]$ using Algorithm~\ref{algo:segment_search_1D_budgetLimited} with 11 initial uniformly-spaced points, including the endpoints, on a grid of size $N=5000$, 
$E^{\max,\textrm{total}}=30$ total function evaluations, and $E^{\max,\textrm{itr}}=1$. Smoothing parameters are set to $\alpha=0$ and $\mu=0.01$.

\begin{figure} [h!]	
  \begin{subfigure}[b]{0.34\textwidth}
    \includegraphics[width=\textwidth]{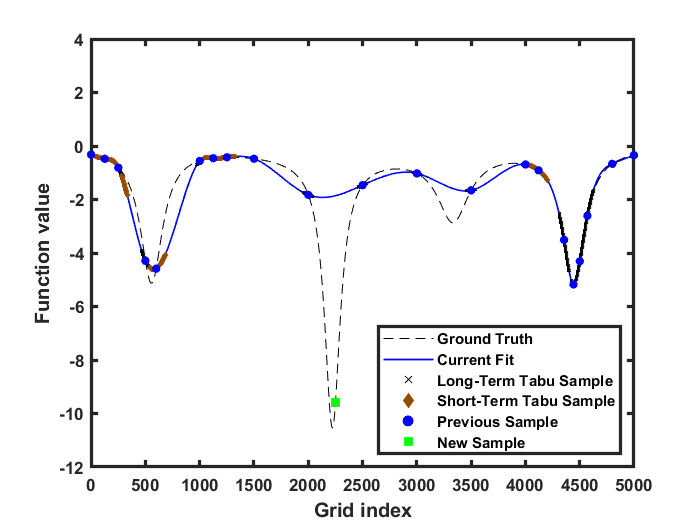}
    \caption{Iteration 22}
    \label{fig:shekel_tabu_demo_iter22}
  \end{subfigure}
    \begin{subfigure}[b]{0.34\textwidth}
    \includegraphics[width=\textwidth]{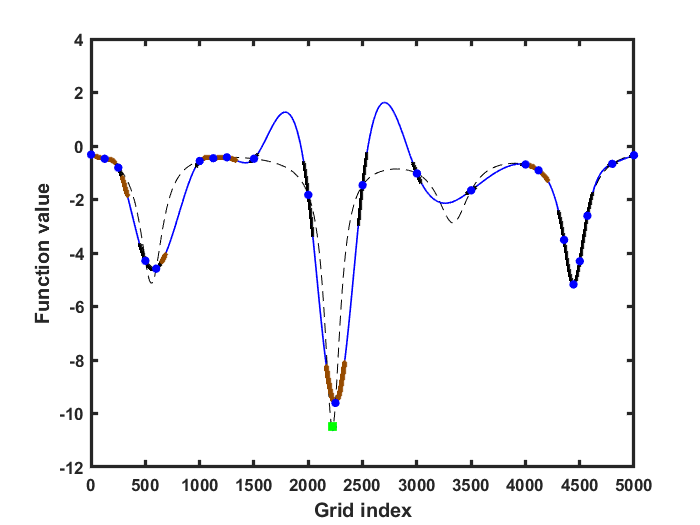}
    \caption{Iteration 23}
    \label{fig:shekel_tabu_demo_iter23}
  \end{subfigure}
      \begin{subfigure}[b]{0.34\textwidth}
      \includegraphics[width=\textwidth]{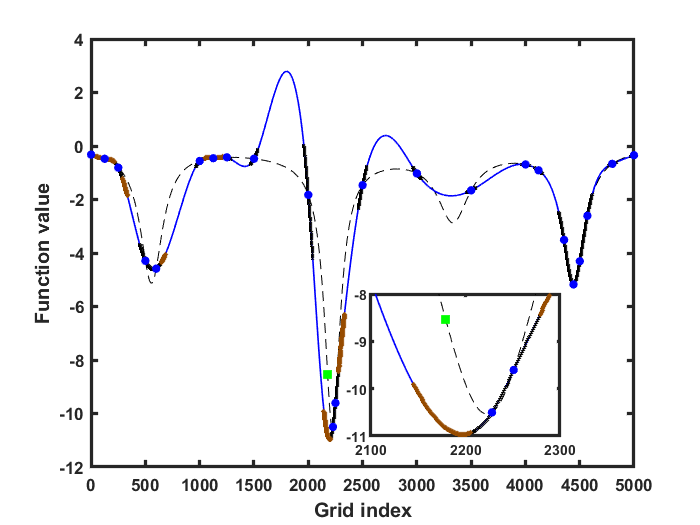}
    \caption{Iteration 24}
      \label{fig:shekel_tabu_demo_iter24}
    \end{subfigure}
  \caption{Example showing short- and long-term tabu neighborhoods (indices) and aspiration criteria being met for the Shekel benchmark function given a maximum of 30 function evaluations. (a) In iteration 22 of Algorithm~\ref{algo:segment_search_1D_budgetLimited}, a new sample is suggested that will lead to a new incumbent minimizer. (b) In the subsequent iteration after the fit has been updated, a new minimizer is found and deemed tabu. However, the first aspiration criterion is satisfied, so this point is accepted. (c) In iteration 24, another new sample deemed tabu is found and the second aspiration criterion is met (note that the minimum objective value reduced by at least 1\% of the range of the fit in the previous iteration). The inset of (c) demonstrates our \texttt{sampleAroundTheBend()} subroutine where a new sample is evaluated at index 2179, which is slightly ``off center'' relative to the fit's minimizer at index 2197.}
  \label{fig:shekel_tabu_demo}
\end{figure}

Since neighborhoods are particularly easy to visualize in one dimension, we explicitly show short- and long-term tabu grid indices around each sample.  
In Subfigure~\ref{fig:shekel_tabu_demo_iter22}, iteration 22 is depicted (i.e., the surrogate based on 21 samples plus the newly-identified 22nd sample).  Although index $i=2131$ is technically a local minimum since $\fhat_i < \min\{\fhat_{i-1},\fhat_{i+1}\}$, it does not satisfy the criterion $\fhat_i < \min\{\fhat_{i-1},\fhat_{i+1}\} - \delta^{\min}$ given in Step~\ref{step:minima_hunt_budgetLimited} of Algorithm~\ref{algo:segment_search_1D_budgetLimited} because $\delta^{\min} = \hat{F}^{\textrm{range}} \times \num{e-6} = \num{4.85e-6}$. 
Consequently, no non-tabu extrema are identified and an exploration step is taken.  There are five consecutive intervals $[1501,2001], \dots, [3001,3501]$ with the same number of unexplored points.  Thus, Algorithm~\ref{algo:find_largest_unexplored_interval} breaks the tie by choosing the one with the smallest $\fhat$ value, which turns out to be the interval whose midpoint is at index 2251, labeled ``New Sample'' in Subfigure~\ref{fig:shekel_tabu_demo_iter22}.

In iteration 23 shown in Subfigure~\ref{fig:shekel_tabu_demo_iter23}, aspiration criterion 1 is invoked since $\fhat_{2249} = -9.6058 \leq -9.4927 = \ftrue_{2251} + 0.01 \hat{F}^{\textrm{range}}$ where $\ftrue_{2251} = -9.6050$ and $\hat{F}^{\textrm{range}} = 11.2314$.
(Note that, in iteration 23, one component of aspiration criterion 2 is also satisfied: the minimum objective function value improved by 39.44\% of the current range of the entire fit. However, index 2249 is ``too close'' to index 2251 and thus the aspiration criterion is not invoked.)
In iteration 24 shown in Subfigure~\ref{fig:shekel_tabu_demo_iter24}, aspiration criterion 2 is invoked because grid index 2197 is ``near, but not too close to'' index 2228 where a new minimum was discovered in the previous iteration and the minimum objective function value improved by 6.52\% of $\hat{F}^{\textrm{range}}$. 
(Note that, in iteration 24, one component of aspiration criterion 1 is satisfied: its objective function value is within 1\% of the best known objective function value. However, the maximum number of neighbors is set to  $N^{\max,\textrm{nbrs}}=1$ and sampling at this point would exceed this limit.)

The astute observer will also spot a subtle detail when transitioning from iteration 22 to 23.  In iteration 22 (Subfigure~\ref{fig:shekel_tabu_demo_iter22}), the right most valley near index 4500 has some non-tabu points to the left of the local minimum. In iteration 23 (Subfigure~\ref{fig:shekel_tabu_demo_iter23}), these non-tabu points disappear, i.e., become tabu.  Why? As outlined in Algorithm~\ref{algo:manage_tabu_struct} and explained in the associated text, the long-term tabu grid distance threshold $\tabuGridDistanceThresholdLong_i$ for $i \in \sampled$ is a function of sample $i$'s distance to $\hat{F}^{\min}$ and $\hat{F}^{\max}$. In iteration 22, the sampled indices near index 4500 are very close to the global minimum of $\vfhat$ labeled ``Current Fit'' in iteration 22. However, after a new, and much lower, minimum is discovered in iteration 22, the samples near index 4500 are no longer close to the global minimum of the ``Current Fit'' in iteration 23.  Thus, their corresponding $\tabuGridDistanceThresholdLong_i$ values increase leading to larger long-term tabu neighborhoods for these points.   

A detailed comparison of Algorithm~\ref{algo:segment_search_1D_budgetLimited} with \texttt{bayesopt} is shown in Figure~\ref{fig:out_shekel_1Dslice}, which reveals that \texttt{bayesopt} finds a better incumbent in 20 iterations (although not quite a global minimum), while Algorithm~\ref{algo:segment_search_1D_budgetLimited} produces a better approximation given 30 and 40 total function evaluations.  Both methods are comparable with 50 total function evaluations in terms of optimality and surrogate quality.

\section{Theory and properties of the LineWalker algorithm} \label{sec:theory}
In this section, we analyze some theoretical properties of the LineWalker algorithms -- Algorithms~\ref{algo:segment_search_1D_budgetLimited} and \ref{algo:LineWalker_pure}. To get a better intuition and understanding of the LineWalker algorithms, in Section~\ref{sec:theory_alternative_perspectives}, we present results on alternative views on the function approximation, as this forms the basis of the proposed algorithms.
Section~\ref{sec:sampling_strategy} motivates our sampling strategy. 
The main convergence result is presented in Section~\ref{sec:conv}. 

\subsection{Theoretical underpinnings: alternative perspectives} \label{sec:theory_alternative_perspectives}

We start by providing an alternative perspective on how our discrete function approximation is constructed. The discrete surrogate can either be viewed as the result of minimizing the sum of squares error of the fit with regularization on the function's first and second derivatives, as in \eqref{model:unconstrained_opt_for_approximate_f}, or as minimizing the sum of squares error with constraints on the first and second derivative. These alternative interpretations are formally stated in Theorem~\ref{thm:constrained_opt_connection}.  Theorem~\ref{thm:surrogate_matrix_is_positive_semidefinite} shows that the matrix $\v{A} + \diag(\v{s})$ is positive semidefinite, revealing that the linear solve in Equation~\eqref{eq:least_squares_linear_system} required to generate our surrogate $\vfhat$ possesses ``nice'' structure. After which, we provide some connections between our surrogate and that of Gaussian Process Regression. 

\begin{theorem} \label{thm:constrained_opt_connection}
For every $\alpha,\mu \geq 0$ there exist $M_1, M_2 \geq 0$ such that the minimizer of 
\begin{equation} \label{model:constrained_nonlinear_regression2}
\begin{aligned}
\min_{\vfhat}~~& \sum_{i=1}^N s_i (\fhat_i - \ftrue_i )^2 && \\
\st~~ 
	& \sum_{i=1}^{N-1} (\fhat_{i+1} - \fhat_i)^2 \leq M_1 & & \\
	& \sum_{i=2}^{N-1} (\fhat_{i+1} + \fhat_{i-1} - 2 \fhat_i)^2 \leq M_2 & & \\
	& \fhat_i \in \Re \qquad \forall i \in \mc{I} & &  
\end{aligned}
\end{equation} 
is also a unique minimizer of the least-squares problem \eqref{model:unconstrained_opt_for_approximate_f}.
\end{theorem}
\begin{proof}
Let $\v{f}^*$ denote the minimizer of the least-squares problem \eqref{model:unconstrained_opt_for_approximate_f}, with given $\alpha,\mu \geq 0$. Next, we choose $M_1 =  \sum_{i=1}^{N-1} (f^*_{i+1} - f^*_i)^2$ and $M_2 =  \sum_{i=2}^{N-1} (f^*_{i+1} + f^*_{i-1} - 2 f^*_i)^2$, i.e., we choose $M_1$ and $M_2$ such that both constraints in \eqref{model:constrained_nonlinear_regression2} hold with equality if we plug in $\mathbf{f}^*$. We define the Lagrangian function of \eqref{model:constrained_nonlinear_regression2} as  $L(\mathbf{f},\lambda_1,\lambda_2)=\sum_{i=1}^N s_i (\fhat_i - \ftrue_i )^2  + \lambda_1\left(\sum_{i=1}^{N-1} (\fhat_{i+1} - \fhat_i)^2 -M_1\right) +\lambda_2 \left( \sum_{i=2}^{N-1} (\fhat_{i+1} + \fhat_{i-1} - 2 \fhat_i)^2 - M_2\right) $. For $\mathbf{f}^*$ to be optimal for \eqref{model:constrained_nonlinear_regression2} it must satisfy the first order optimality conditions 
\begin{align}
& \nabla_{\mathbf{f}}L(\mathbf{f}^*,\lambda_1,\lambda_2) = \mathbf{0},\\
&  \sum_{i=1}^{N-1} (f^*_{i+1} - f^*_i)^2 \leq M_1,\\
& \sum_{i=2}^{N-1} (f^*_{i+1} + f^*_{i-1} - 2 f^*_i)^2 \leq M_2, \\
&  \lambda_1\left(\sum_{i=1}^{N-1} (f^*_{i+1} - f^*_i)^2 - M_1\right) = 0,\\
&  \lambda_2\left(\sum_{i=2}^{N-1} (f^*_{i+1} + f^*_{i-1} - 2 f^*_i)^2 - M_2\right) = 0,
\end{align}
for some $\lambda_1, \lambda_2 \geq 0$. The first optimality condition, the stationarity condition, holds by setting $\lambda_1 = \alpha$ and  $\lambda_2 = \mu$, as the objective in problem \eqref{model:unconstrained_opt_for_approximate_f} coincides with $L(\v{f},\lambda_1,\lambda_2)$ for this choice of $\lambda_1$ and $\lambda_2$. The other four optimality conditions, primal feasibility and complementary slackness, are clearly satisfied due to the choice of $M_1$ and $M_2$. Due to the strict convexity of the objective function in \eqref{model:constrained_nonlinear_regression2}, we know that the solution is a unique global minimizer. Thus, by properly selecting $M_1$ and $M_2$, the minimizer of  problem  \eqref{model:constrained_nonlinear_regression2}  is also a minimizer of the least-squares problem \eqref{model:unconstrained_opt_for_approximate_f}. \qed
\end{proof}

Formulation~\eqref{model:constrained_nonlinear_regression2} shows that other constraints could easily be imposed on the function approximation to utilize prior knowledge. For example, if one knew bounds on $\vfhat$, or its derivatives,  at the outset, these could easily be incorporated into \eqref{model:constrained_nonlinear_regression2} to improve the function approximation.
Such a bounding approach is quite common and is analogous to what is done in ridge regression (see, e.g., Section 3.4.1 of \cite{hastie2009elements}).

While Equation~\eqref{eq:smoothing_matrix_A} shows that $\v{A}$ is sparse, symmetric, and pentadiagonal, the next theorem confirms that $\v{A}$ is also positive semidefinite.  More importantly, these properties reveal that the linear solve in Equation~\eqref{eq:least_squares_linear_system} is highly structured. 
\begin{theorem} \label{thm:surrogate_matrix_is_positive_semidefinite}
The matrix $\v{A} + \diag(\v{s})$ is positive semidefinite.
\end{theorem}
\begin{proof}
First, $\diag(\v{s})$ is diagonally dominant and hence positive semidefinite.
Next, we show that $\v{A}$ can be written as a nonnegative linear combination of two positive semidefinite matrices $\v{Q}_1$ and $\v{Q}_2$. 
Since sums of squares are nonnegative, 
we have 
\begin{equation}
0
\leq
\sum_{i=1}^{N-1} (\fhat_{i+1} - \fhat_i)^2
=
\sum_{i=1}^{N-1} [\fhat_{i} ~ \fhat_{i+1}]^{\top}\v{B} 
\begin{bmatrix*}[l]
\fhat_{i} \\
\fhat_{i+1}
\end{bmatrix*} 
=
\sum_{i=1}^{N-1} \vfhat^{\top} \v{C}_i \vfhat
=
\vfhat^{\top}\v{Q}_1 \vfhat
\qquad \forall \vfhat \in \Re^N 
\end{equation}
where $\v{B} = \begin{bmatrix*}[r]
 1 & -1 \\
-1 &  1
\end{bmatrix*}$, 
$\v{C}_i$ is an $N \times N$ matrix whose $j$th row and $k$th column ($c_{i,j,k}$) satisfy
\begin{equation}
c_{i,j,k} = \left\{ 
\begin{array}{ll}
b_{j,k} & j,k \in \{i,i+1\} \\
0 & \text{o.w.} \\
\end{array}
\right.
\text{for}~i=1,\dots,N-1,
\end{equation}
and $\v{Q}_1 = \sum_{i=1}^{N-1} \v{C}_i$. Hence, $\v{Q}_1$ is positive semidefinite.
By a similar line of reasoning, we have 
\begin{equation}
0
\leq
\sum_{i=2}^{N-1} (\fhat_{i+1} + \fhat_{i-1} - 2 \fhat_i)^2
=
\sum_{i=2}^{N-1} \v{g}_i^{\top}\v{H}\v{g}_i 
=
\sum_{i=2}^{N-1} \vfhat^{\top}\v{M}_i \vfhat
=
\vfhat^{\top} \v{Q}_2 \vfhat
\qquad \forall \vfhat \in \Re^N 
\end{equation}
where $\v{g}_i = (\fhat_{i-1},\fhat_i,\fhat_{i+1})^{\top}$, $\v{H} = \begin{bmatrix*}[r]
 1 & -2 &  1 \\
-2 &  4 & -2 \\
 1 & -2 &  1
\end{bmatrix*}$, 
$\v{M}_i$ is an $N \times N$ matrix such that
\begin{equation}
m_{i,j,k} = \left\{ 
\begin{array}{ll}
h_{j,k} & j,k \in \{i-1,i,i+1\} \\
0 & \text{o.w.} \\
\end{array}
\right.
\text{for}~i=2,\dots,N-1,
\end{equation} 
and $\v{Q}_2 = \sum_{i=2}^{N-1} \v{M}_i$.
Hence, $\v{Q}_2$ is positive semidefinite.
Since $\v{A} = \lambda \v{Q}_1 + \mu \v{Q}_2$ (with $\lambda,\mu \in \Re_+$) and positive semidefiniteness is preserved under addition and nonnegative scaling, it follows that $\v{A} + \diag(\v{s})$ is positive semidefinite. \qed
\end{proof}

We continue analyzing the function approximation and show a weak resemblance to Gaussian Process Regression. Following the previously introduced notation in \eqref{eq:least_squares_linear_system}, 
the surrogate $\vfhat$ is given by 
\begin{equation*}
\vfhat = [\v{A}+\diag(\v{s})]^{-1}\diag(\v{s}){\vftrue}.
\end{equation*}
By defining  $\v{W}=[\v{A}+\diag(\v{s})]^{-1}$ and $\v{y}=\diag(\v{s})\vftrue$, we can write the function approximation as
\begin{equation*}
\vfhat = \v{W}\v{y}.
\end{equation*}
Furthermore, the approximated function value at the grid point $i$ is given by
\begin{equation} \label{eq:weighted_view}
\fhat_i =  \sum_{j =1}^N w_{ij}y_j = \sum_{j \in \sampled} w_{ij}\ftrue_j=\v{w}_i\v{f}_{1:t},
\end{equation}
where $w_{ij}$ corresponds to the entry in row $i$ and column $j $ of $\v{W}$, $\v{w}_i$ contains the elements $j\in \sampled$ of the $i$th row, and $\v{f}_{1:t}$ is a $t$-dimensional vector denoting the $t$ sampled function values.  
Equation~\eqref{eq:weighted_view} has an interesting interpretation: The approximation at an unsampled point $i \in \mc{I} \setminus \sampled$ is a weighted sum of the function values at already sampled points $j \in \sampled$.  Next, we show that this is somewhat similar to the estimation in GPR. Using the notation of \citet{brochu2010tutorial}, the equation for the mean prediction of a new point $x_{t+1}$ in GPR is
\begin{equation*} \label{eq:mean_GPR_new_sample}
\mu(x_{t+1}) = \v{k}^{\top}\v{K}^{-1}\v{f}_{1:t},
\end{equation*}   
where $\v{K}$ is a $t \times t$ kernel matrix, $\v{k} = [k(x_1,x_{t+1}),\ldots,k(x_t,x_{t+1})]^{\top}$ is a $t$-dimensional vector denoting the kernel distance between the new point $x_{t+1}$ and each sampled point. Thus, in some specific circumstances and for a specific choice of kernel, such that $\v{k}^{\top}\v{K}^{-1}=\v{w}_i$, the predictions of the LineWalker algorithms and GPR would be equivalent at this point. Note that we cannot guarantee that there exists a kernel satisfying $\v{k}^{\top}\v{K}^{-1}=\v{w}_i$, and it is unlikely that any kernel could satisfy this in each iteration. Nonetheless, this shows an interesting resemblance of the LineWalker function approximation and GPR.

\subsection{Motivation for the sampling strategy} \label{sec:sampling_strategy}
Where to sample is an essential question when trying to improve a surrogate model. Here, we provide a brief motivation and intuition to the strategy of sampling at extrema of the surrogate function. Sampling at, or close to, minima of the surrogate function seems natural when searching for minima of the true function.  However, when the goal is to improve the surrogate, an ideal sampling strategy could be to sample at points where the error between the true function and surrogate function is the greatest. The error function $e$ can be defined as 
\begin{equation}
e(x) = \left(\ftrue(x) -  \fhat(x)\right)^2.
\end{equation}
Finding the maximizer of the error function is, in general, not tractable as $\ftrue$ is unknown. But, in some circumstances, the extrema of $\hat{f}$ will also be good estimates for the extrema of $e$.  
 
Let's consider three specific cases: I) $\ftrue$ is a constant function, II) $\ftrue$ is a piecewise constant function, and III) $\ftrue$ is an affine function with a moderate slope. For the first case, it is clear that the extrema of $\fhat$ will also be extrema of $e$ as the derivative of $e$ is zero at these points. For the second case, the derivative of $e$ will be zero at the extreme points of $\fhat$, but some extreme points of $e$ may also be at the non-differentiable points where $\ftrue$ make a step change. As the location of the non-differentiable points of $\ftrue$ are unknown, there might be some additional extreme points that we cannot locate. In the third case, if $\fhat$ has enough curvature at an extreme point, i.e., the absolute value of the second derivative is large enough, in relation to the slope of $\ftrue$, then an extremum of $e$ will occur close to this point. These cases might seem unlikely, but keep in mind these cases might appear locally for more general functions $\ftrue$.
 
The arguments presented above do not imply that the sampling strategy is ideal, but serve as a motivation for why it can be an efficient strategy. As described in Section~\ref{sec:enhanced_linewalker_algorithms}, we also propose some modifications to the sampling strategy to promote more exploration in unsampled regions. 

\subsection{Proof of convergence to a global minimum/maximum}\label{sec:conv}
A relevant question for any optimization algorithm is whether it is guaranteed to find a global minimum/maximum or not. In the setting of expensive black box functions, such convergence proofs can become somewhat irrelevant, as the sampling budget is typically restricted to the extent that meaningful bounds or optimality proofs cannot be obtained. However, from an algorithmic perspective, such convergence proofs are still valuable as they also serve as a ``correctness'' certificate of the optimization algorithm.  

We prove that the \texttt{LineWalker-full} (Algorithm~\ref{algo:segment_search_1D_budgetLimited}) converges to an $\epsilon$-accurate global minimum or maximum if we allow enough function evaluations and if the grid is chosen small enough. But, we first provide a clear definition of an $\epsilon$-accurate global minimum or maximum. 

\begin{definition}
Let $x$ be a global minimum (or maximum) of the function $f$ on $[a, b]$ and let $\epsilon > 0$. Then $y$ is an $\epsilon$-accurate global minimum (or maximum) if
\begin{equation*}
f(y) \leq f(x) + \epsilon \quad (\geq f(x) - \epsilon \ \ \text{for maxima}).
\end{equation*}
\end{definition}

The main convergence property is presented in the following theorem.
\begin{theorem}
If the function $f^\text{true}$ is Lipschitz continuous with Lipschitz constant $L$, then the \emph{\texttt{LineWalker-full}} algorithm will find all $\epsilon$-accurate global minima or maxima on the interval $[a , b]$ by setting the number of grid points $N$ and maximum allowed function evaluations $E^{\text{max,total}}$ as
\begin{equation}
N = E^{\text{max,total}} = \left\lceil\frac{L|b-a|}{\epsilon}\right\rceil,
\end{equation}
where $\lceil \cdot \rceil$ denotes the round up operator.
\end{theorem}
\begin{proof}
With this choice of $N$, the distance between any two adjacent grid points $x_i$ and $x_{i+1}$ is bounded from above by $\frac{\epsilon}{L}$. As $f^{\text{true}}$ is Lipschitz continuous, we know that 
\begin{equation*}
\left|f^{\text{true}}(x_i) - f^{\text{true}}(y) \right| \leq L | x_i -x_{i+1}| \leq L \frac{\epsilon}{L} \leq \epsilon \quad \forall y \in [x_i , x_{i+1}].
\end{equation*} 
Thus, the function $f^{\text{true}}$ varies by at most $\epsilon$ in between any two grid points. By the choice of maximum allowed function evaluations, all grid points will eventually be explored proving that the algorithm will find all $\epsilon$-accurate global minima or maxima.  \qed
\end{proof}

The theorem shows that the algorithm can find all global extrema, but it does not provide an insight into the computational efficiency, which is experimentally evaluated in Section~\ref{sec:Numerical_experiments}. However, the theorem does suggest a suitable choice for the number of grid points if a rough estimate of the Lipschitz constant is known.

\section{Numerical experiments} \label{sec:Numerical_experiments}

This section details our numerical experiments and showcases the performance of our LineWalker algorithms against state-of-the-art methods.  
Subsection~\ref{sec:numerical_setup} outlines the suite of benchmark functions used for comparison, the algorithms compared, and the key performance metrics used for evaluation.
Subsection~\ref{sec:optimality_comparison} compares the various methods from that vantage point of a DFO practitioner seeking optimality.
Subsection~\ref{sec:surrogate_comparison} contrasts the competing algorithms in terms of overall approximation quality.

\subsection{Experimental set up} \label{sec:numerical_setup}

\subsubsection{Test suite of functions}

We consider 20 one-dimensional test functions to demonstrate the suitability of our LineWalker algorithms. 
Figure~\ref{fig:test_fnc_gallery} depicts each function on its respective domain. 
Table~\ref{tab:test_fnc_functional_form_table} provides the precise analytical form of each test function, along with the set of global minimizers and minimum objective function values. Table~\ref{tab:test_fnc_gallery_table} categorizes these test functions based on their shape and number of extrema.  Throughout we use the terms ``test function'' and ``instance'' interchangeably as the latter is more common in optimization parlance.

\begin{figure}
  \begin{subfigure}[b]{0.24\textwidth}
    \includegraphics[width=\linewidth]{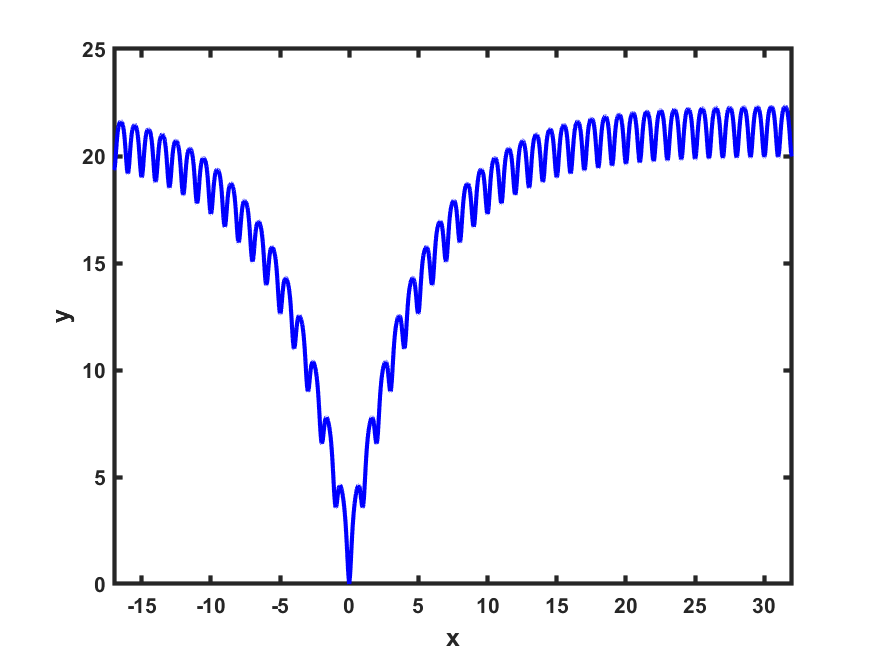}
    \caption{\href{http://www.sfu.ca/~ssurjano/ackley.html}{Ackley}}
    \label{fig:gallery_ackley}
  \end{subfigure}
  \begin{subfigure}[b]{0.24\textwidth}
    \includegraphics[width=\textwidth]{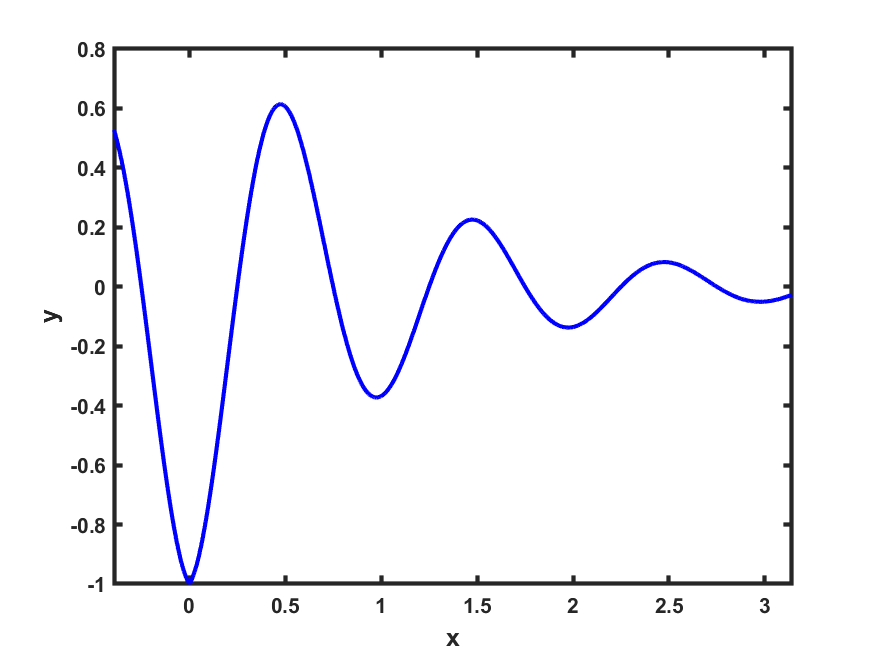}
    \caption{{\tiny \href{http://hyperphysics.phy-astr.gsu.edu/hbase/oscda.html}{Damped harmonic oscillator}}}
    \label{fig:gallery_damped_harmonic_oscillator}
  \end{subfigure}
  \begin{subfigure}[b]{0.24\textwidth}
    \includegraphics[width=\textwidth]{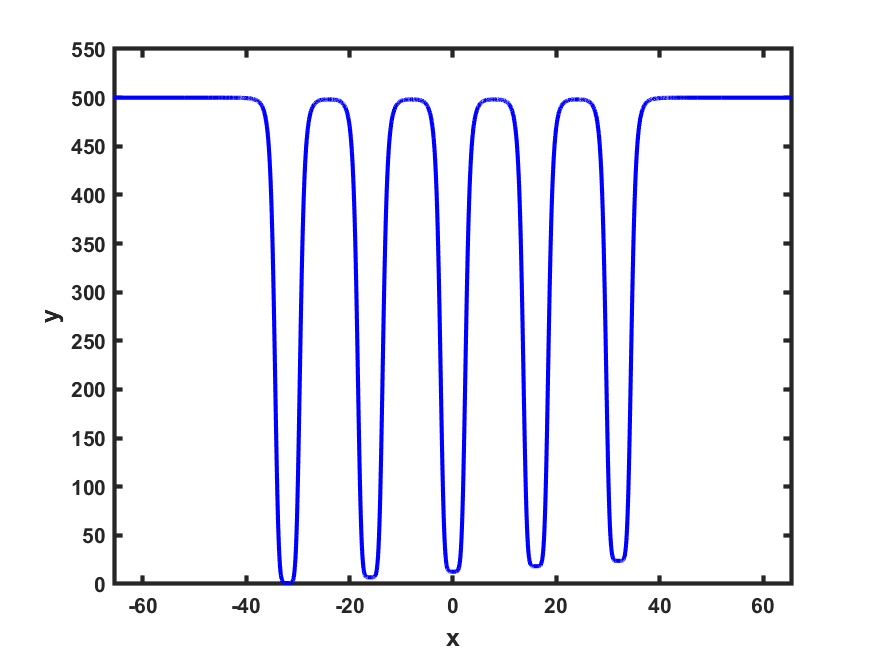}
    \caption{\href{http://www.sfu.ca/~ssurjano/dejong5.html}{dejong5}}
    \label{fig:gallery_dejong5_1Dslice}
  \end{subfigure}
  \begin{subfigure}[b]{0.24\textwidth}
    \includegraphics[width=\textwidth]{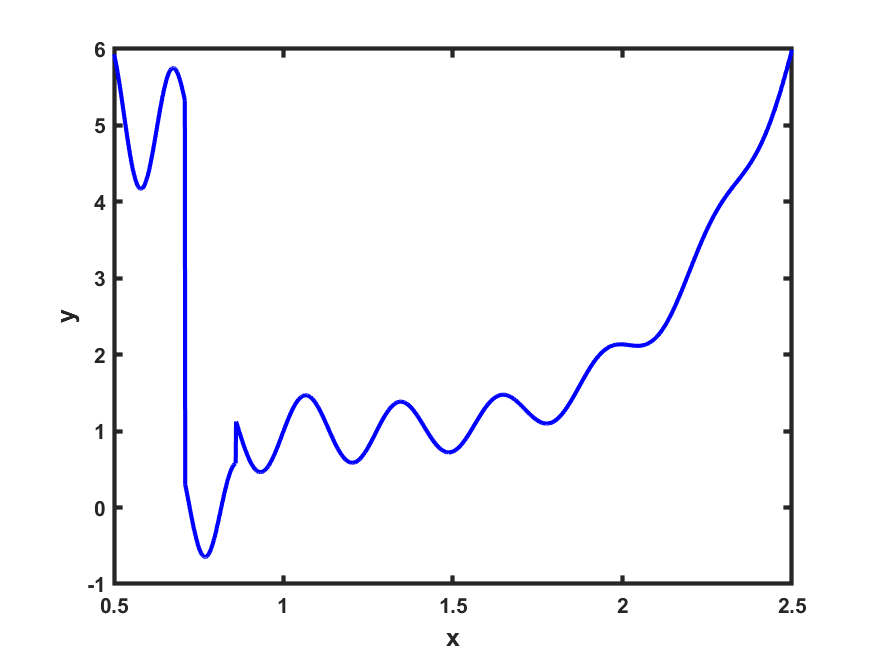}
    \caption{{\tiny \href{http://www.sfu.ca/~ssurjano/grlee12.html}{Grimacy \& Lee Variant} [\sout{S}]}}
    \label{fig:gallery_grlee12Step}
  \end{subfigure}
  \begin{subfigure}[b]{0.24\textwidth}
    \includegraphics[width=\textwidth]{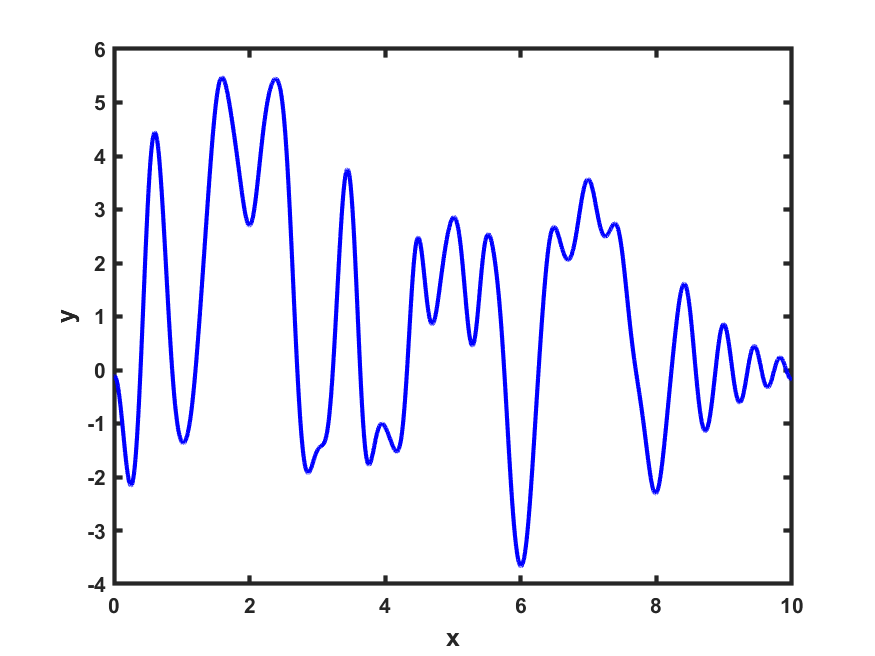}
    \caption{\href{http://www.sfu.ca/~ssurjano/langer.html}{Langermann}}
    \label{fig:gallery_langer}
  \end{subfigure}
  \begin{subfigure}[b]{0.24\textwidth}
    \includegraphics[width=\textwidth]{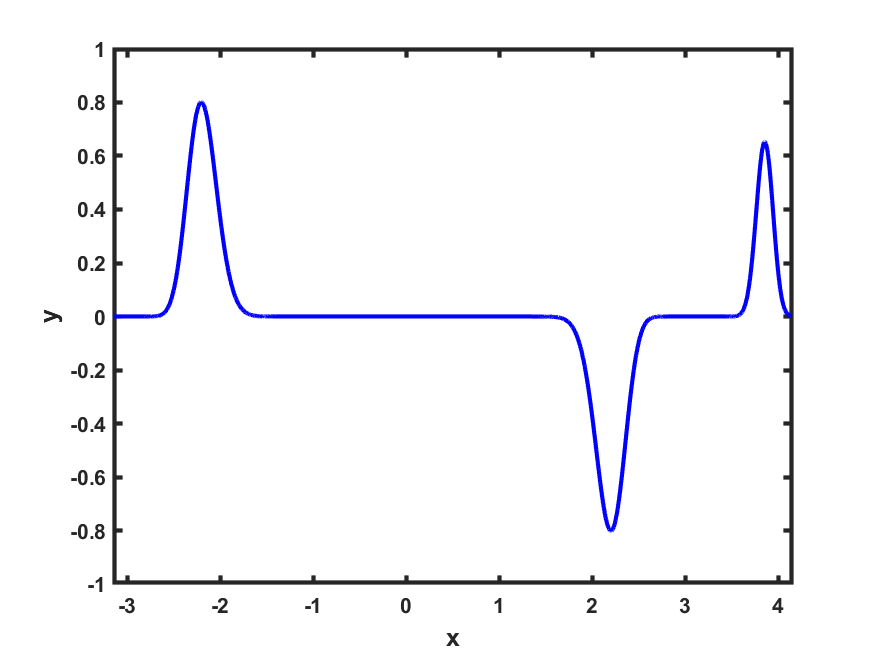}
    \caption{\href{http://www.sfu.ca/~ssurjano/michal.html}{Michalewicz}}
    \label{fig:gallery_michal}
  \end{subfigure}
  \begin{subfigure}[b]{0.24\textwidth}
    \includegraphics[width=\textwidth]{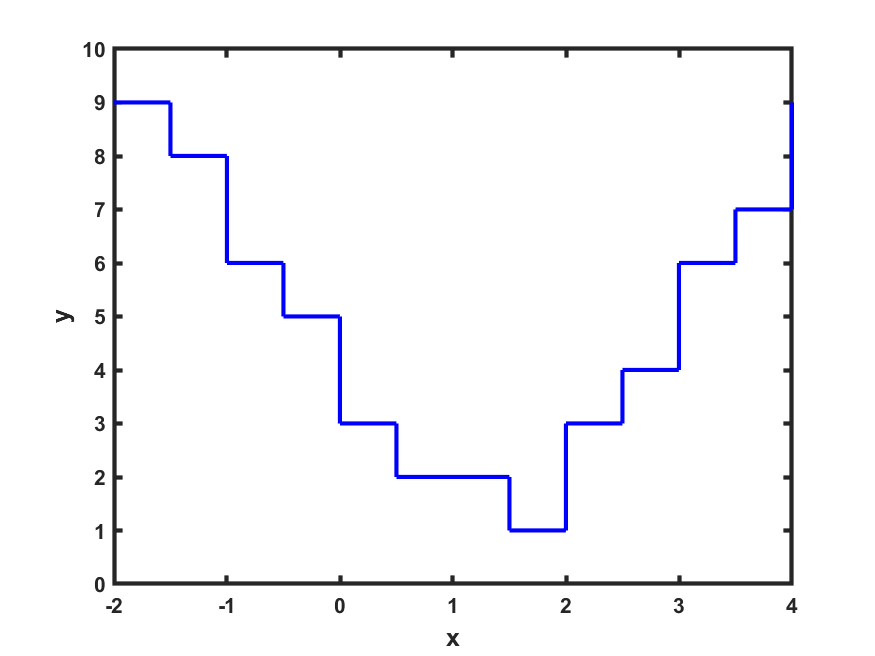}
    \caption{Plateau [\sout{S}]}
    \label{fig:gallery_plateau_1Dslice}
  \end{subfigure}
  \begin{subfigure}[b]{0.24\textwidth}
    \includegraphics[width=\textwidth]{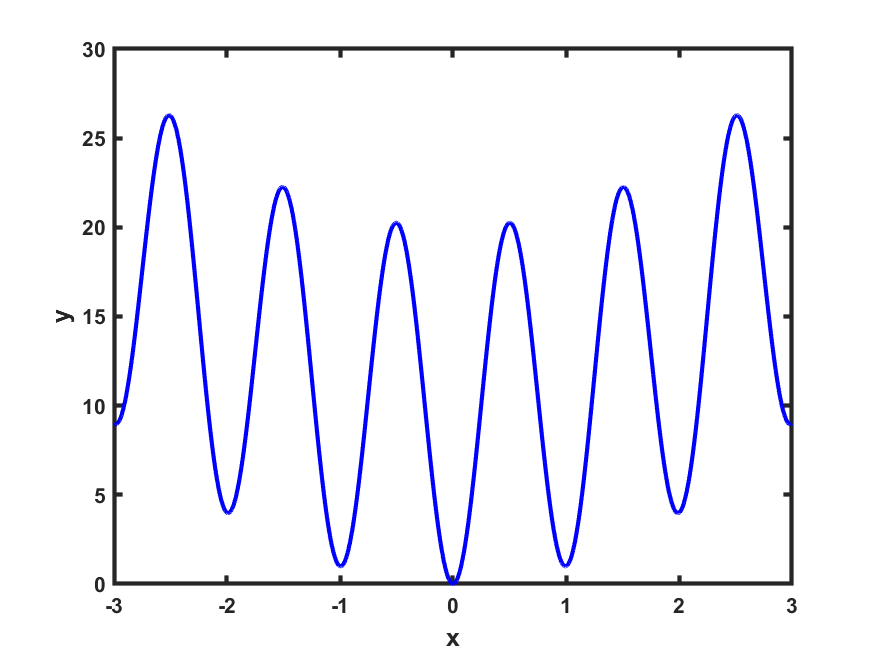}
    \caption{\href{http://www.sfu.ca/~ssurjano/rastr.html}{Rastrigin}}
    \label{fig:gallery_rastr}
  \end{subfigure}
  \begin{subfigure}[b]{0.24\textwidth}
    \includegraphics[width=\textwidth]{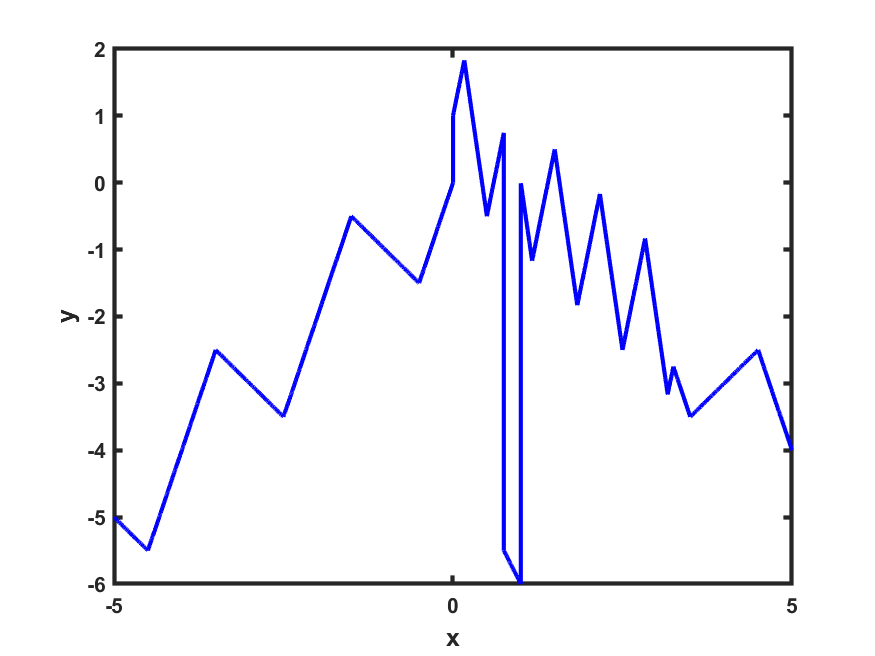}
    \caption{SawtoothD [\sout{S}]}
    \label{fig:gallery_sawtoothD}
  \end{subfigure}
  \begin{subfigure}[b]{0.24\textwidth}
    \includegraphics[width=\textwidth]{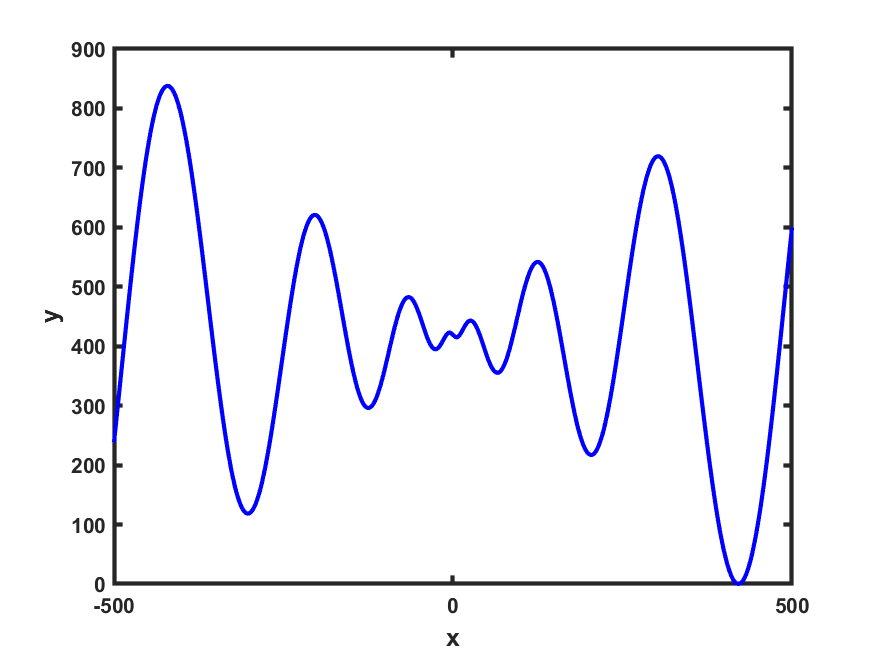}
    \caption{\href{http://www.sfu.ca/~ssurjano/schwef.html}{Schwefel}}
    \label{fig:gallery_schwef}
  \end{subfigure}
  \begin{subfigure}[b]{0.24\textwidth}
    \includegraphics[width=\textwidth]{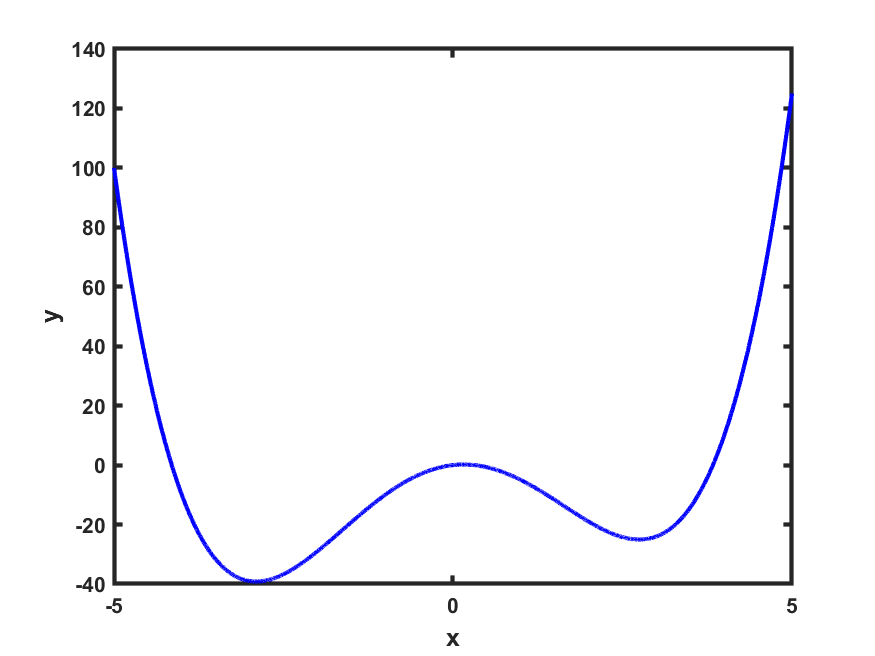}
    \caption{\href{http://www.sfu.ca/~ssurjano/stybtang.html}{Styblinski-Tang}}
    \label{fig:gallery_stybtang}
  \end{subfigure}
  \begin{subfigure}[b]{0.24\textwidth}
    \includegraphics[width=\textwidth]{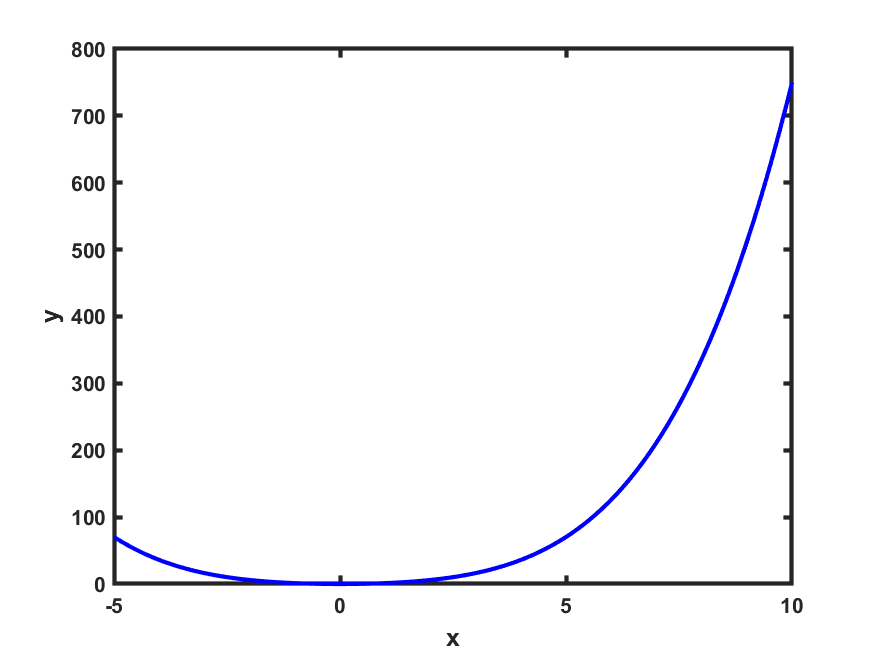}
    \caption{\href{http://www.sfu.ca/~ssurjano/zakharov.html}{Zakharov}}
    \label{fig:gallery_zakharov}
  \end{subfigure}
  \begin{subfigure}[b]{0.24\textwidth}
    \includegraphics[width=\textwidth]{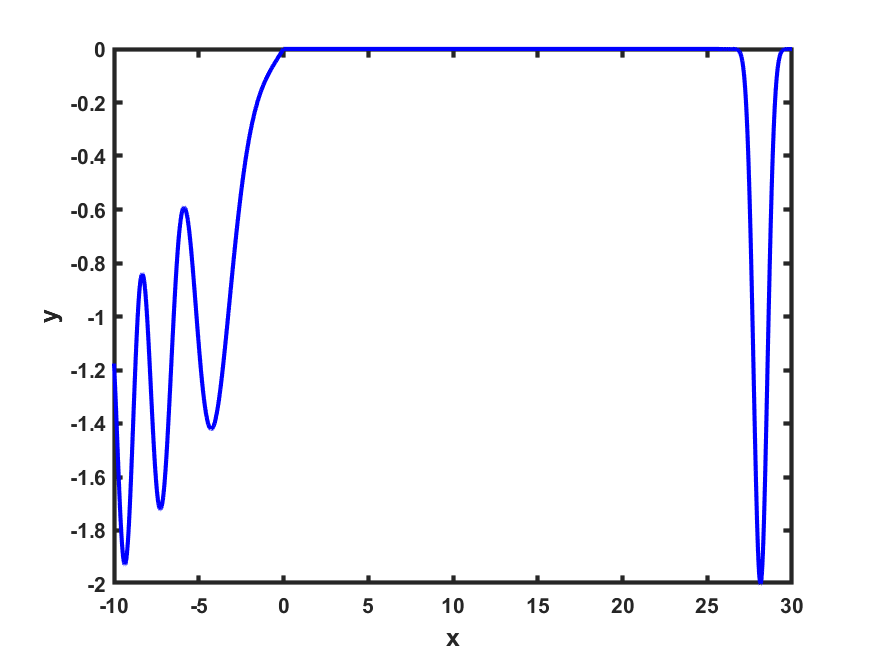}
    \caption{Easom-Schaffer2A}
    \label{fig:gallery_easom_schaffer2A}
  \end{subfigure}
  \begin{subfigure}[b]{0.24\textwidth}
    \includegraphics[width=\textwidth]{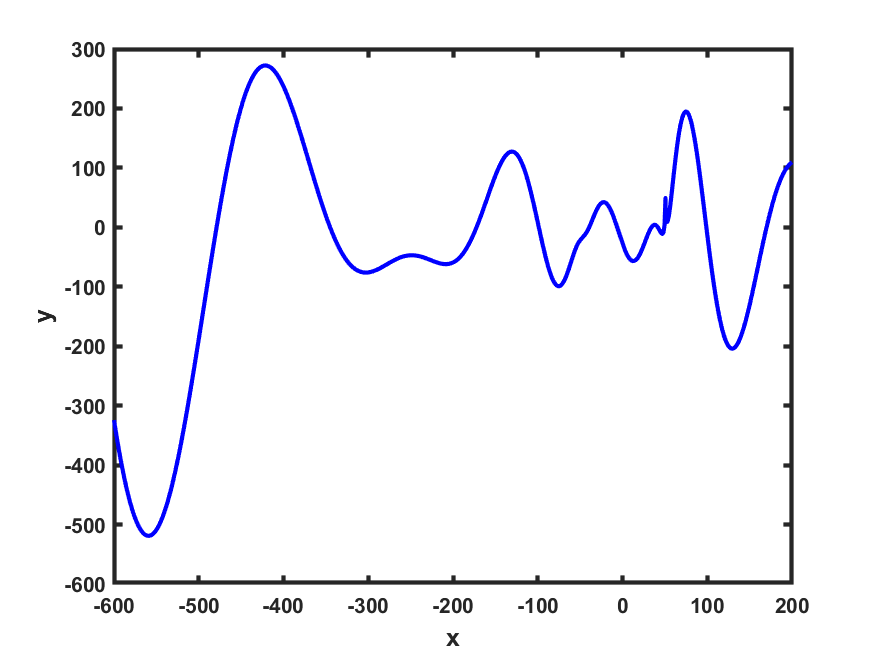}
    \caption{\href{http://www.sfu.ca/~ssurjano/egg.html}{Egg2} [\sout{S}]}
    \label{fig:gallery_egg2}
  \end{subfigure}
  \begin{subfigure}[b]{0.24\textwidth}
    \includegraphics[width=\textwidth]{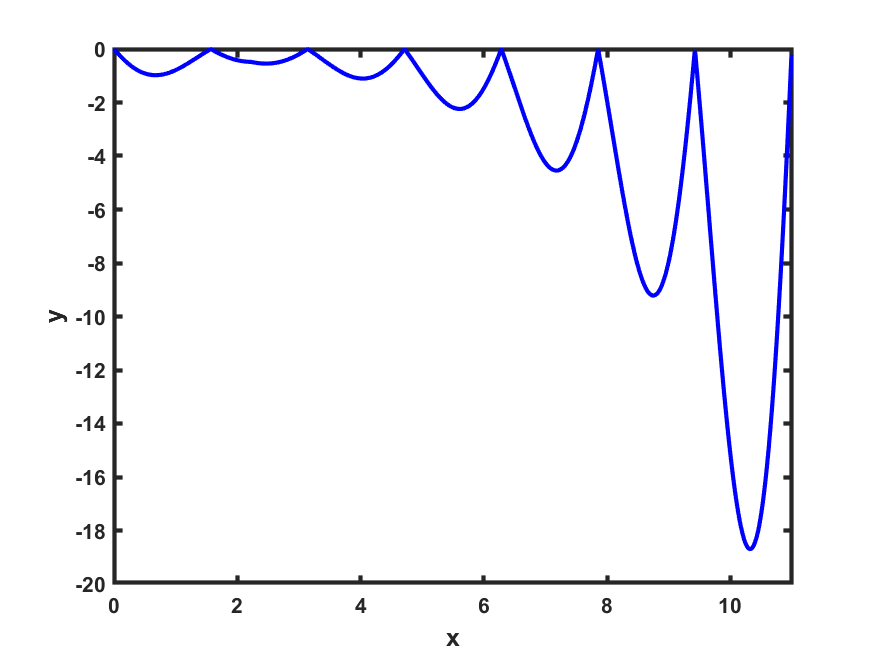}
    \caption{\href{http://www.sfu.ca/~ssurjano/holder.html}{Holder} [\sout{S}]}
    \label{fig:gallery_holder}
  \end{subfigure}
  \begin{subfigure}[b]{0.24\textwidth}
    \includegraphics[width=\textwidth]{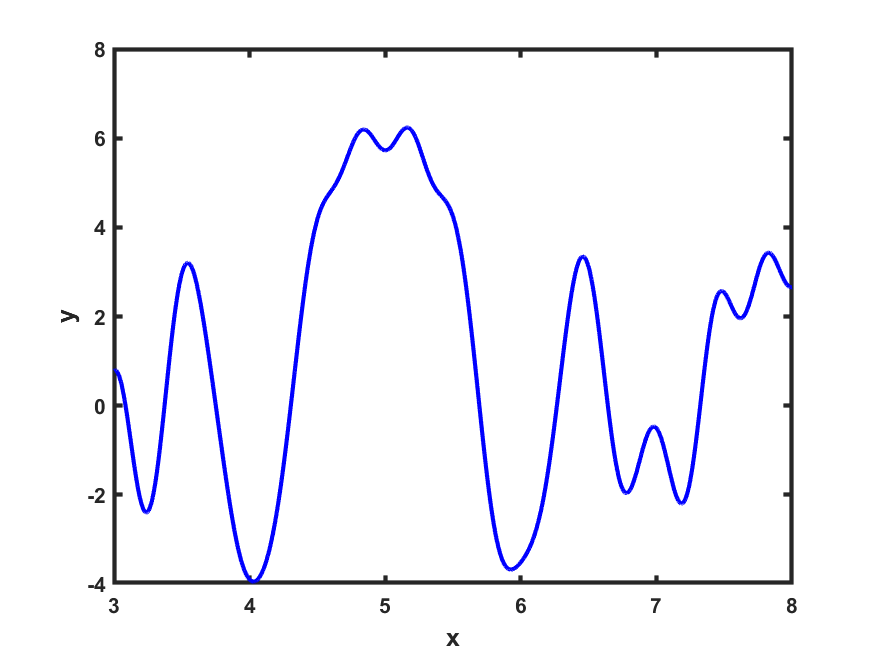}
    \caption{\href{http://www.sfu.ca/~ssurjano/langer.html}{Langermann2}}
    \label{fig:gallery_langermann2A}
  \end{subfigure}
  \begin{subfigure}[b]{0.24\textwidth}
    \includegraphics[width=\textwidth]{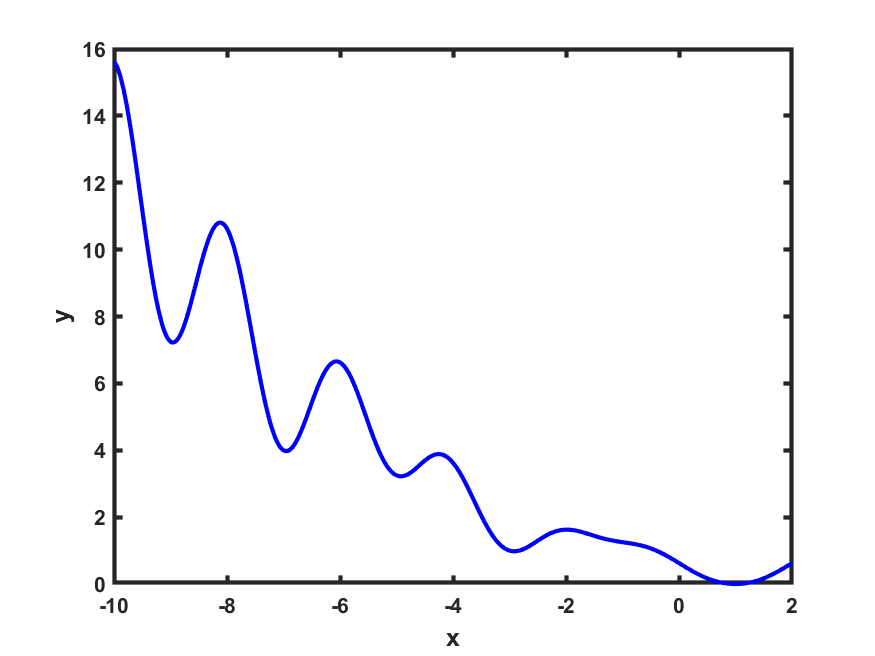}
    \caption{\href{http://www.sfu.ca/~ssurjano/levy.html}{Levy}}
    \label{fig:gallery_levy}
  \end{subfigure}
  \begin{subfigure}[b]{0.24\textwidth}
    \includegraphics[width=\textwidth]{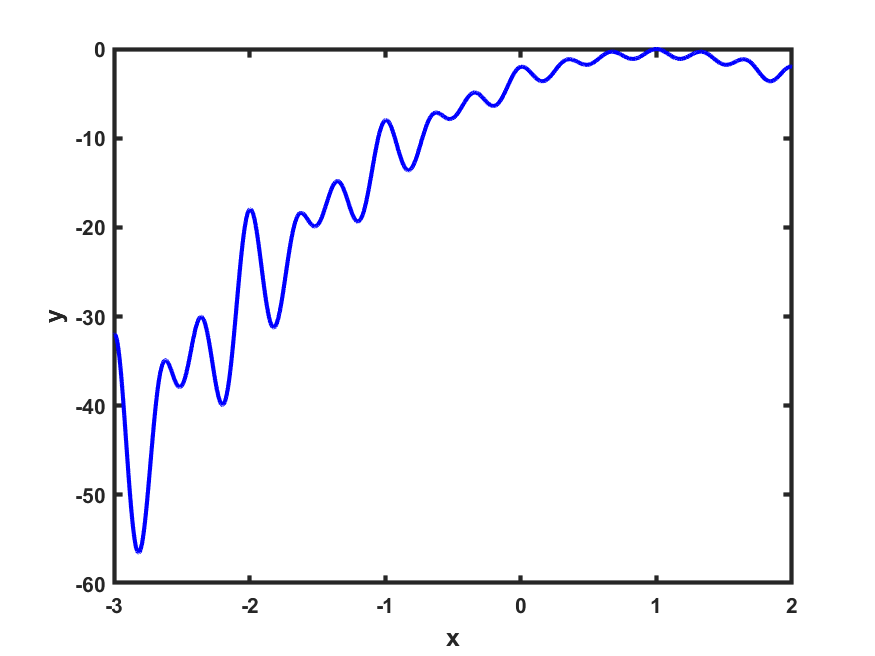}
    \caption{\href{http://www.sfu.ca/~ssurjano/levy13.html}{Levy13}}
    \label{fig:gallery_levy13}
  \end{subfigure}
  \begin{subfigure}[b]{0.24\textwidth}
    \includegraphics[width=\textwidth]{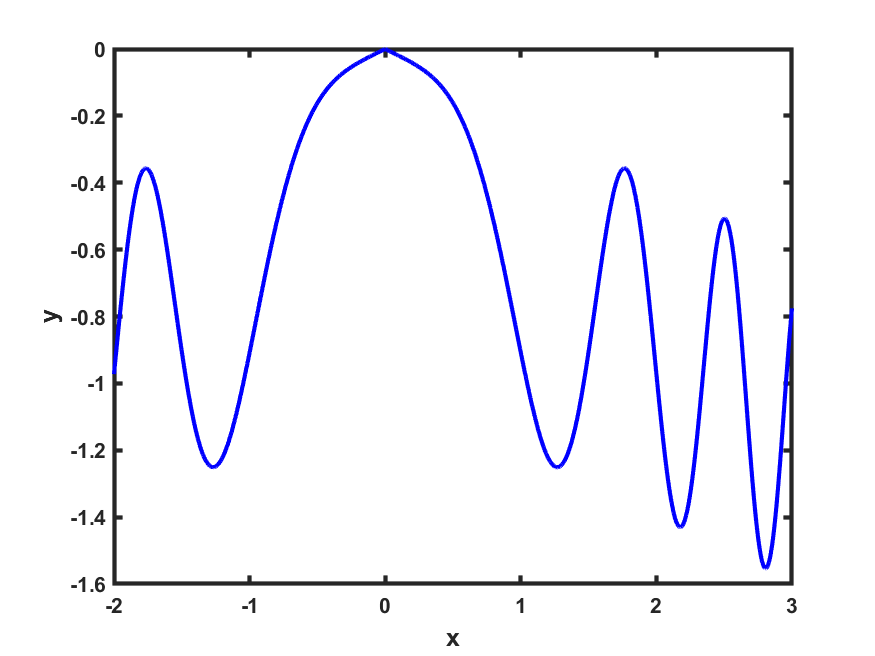}
    \caption{\href{http://www.sfu.ca/~ssurjano/schaffer2.html}{Schaffer2A} [\sout{S}]}
    \label{fig:gallery_schaffer2A}
  \end{subfigure}
  \begin{subfigure}[b]{0.24\textwidth}
    \includegraphics[width=\textwidth]{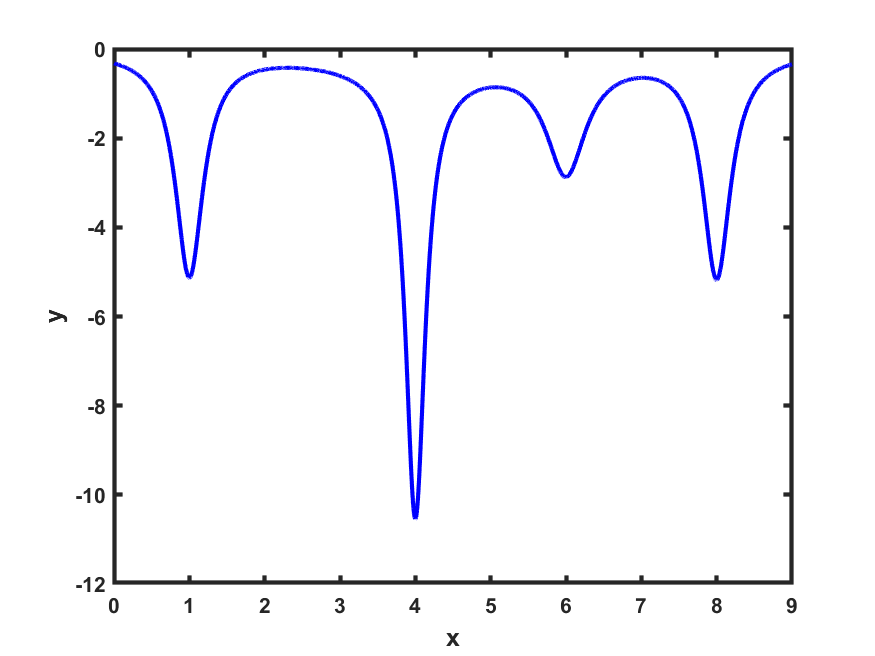}
    \caption{\href{http://www.sfu.ca/~ssurjano/shekel.html}{Shekel}}
    \label{fig:gallery_shekel}
  \end{subfigure}
  \caption{Gallery of test functions. Functions with an [\sout{S}] are non-smooth.}
  \label{fig:test_fnc_gallery}
\end{figure}

As shown in Table~\ref{tab:test_fnc_gallery_table}, 17 of the test functions come from the website by \cite{surjanovic2013virtual}, while the remaining three (Damped Harmonic Oscillator, Plateau, SawtoothD) come from other sources.  We included functions from different categories according to the classification by \cite{surjanovic2013virtual}.  In general, bowl- and plate-shaped (i.e., strictly convex) functions are ``easy'' for all of the algorithms that we explored in the sense that the algorithms are quite effective at optimizing these functions and creating a high-quality surrogate.  Consequently, we only included a single convex function (Zakharov) and instead focused more heavily on nonconvex functions with few or many local extrema.  Since some functions possess extrema occurring at regular intervals (which we call ``periodic'' functions, perhaps with some abuse of terminology), we also include some functions that are nonconvex and aperiodic.  Such functions are particularly challenging because a certain amount of exploration must occur in order to find all extrema and/or obtain a high-quality fit. The Eason-Schaffer2A and Michalewicz functions are two such functions that are constant on a large interval where sampling is needed to ensure that no hidden extrema are present. Despite the caveats that our LineWalker methods are meant for smooth functions, we also consider several non-smooth functions to demonstrate how our methods perform under non-ideal conditions. For example, the plateau function possesses a staircase structure giving rise to an infinite number of local extrema. 

\subsubsection{Algorithms compared}
To benchmark the performance of our LineWalker algorithms, we compare against MATLAB's Bayesian optimization algorithm \href{https://www.mathworks.com/help/stats/bayesopt.html}{\texttt{bayesopt}} in MATLAB R2018b, NOMAD 3.9.1 \cite{Le2011a}, MATLAB's \href{https://www.mathworks.com/help/matlab/ref/fminbnd.html}{\texttt{fminbnd}} function, MATLAB's \href{https://www.mathworks.com/help/matlab/ref/fminsearch.html}{\texttt{fminsearch}} function, and ALAMO version 2021.12.28 \cite{cozad2014learning}. \texttt{bayesopt} proves to be quite competitive on these instances and therefore illustrates the power of GPR on one-dimensional nonconvex functions. We chose NOMAD as the DFO literature shows that NOMAD continues to be one of the best DFO solvers available. MATLAB's \texttt{fminbnd} and \texttt{fminsearch} may seem like odd choices, but many non-experts use these algorithms out of sheer convenience since they are readily available in MATLAB. Finally, ALAMO is widely regarded as a strong surrogate modeling tool.  We briefly describe each below and how we use it.

All algorithms are furnished the same initial samples, when initial samples are required. More precisely, for all algorithms except \texttt{fminbnd}, which does not require an initial starting point, we generated an initial set $\mc{I}^{\textrm{init}} = \{1 + (N-1) \tfrac{(i-1)}{10}: i=1,\dots,11\}$ of 11 uniformly-spaced grid indices corresponding to the samples $\{x^L + (x^U-x^L) \tfrac{(i-1)}{10}: i=1,\dots,11\}$ samples on the interval $[x^L,x^U]$ (see the ``Domain'' column of Table~\ref{tab:test_fnc_functional_form_table}).
For the algorithms (\texttt{fminsearch} and NOMAD) that only require a single starting point, we furnished them the set $\mc{I}^{\textrm{init}}$ and then, via a loop, launched a search beginning from each starting point. 
For the algorithms that make use of a sample set to construct an initial surrogate (LineWalker variants, \texttt{bayesopt}, and ALAMO), $\mc{I}^{\textrm{init}}$ was used for exactly this purpose. 

\texttt{LineWalker-pure} (Algorithm~\ref{algo:LineWalker_pure}) and its enhanced version \texttt{LineWalker-full} (Algorithm~\ref{algo:segment_search_1D_budgetLimited}) use the parameters listed in Table~\ref{table:linewalker_param_values}.
We set $E^{\max,\textrm{itr}}=1$ so that only one function evaluation is made per iteration, consistent with what other methods require. We use $N=$ 5,000 grid points for all test functions except for Ackley where 10,000 gridpoints are needed because, with only 5,000 grid points, there is no way to be within 0.01 of the true optimal objective function value of 0, i.e., there does not exist a grid point near the true optimum with an objective function value satisfying the optimality requirement. Doing so results in a single grid point (at index 3470 corresponding to $x=\num{-2.0002e-04}$) having an objective function value capable of satisfying the ``solve'' requirement~\ref{def:instance_is_solved}; all other points do not satisfy this condition.

\begin{table} 
\caption{LineWalker parameter values for computational experiments.}
\label{table:linewalker_param_values}
\begin{tabular}{lcl}
\toprule
\textbf{Parameter} & \textbf{Value} & \textbf{Comment} \\
\midrule
$E^{\max,\textrm{total}}$ & $\{20,30,40,50\}$ & maximum number of total function evaluations allowed \\
$E^{\max,\textrm{itr}}$ & 1 & maximum number of function evaluations allowed per major iteration \\
$\tabuGridDistanceThresholdShort_i$ & $N/(2 E^{\max,\textrm{total}})$ & short-term tabu grid distance threshold for grid index $i \in \mc{I}$ \\
$\tabuGridDistanceThresholdLong_i$ & dynamic & long-term tabu grid distance threshold for grid index $i \in \mc{I}$ \\
$N$ & 5,000 & number of equally-spaced grid points ($N=10,000$ for ackley) \\
$\alpha$ & 0 & first-derivative smoothing parameter \\
$\mu$ & 0.01 &  second-derivative smoothing parameter \\
$\delta^{\min}$ & $\hat{F}^{\textrm{range}} \times \num{e-6}$ & objective function tolerance for local minima \\
$\delta^{\max}$ & $\hat{F}^{\textrm{range}} \times \num{e-6}$ & objective function tolerance for local maxima \\
$\tabuTenureShort$ & 5 & Short-term tabu tenure is initialized to 5, but changes dynamically \\
$\theta$ & 0.01 & Maximum fractional deviation from optimum in Algorithm~\ref{algo:sample_around_the_bend} \\
$\nu^{\min}$ & 0.10 & Minimum grid point separation multiplier  \\
$\nu^{\max}$ & 0.25 & Maximum grid point separation multiplier  \\
\bottomrule
\end{tabular}
\end{table}

The \texttt{bayesopt} algorithm is a powerful and versatile algorithm whose main purpose is to find a global minimum of a (possibly multivariate and stochastic) black box function. While optimizing this function, it also produces a surrogate function based on a GPR's posterior mean distribution. We compare the resulting fit from our LineWalker algorithms with this posterior mean distribution.  Note that \texttt{bayesopt} is first and foremost designed to ``chase global minima,'' not to construct an accurate surrogate model at every point in the domain. We supply \texttt{bayesopt} with a vector of initial sample locations (\texttt{InitialX}).
We flag that the objective function is deterministic.
Most importantly, we use the \texttt{expected-improvement-plus} acquisition function described in \href{https://www.mathworks.com/help/stats/bayesian-optimization-algorithm.html}{MATLAB's Bayesian optimization algorithm}.

NOMAD (Nonlinear Optimization with the MADS algorithm) is ``a C++ implementation of the Mesh Adaptive Direct Search algorithm (MADS), designed for difficult blackbox optimization problems'' \url{https://www.gerad.ca/en/software/nomad/}. NOMAD can handle nonsmooth functions, constraints, as well as integer and categorical decision variables. It regularly appears as one of the most consistent and dominant DFO solvers in the literature and serves as an important state-of-the-art benchmark. 
Since NOMAD requires a starting point, we loop over all of the initial samples and perform the search from this starting point.  In this way, all methods are privy to the same initial samples. 
We set $\texttt{min\_mesh\_size} = \num{1e-4}$ and $\texttt{initial\_mesh\_size} = 10$.
Otherwise, we used the default parameter settings. 

MATLAB's \texttt{fminbnd} attempts to find a local minimum of a one-dimensional function on a bounded interval using an algorithm based on golden section search and parabolic interpolation.
For a strictly unimodal function with an extremum in the interior of the domain, it will find the extremum and do so in the most asymptotically economic manner, i.e., using the fewest function evaluations for a prescribed accuracy \citet{snyman2018line}. 
According to the online documentation \url{https://www.mathworks.com/help/matlab/ref/fminbnd.html}, ``Unless the left endpoint $x_1$ [of the domain interval] is very close to the right endpoint $x_2$, \texttt{fminbnd} never evaluates [the function] at the endpoints, so [the function] need only be defined for $x$ in the interval $x_1 < x < x_2$.''  This is not a concern since no global minimum occurs at an endpoint in our benchmark suite.
We set $\texttt{TolFun} = \num{1e-6}$. 
Otherwise, we used the default parameter settings. 

Applicable to multivariate functions, MATLAB's \texttt{fminsearch} uses the Nelder-Mead simplex algorithm, ``perhaps the most widely used direct-search method'' \cite[p.6]{larson2019derivative}, as described in \cite{lagarias1998convergence}.  We were particularly interested to understand its performance in light of the high expectations placed upon it by DFO experts:
``Nelder-Mead is an incredibly popular method, in no small part due to its inclusion in Numerical Recipes [Press et al., 2007], which has been cited over 125,000 times and no doubt used many times more. The method (as implemented by \cite{lagarias2012convergence}) is also the algorithm underlying \texttt{fminsearch} in MATLAB'' \citep[p.7]{larson2019derivative}.
Like NOMAD, \texttt{fminsearch} requires the user to provide an initial starting point, so we adopted the same loop as used for NOMAD. 
Since \texttt{fminsearch} does not handle variable bounds, we check that each solution returned is feasible, i.e., within the function's given domain.
We set $\texttt{TolFun} = \num{1e-6}$. 
Otherwise, we used the default parameter settings.  

Finally, we compared against ALAMO (Automated Learning of Algebraic Models), which claims to be ``the only software that can impose physical constraints on machine learning models, enabling users to build accurate models even from small datasets'' \url{https://minlp.com/alamo-modeling-tool}.
In contrast to \texttt{bayesopt}, ALAMO's main purpose is to construct accurate and simple algebraic surrogate models from data.  These algebraic surrogates can then be plugged back into a larger optimization or simulation tool to reduce computation time.  

All experiments (excluding running ALAMO, which is a standalone software) were conducted in MATLAB R2018b (9.5.0.944444).
All computing was performed on a Dell Precision 5540 x64-based PC laptop running Intel(R) Core(TM) i7-9850H CPU @ 2.60GHz, 2592 Mhz, 6 Core(s), 12 Logical Processor(s).
For each algorithm, we supplied the same set of 11 uniformly-spaced starting points, two of which correspond to the endpoints of the domain. 

\subsubsection{Key performance metrics} \label{sec:metrics}

While there are many possible metrics to assess performance for DFO methods and surrogate builders, we chose to focus on two commonly used metrics -- (1) number of instances solved and (2) total absolute scaled error of the resulting fit -- as a function of the number of samples taken.  Both are defined and described below.
Conspicuously absent from this list is computation time.  Since our assumption is that the black box function evaluations themselves could potentially be very time consuming (requiring hours to days), we found the remaining time negligible in the grand scheme of things.
For our LineWalker algorithms, the primary bottleneck is solving the linear system in Step~\ref{step:solve_linear_system} in Algorithm~\ref{algo:segment_search_1D}, which depends on the number of grid points $N$ used in the surrogate. This system can be solved in seconds for $N=10k$. See also Figure~\ref{fig:cpu_time_per_iteration_shekel}. 
For \texttt{bayesopt}, optimizing the acquisition function was typically the main bottleneck.
In general, we observed that the non-function evaluation computation time for both the LineWalker algorithms and \texttt{bayesopt} was on the order of seconds to tens of seconds and thus would pale in comparison to the time spent performing function evaluations for computationally-expensive, real-world black box simulations and oracle functions. 

\textbf{Number of instances solved}. We follow typical requirements used in the DFO community in which ``A solver [is] considered to have successfully solved a problem [i.e., optimized a test function] if it returned a solution with an objective function value within 1\% or 0.01 of the global optimum, whichever was larger'' \cite{ploskas2021review}.
Mathematically, let $f^{*,\textrm{true}}$ and $\fhat^*$ denote the optimal objective function value of the true underlying function and the approximate function, respectively. Let $x^{*,\textrm{eval}}$ denote the sample with the smallest evaluated objective function value. A test function is declared ``solved'' if 
\begin{equation} \label{def:instance_is_solved}
|\fhat^*-f^{*,\textrm{true}}| \leq 0.01 \max\{1,|f^{*,\textrm{true}}|\} 
\qquad
\textrm{or}
\qquad
|\fhat(x^{*,\textrm{eval}})-f^{*,\textrm{true}}| \leq 0.01 \max\{1,|f^{*,\textrm{true}}|\} 
\end{equation}

\textbf{Total absolute scaled error}. See Section~\ref{sec:surrogate_comparison}.

\subsection{Optimality comparison} \label{sec:optimality_comparison}

We first compare our LineWalker algorithms with other methods as one would do in a DFO context where the key metric is how many instances are solved within a given budget of function evaluations. Salient DFO-related observations are listed below.

\begin{figure}[h!] 
\centering
\includegraphics[width=12cm]{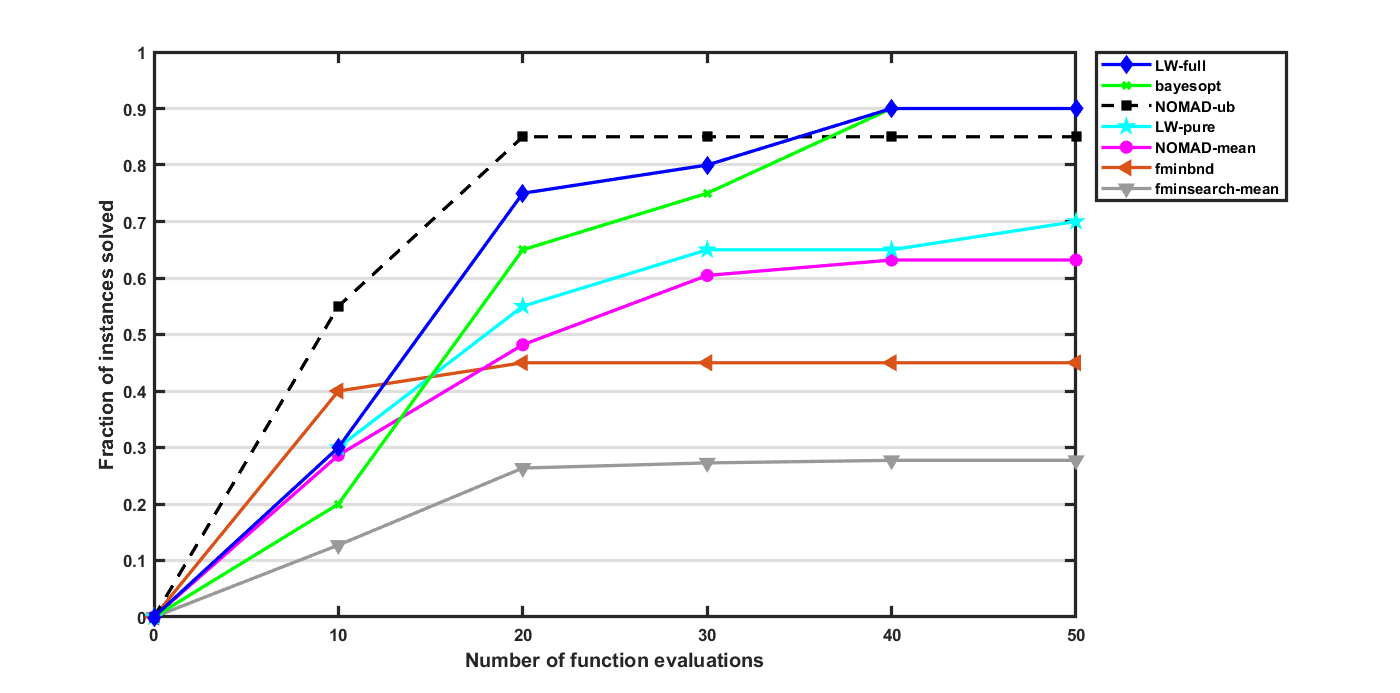}
\caption{Fraction of instances solved (see definition \eqref{def:instance_is_solved}) as a function of the number of black box evaluations. 
\texttt{LW-full} refers to Algorithm~\ref{algo:segment_search_1D_budgetLimited}. 
\texttt{LW-pure} refers to Algorithm~\ref{algo:LineWalker_pure}.
NOMAD-ub is not empirically achievable and therefore serves as an optimistic (or oracle-initialized) upper bound.
No algorithm could solve the SawtoothD instance.
}
\label{fig:fraction_of_instances_solved}
\end{figure} 

\textbf{fminsearch}. MATLAB's \texttt{fminsearch} function performed relatively poorly, solving just over 25\% of all instances on average, revealing that the popular Nelder-Mead simplex approach is challenged on our benchmark functions. 

\textbf{fminbnd}. MATLAB's \texttt{fminbnd} function, which relies on golden section search and parabolic interpolation, performs the best out of all methods when limited to only 10 samples, finding a global minimum in 8 out of 20 instances. However, when given a larger sample budget, it can only find one more global minimum and fares poorly. Recall that it is theoretically the ``best'' one can do on unimodal functions. 

\textbf{NOMAD}. As shown under ``NOMAD-mean,'' the state-of-the-art DFO solver NOMAD's average performance surpasses \texttt{fminbnd} with a budget of 30 samples or greater.  On average, NOMAD is able to solve just over half of the instances. Meanwhile, an upper bound on NOMAD's performance (``NOMAD-ub'') shows that if one had an oracle and could select the best starting point amongst the 11 that we offered, NOMAD can perform quite well, but is still unable to solve 15\% of the instances.  As a reminder, this performance should not be considered practical since it requires prior knowledge of the best starting point.

\textbf{Bayesian optimization (bayesopt)}. MATLAB's \texttt{bayesopt} function exhibits superior performance amongst the state-of-the-art methods given 20 or more function evaluations, identifying a global minimum in 18 of the 20 instances.

\textbf{LineWalker}. Somewhat surprisingly, our most basic \texttt{LineWalker-pure} (``LW-pure'') Algorithm~\ref{algo:LineWalker_pure}, which is essentially an extrema hunter procedure with little to no exploration, is able to outperform NOMAD on this benchmark suite. It solves 11 out of 20 instances with just 20 samples and ultimately solves 3 more instances with a budget of 50 samples. It is inferior, however, to \texttt{bayesopt} with 20 or more samples. Meanwhile, \texttt{LineWalker-full} (``LW-full'') Algorithm~\ref{algo:segment_search_1D_budgetLimited}, which includes a tabu structure and exploration, proves to be quite competitive with \texttt{bayesopt}, even achieving more ``solve'' successes when given a budget of fewer than 40 function evaluations. 

\texttt{LW-full}, \texttt{LW-pure}, and \texttt{bayesopt} fail to solve two instances (Ackley and SawtoothD) within 50 function evaluations. For Ackley, \texttt{LW-full} returns the sample at index 3468 as the minimizer with objective function value -0.10. As mentioned above, index 3470 is the only index that satisfies the optimality criteria.

\subsection{Surrogate comparison} \label{sec:surrogate_comparison}

While global optimality is one useful metric for comparing methods, it does not tell the full story. Practitioners may also wish to assess, and have confidence in, the overall function approximation obtained from a given method.  
Figure~\ref{fig:damped_harmonic_oscillator_surrogate_comparison} compares the approximation quality of \texttt{LineWalker-full}, \texttt{bayesopt}, and ALAMO on the damped harmonic oscillator function, a rather well-behaved, oscillating nonconvex test instance with a constant periodicity.  The results are striking.  While all three methods are nearly able to identify a global optimum (\texttt{LineWalker-full} and \texttt{bayesopt} solve this instance), \texttt{LineWalker-full} provides a much better fit relative to the other methods. Indeed, with fewer than 40 function evaluations, \texttt{bayesopt} tends to underestimate every local (non-boundary) maxima. A detailed visual comparison between \texttt{LineWalker-full} and \texttt{bayesopt} for all 20 functions is given in the appendix.      

\begin{figure} [h!]	
  \begin{subfigure}[b]{0.33\textwidth}
    \includegraphics[width=\textwidth]{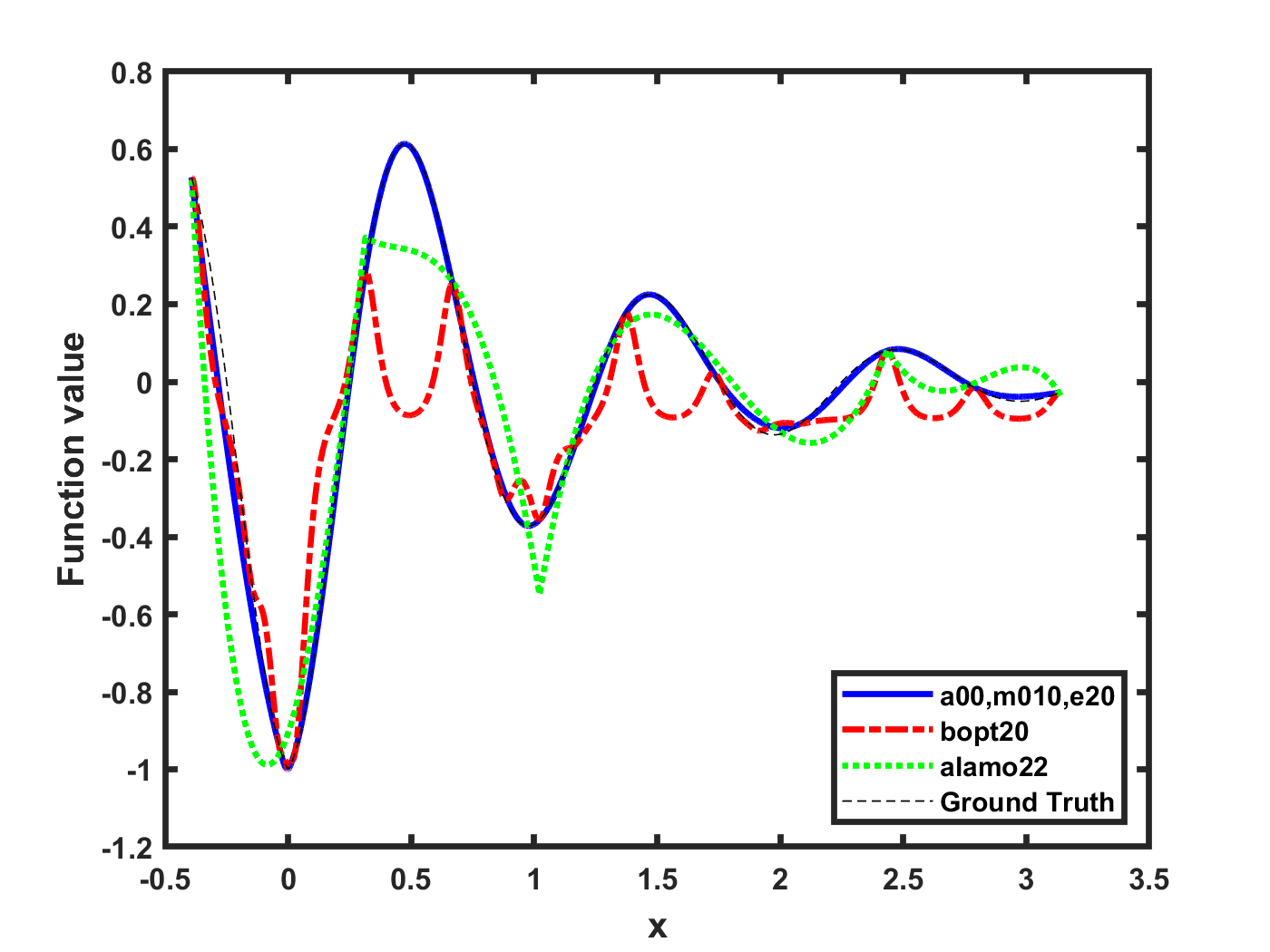}
    \caption{$E^{\max,\textrm{total}} = 20$}
    \label{fig:damped_harmonic_oscillator_e20_comparison}
  \end{subfigure}
    \begin{subfigure}[b]{0.33\textwidth}
    \includegraphics[width=\textwidth]{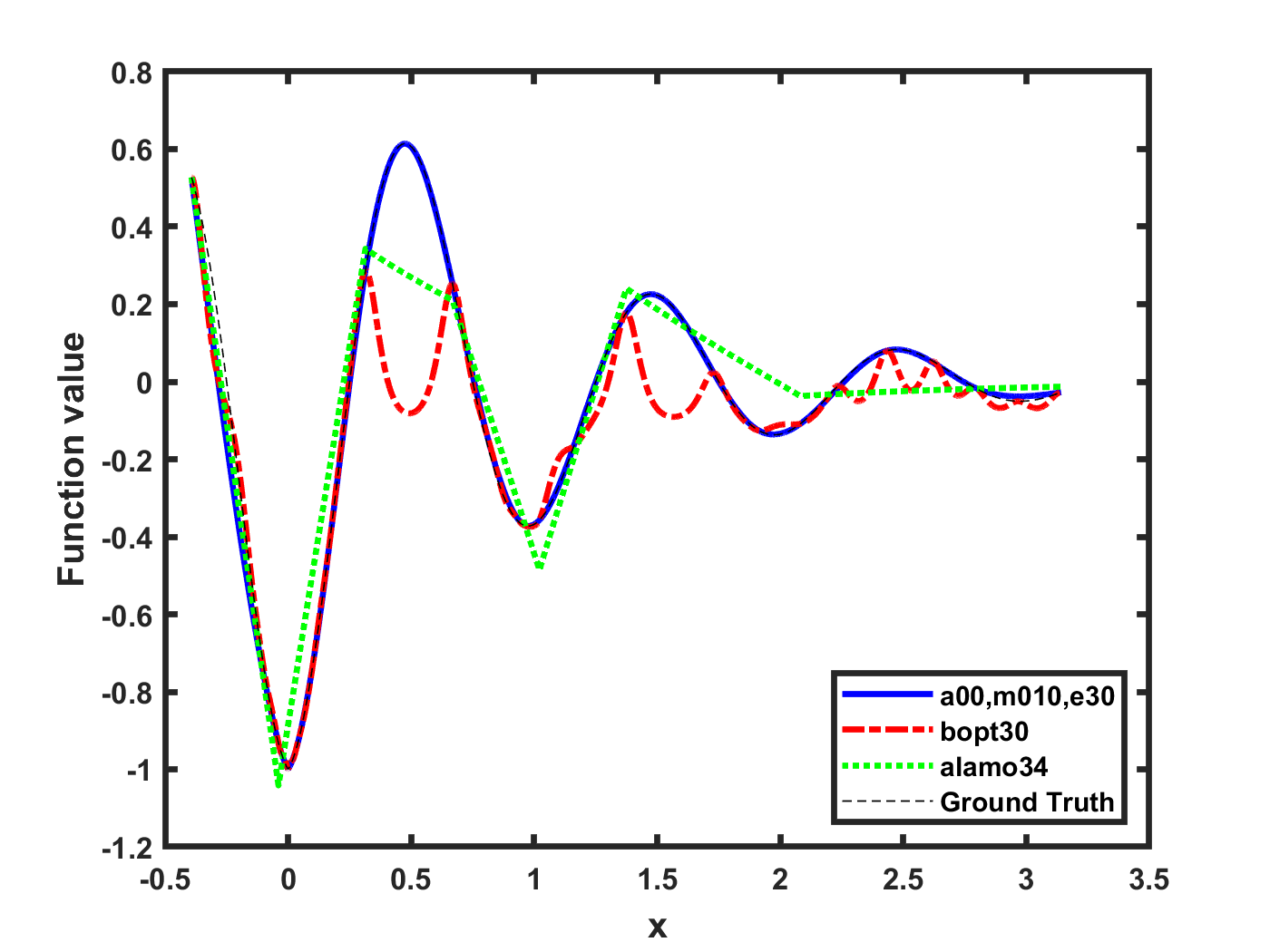}
    \caption{$E^{\max,\textrm{total}} = 30$}
    \label{fig:damped_harmonic_oscillator_e30_comparison}
  \end{subfigure}
      \begin{subfigure}[b]{0.33\textwidth}
      \includegraphics[width=\textwidth]{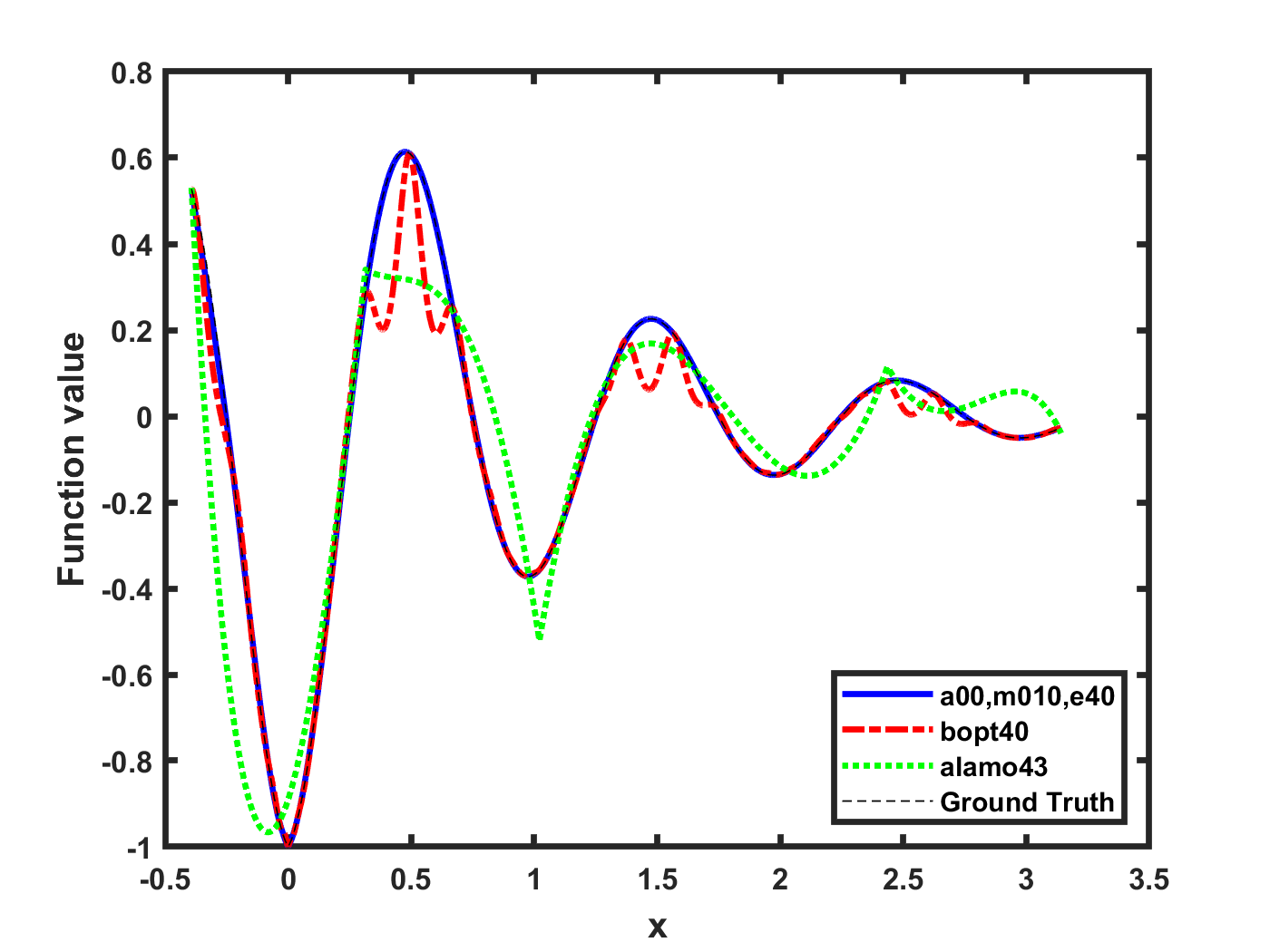}
    \caption{$E^{\max,\textrm{total}} = 40$}
      \label{fig:damped_harmonic_oscillator_e40_comparison}
    \end{subfigure}
  \caption{Surrogate comparison, on the damped harmonic oscillator function, between \texttt{LineWalker-full} (``a00,m010*'', i.e., with parameters $\alpha=0$ and $\mu=0.10$), \texttt{bayesopt} (``bopt*''), and ALAMO given a minimum budget of 20, 30, and 40 function evaluations.
  \texttt{LineWalker-full} effectively coincides with the ground truth in all cases, while the others do not.  \texttt{LineWalker-full} and \texttt{bayesopt} use the precise budget allotted, while ALAMO sometimes required additional function evaluations. The two-digit number following each legend name indicates the total number of function evaluations used.}
  \label{fig:damped_harmonic_oscillator_surrogate_comparison}
\end{figure}

Figure~\ref{fig:surrogate_quality_TASE} attempts to quantify the approximation quality of each method using a metric called Total Absolute Scaled Error (TASE). 
Figure~\ref{fig:surrogate_quality_TASE} depicts the mean TASE over all 20 benchmark functions.
TASE is based on the Mean Absolute Scaled Error (MASE) metric, proffered in \cite{hyndman2006another}, which has garnered considerable attention in the forecasting community.  TASE is a scale-free error metric that relates the total maximum absolute error of one method (e.g., a new or superior method) to that of a reference method. Like MASE, by being scale-free, it allows one to compare algorithmic performance across different benchmark functions whose scale may vary significantly. Loosely speaking, TASE allows one to quantify how much of the (total absolute) error of a reference method is explained by another, typically superior method.

\begin{figure}[h!] 
\centering
\includegraphics[width=12cm]{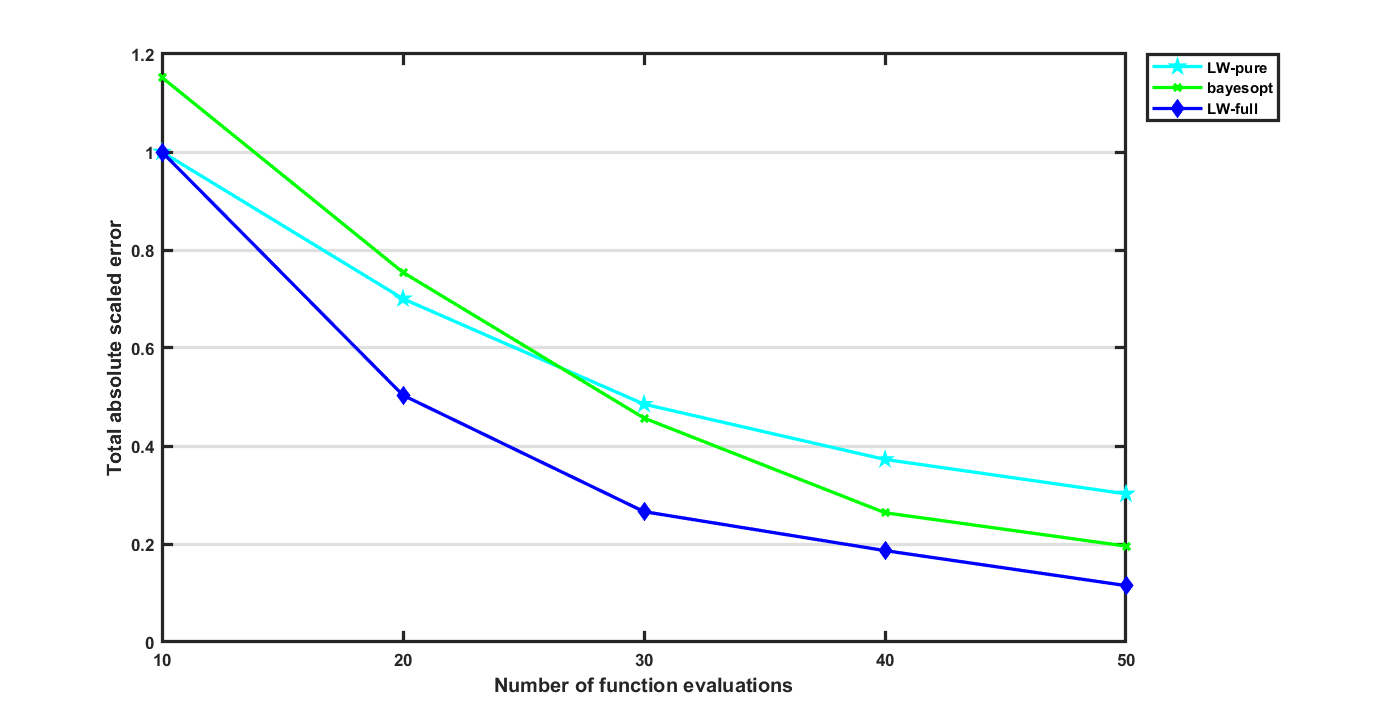}
\caption{Average total absolute scaled error (see Definition~\ref{def:TASE}) as a function of the number of black box evaluations. }
\label{fig:surrogate_quality_TASE}
\end{figure} 

Using notation similar to that given in \cite{hyndman2006another}, let
\begin{equation}
e_{i}^m = \fhat_i - \ftrue_i
\end{equation}
denote the error at grid index $i \in \mc{I}$ of an approximate function $\fhat$ trained (obtained) using method $m$ with exactly $E$ function evaluations. A more verbose notation for $e_{i}^m$ is $e_{i}^m(\ftrue,E^m)$ to emphasize that the error depends on the true function $\ftrue$ and a surrogate constructed from method $m$'s $E^m$ function evaluations.
The total absolute scaled error of method $m$ for the true function $\ftrue$, trained using exactly $E^m$ function evaluations, relative to a reference method ``ref'' trained using $E^{\textrm{ref}}$ function evaluations, is defined as
\begin{equation} \label{def:TASE}
\textrm{TASE}_{\ftrue,E^m,E^{\textrm{ref}}}^m = 
\frac{\sum_{i \in \mc{I}}|e_{i}^m|}{\sum_{i \in \mc{I}}|e_{i}^{\textrm{ref}}|}. 
\end{equation}
The denominator sums the absolute error associated with a reference (typically, na\"{i}ve) method using $E^{\textrm{ref}}$ function evaluations, and 
the numerator sums the absolute error associated with a new method $m$ using $E^m$ function evaluations. 
The mean TASE for method $m$, using $E^m$ function evaluations, with respect to a reference method using $E^{\textrm{ref}}$ function evaluations is 
\begin{equation}
\overline{\textrm{TASE}}_{E^m,E^{\textrm{ref}}}^m 
= \tfrac{1}{20} \sum_{\ftrue} \textrm{TASE}_{\ftrue,E^m,E^{\textrm{ref}}}^m 
\end{equation}
In our experiments, our reference ``method'' is the fit provided by solving the system of linear equations in \eqref{eq:least_squares_linear_system} (which is the same as \texttt{LineWalker-pure} and \texttt{LineWalker-full}) using only $E^{\textrm{ref}}=11$ function evaluations. Recall that these are the same 11 uniformly-spaced samples used to initialize both \texttt{LineWalker} methods and \texttt{bayesopt}. Thus, we always have $E^m \geq E^{\textrm{ref}}$ so that the surrogate used in the numerator always has at least as many function evaluations as the method used in the denominator.  The reference method has the same $\alpha$ and $\mu$ parameters as all \texttt{LineWalker} methods. 

With only 11 samples (the same samples used for each method), both \texttt{LineWalker} algorithms construct a better overall fit than \texttt{bayesopt}, which is why \texttt{bayesopt}'s TASE is nearly 1.2 times that of both \texttt{LineWalker} methods.  As all methods sample more points, their fit improves revealing an across-the-board reduction in TASE.  The fact that \texttt{LineWalker-pure}'s TASE reductions are less rapid than those of \texttt{LineWalker-full} indicates that the additional enhancements in \texttt{LineWalker-full} are responsible for accelerating the fit improvement. This result should not be surprising as \texttt{LineWalker-pure} is essentially an exploitation method, sampling extrema of the current approximation until the approximation fails to identify new extrema to sample. In contrast, \texttt{LineWalker-full} incorporates an explicit exploration component, as well as a simple tabu structure, to promote more diverse sampling, which ultimately leads to better approximation quality.

\subsection{LineWalker hyperparameters and their impact}

Our \texttt{LineWalker} algorithms depend on the hyperparameters listed in Table~\ref{table:linewalker_param_values}, the most influential of which are $N$, $\alpha$, and $\mu$, as they govern the surrogate quality, much like the kernel in Gaussian Process Regression.
Recall that $N$ denotes the number of grid indices, while $\alpha$ and $\mu$ regulate the penalty on the first- and second-derivative changes (computed via finite differences), respectively. 
We obtain the same results using $N=10,000$ grid indices for all test instances with a small increase in computation time (i.e., tens of seconds -- see Figure~\ref{fig:cpu_time_per_iteration_shekel}).  We also obtain the same results setting $\mu=0.001$. For a highly nonconvex function and a small optimality tolerance, $N$ should be chosen sufficiently large (e.g., $N=10,000$). Otherwise, our results indicate that a smaller number of grid indices suffices.

The tabu search-related parameters appear to have a less pronounced impact and mainly affect the sampling selection, which ultimately impacts the number of samples needed to declare optimality.
The short-term tabu tenure parameter $\tabuTenureShort$ governs the number of iterations until the neighborhood of a previously sampled surrogate extrema is revisited. If one wishes to exploit more than explore, then one can reduce $\tabuTenureShort$ or turn off the short-term tabu tenure completely.

\section{Conclusions and Future Work} \label{sec:conclusions}

Contrary to what some DFO aficionados may profess, this work has shown that one-dimensional, deterministic line search for nonconvex black box functions is not a solved problem and that there is still room for improvement. In doing so, we introduced two line search algorithms -- \texttt{LineWalker-pure} and \texttt{LineWalker-full} -- that chase a novel surrogate's extrema to guide the sampling strategy. Whereas \texttt{LineWalker-pure} behaves in a highly ``exploitative'' manner, which may lead to oversampling in a particular neighborhood, \texttt{LineWalker-full} incorporates additional tabu search concepts to induce a balance between exploration and exploitation. Somewhat surprisingly, even our most na\"{i}ve \texttt{LineWalker-pure} implementation is superior to NOMAD, a leading state-of-the-art DFO solver. Off-the-shelf methods like \texttt{fminbound} and \texttt{fminsearch} (Nelder-Mead simplex) perform rather poorly on our benchmark suite of (mostly) complicated nonconvex functions. Bayesian optimization proves to be the most competitive with our enhanced \texttt{LineWalker-full} method, while the latter yields, on average, superior surrogate approximations for any fixed number of function evaluations.

Empirical evidence suggests that our underlying ``extrema hunting'' philosophy, a cornerstone of our \texttt{LineWalker} algorithms, is a sensible strategy. Thus, perhaps unexpectedly, even if our goal is to find a global minimum, we show that there are some benefits to sampling maxima of our surrogate as these samples improve the surrogate quality (see, e.g., Figure~\ref{fig:out_holder_1Dslice}) and ultimately guide the algorithm to better samples. In other words, our surrogate's extrema appear, at least empirically, to be more ``information-rich'' than other unexplored samples.
 

Given LineWalker's success in one dimensional function approximation, it is natural to ask how the method scales to higher dimensions. Clearly, the discretization, which affords considerable flexibility in one dimension, is the algorithm's Achilles heel as it quickly becomes prohibitive in even two or three dimensions.  For example, given $N=10^4$, one would need $N^2 = 10^8$ grid points in $\Re^2$ alone. Moreover, the tabu structures must be modified to explicitly define neighborhoods and what it means to ``sample around the bend.''  Bayesian optimization and basis function approaches are better-suited for higher dimensional surrogates.

Another potential criticism of our approach is that we make no attempt to offer (deterministic or probabilistic) bounds on our surrogate.  This choice stems from our assumption that no prior information is known beyond the presence of a deterministic smooth (or mostly smooth) function. 
In contrast, if one assumes that more information is available, then Bayesian optimization, for example, may be a worthy choice as one can assume a prior distribution on the data. Even in the one-dimensional setting, however, this assumption may lead to overconfidence. For example, inspecting the interval $[2.5,3]$ in Figure~\ref{fig:out_schaffer2A_1Dslice}(a), the lower bounds suggested by \texttt{bayesopt} are clearly wrong because it has not sampled the interior of this interval to detect a better solution. Instead of using an acquisition function to balance exploration and exploitation, our algorithms identify all non-tabu extrema of the surrogate and sort them increasing order.  If no non-tabu candidates are identified, then we simply find the largest unexplored interval, breaking ties by choosing the one with the smallest (in terms of objective function value) endpoints.

There are numerous opportunities for future work and extensions.
One could 
(i) incorporate other information, e.g., Lipschitz constant bounds, into the objective function \eqref{model:unconstrained_opt_for_approximate_f} or constraints of \eqref{model:constrained_nonlinear_regression2}; 
(ii) sample saddle points, not just extrema, of the surrogate; 
(iii) randomize the \texttt{sort()} function in Step~\ref{step:sort_points_in_ascending_order} of Algorithm~\ref{algo:segment_search_1D_budgetLimited} so that candidate samples are sorted in a random order and thus the non-tabu candidate with the lowest approximate objective function value is not always selected first;
(iv) pursue an ensemble approach in which multiple surrogates are simultaneously constructed and used to generate multiple candidate samples. For example, one could choose different values for the regularization parameter value $(\alpha,\mu)$. One would have to determine the weight to ascribe to each surrogate to determine which candidate to select for sampling. 

\small
\bibliographystyle{plainnat}
\bibliography{linewalker_refs}

\newpage
\section{Appendix}

In this section, we offer detailed supplementary material to give a more holistic picture of our approach.
Section~\ref{sec:benchmark_suite_of_functions} furnishes detailed information about our benchmark suite of functions.
Section~\ref{sec:cpu_time_per_iteration} shows the CPU time per iteration for \texttt{LineWalker-full} as a function of the number $N$ of grid indices.
Section~\ref{sec:alamo} explains our experiments with ALAMO.
Section~\ref{sec:linewalker_vs_bayesopt} showcases a detailed visual comparison between \texttt{bayesopt} and \texttt{LineWalker-full}.

\subsection{Benchmark suite of functions} \label{sec:benchmark_suite_of_functions}

Table~\ref{tab:test_fnc_gallery_table} categorizes these test functions based on their shape and number of extrema.
Tables~\ref{tab:test_fnc_functional_form_table} and \ref{tab:test_fnc_functional_form_table2} provide the precise analytical form of each test function, along with the set of global minimizers and minimum objective function values.

\begin{table} [h!]
\centering														
\begin{tabular}{p{4cm}p{3cm}cccc}																
\toprule												
	\textbf{Function Name}	&	\textbf{Category}	&	\textbf{Periodic}	&	\textbf{Nonsmooth}	&	\textbf{\# Max}	&	\textbf{\# Min}	\\
\hline												
\href{http://www.sfu.ca/~ssurjano/ackley.html}{	ackley	} &	Many local extrema	&	$\checkmark$	&		&	48	&	49	\\
{\footnotesize \href{http://hyperphysics.phy-astr.gsu.edu/hbase/oscda.html}{	damped harmonic oscillator	}} &	Few local extrema	&	$\checkmark$	&		&	4	&	3	\\
\href{http://www.sfu.ca/~ssurjano/dejong5.html}{	dejong5	} &	Steep ridges	&		&		&	6	&	5	\\
\href{http://www.sfu.ca/~ssurjano/grlee12.html}{	grlee12	} &	Many local extrema	&		&	$\checkmark$	&	7	&	6	\\
\href{http://www.sfu.ca/~ssurjano/langer.html}{	langer	} &	Many local extrema	&		&		&	16	&	16	\\
\href{http://www.sfu.ca/~ssurjano/michal.html}{	michal	} &	Steep ridges	&		&		&	7	&	10	\\
\href{http://infinity77.net/global_optimization/test_functions_nd_P.html#go_benchmark.Plateau}{	plateau	} &	Bowl-shaped	&		&	$\checkmark$	&	Infinite	&	Infinite	\\
\href{http://www.sfu.ca/~ssurjano/rastr.html}{	rastrigin	} &	Many local extrema	&	$\checkmark$	&		&	7	&	6	\\
	sawtoothD	&	Many local extrema	&	$\checkmark$	&	$\checkmark$	&	10	&	10	\\
\href{http://www.sfu.ca/~ssurjano/schwef.html}{	schwefel	} &	Many local extrema	&	$\checkmark$	&		&	7	&	7	\\
\href{http://www.sfu.ca/~ssurjano/stybtang.html}{	stybtang	} &	Simple	&		&		&	2	&	1	\\
\href{http://www.sfu.ca/~ssurjano/zakharov.html}{	zakharov	} &	Plate-shaped	&		&		&	1	&	0	\\
	Easom-Schaffer2A	&	Few local extrema	&		&		&	3	&	4	\\
\href{http://www.sfu.ca/~ssurjano/egg.html}{	egg2	} &	Many local extrema	&		&	$\checkmark$	&	7	&	8	\\
\href{http://www.sfu.ca/~ssurjano/holder.html}{	holder	} &	Many local extrema	&	$\checkmark$	&	$\checkmark$	&	6	&	7	\\
\href{http://www.sfu.ca/~ssurjano/langer.html}{	langer2	} &	Many local extrema	&		&		&	7	&	7	\\
\href{http://www.sfu.ca/~ssurjano/levy.html}{	levy	} &	Few local extrema	&	$\checkmark$	&		&	4	&	5	\\
\href{http://www.sfu.ca/~ssurjano/levy13.html}{	levy13	} &	Many local extrema	&	$\checkmark$	&		&	14	&	15	\\
\href{http://www.sfu.ca/~ssurjano/schaffer2.html}{	schaffer2A	} &	Few local extrema	&		&	$\checkmark$	&	4	&	4	\\
\href{http://www.sfu.ca/~ssurjano/shekel.html}{	shekel	} &	Few local extrema	&		&		&	3	&	4	\\
\bottomrule																
\end{tabular}																
\caption{Gallery of test functions. Category: This classification is taken from \cite{surjanovic2013virtual} with some amendments and additions where appropriate.
Periodic: There is regularity/frequency to the spacing between the extrema over the entire domain.
Nonsmooth: Does not possess continuous derivatives over the domain.
All functions are lower semi-continuous, except plateau.
The number of local minima (\# Min) and maxima (\# Max) excludes endpoints.}
  \label{tab:test_fnc_gallery_table}
\end{table}

\begin{landscape}
\begin{table}
\centering
\begin{tabular}{p{2cm} ccrr}														
\toprule															
	\textbf{Function Name}	&	\textbf{$f(x)$}	&				\textbf{Domain}		&	\textbf{$x^*$}	&	\textbf{$f(x^*)$}	\\	
\hline															
\href{http://www.sfu.ca/~ssurjano/ackley.html}{	ackley	} &	$-20 \exp(-0.2 x) -\exp(\cos(2\pi x)) + 20 + \exp(1)$	&	$[	-17	,	32	$]	&	0	&	0	\\	[10pt]
{\scriptsize \href{http://hyperphysics.phy-astr.gsu.edu/hbase/oscda.html}{	damped harmonic oscillator	}} &	$-\exp(-|x|)\cos(2\pi |x|)$	&		$[-\pi/8	,	\pi]$		&	0	&	-1	\\	[10pt]
\href{http://www.sfu.ca/~ssurjano/dejong5.html}{	dejong5	} &	$\begin{array}{l} \Bigg( .002 + \sum_{i=1}^{25} \frac{1}{i + (x_1 - a_{1i})^6 + (x_2 - a_{2i})^6} \Bigg)^{-1} \textrm{ where } \\ \v{a}= {\tiny \begin{bmatrix} -32 & -16 &   0 &  16 &  32 & -32 & \cdots &  0 & 16 & 32 \\ -32 & -32 & -32 & -32 & -32 & -16 & \cdots & 32 & 32 & 32 \\ \end{bmatrix}} \end{array}$	&	$[	-65.536	,	65.536	$]	&	$-31.97600$	&	0.99800	\\	[30pt]
\href{http://www.sfu.ca/~ssurjano/grlee12.html}{	grlee12	} &	$\left\{ \begin{array}{cl} \frac{\sin(10\pi x^{1.10})}{2x} + (x-1)^4 + 5 & \mbox{if $x < 0.71$} \\  \frac{\sin(10\pi x^{0.75})}{2x} + (x-1)^4 + 5 & \mbox{if $0.71 \leq x \leq 0.86$} \\   \frac{\sin(10\pi x^{0.75})}{2x} + (x-1)^4 + 1 & \mbox{if $x > 0.86$} \\ \end{array} \right.$	&	$[	0.5	,	2.5	$]	&	0.76879	&	$-0.64708$	\\	[30pt]
\href{http://www.sfu.ca/~ssurjano/langer.html}{	langer	} &	$\begin{array}{l} \sum_{i=1}^{m=5} c_i \exp\Big( -\tfrac{1}{\pi}(x-a_i)^2 \Big) \cos\Big( \pi(x-a_i)^2 \Big), \\ \textrm{ where } \v{c}=(1, 2, 5, 2, 3), \v{a}=(3, 5, 2, 1, 7) \\ \end{array}$	&	$[	0	,	10	$]	&	6.00295	&	$-3.66452$	\\	[20pt]
\href{http://www.sfu.ca/~ssurjano/michal.html}{	michal	} &	$-\sin(x)\sin^{20}\Big(\tfrac{x^2}{\pi}\Big)$	&	$[	0	,	13	$]	&	8.00922	&	$-0.98795$	\\	[10pt]
\href{http://infinity77.net/global_optimization/test_functions_nd_P.html#go_benchmark.Plateau}{	plateau	} &	$| \lfloor x \rfloor | + | \lfloor 2x-3 \rfloor |$	&	$[	-2	,	4	$]	&	[1.5,2)	&	1	\\	[10pt]
\href{http://www.sfu.ca/~ssurjano/rastr.html}{	rastrigin	} &	$10 + x^2 - 10\cos(2\pi x)$	&	$[	-3	,	3	$]	&	0	&	0	\\	[10pt]
	sawtoothD	&	$\left\{ \begin{array}{ll} \tfrac{2}{\pi}\sin^{-1}\big(\sin( \pi x)\big) - |x| + 0 & \mbox{if $x \leq 0$} \\ \tfrac{2}{\pi}\sin^{-1}\big(\sin(3\pi x)\big) - |x| + 1 & \mbox{if $x \in (0,0.75)$} \\ \tfrac{2}{\pi}\sin^{-1}\big(\sin( \pi x)\big)       - 6 & \mbox{if $x \in [0.75,1]$} \\ \tfrac{2}{\pi}\sin^{-1}\big(\sin(3\pi x)\big) - |x| + 1 & \mbox{if $x \in (1,3.25)$} \\ \tfrac{2}{\pi}\sin^{-1}\big(\sin( \pi x)\big) - |x| + 1 & \mbox{if $x \geq 3.25$} \\ \end{array} \right.$	&	$[	-5	,	5	$]	&	1	&	$-6$	\\	[40pt]
\bottomrule															
\end{tabular}	
\caption{Functional form of 10 test functions. dejong5 is a 2-dimensional function, which we evaluate on the line segment $\{(x_1,x_2): x_1=x_2, x_1 \in [-65.536,65.536]\}$.
plateau is a 2-dimensional function $f(\v{x})=\sum_{i=1}^2 | \lfloor x_i \rfloor |$, which we evaluate on the 2D line segment with endpoints $\v{x}_1=(-2,-7)$ and $\v{x}_2=(4,5)$ or equivalently on the 2D domain $\{(x_1,x_2) : x_2=2x_1 - 3, x_1 \in [-2,4]\}$.
}
  \label{tab:test_fnc_functional_form_table}													
\end{table}

\begin{table}
\centering
\begin{tabular}{p{2cm} ccrr}														
\toprule															
	\textbf{Function Name}	&	\textbf{$f(x)$}	&				\textbf{Domain}		&	\textbf{$x^*$}	&	\textbf{$f(x^*)$}	\\	
\hline															
\href{http://www.sfu.ca/~ssurjano/schwef.html}{	schwefel	} &	$418.9829 - x\sin(\sqrt{|x|})$	&	$[	-500	,	500	$]	&	420.96870	&	\num{1.27278e-05}	\\	[10pt]
\href{http://www.sfu.ca/~ssurjano/stybtang.html}{	stybtang	} &	$\tfrac{1}{2}(x^4 - 16x^2 + 5x)$	&	$[	-5	,	5	$]	&	$-2.903534$	&	$-39.16599$	\\	[10pt]
\href{http://www.sfu.ca/~ssurjano/zakharov.html}{	zakharov	} &	$\tfrac{3}{2} x^2 + \tfrac{1}{2} x^4$	&	$[	-5	,	10	$]	&	0	&	0	\\	[10pt]
	easom-schaffer2A	&	$\begin{array}{ll}
    -2\cos^2(w) \exp\Big(-2(w-\pi)^2\Big), \textrm{where } w=x-25 & \textrm{if } x \geq 0 \\
    -0.5 - \frac{\sin^2(w^2) - 0.5}{(1 + 0.001 w^2)^2} - 0.1 |w|, \textrm{where } w=0.3x & \textrm{o.w.}
\end{array}$	&	$[	-10	,	30	$]	&	28.14363	&	-2	\\	[10pt]
\href{http://www.sfu.ca/~ssurjano/egg.html}{	egg2	} &	$-(x+47) \sin\Big(\sqrt{\big\vert x+\tfrac{x^{1/3}}{2}+47 \big\vert}\Big)-x \sin\Big(\sqrt{\big\vert x^{2/3}-47\big\vert}\Big)$	&	$[	-600	,	200	$]	&	-559.35187	&	-518.98768	\\	[10pt]
\href{http://www.sfu.ca/~ssurjano/holder.html}{	holder	} &	$- \Big\vert \sin(x)\cos(x)+\exp\Big(\big\vert 1 - \frac{\sqrt{2x^2}}{\pi} \big\vert \Big) \Big\vert$	&	$[	0	,	11	$]	&	10.32006	&	-18.69332	\\	[10pt]
\href{http://www.sfu.ca/~ssurjano/langer.html}{	langer2	} &	langer with $\v{a}=(5, 1, 5, 2, 8)$	&	$[	3	,	8	$]	&	4.02921	&	-3.94660	\\	[10pt]
\href{http://www.sfu.ca/~ssurjano/levy.html}{	levy	} &	$\begin{array}{l} \sin^2(\pi w) + (w-1)^2[1+\sin^2(2\pi w)], \\ \textrm{ where } w=1+\tfrac{x-1}{4} \\ \end{array}$	&	$[	-10	,	2	$]	&	1	&	0	\\	[10pt]
\href{http://www.sfu.ca/~ssurjano/levy13.html}{	levy13	} &	$-\sin^2(3\pi x) - (x-1)^2[2+\sin^2(3\pi x)+\sin^2(2\pi x)]$	&	$[	-3	,	2	$]	&	-2.81896	&	-56.48262	\\	[10pt]
\href{http://www.sfu.ca/~ssurjano/schaffer2.html}{	schaffer2A	} &	$-0.5 - \frac{\sin^2(x^2) - 0.5}{(1 + 0.001 x^2)^2} - 0.2 |x|$	&	$[	-2	,	3	$]	&	2.80596	&	-1.55304	\\	[10pt]
\href{http://www.sfu.ca/~ssurjano/shekel.html}{	shekel	} &	$\sum_{i=1}^{10} \Big( \sum_{j=1}^4 (x-C_{ji})^2 + \beta_i \Big)^{-1}$	&	$[	0	,	9	$]	&	4	&	-10.53626	\\	[10pt]
\bottomrule															
\end{tabular}	
\caption{Functional form of remaining 10 test functions.
}
  \label{tab:test_fnc_functional_form_table2}													
\end{table}

\end{landscape}

\subsection{CPU time per iteration} \label{sec:cpu_time_per_iteration}

While we assume that computation time will be dominated by calls to an expensive simulator or oracle, for sake of completeness, we provide some empirical evidence of \texttt{LineWalker-full}'s computation time. Specifically, using the Shekel function, we varied the grid size $N$ from 1k to 14k points. For each value of $N$, we took 50 samples, of which the first 11 are taken before the main loop.  Figure~\ref{fig:cpu_time_per_iteration_shekel} shows that, for $N=$5k, it takes \texttt{LineWalker-full} roughly 0.5 s, and nearly 3 s for $N=$10k. The equation in the figure indicates that the CPU time increases at a cubic rate in $N$, which is not surprising since \texttt{LineWalker-full} must solve a linear system whose computational complexity $O(N^3)$.
\begin{figure}[h!] 
\centering
\includegraphics[width=8cm]{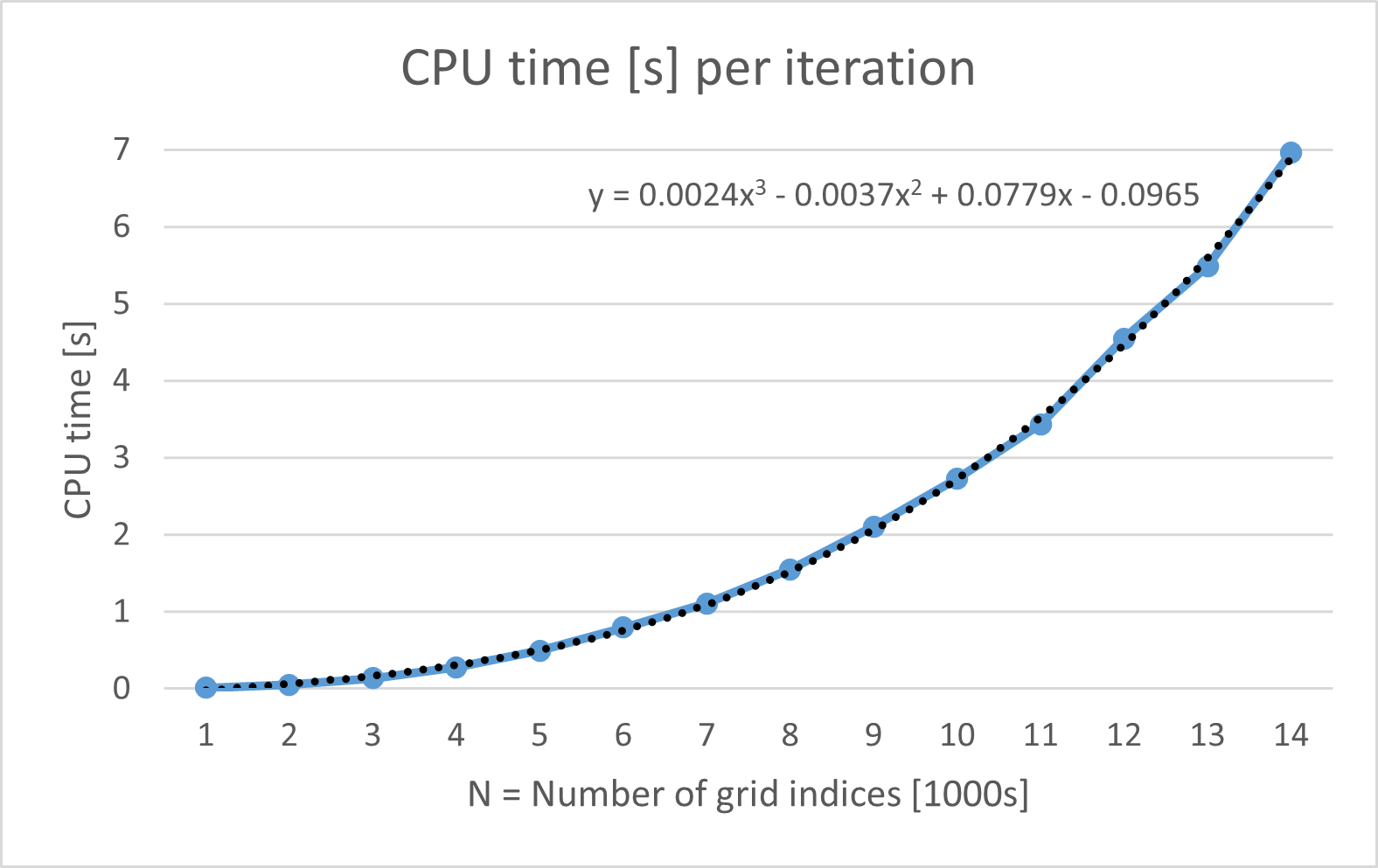}
\caption{CPU time [s] per iteration for \texttt{LineWalker-full} as a function of the number $N$ of grid indices. The polynomial fit equation suggests a cubic relationship in the CPU time per iteration ($y$) and the number of grid indices ($x$). 
}
\label{fig:cpu_time_per_iteration_shekel}
\end{figure} 

\subsection{ALAMO} \label{sec:alamo}
In our experiments with ALAMO version 2021.12.28,
we permitted ALAMO to include the following functions in its surrogate model construction:
constant; linear; logarithmic; exponential; sine; cosine; monomials with powers 0.5, 2, 3, 4, and 5; pairwise combinations of basis functions (see line ``multi2power 2 3 4”); and Gaussian radial basis functions.
Out of fairness, we did not include any custom basis functions as we attempted to mimic the likely assumptions an agnostic user with a truly unknown black box function would do with ALAMO. 
We provided ALAMO with 11 initial samples (the same points given to all methods) from which to begin the surrogate construction.  
We minimized mean square error (MSE) so that ALAMO could perform well in the RMSE metric even though this might lead to surrogates with a larger number of basis functions.
Figure~\ref{fig:alamo_input_file_for_grlee12Step} shows an example ALAMO input file for grlee12Step.

Critical to ALAMO's exploration is the choice of adaptive sampling technique.  We selected the popular DFO method SNOBFIT rather than a random sampler to avoid stochasticity and having to average results.  While this pairing of ALAMO and SNOBFIT worked reasonably well for eight out of the first twelve functions, four functions - ackley, dejong5, langer, and schwefel - caused issues that we do not understand and could not resolve.  Specifically, we could not force SNOBFIT to make additional samples for these four functions.  With no additional samples, the surrogate quality stagnated as no improvements were made based on additional information.  Since we could not resolve this issue (even after trying different objective function criteria, e.g., BIC and AICc), we chose not to compare against ALAMO for these functions. 
Finally, for reasons that we do not understand, we could not force SNOBFIT to evaluate only one new sample point in each iteration even after setting the ALAMO parameter \texttt{maxpoints} to 1.  As a consequence, we permitted ALAMO to have more samples than the other methods to construct its surrogates.  This explains why in  Figure~\ref{fig:damped_harmonic_oscillator_surrogate_comparison}, ALAMO performed 22, 34, and 43 function evaluations when the other methods were given a strict limit of 20, 30, and 40 function evaluations, respectively.

\begin{figure}
\begin{verbatim}
! ALAMO input file for grlee12Step
ninputs 1
noutputs 1
xmin 0.500000
xmax 2.500000
ndata 11
minpoints 1
maxpoints 1
xlabels X1
zlabels Z
expfcns 1
linfcns 1
logfcns 1
sinfcns 1
cosfcns 1
constant 1
monomialpower 0.5 2 3 4 5
multi2power 2 3 4
grbfcns 1
rbfparam 1.0
modeler 5
simulator grlee12Step_alamo.exe
maxiter 10
sampler 2
simin SIMIN
simout SIMOUT
BEGIN_DATA
0.500000 5.930875
0.700040 5.504902
0.900080 0.618779
1.100120 1.329576
1.300160 1.209880
1.500200 0.734344
1.700240 1.340402
1.900280 1.801388
2.100320 2.231071
2.300360 4.043257
2.500000 5.989909
END_DATA
\end{verbatim}
\caption{ALAMO input file for the test function grlee12Step}
\label{fig:alamo_input_file_for_grlee12Step}
\end{figure}

\newpage
\subsection{LineWalker vs. bayesopt: visual comparison of each benchmark function approximation} \label{sec:linewalker_vs_bayesopt}

Figures~\ref{fig:out_ackley}-\ref{fig:out_zakharov} provide a tantalizing visual comparison between \texttt{LineWalker-full} and \texttt{bayesopt} for all 20 functions given a budget $E^{\max,\textrm{total}}$ of 20, 30, 40, and 50 function evaluations.
No legend is given for the \texttt{LineWalker-full} figures; see the Figure~\ref{fig:rastrigin_demo_iter1} legend for details. 
The \texttt{bayesopt} legend requires some explanation. Because the objective function is noise-free (i.e., deterministic), the ``Model mean'' and ``Noise error bars'' coincide and represent the mean of the GPR posterior distribution.  Meanwhile, the ``Model error bars'' show the 95\% confidence bounds for the posterior mean.  
 
In addition to our motivating example shown in Figure~\ref{fig:damped_harmonic_oscillator_surrogate_comparison}
\texttt{LineWalker-full} produces a much better surrogate than \texttt{bayesopt} on several functions:
\begin{itemize}
\item Eason-Schaffer2A (Figure~\ref{fig:out_easom_schaffer2A_1Dslice}): Along the long plateau in the interval $[0,24]$, \texttt{bayesopt} produces a surrogate that looks far more like a ``sagging electric cable transmission line'' than a straight line fit.  
\item Grimacy \& Lee (Figure~\ref{fig:out_grlee12Step}): \texttt{bayesopt} never resolves the local maximum at $x \approx 0.7$ such that, even with 50 function evaluations, the mean predicted objective value at this point according to the posterior distribution is 7, not 6. \texttt{LineWalker-full} resolves this local maximum. 
\item Holder (Figure~\ref{fig:out_holder_1Dslice}): \texttt{LineWalker-full} is far more successful at finding many of the function's peaks than \texttt{bayesopt}.
\item Langer (Figure~\ref{fig:out_langer}): \texttt{LineWalker-full} produces a more accurate approximation than \texttt{bayesopt} at the local extrema near $x=0.5$ (grid index 300) and $x=3.75$ (grid index 1750), and in the interval $[6.5,7.5]$.
\item Levy13 (Figure~\ref{fig:out_levy13_1Dslice}): When the budget $E^{\max,\textrm{total}} \leq 40$, \texttt{bayesopt} underestimates the local maximum around $x=-2.25$ (grid index 700) and, with $E^{\max,\textrm{total}} \leq 20$, struggles in the interval $[-1,2]$.
\item SawtoothD (Figure~\ref{fig:out_sawtoothD}): \texttt{bayesopt} produces a surrogate that looks more like a ``sagging electric cable transmission line'' relative to the true function and what \texttt{LineWalker-full} generates.
\item Shekel (Figure~\ref{fig:out_shekel_1Dslice}): Although \texttt{bayesopt} finds a near global minimum in 20 samples while \texttt{LineWalker-full}, its surrogate is inferior to that of  \texttt{LineWalker-full} when $E^{\max,\textrm{total}} \in \{30,40\}$.
\end{itemize}

\begin{figure} [h!]
\centering
\begin{subfigure}[b]{0.270\textwidth}
\includegraphics[width=\textwidth]{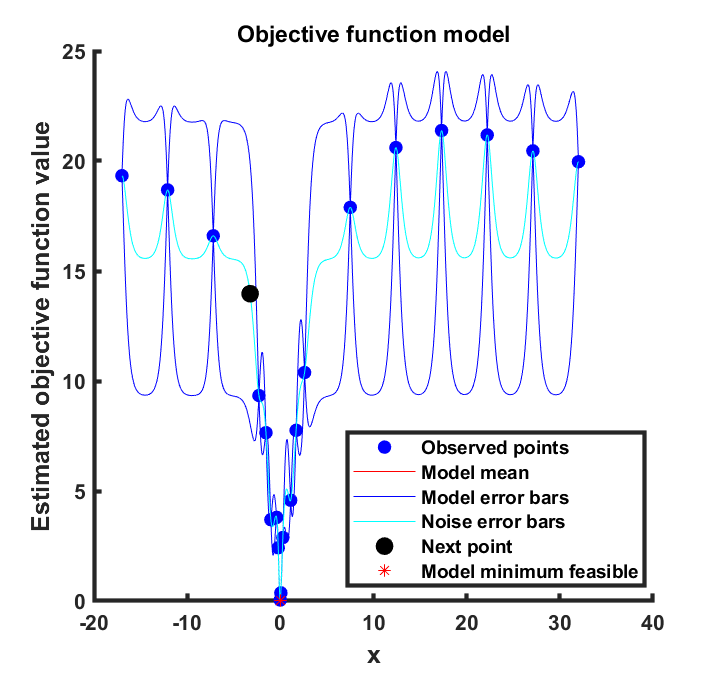}
\caption{bayesopt20}
\end{subfigure}
\begin{subfigure}[b]{0.330\textwidth}
\includegraphics[width=\textwidth]{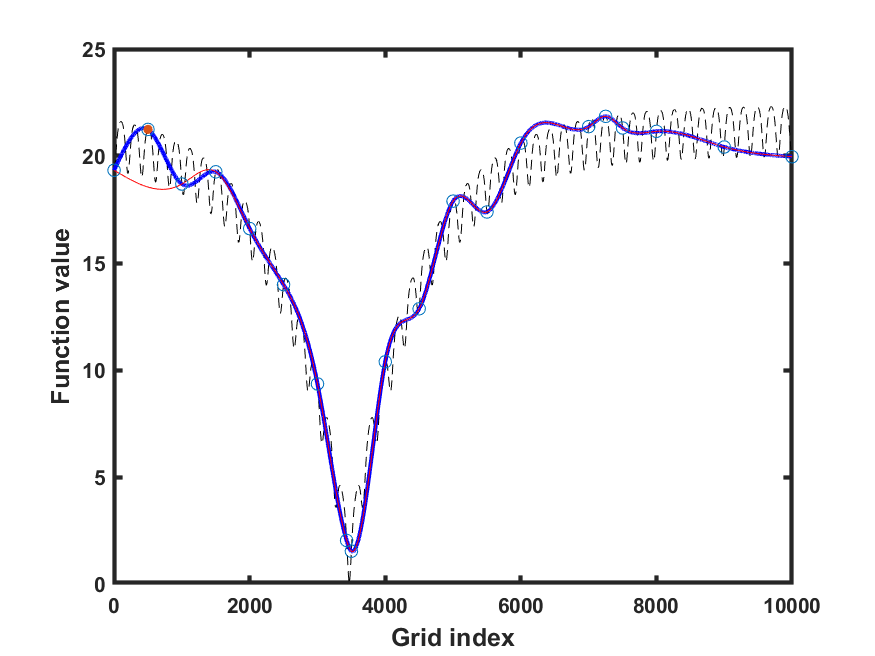}
\caption{LineWalker20}
\end{subfigure}
\newline
\begin{subfigure}[b]{0.270\textwidth}
\includegraphics[width=\textwidth]{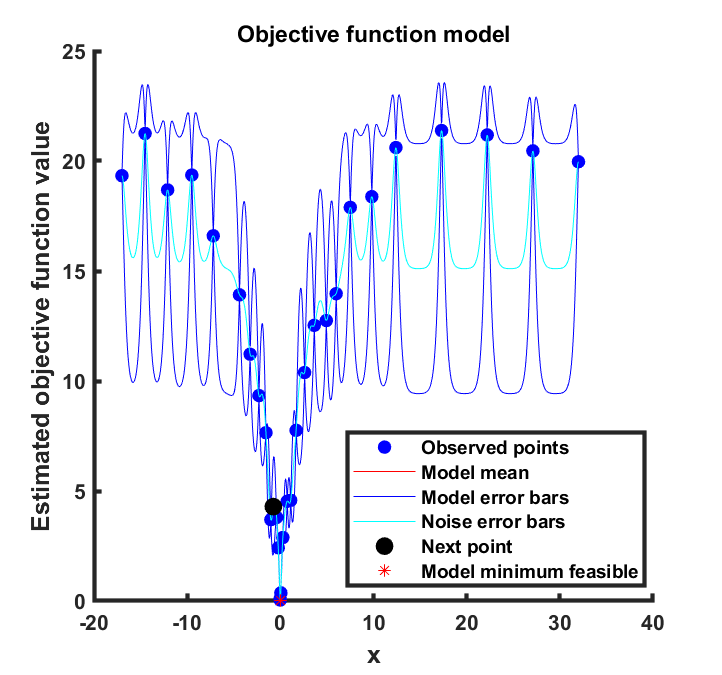}
\caption{bayesopt30}
\end{subfigure}
\begin{subfigure}[b]{0.330\textwidth}
\includegraphics[width=\textwidth]{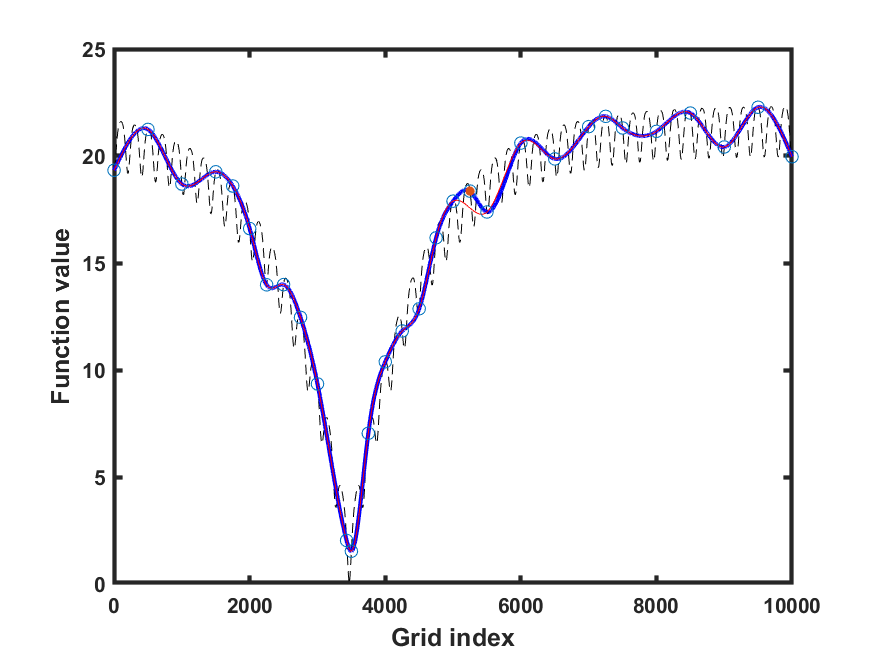}
\caption{LineWalker30}
\end{subfigure}
\newline
\begin{subfigure}[b]{0.270\textwidth}
\includegraphics[width=\textwidth]{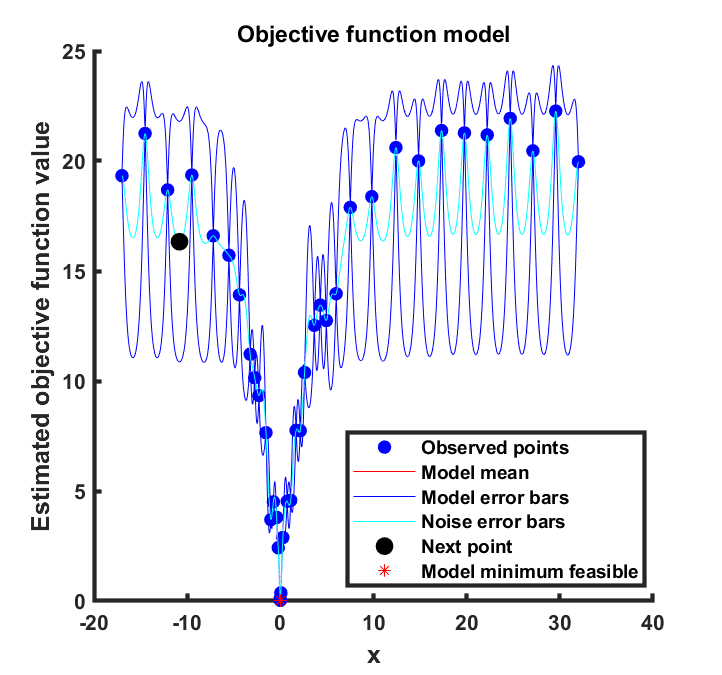}
\caption{bayesopt40}
\end{subfigure}
\begin{subfigure}[b]{0.330\textwidth}
\includegraphics[width=\textwidth]{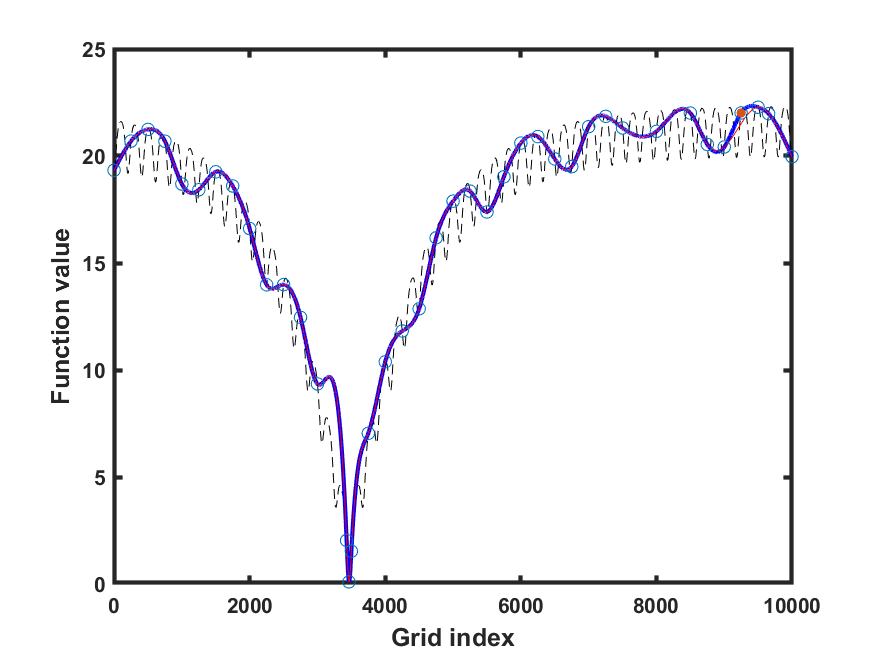}
\caption{LineWalker40}
\end{subfigure}
\newline
\begin{subfigure}[b]{0.270\textwidth}
\includegraphics[width=\textwidth]{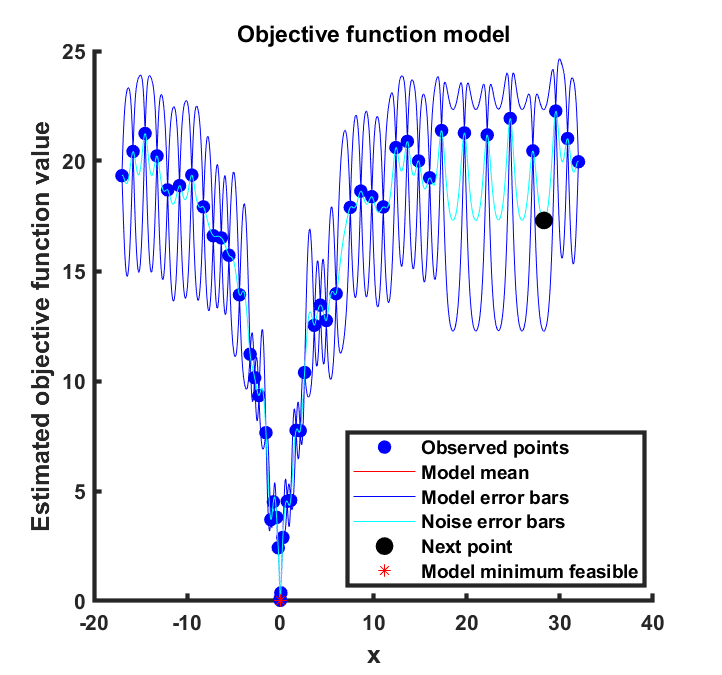}
\caption{bayesopt50}
\end{subfigure}
\begin{subfigure}[b]{0.330\textwidth}
\includegraphics[width=\textwidth]{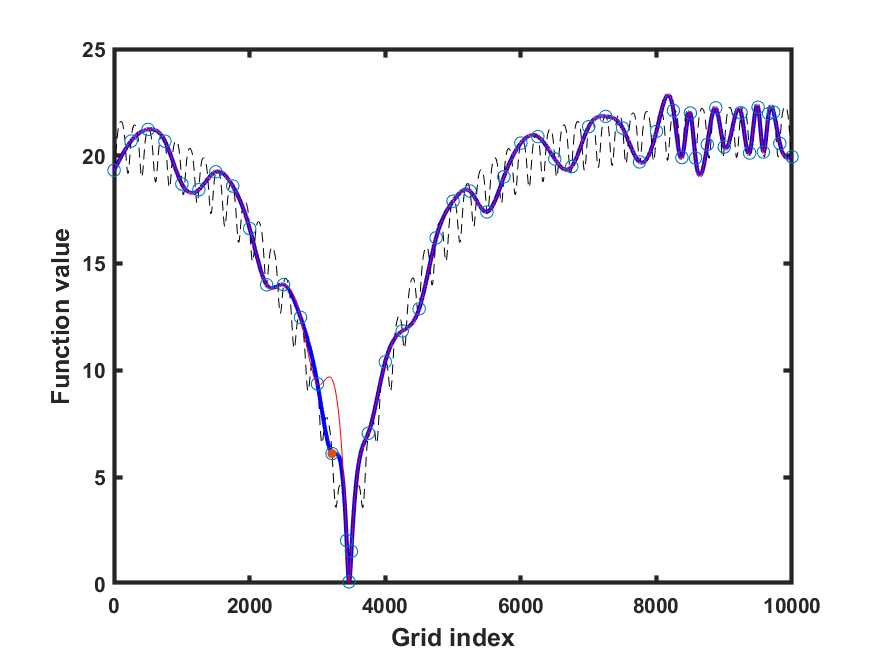}
\caption{LineWalker50}
\end{subfigure}
\newline
\caption{ackley. Left column = \texttt{bayesopt}. Right column = \texttt{LineWalker-full}}
\label{fig:out_ackley}
\end{figure}

\begin{figure} [h!]
\centering
\begin{subfigure}[b]{0.270\textwidth}
\includegraphics[width=\textwidth]{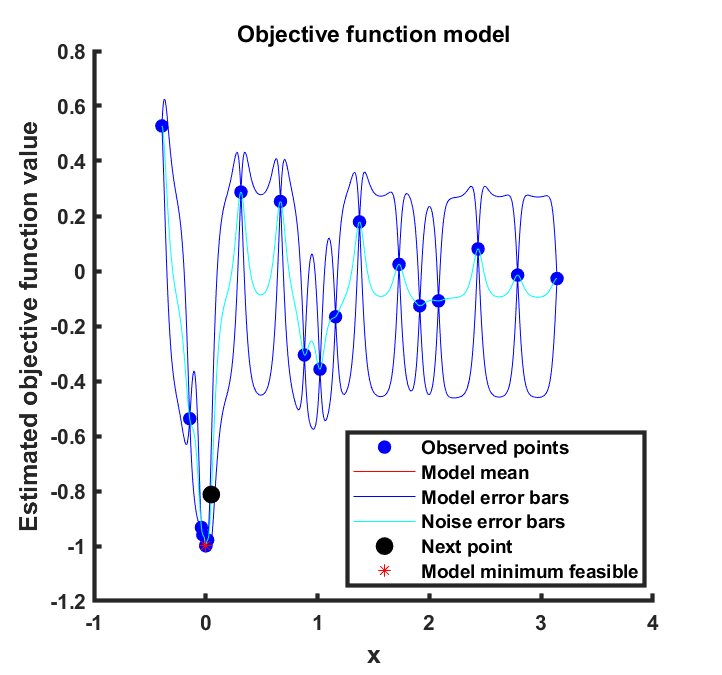}
\caption{bayesopt20}
\end{subfigure}
\begin{subfigure}[b]{0.330\textwidth}
\includegraphics[width=\textwidth]{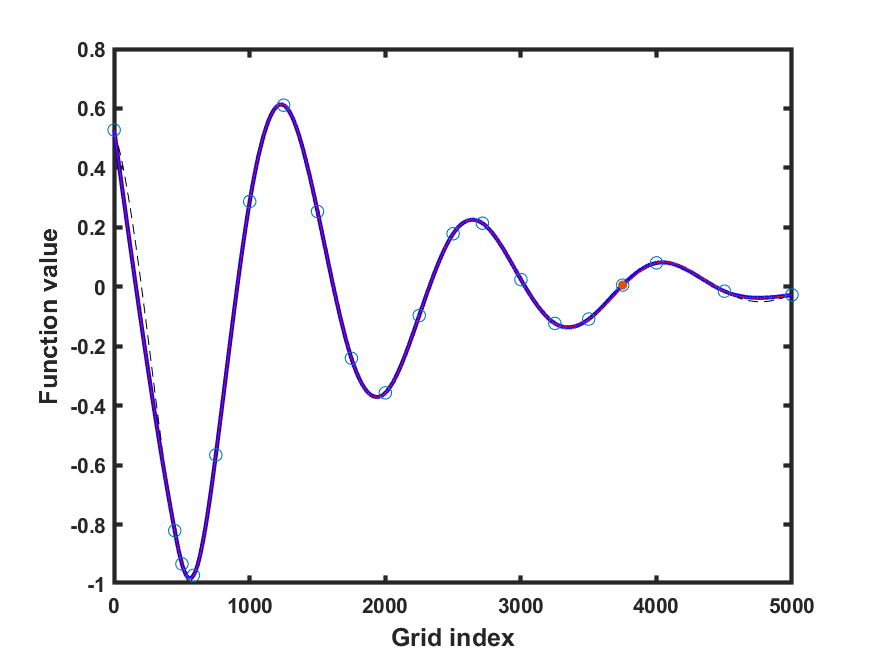}
\caption{LineWalker20}
\end{subfigure}
\newline
\begin{subfigure}[b]{0.270\textwidth}
\includegraphics[width=\textwidth]{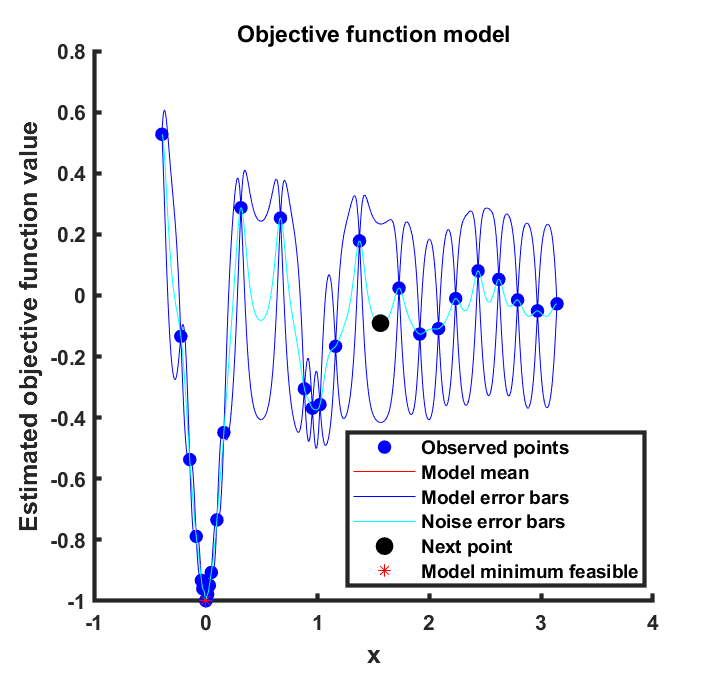}
\caption{bayesopt30}
\end{subfigure}
\begin{subfigure}[b]{0.330\textwidth}
\includegraphics[width=\textwidth]{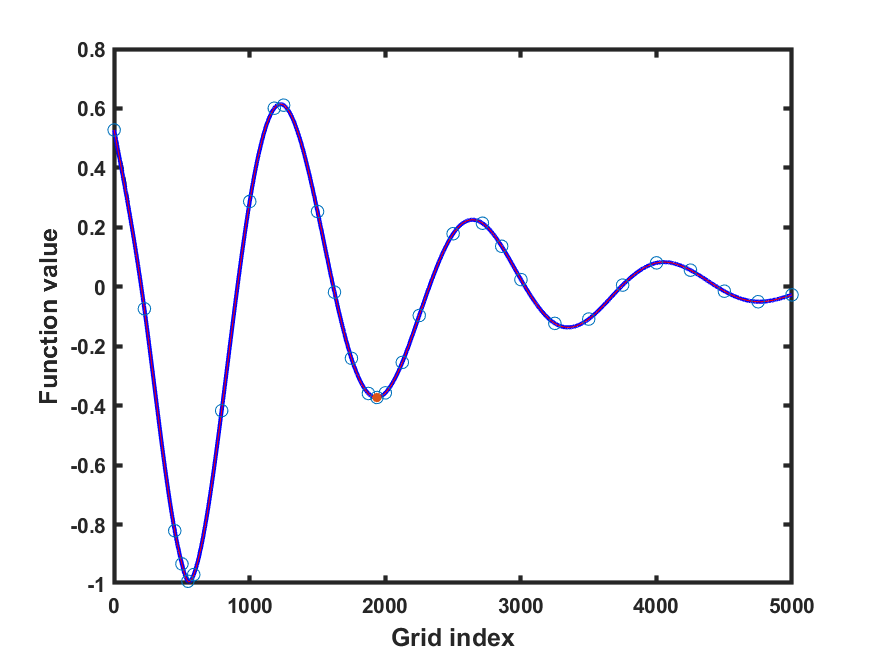}
\caption{LineWalker30}
\end{subfigure}
\newline
\begin{subfigure}[b]{0.270\textwidth}
\includegraphics[width=\textwidth]{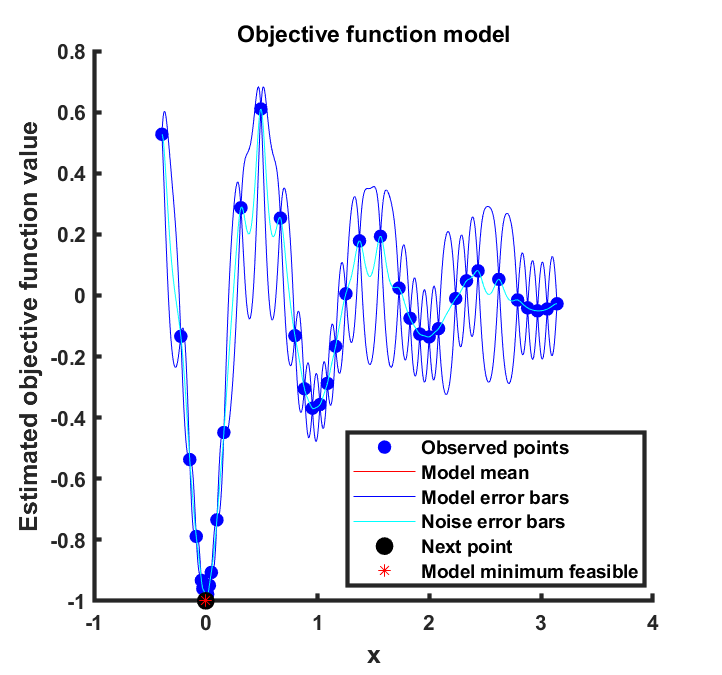}
\caption{bayesopt40}
\end{subfigure}
\begin{subfigure}[b]{0.330\textwidth}
\includegraphics[width=\textwidth]{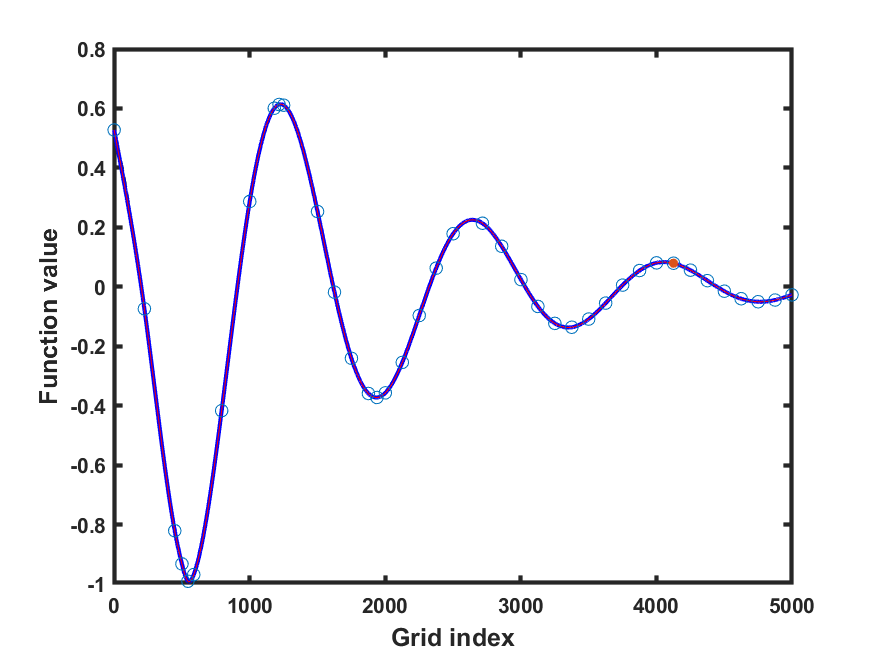}
\caption{LineWalker40}
\end{subfigure}
\newline
\begin{subfigure}[b]{0.270\textwidth}
\includegraphics[width=\textwidth]{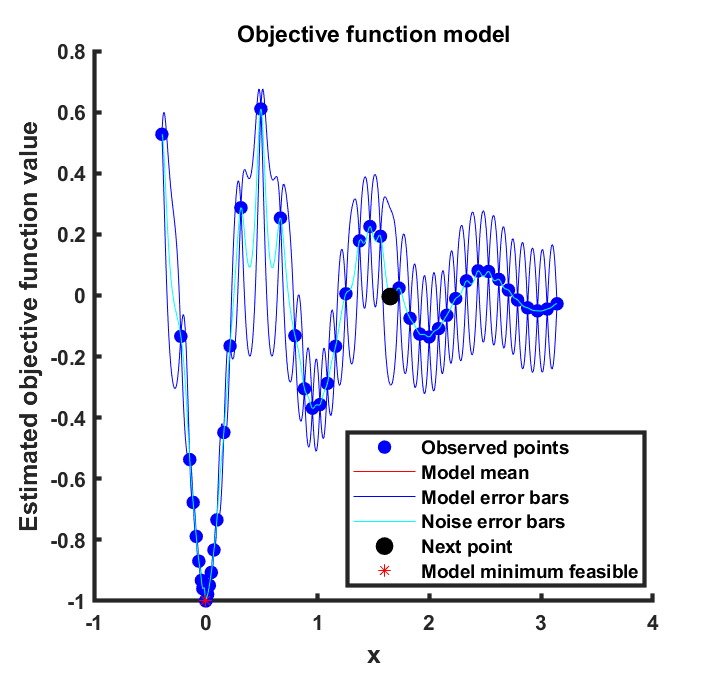}
\caption{bayesopt50}
\end{subfigure}
\begin{subfigure}[b]{0.330\textwidth}
\includegraphics[width=\textwidth]{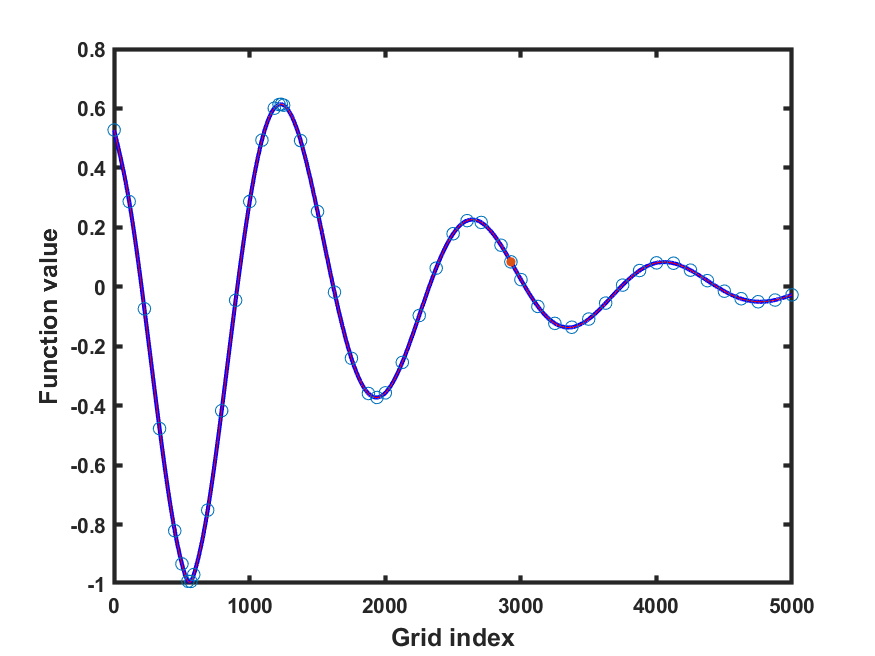}
\caption{LineWalker50}
\end{subfigure}
\newline
\caption{damped harmonic oscillator. Left column = \texttt{bayesopt}. Right column = \texttt{LineWalker-full}}
\label{fig:out_damped_harmonic_oscillator}
\end{figure}

\begin{figure} [h!]
\centering
\begin{subfigure}[b]{0.270\textwidth}
\includegraphics[width=\textwidth]{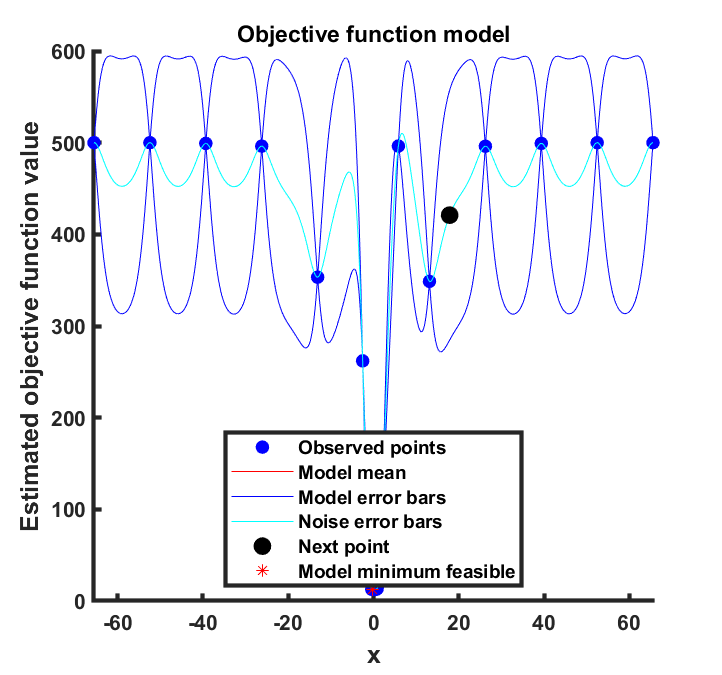}
\caption{bayesopt20}
\end{subfigure}
\begin{subfigure}[b]{0.330\textwidth}
\includegraphics[width=\textwidth]{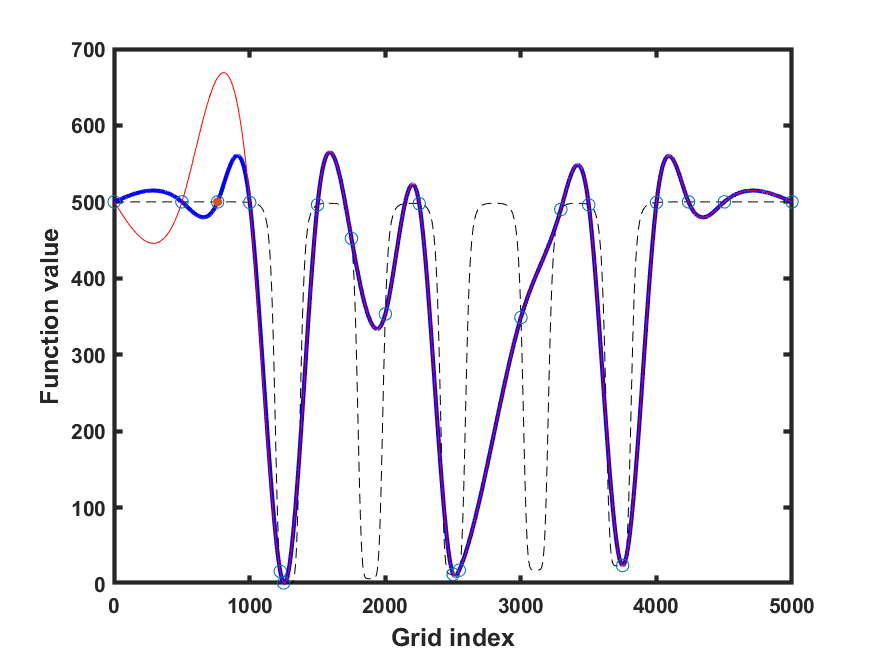}
\caption{LineWalker20}
\end{subfigure}
\newline
\begin{subfigure}[b]{0.270\textwidth}
\includegraphics[width=\textwidth]{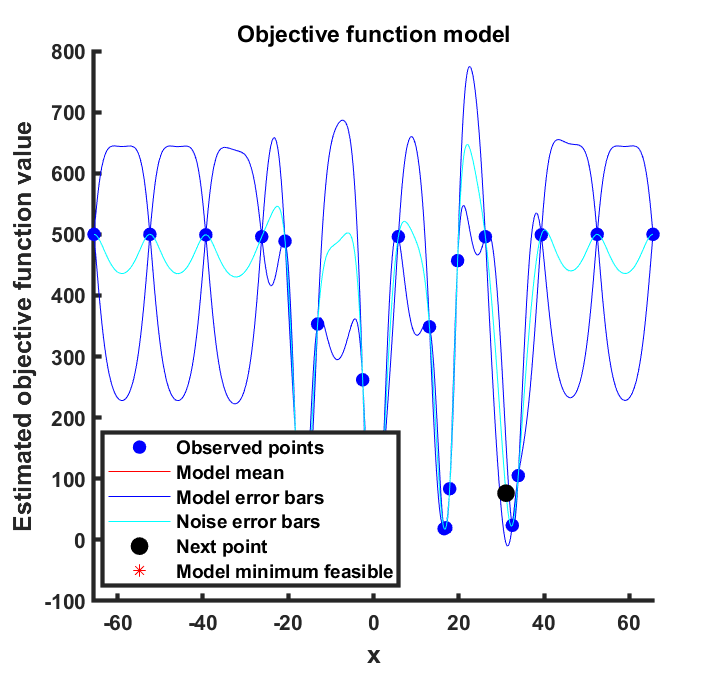}
\caption{bayesopt30}
\end{subfigure}
\begin{subfigure}[b]{0.330\textwidth}
\includegraphics[width=\textwidth]{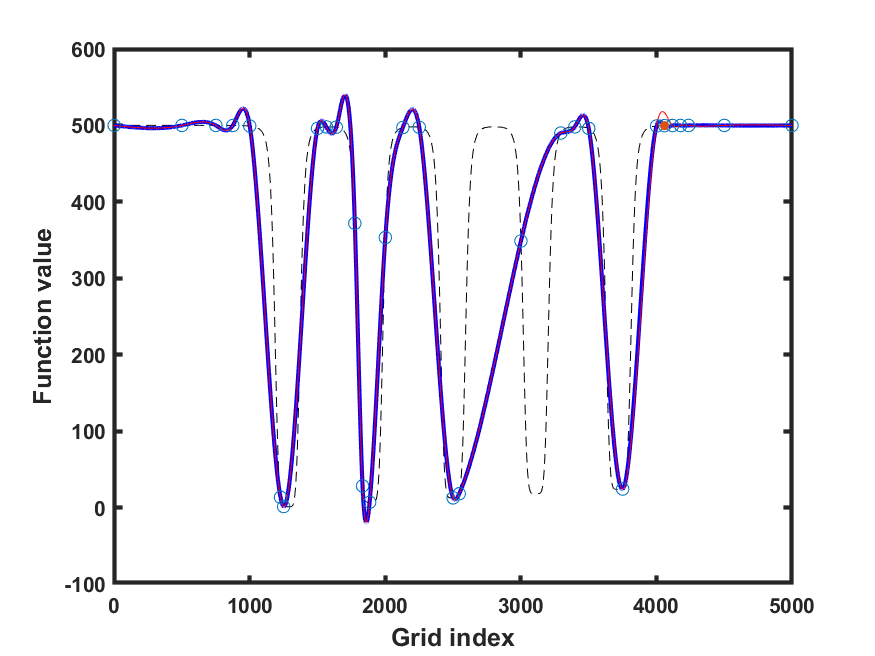}
\caption{LineWalker30}
\end{subfigure}
\newline
\begin{subfigure}[b]{0.270\textwidth}
\includegraphics[width=\textwidth]{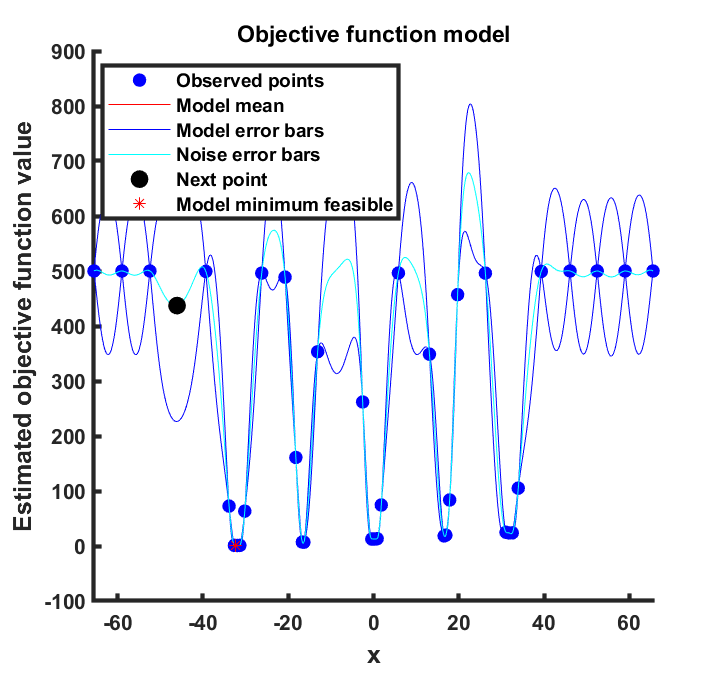}
\caption{bayesopt40}
\end{subfigure}
\begin{subfigure}[b]{0.330\textwidth}
\includegraphics[width=\textwidth]{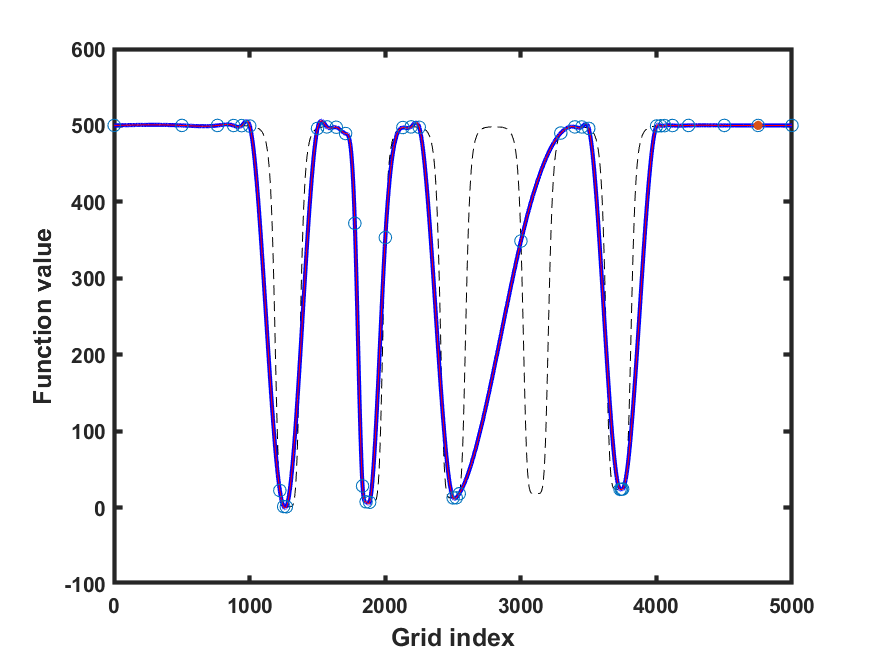}
\caption{LineWalker40}
\end{subfigure}
\newline
\begin{subfigure}[b]{0.270\textwidth}
\includegraphics[width=\textwidth]{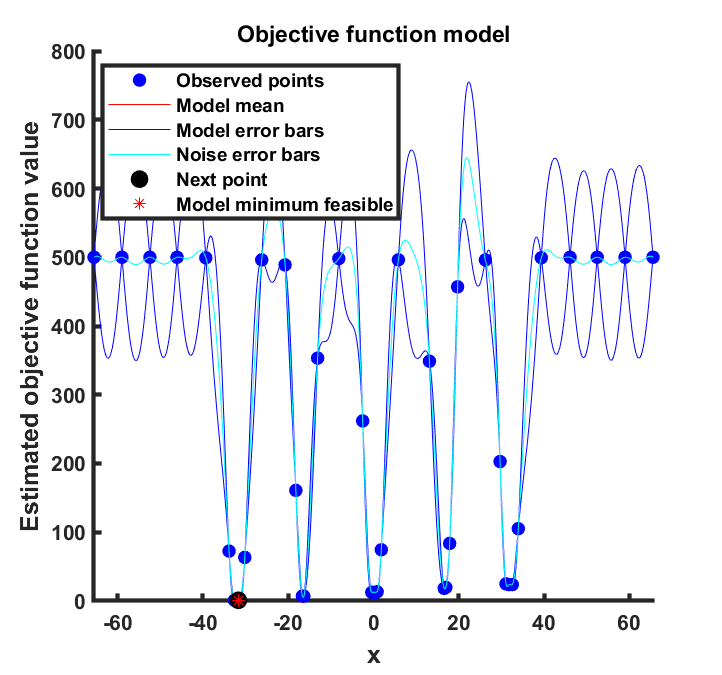}
\caption{bayesopt50}
\end{subfigure}
\begin{subfigure}[b]{0.330\textwidth}
\includegraphics[width=\textwidth]{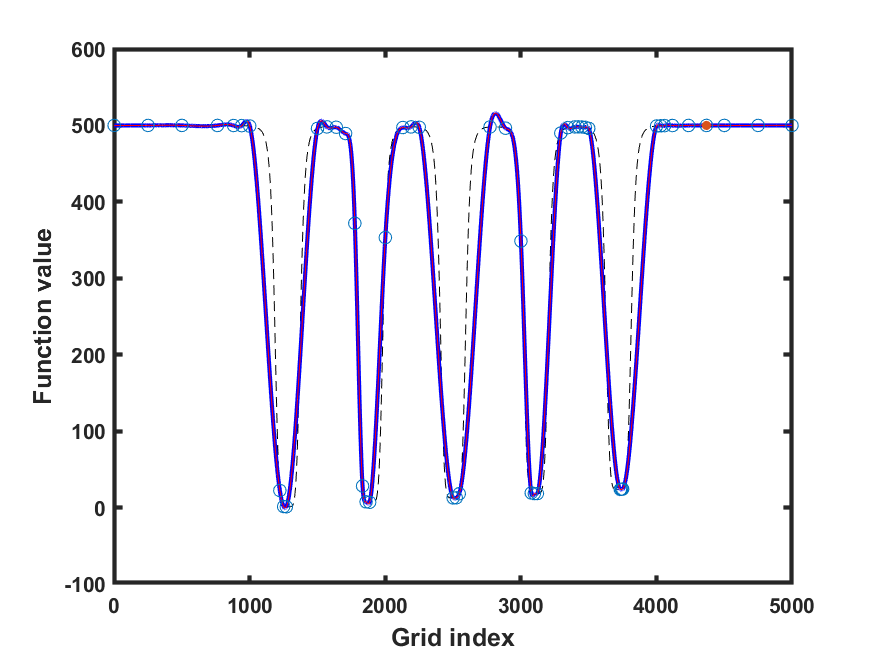}
\caption{LineWalker50}
\end{subfigure}
\newline
\caption{dejong5. Left column = \texttt{bayesopt}. Right column = \texttt{LineWalker-full}}
\label{fig:out_dejong5_1Dslice}
\end{figure}

\begin{figure} [h!]
\centering
\begin{subfigure}[b]{0.270\textwidth}
\includegraphics[width=\textwidth]{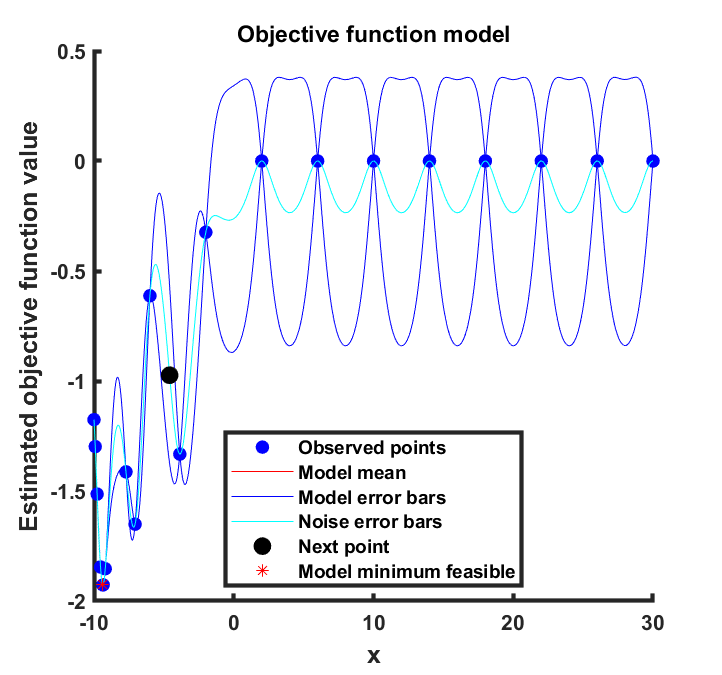}
\caption{bayesopt20}
\end{subfigure}
\begin{subfigure}[b]{0.330\textwidth}
\includegraphics[width=\textwidth]{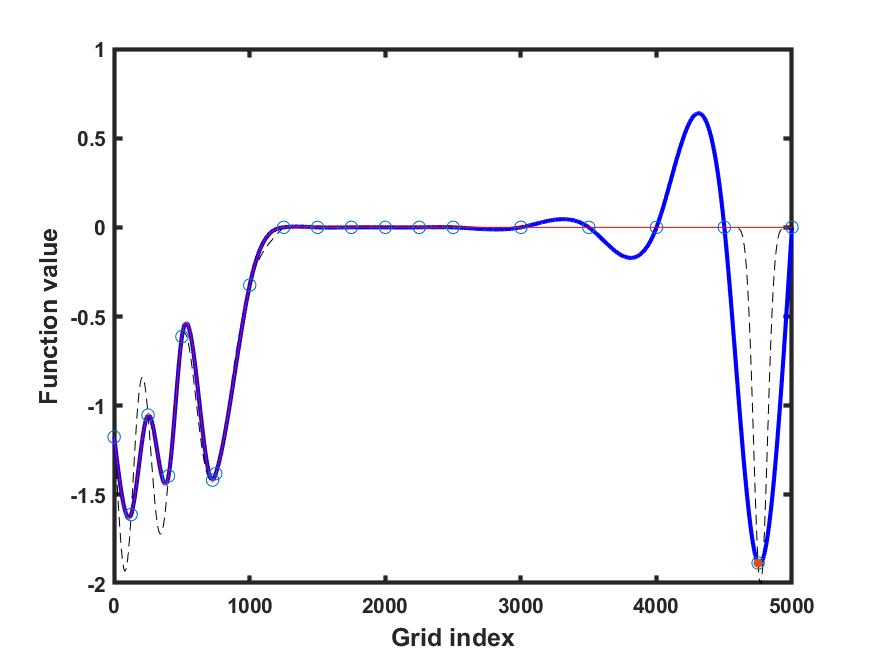}
\caption{LineWalker20}
\end{subfigure}
\newline
\begin{subfigure}[b]{0.270\textwidth}
\includegraphics[width=\textwidth]{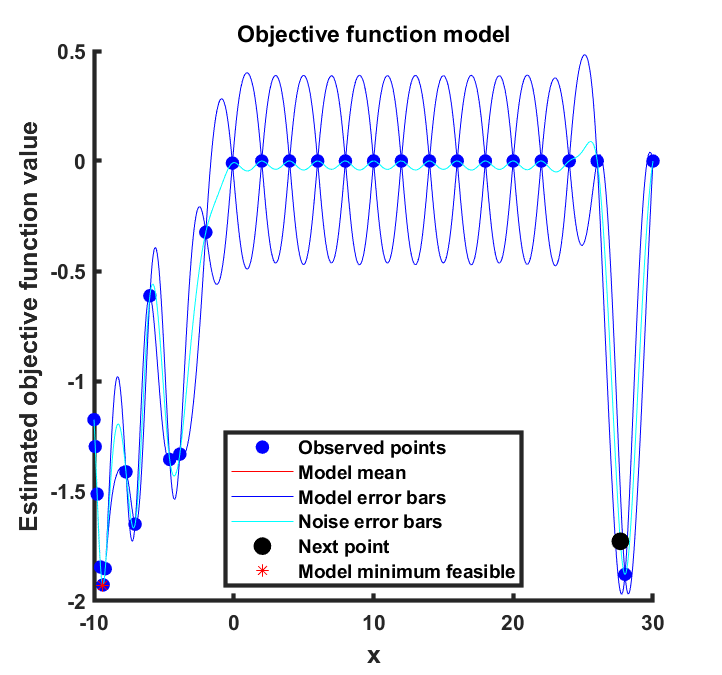}
\caption{bayesopt30}
\end{subfigure}
\begin{subfigure}[b]{0.330\textwidth}
\includegraphics[width=\textwidth]{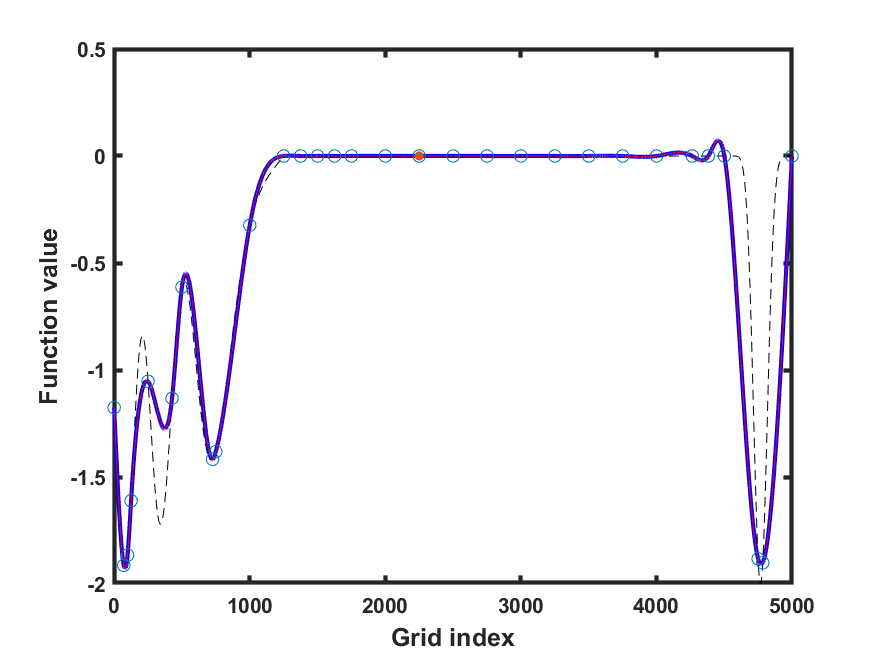}
\caption{LineWalker30}
\end{subfigure}
\newline
\begin{subfigure}[b]{0.270\textwidth}
\includegraphics[width=\textwidth]{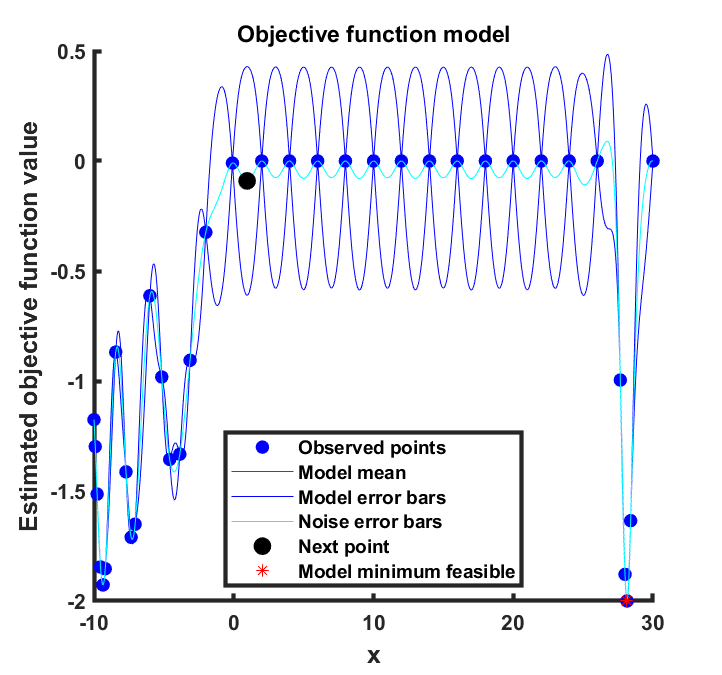}
\caption{bayesopt40}
\end{subfigure}
\begin{subfigure}[b]{0.330\textwidth}
\includegraphics[width=\textwidth]{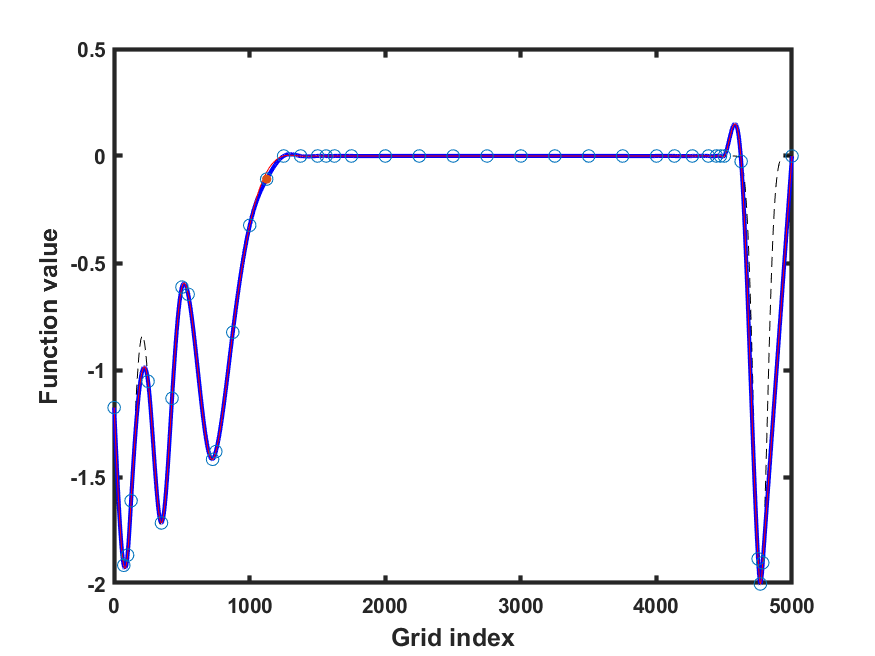}
\caption{LineWalker40}
\end{subfigure}
\newline
\begin{subfigure}[b]{0.270\textwidth}
\includegraphics[width=\textwidth]{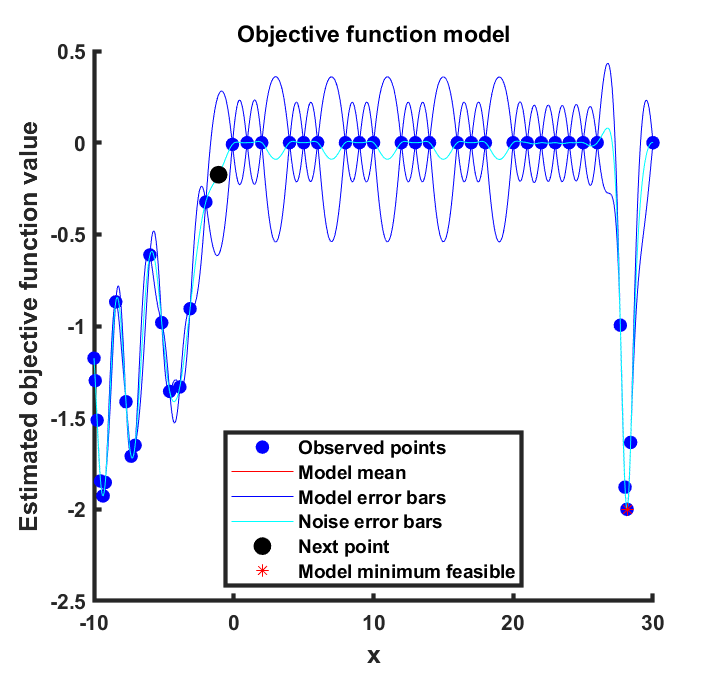}
\caption{bayesopt50}
\end{subfigure}
\begin{subfigure}[b]{0.330\textwidth}
\includegraphics[width=\textwidth]{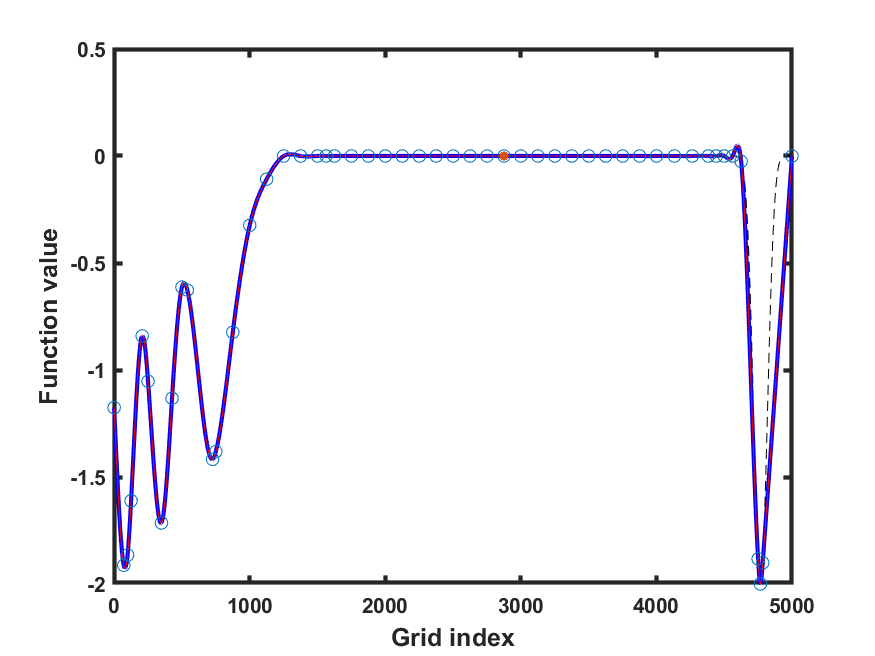}
\caption{LineWalker50}
\end{subfigure}
\newline
\caption{easom schaffer2A. Left column = \texttt{bayesopt}. Right column = \texttt{LineWalker-full}}
\label{fig:out_easom_schaffer2A_1Dslice}
\end{figure}

\begin{figure} [h!]
\centering
\begin{subfigure}[b]{0.270\textwidth}
\includegraphics[width=\textwidth]{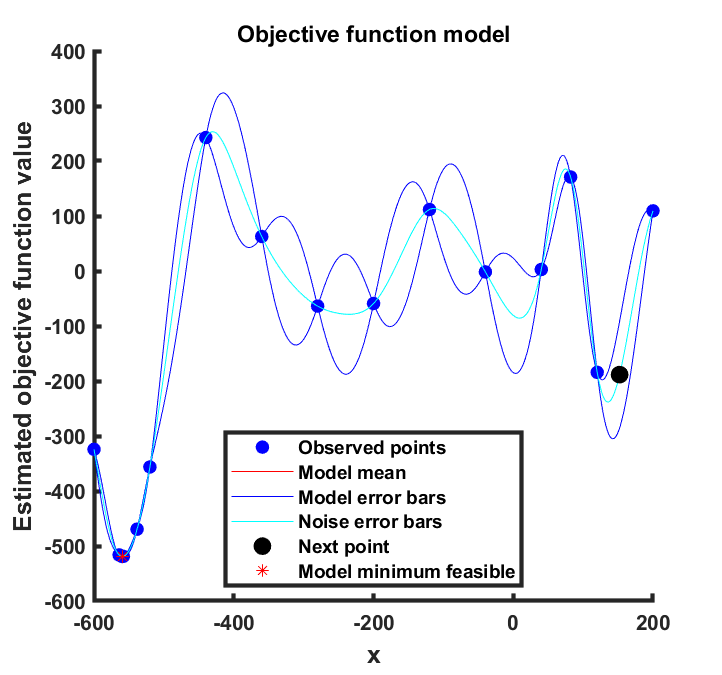}
\caption{bayesopt20}
\end{subfigure}
\begin{subfigure}[b]{0.330\textwidth}
\includegraphics[width=\textwidth]{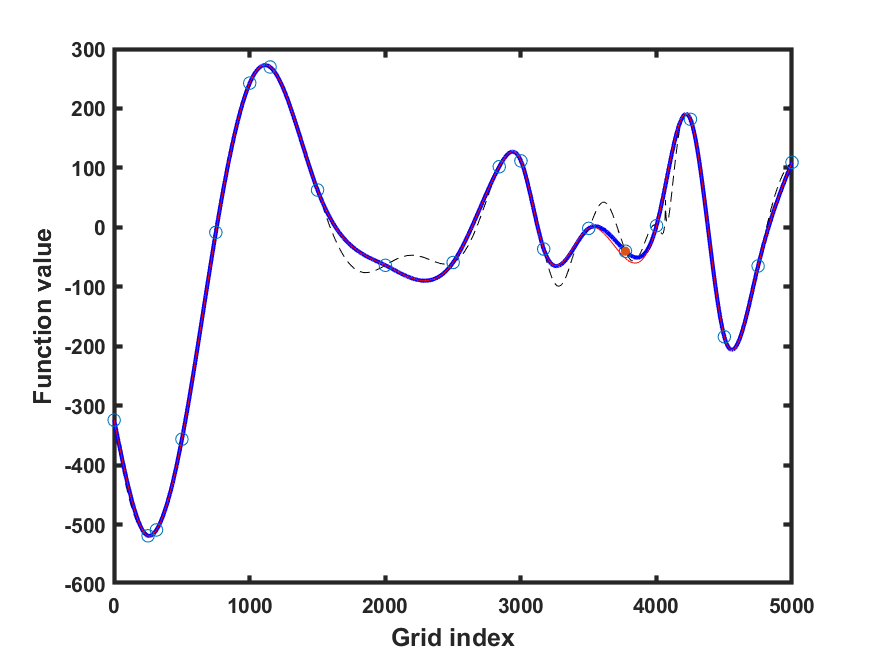}
\caption{LineWalker20}
\end{subfigure}
\newline
\begin{subfigure}[b]{0.270\textwidth}
\includegraphics[width=\textwidth]{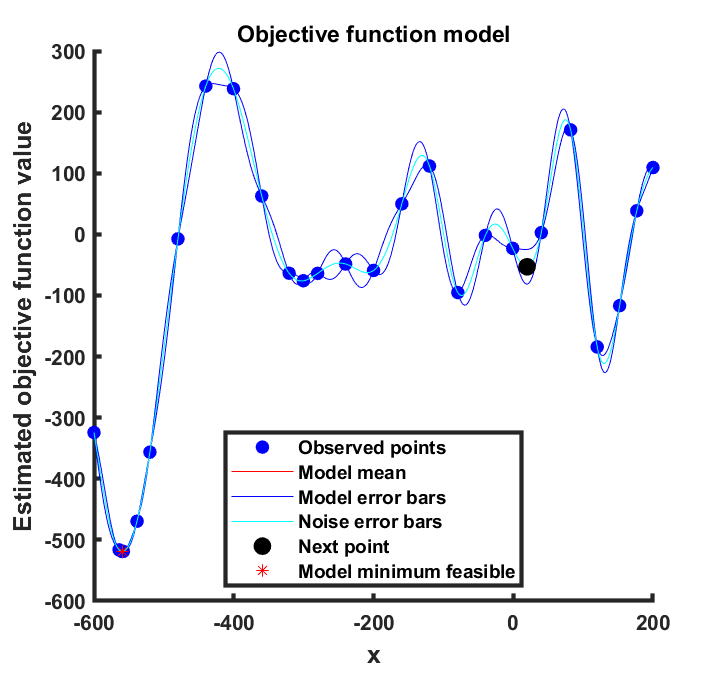}
\caption{bayesopt30}
\end{subfigure}
\begin{subfigure}[b]{0.330\textwidth}
\includegraphics[width=\textwidth]{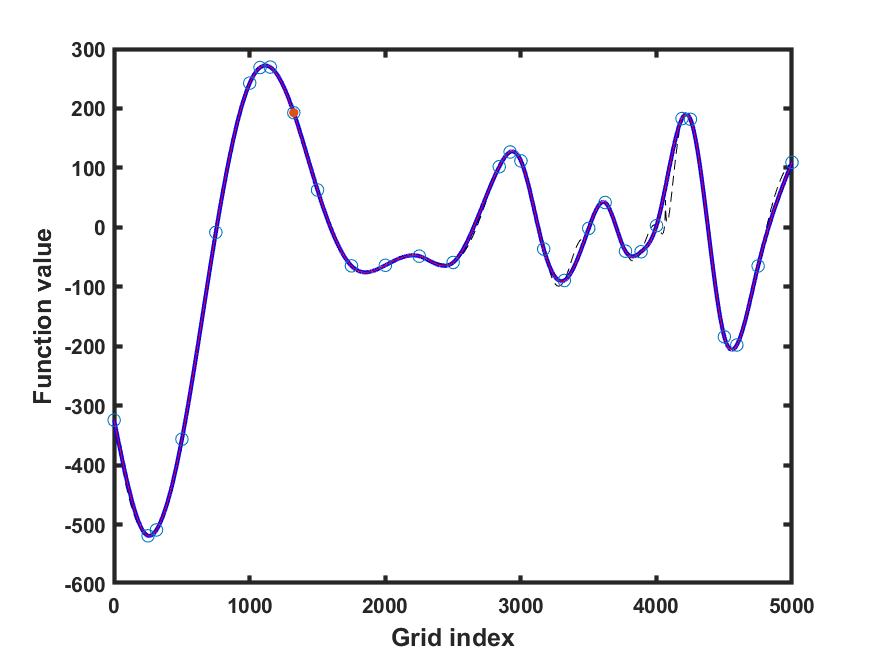}
\caption{LineWalker30}
\end{subfigure}
\newline
\begin{subfigure}[b]{0.270\textwidth}
\includegraphics[width=\textwidth]{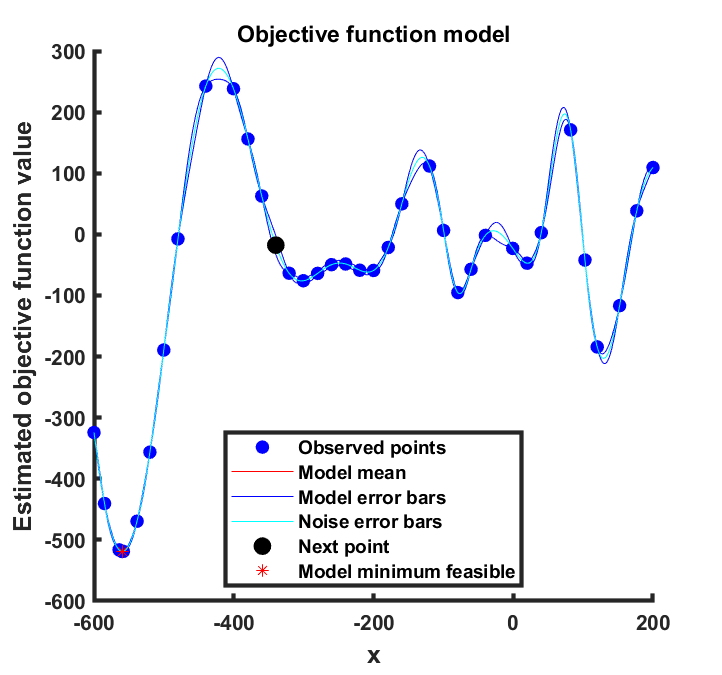}
\caption{bayesopt40}
\end{subfigure}
\begin{subfigure}[b]{0.330\textwidth}
\includegraphics[width=\textwidth]{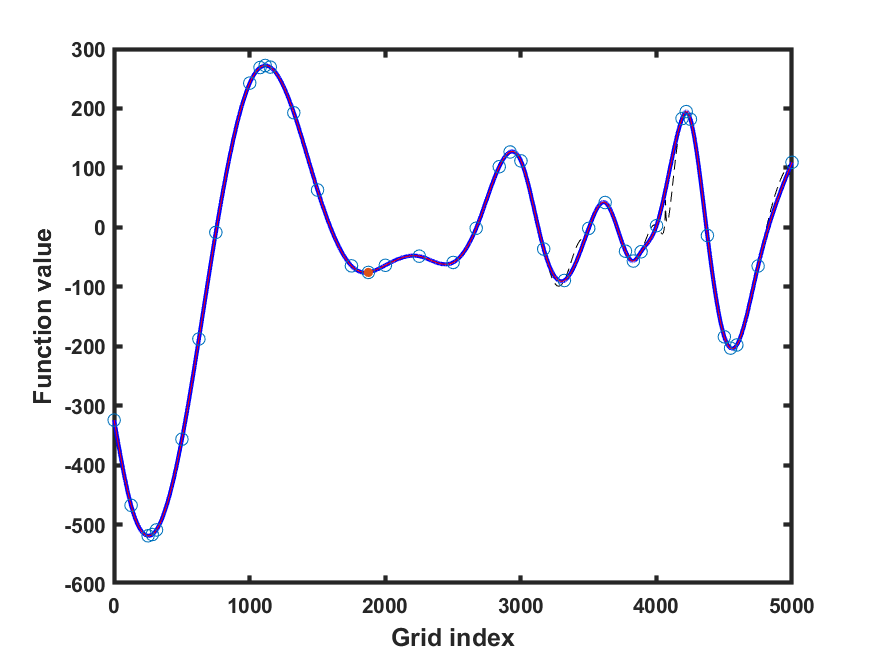}
\caption{LineWalker40}
\end{subfigure}
\newline
\begin{subfigure}[b]{0.270\textwidth}
\includegraphics[width=\textwidth]{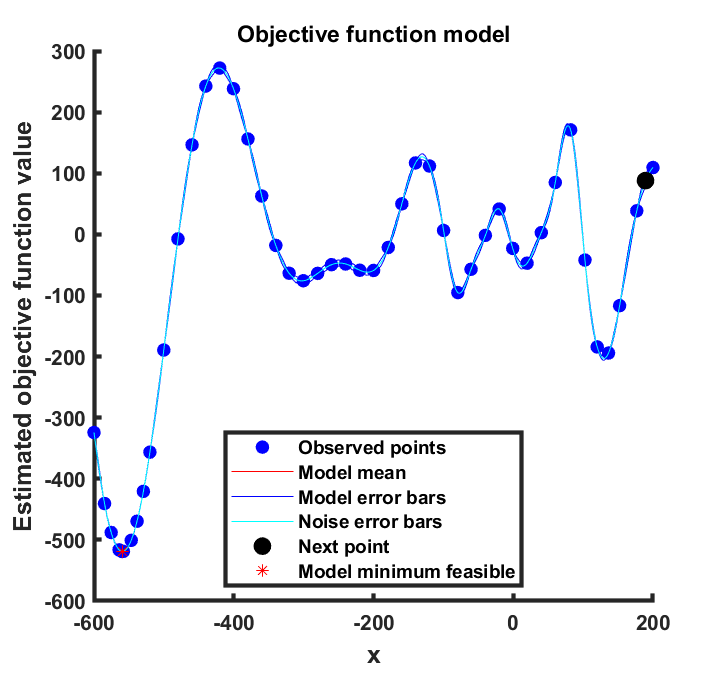}
\caption{bayesopt50}
\end{subfigure}
\begin{subfigure}[b]{0.330\textwidth}
\includegraphics[width=\textwidth]{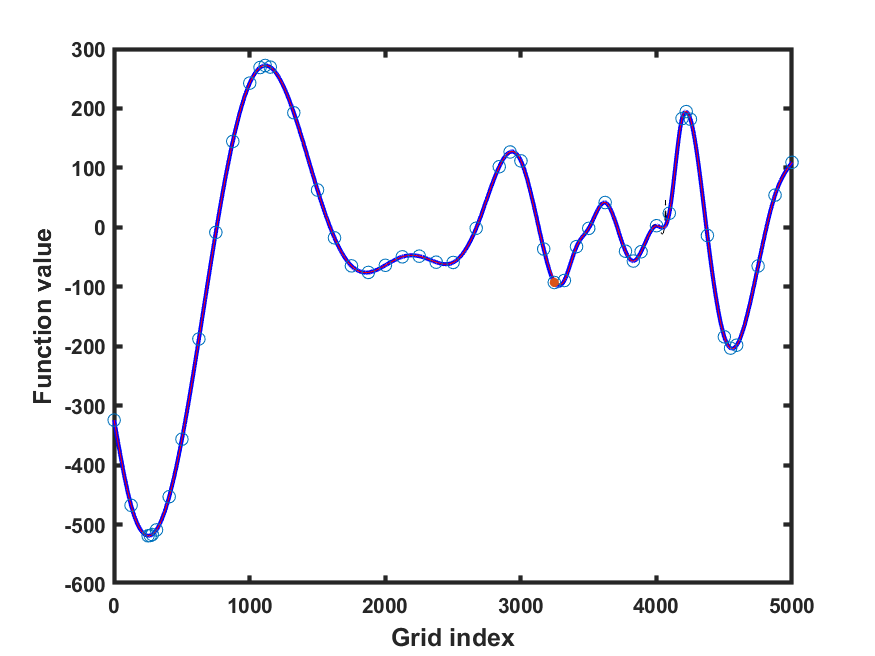}
\caption{LineWalker50}
\end{subfigure}
\newline
\caption{egg2. Left column = \texttt{bayesopt}. Right column = \texttt{LineWalker-full}}
\label{fig:out_egg2_1Dslice}
\end{figure}

\begin{figure} [h!]
\centering
\begin{subfigure}[b]{0.270\textwidth}
\includegraphics[width=\textwidth]{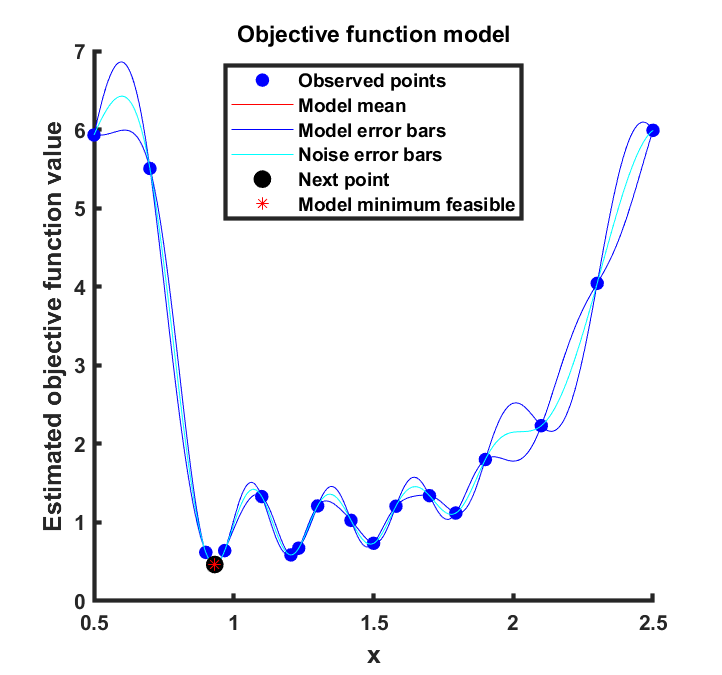}
\caption{bayesopt20}
\end{subfigure}
\begin{subfigure}[b]{0.330\textwidth}
\includegraphics[width=\textwidth]{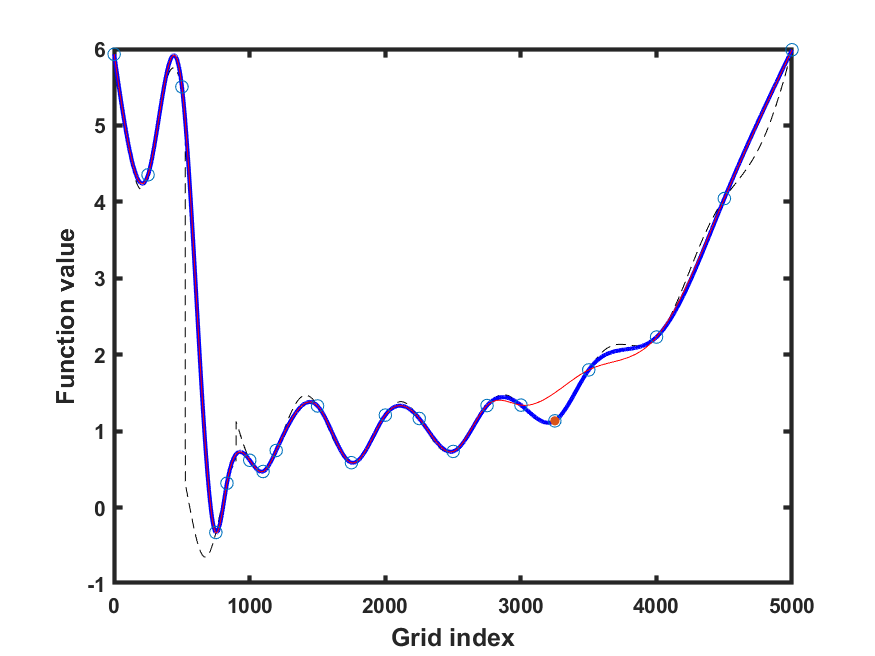}
\caption{LineWalker20}
\end{subfigure}
\newline
\begin{subfigure}[b]{0.270\textwidth}
\includegraphics[width=\textwidth]{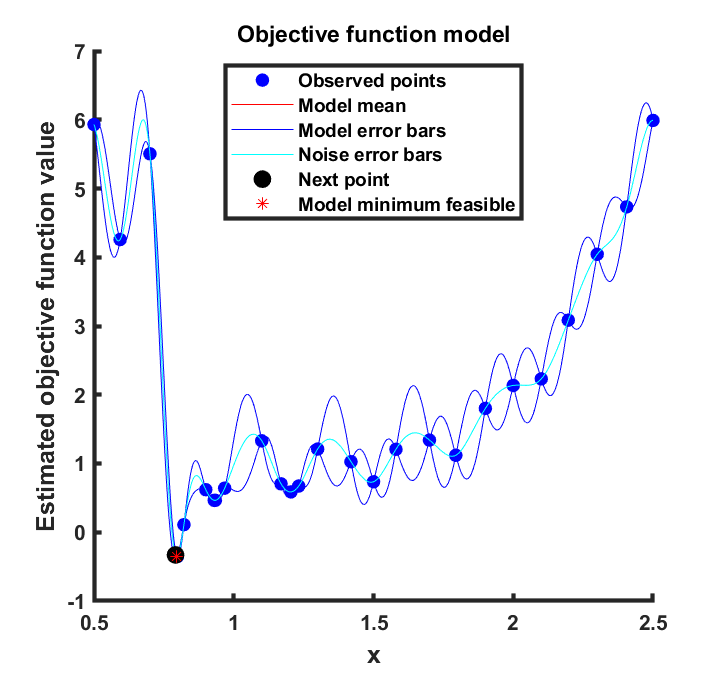}
\caption{bayesopt30}
\end{subfigure}
\begin{subfigure}[b]{0.330\textwidth}
\includegraphics[width=\textwidth]{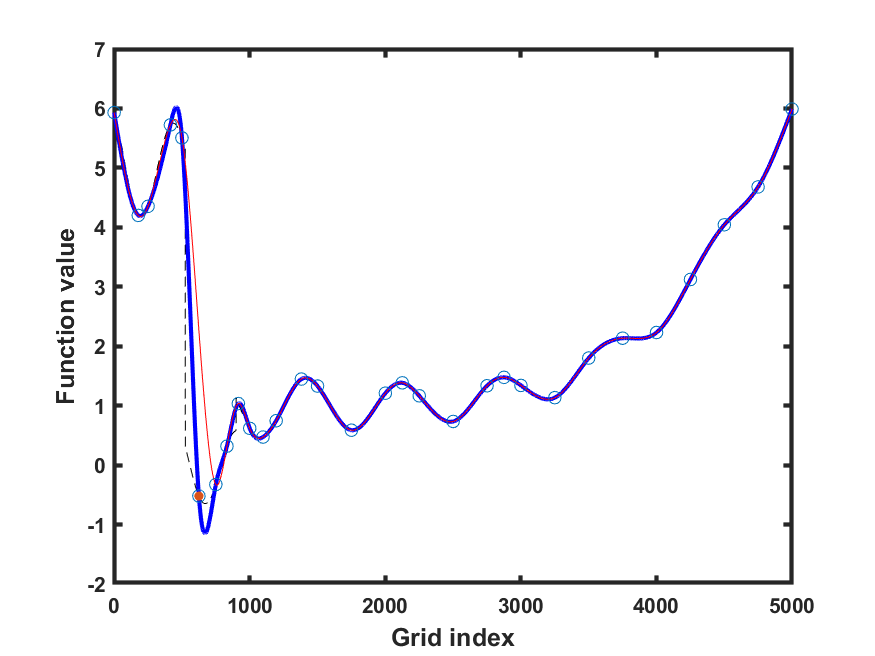}
\caption{LineWalker30}
\end{subfigure}
\newline
\begin{subfigure}[b]{0.270\textwidth}
\includegraphics[width=\textwidth]{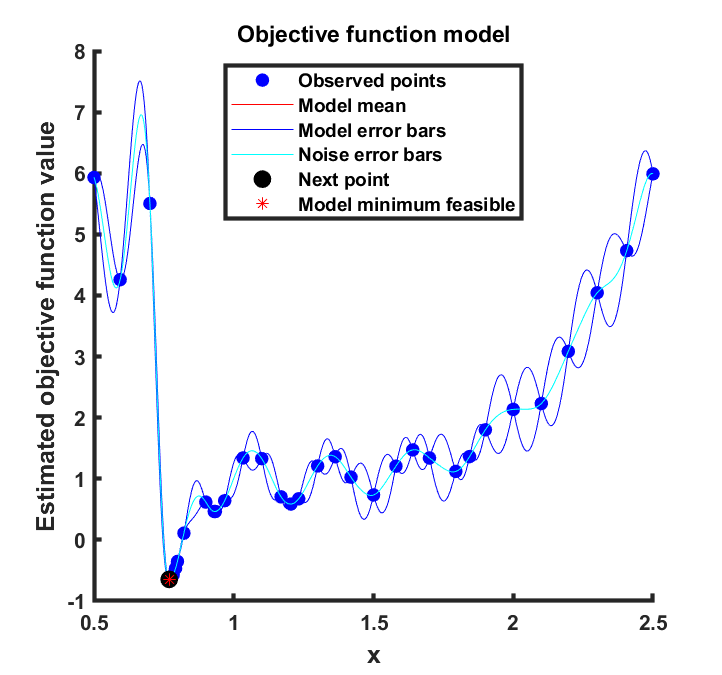}
\caption{bayesopt40}
\end{subfigure}
\begin{subfigure}[b]{0.330\textwidth}
\includegraphics[width=\textwidth]{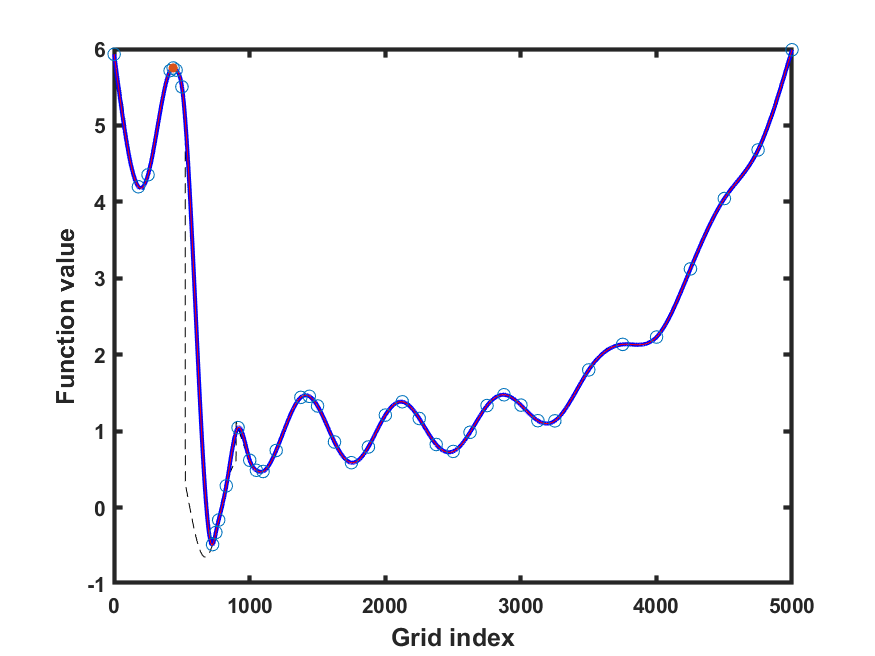}
\caption{LineWalker40}
\end{subfigure}
\newline
\begin{subfigure}[b]{0.270\textwidth}
\includegraphics[width=\textwidth]{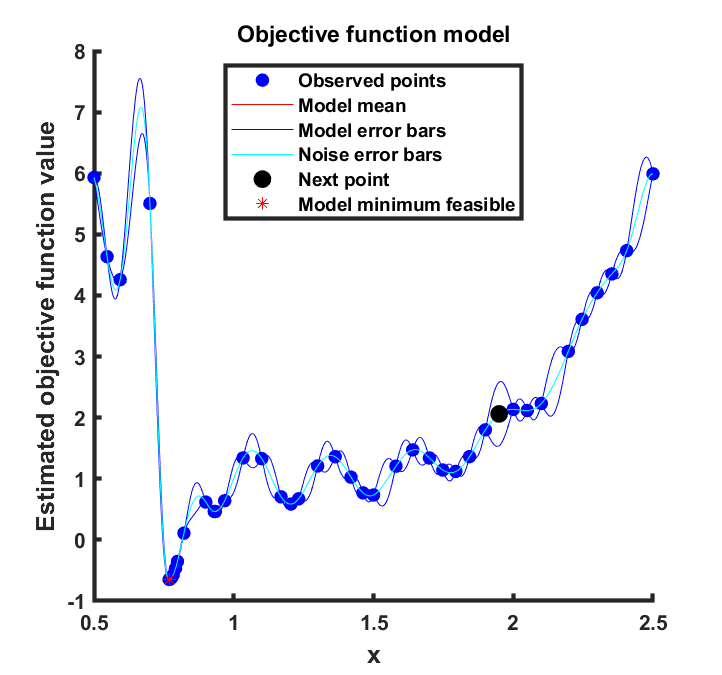}
\caption{bayesopt50}
\end{subfigure}
\begin{subfigure}[b]{0.330\textwidth}
\includegraphics[width=\textwidth]{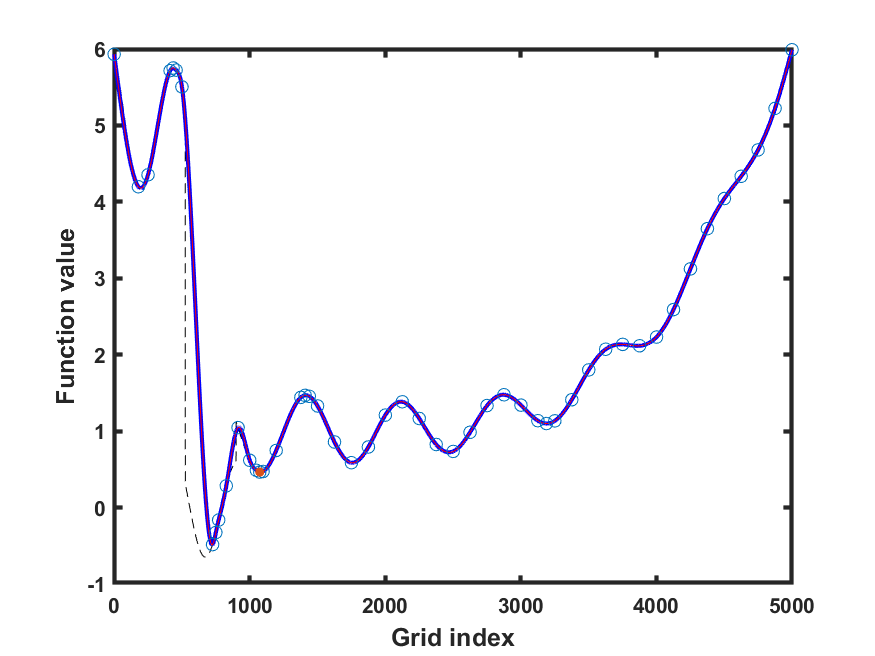}
\caption{LineWalker50}
\end{subfigure}
\newline
\caption{grlee12Step. Left column = \texttt{bayesopt}. Right column = \texttt{LineWalker-full}}
\label{fig:out_grlee12Step}
\end{figure}

\begin{figure} [h!]
\centering
\begin{subfigure}[b]{0.270\textwidth}
\includegraphics[width=\textwidth]{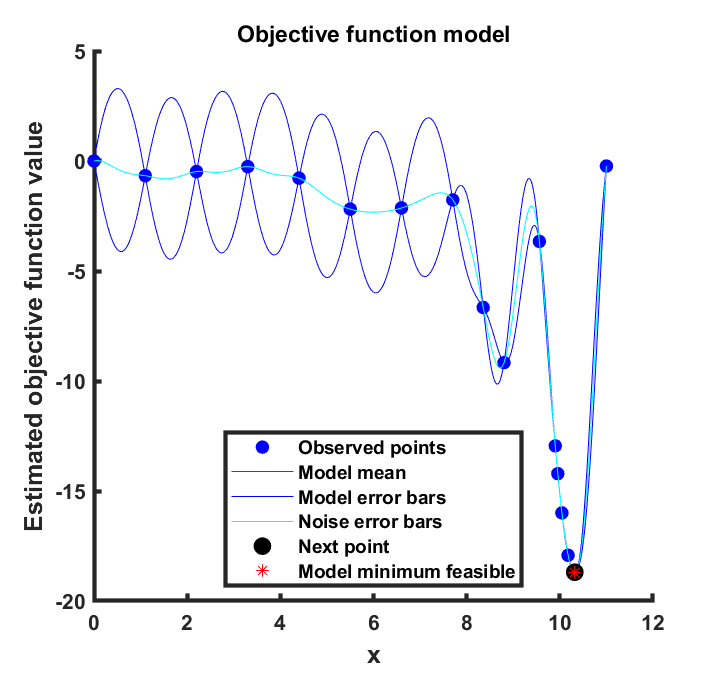}
\caption{bayesopt20}
\end{subfigure}
\begin{subfigure}[b]{0.330\textwidth}
\includegraphics[width=\textwidth]{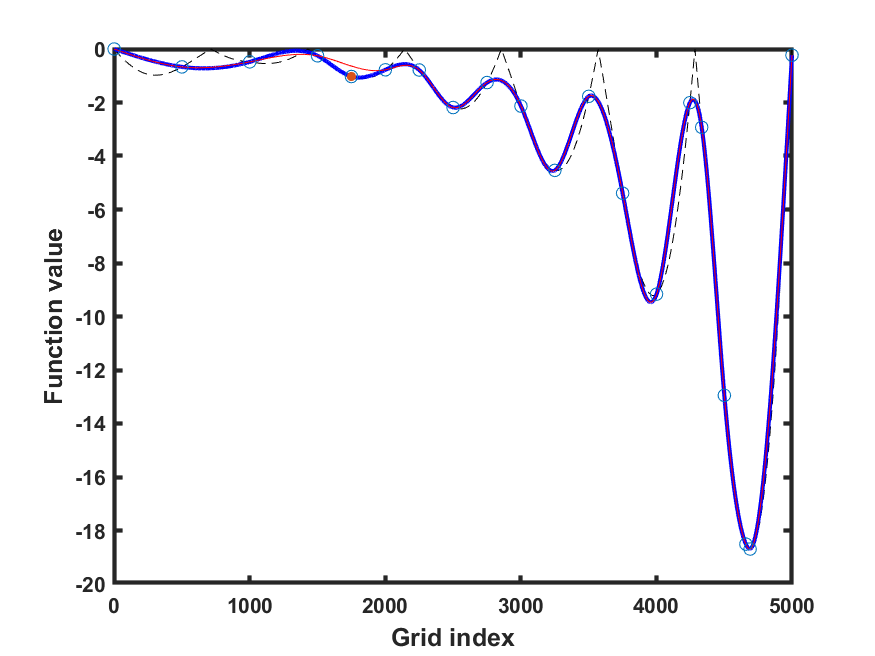}
\caption{LineWalker20}
\end{subfigure}
\newline
\begin{subfigure}[b]{0.270\textwidth}
\includegraphics[width=\textwidth]{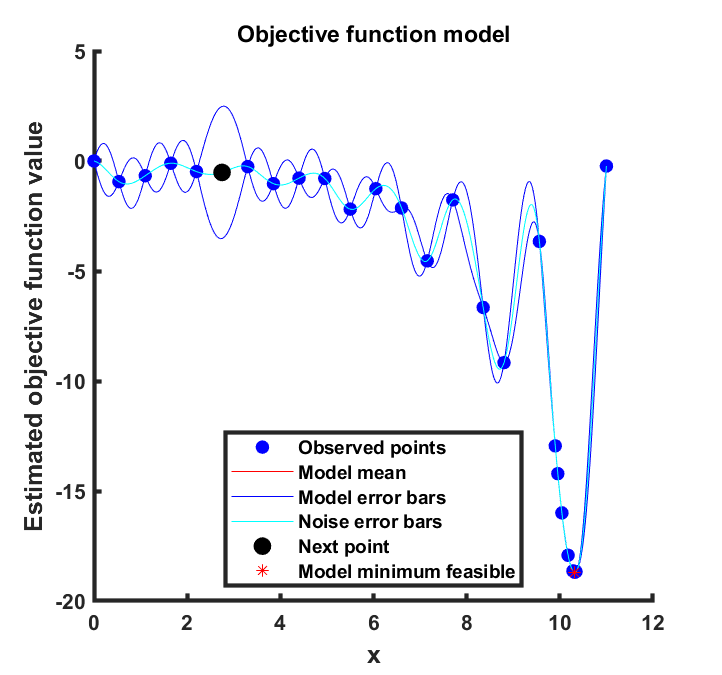}
\caption{bayesopt30}
\end{subfigure}
\begin{subfigure}[b]{0.330\textwidth}
\includegraphics[width=\textwidth]{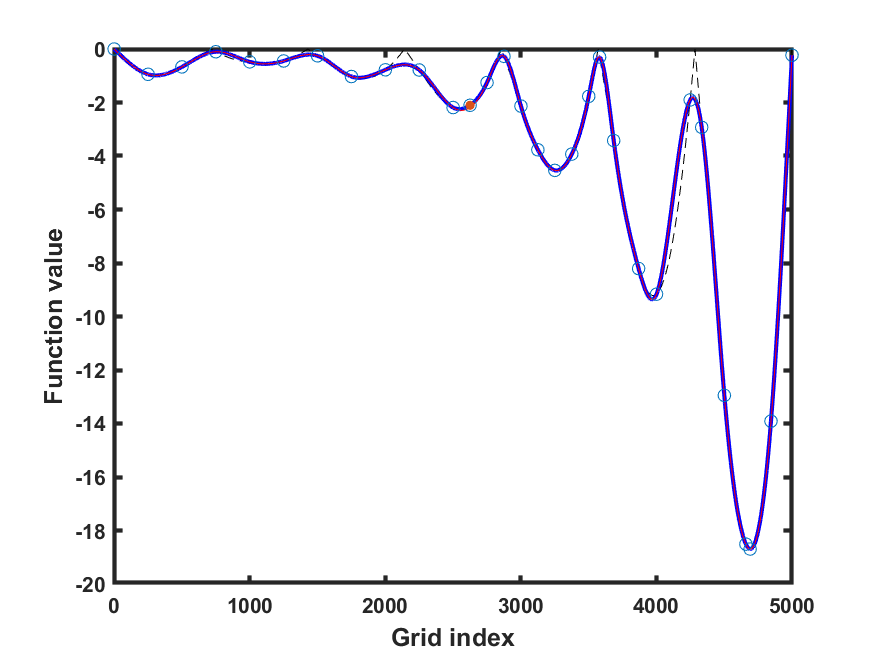}
\caption{LineWalker30}
\end{subfigure}
\newline
\begin{subfigure}[b]{0.270\textwidth}
\includegraphics[width=\textwidth]{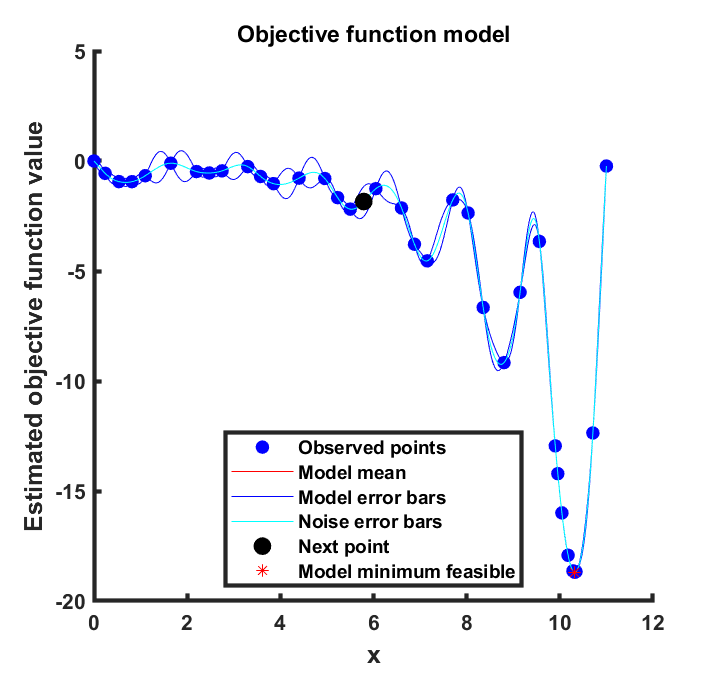}
\caption{bayesopt40}
\end{subfigure}
\begin{subfigure}[b]{0.330\textwidth}
\includegraphics[width=\textwidth]{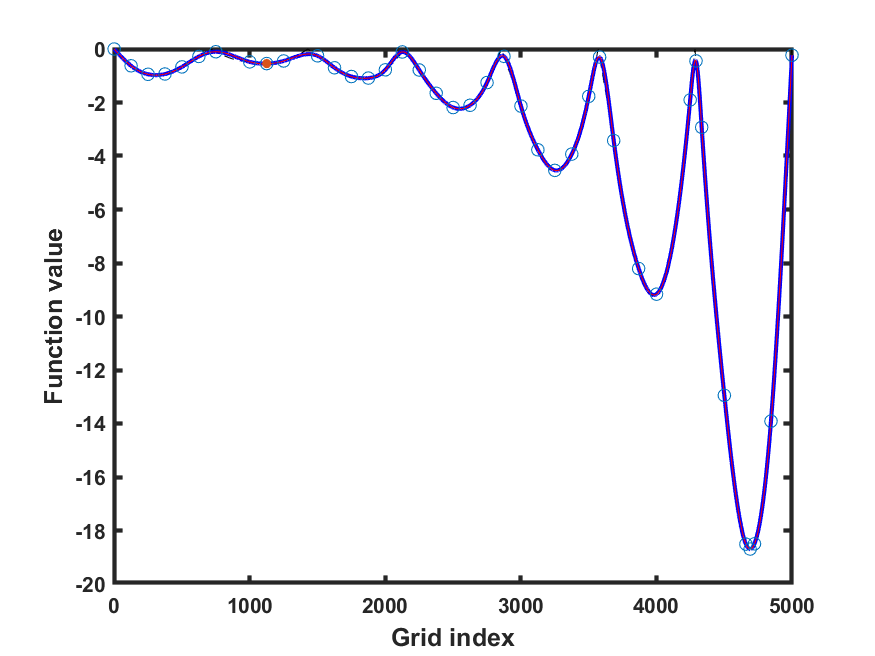}
\caption{LineWalker40}
\end{subfigure}
\newline
\begin{subfigure}[b]{0.270\textwidth}
\includegraphics[width=\textwidth]{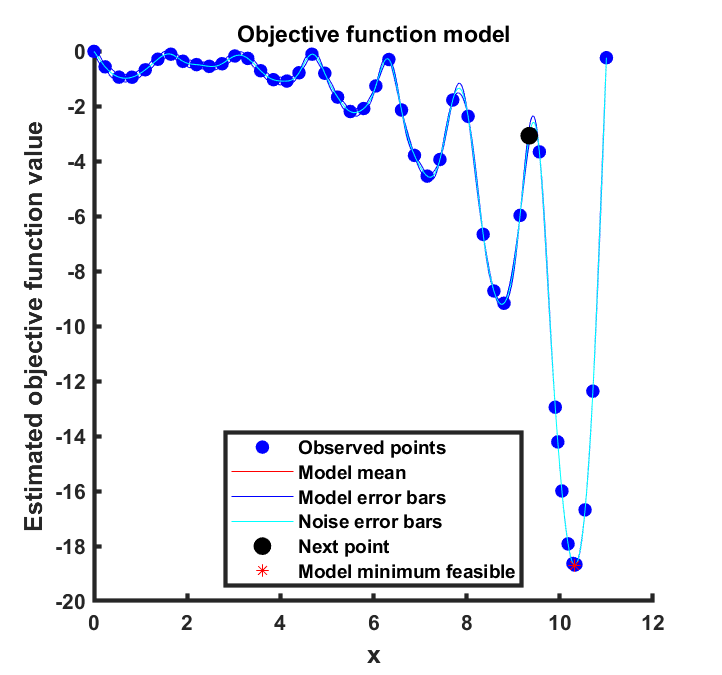}
\caption{bayesopt50}
\end{subfigure}
\begin{subfigure}[b]{0.330\textwidth}
\includegraphics[width=\textwidth]{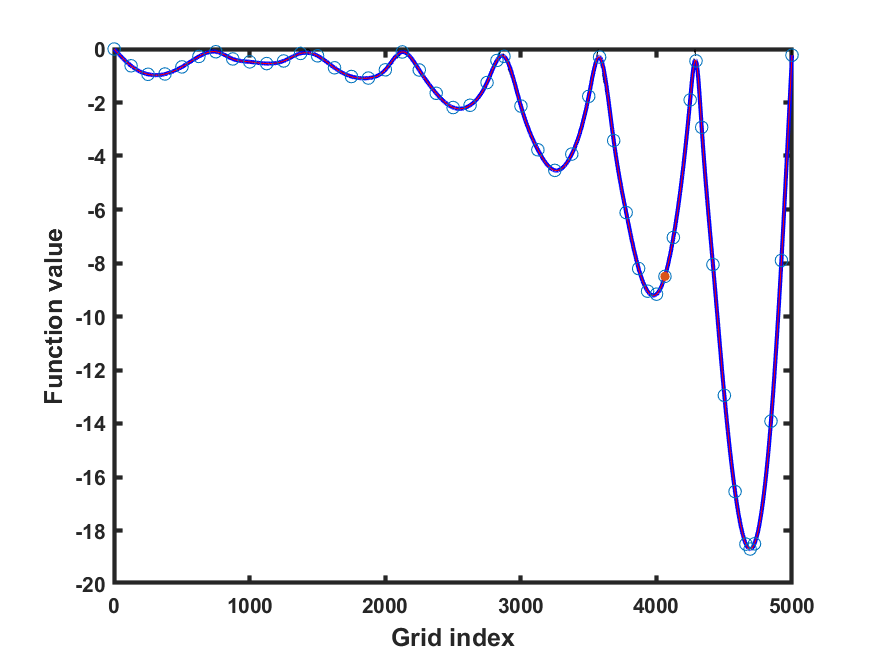}
\caption{LineWalker50}
\end{subfigure}
\newline
\caption{holder. Left column = \texttt{bayesopt}. Right column = \texttt{LineWalker-full}}
\label{fig:out_holder_1Dslice}
\end{figure}

\begin{figure} [h!]
\centering
\begin{subfigure}[b]{0.270\textwidth}
\includegraphics[width=\textwidth]{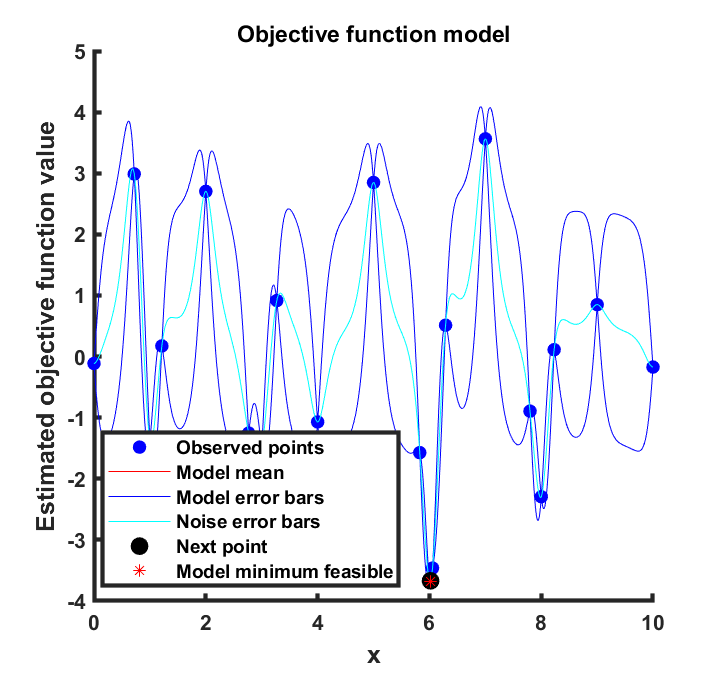}
\caption{bayesopt20}
\end{subfigure}
\begin{subfigure}[b]{0.330\textwidth}
\includegraphics[width=\textwidth]{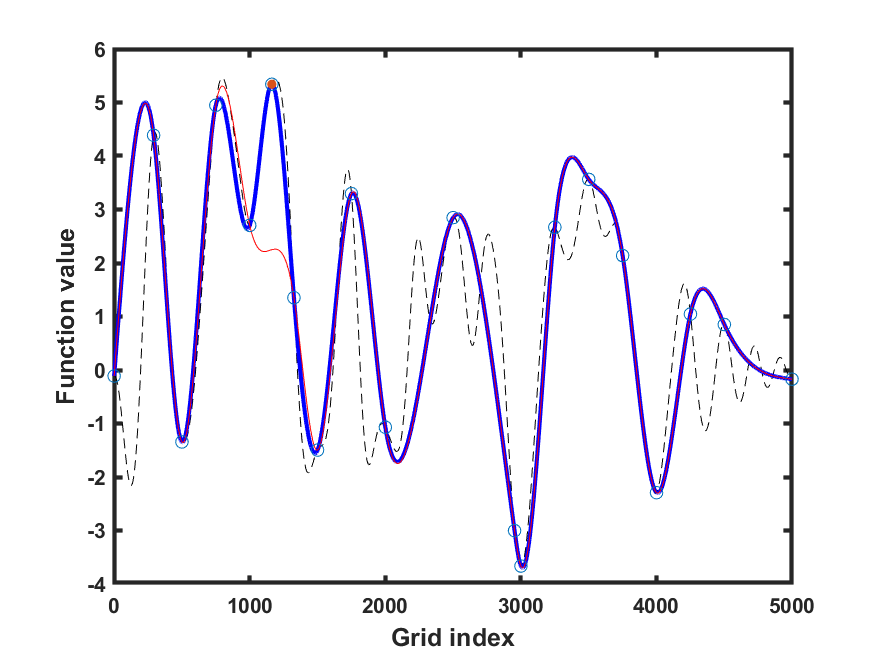}
\caption{LineWalker20}
\end{subfigure}
\newline
\begin{subfigure}[b]{0.270\textwidth}
\includegraphics[width=\textwidth]{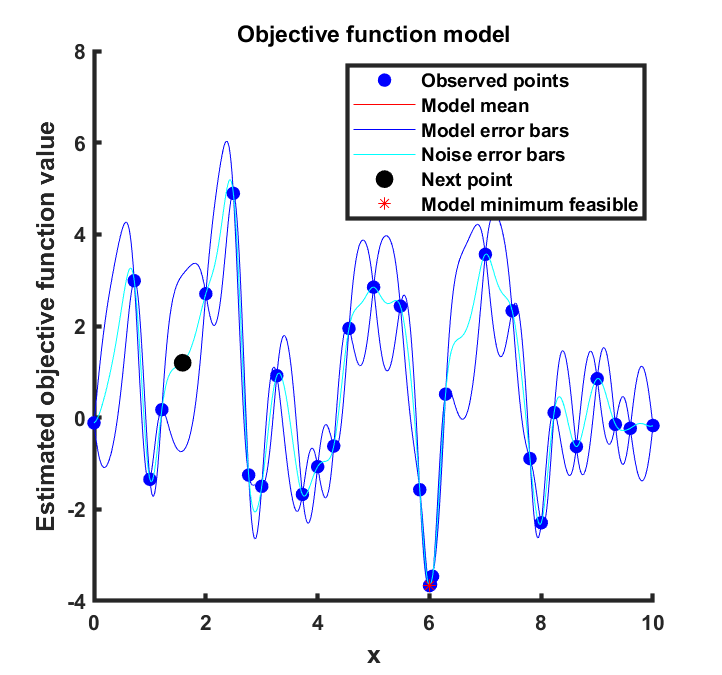}
\caption{bayesopt30}
\end{subfigure}
\begin{subfigure}[b]{0.330\textwidth}
\includegraphics[width=\textwidth]{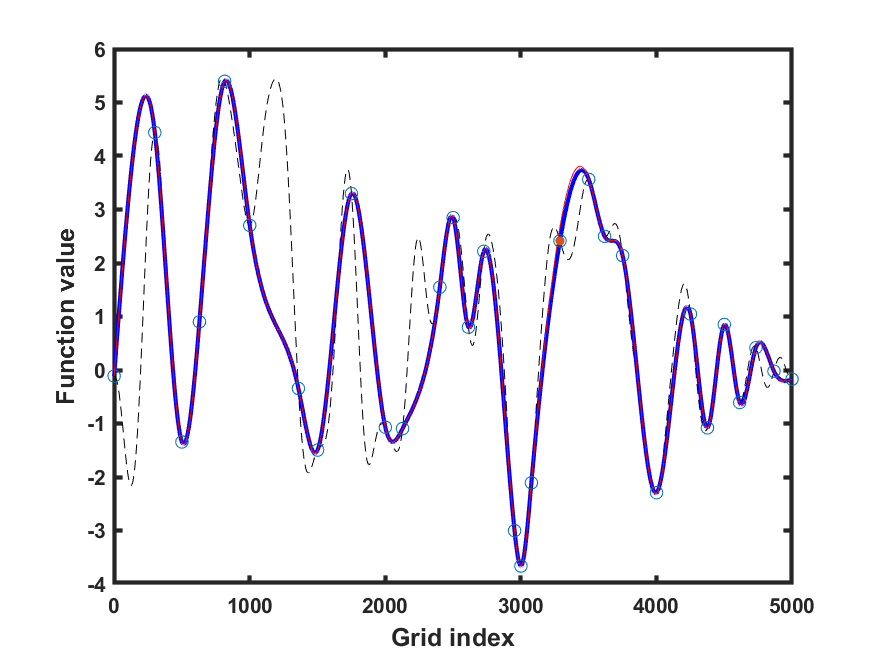}
\caption{LineWalker30}
\end{subfigure}
\newline
\begin{subfigure}[b]{0.270\textwidth}
\includegraphics[width=\textwidth]{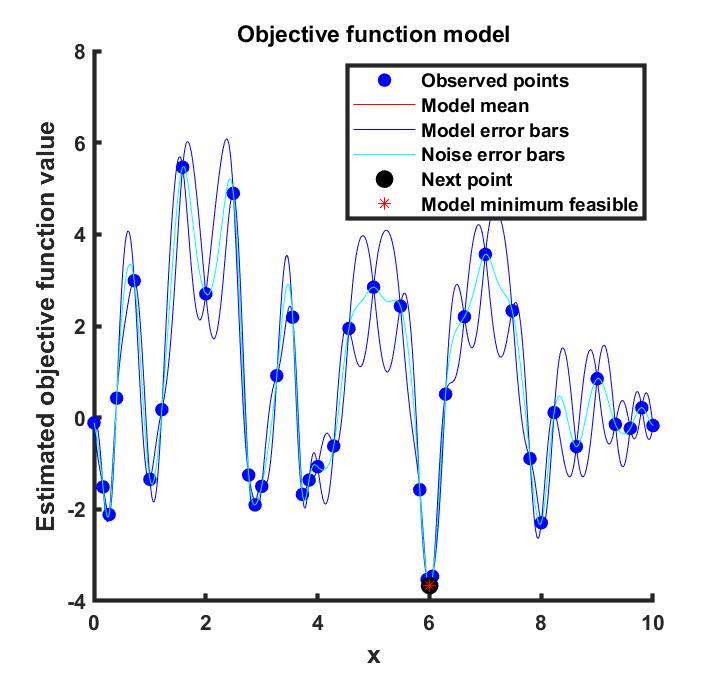}
\caption{bayesopt40}
\end{subfigure}
\begin{subfigure}[b]{0.330\textwidth}
\includegraphics[width=\textwidth]{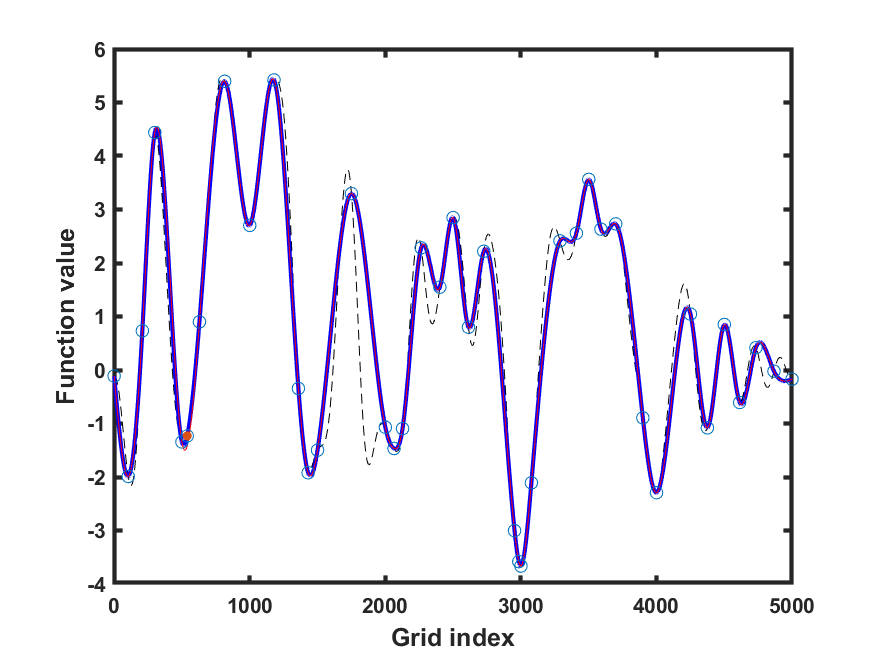}
\caption{LineWalker40}
\end{subfigure}
\newline
\begin{subfigure}[b]{0.270\textwidth}
\includegraphics[width=\textwidth]{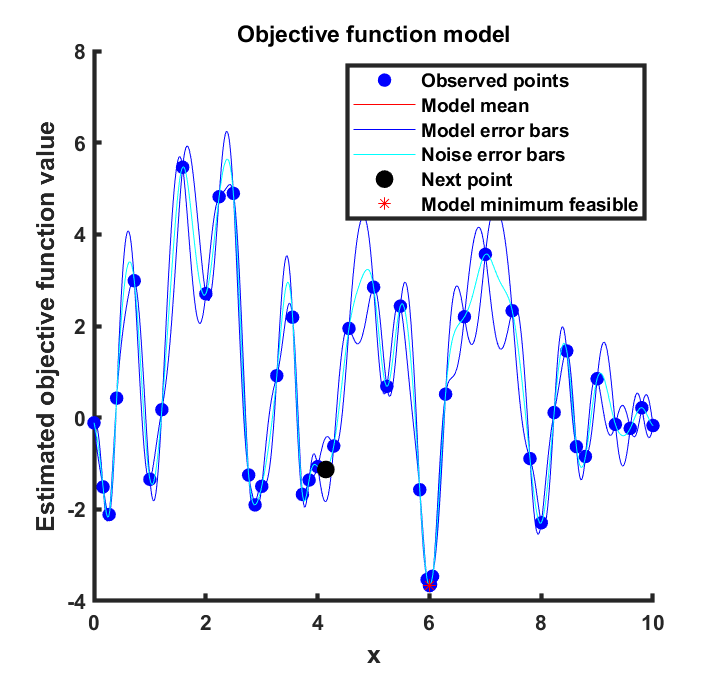}
\caption{bayesopt50}
\end{subfigure}
\begin{subfigure}[b]{0.330\textwidth}
\includegraphics[width=\textwidth]{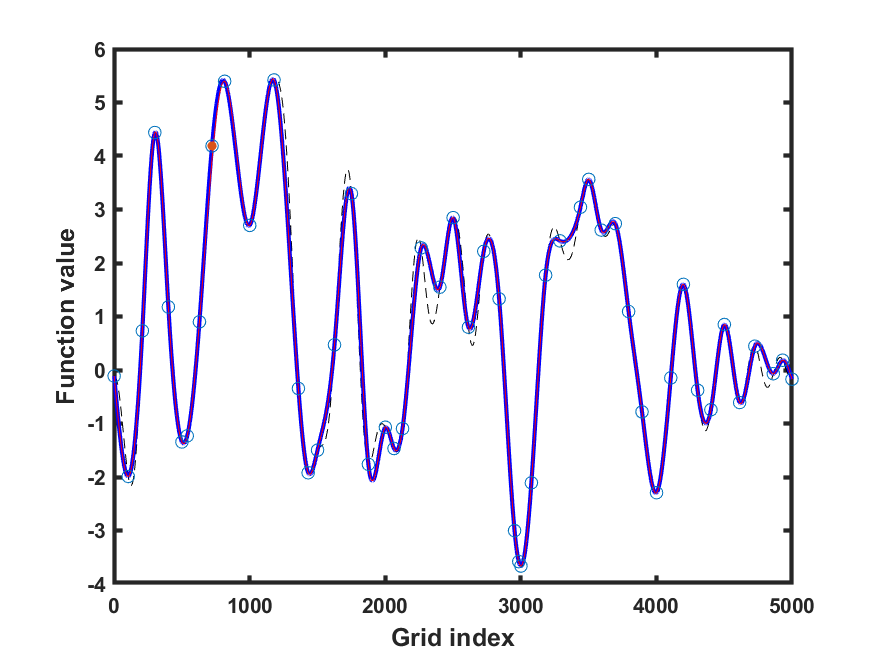}
\caption{LineWalker50}
\end{subfigure}
\newline
\caption{langer. Left column = \texttt{bayesopt}. Right column = \texttt{LineWalker-full}}
\label{fig:out_langer}
\end{figure}

\begin{figure} [h!]
\centering
\begin{subfigure}[b]{0.270\textwidth}
\includegraphics[width=\textwidth]{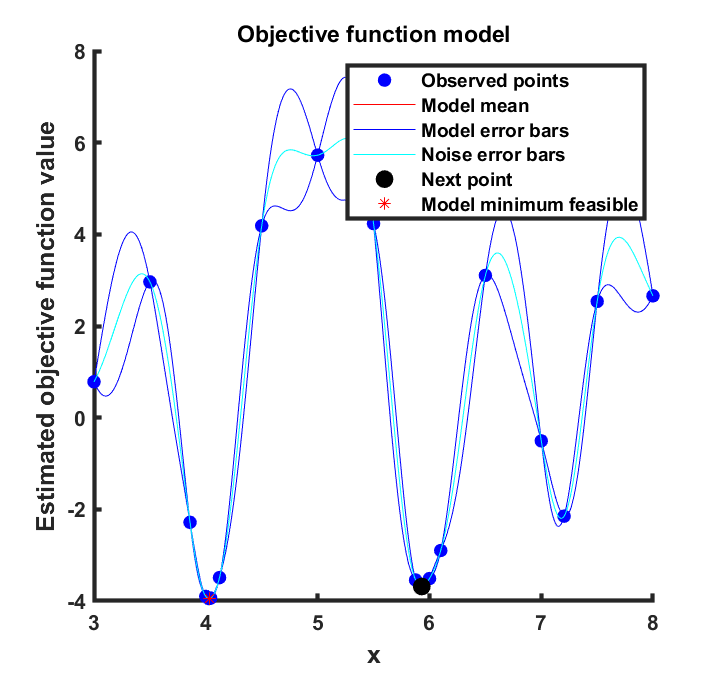}
\caption{bayesopt20}
\end{subfigure}
\begin{subfigure}[b]{0.330\textwidth}
\includegraphics[width=\textwidth]{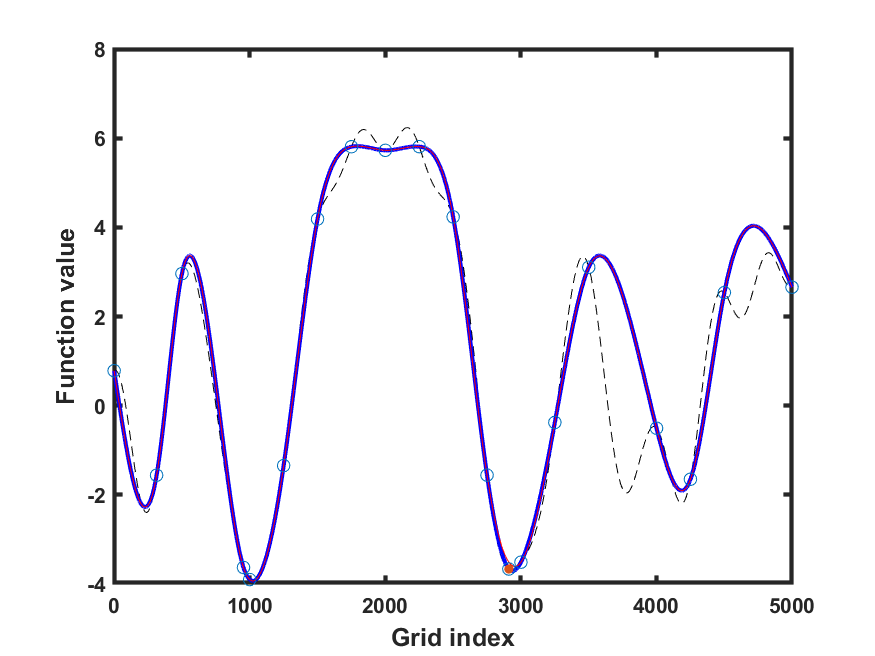}
\caption{LineWalker20}
\end{subfigure}
\newline
\begin{subfigure}[b]{0.270\textwidth}
\includegraphics[width=\textwidth]{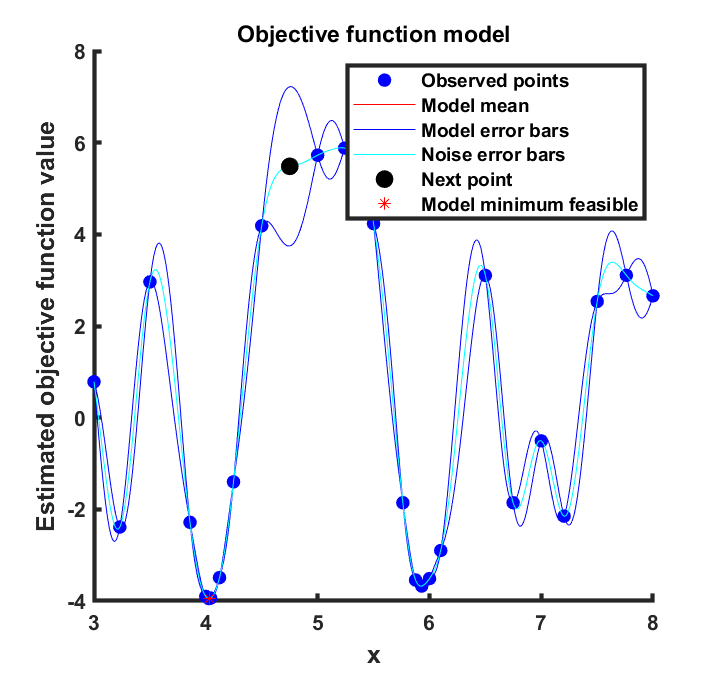}
\caption{bayesopt30}
\end{subfigure}
\begin{subfigure}[b]{0.330\textwidth}
\includegraphics[width=\textwidth]{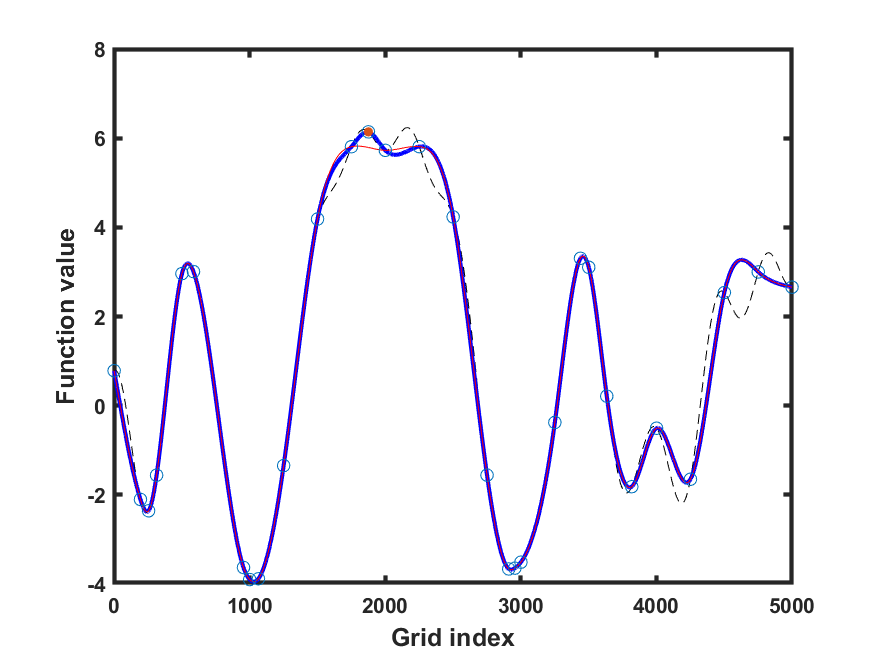}
\caption{LineWalker30}
\end{subfigure}
\newline
\begin{subfigure}[b]{0.270\textwidth}
\includegraphics[width=\textwidth]{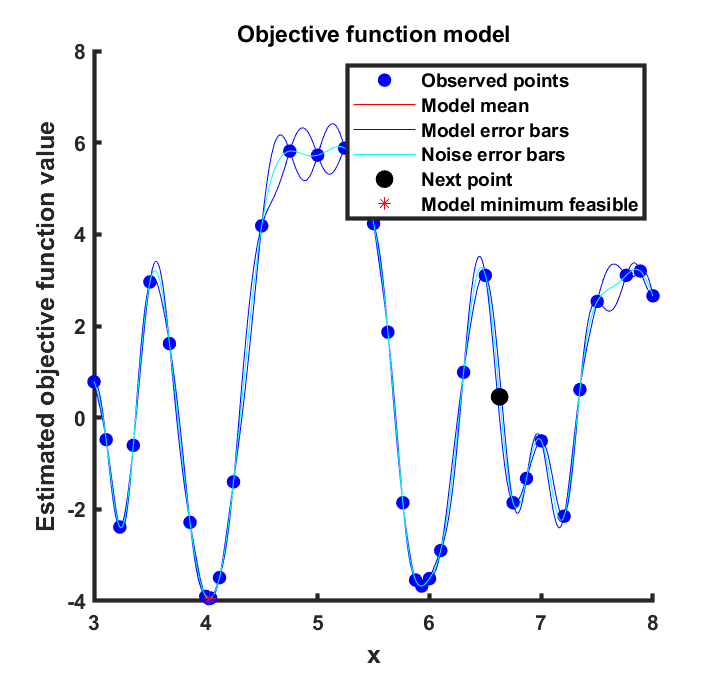}
\caption{bayesopt40}
\end{subfigure}
\begin{subfigure}[b]{0.330\textwidth}
\includegraphics[width=\textwidth]{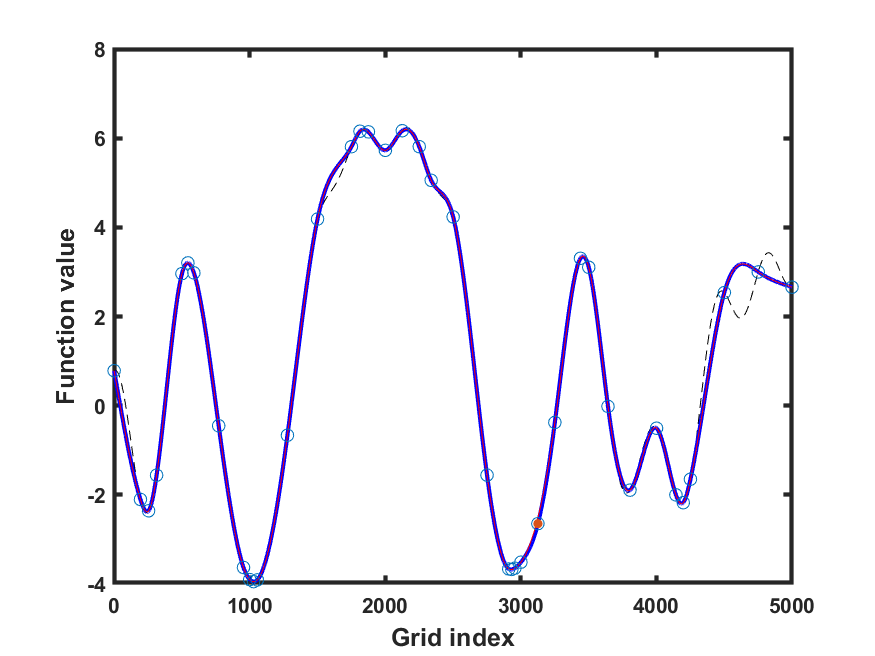}
\caption{LineWalker40}
\end{subfigure}
\newline
\begin{subfigure}[b]{0.270\textwidth}
\includegraphics[width=\textwidth]{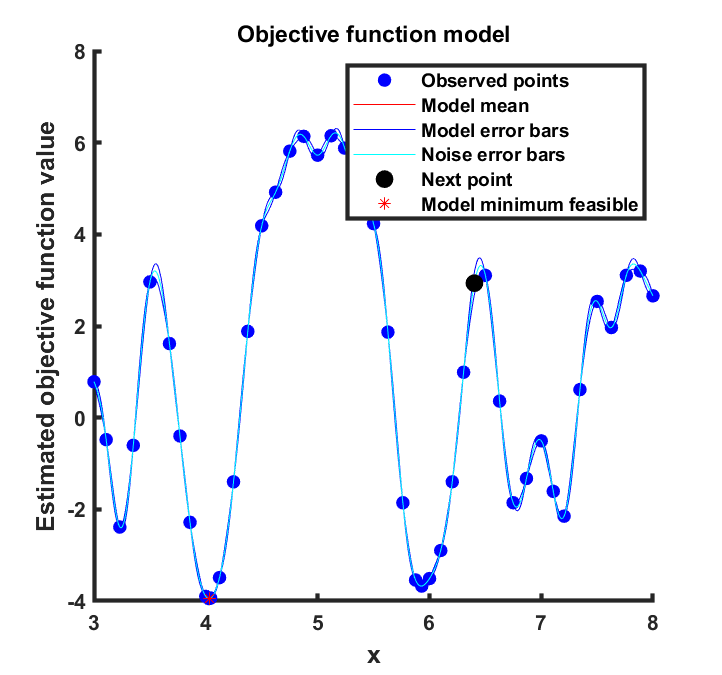}
\caption{bayesopt50}
\end{subfigure}
\begin{subfigure}[b]{0.330\textwidth}
\includegraphics[width=\textwidth]{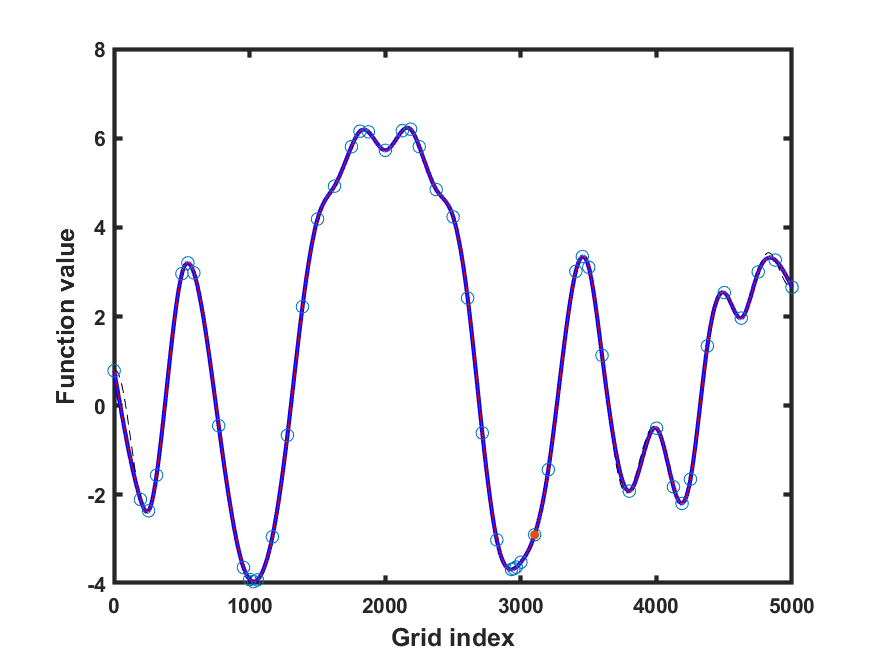}
\caption{LineWalker50}
\end{subfigure}
\newline
\caption{langer2. Left column = \texttt{bayesopt}. Right column = \texttt{LineWalker-full}}
\label{fig:out_langer2}
\end{figure}

\begin{figure} [h!]
\centering
\begin{subfigure}[b]{0.270\textwidth}
\includegraphics[width=\textwidth]{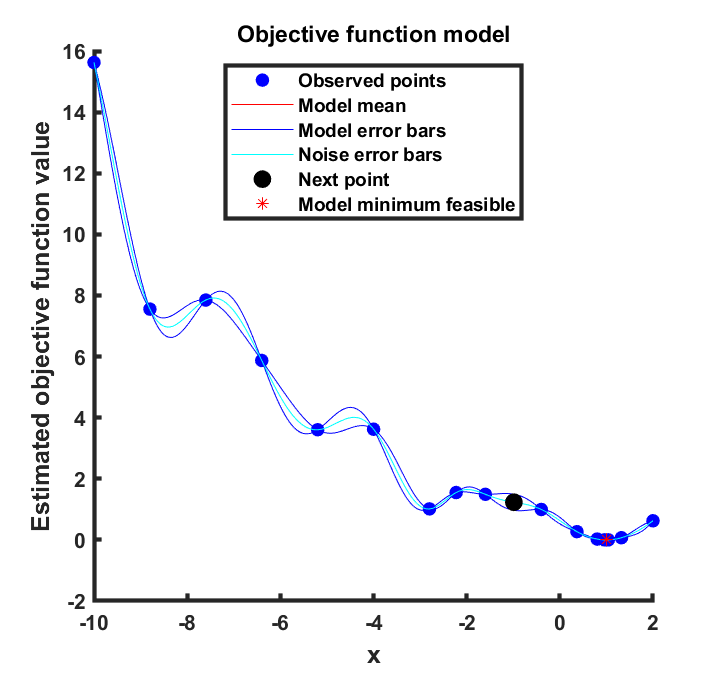}
\caption{bayesopt20}
\end{subfigure}
\begin{subfigure}[b]{0.330\textwidth}
\includegraphics[width=\textwidth]{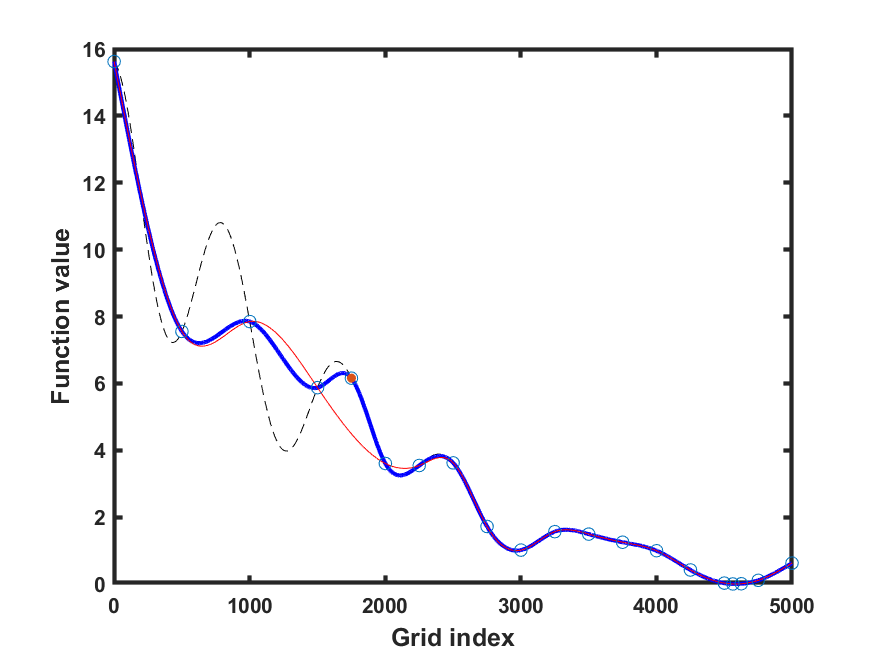}
\caption{LineWalker20}
\end{subfigure}
\newline
\begin{subfigure}[b]{0.270\textwidth}
\includegraphics[width=\textwidth]{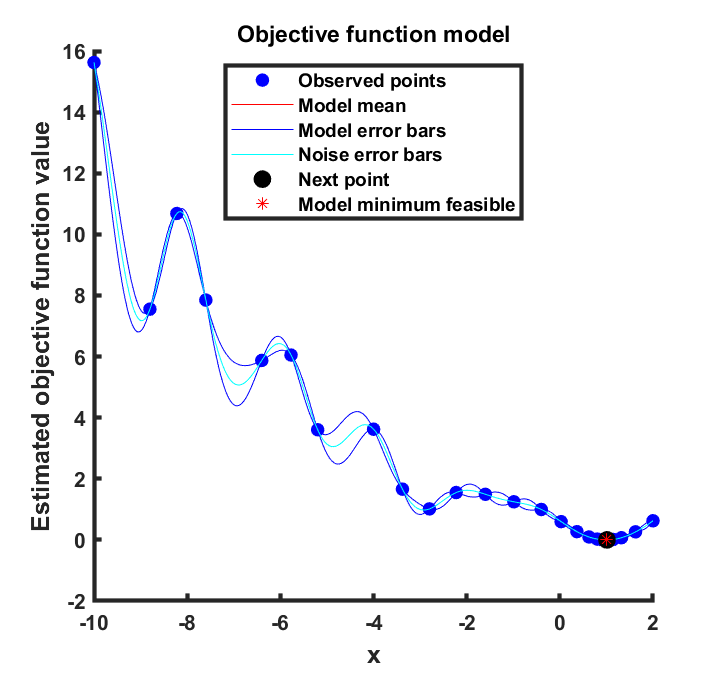}
\caption{bayesopt30}
\end{subfigure}
\begin{subfigure}[b]{0.330\textwidth}
\includegraphics[width=\textwidth]{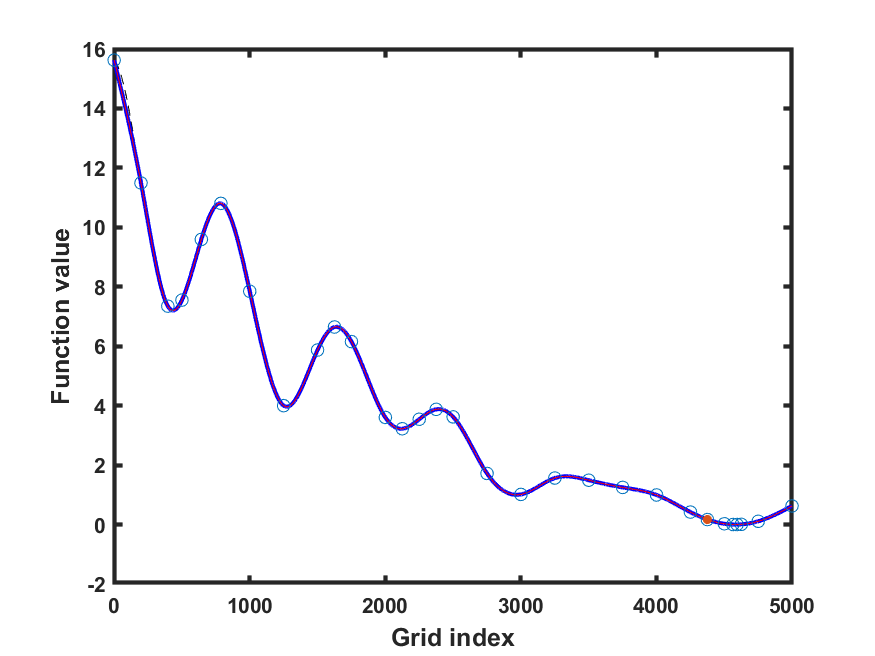}
\caption{LineWalker30}
\end{subfigure}
\newline
\begin{subfigure}[b]{0.270\textwidth}
\includegraphics[width=\textwidth]{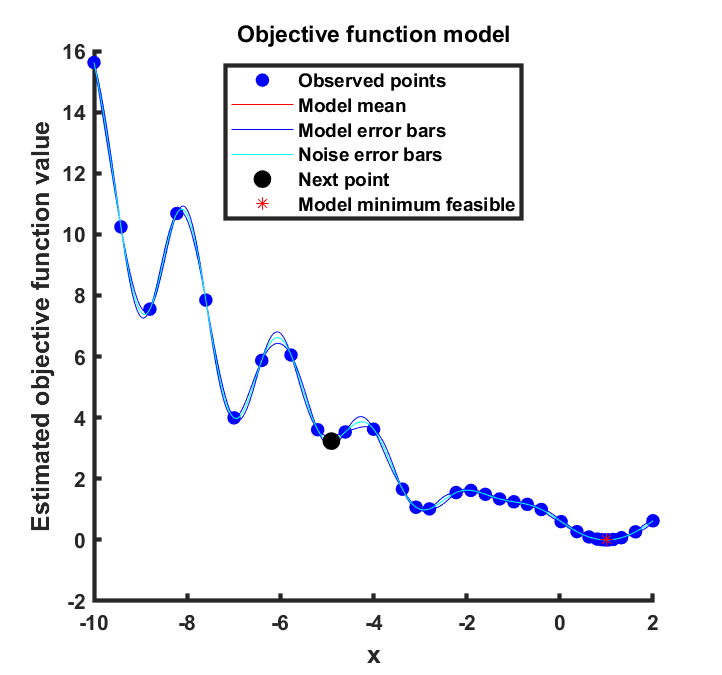}
\caption{bayesopt40}
\end{subfigure}
\begin{subfigure}[b]{0.330\textwidth}
\includegraphics[width=\textwidth]{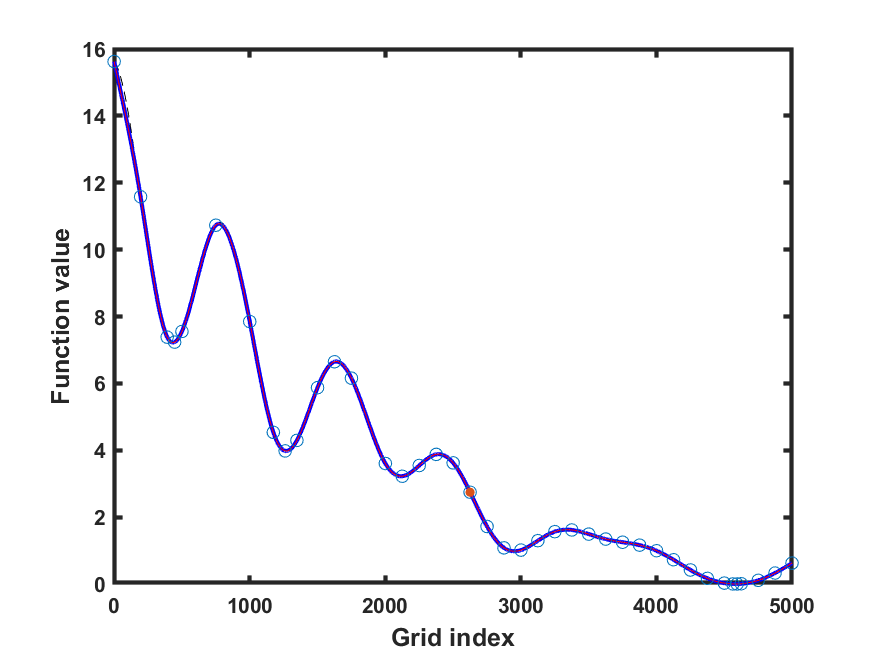}
\caption{LineWalker40}
\end{subfigure}
\newline
\begin{subfigure}[b]{0.270\textwidth}
\includegraphics[width=\textwidth]{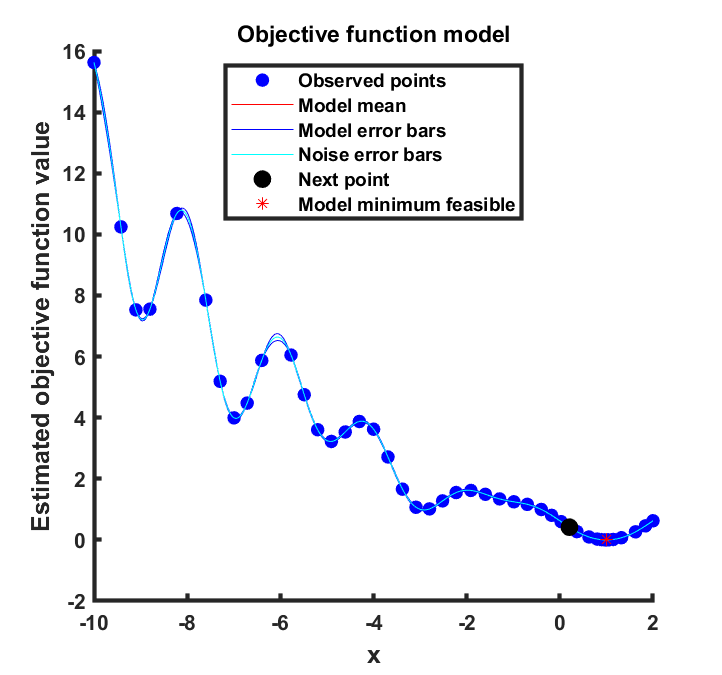}
\caption{bayesopt50}
\end{subfigure}
\begin{subfigure}[b]{0.330\textwidth}
\includegraphics[width=\textwidth]{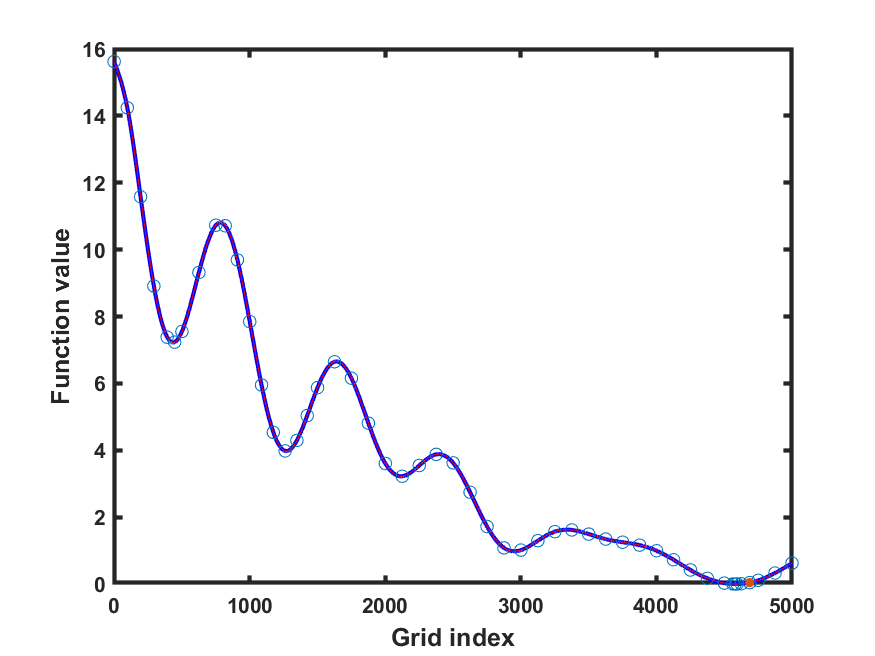}
\caption{LineWalker50}
\end{subfigure}
\newline
\caption{levy. Left column = \texttt{bayesopt}. Right column = \texttt{LineWalker-full}}
\label{fig:out_levy}
\end{figure}

\begin{figure} [h!]
\centering
\begin{subfigure}[b]{0.270\textwidth}
\includegraphics[width=\textwidth]{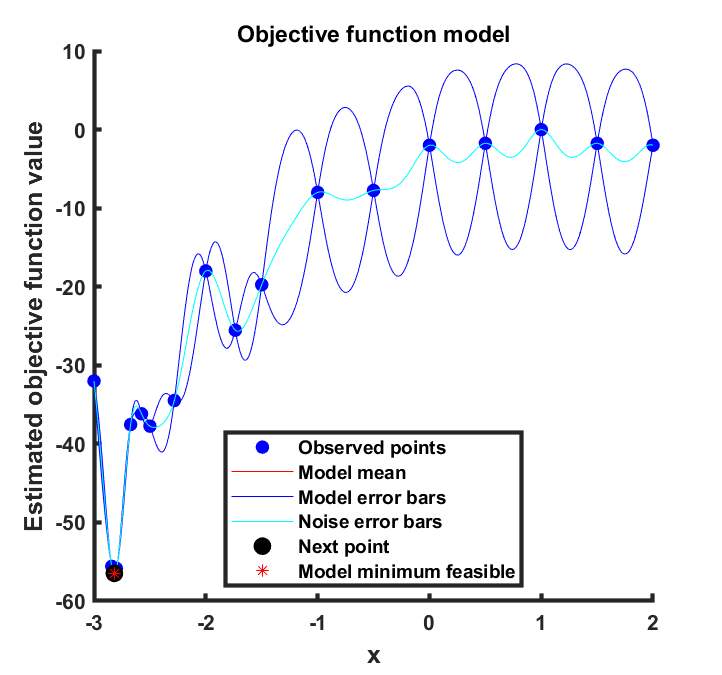}
\caption{bayesopt20}
\end{subfigure}
\begin{subfigure}[b]{0.330\textwidth}
\includegraphics[width=\textwidth]{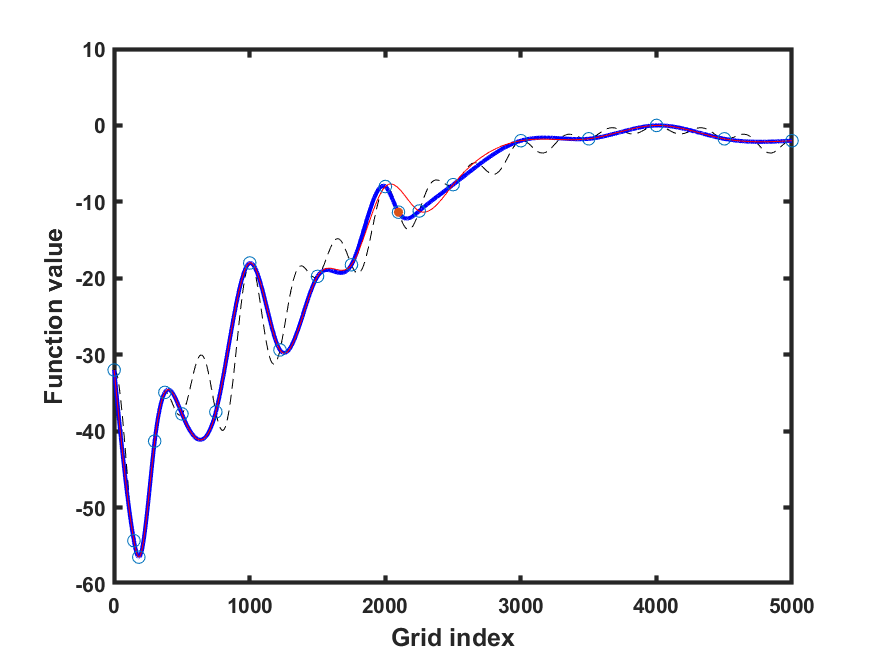}
\caption{LineWalker20}
\end{subfigure}
\newline
\begin{subfigure}[b]{0.270\textwidth}
\includegraphics[width=\textwidth]{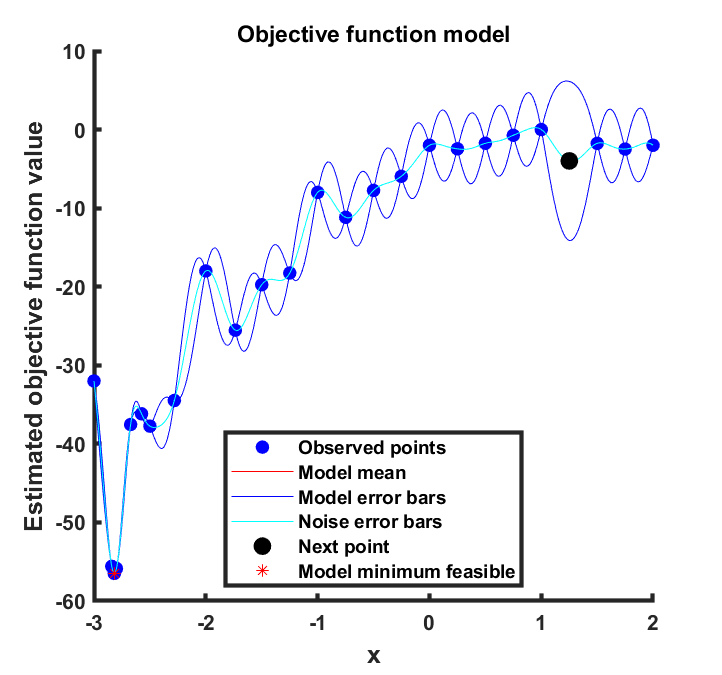}
\caption{bayesopt30}
\end{subfigure}
\begin{subfigure}[b]{0.330\textwidth}
\includegraphics[width=\textwidth]{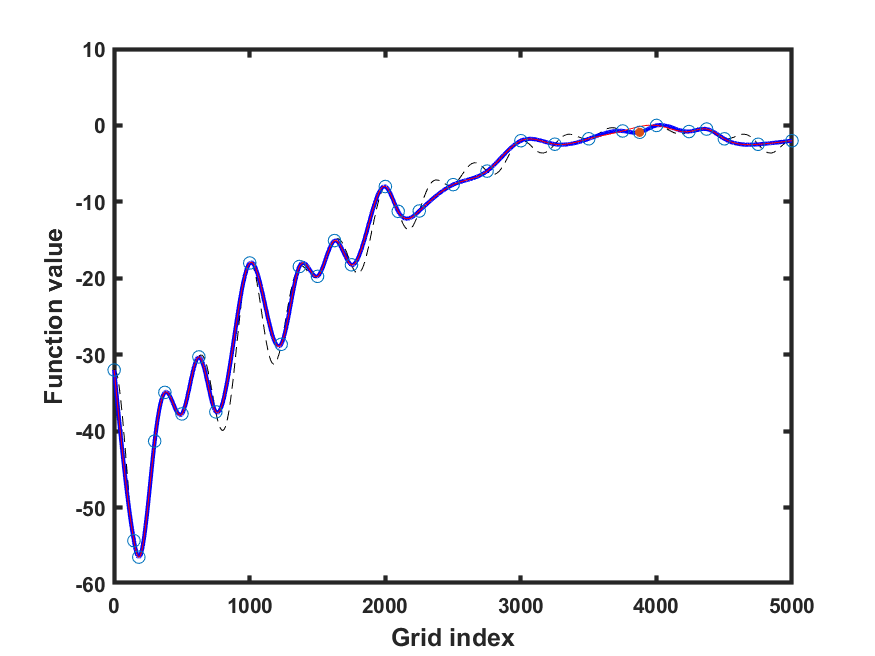}
\caption{LineWalker30}
\end{subfigure}
\newline
\begin{subfigure}[b]{0.270\textwidth}
\includegraphics[width=\textwidth]{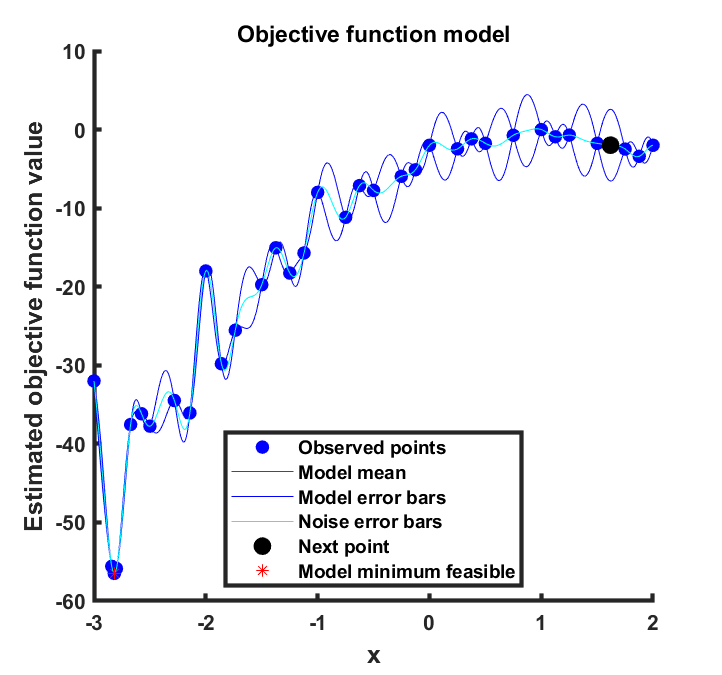}
\caption{bayesopt40}
\end{subfigure}
\begin{subfigure}[b]{0.330\textwidth}
\includegraphics[width=\textwidth]{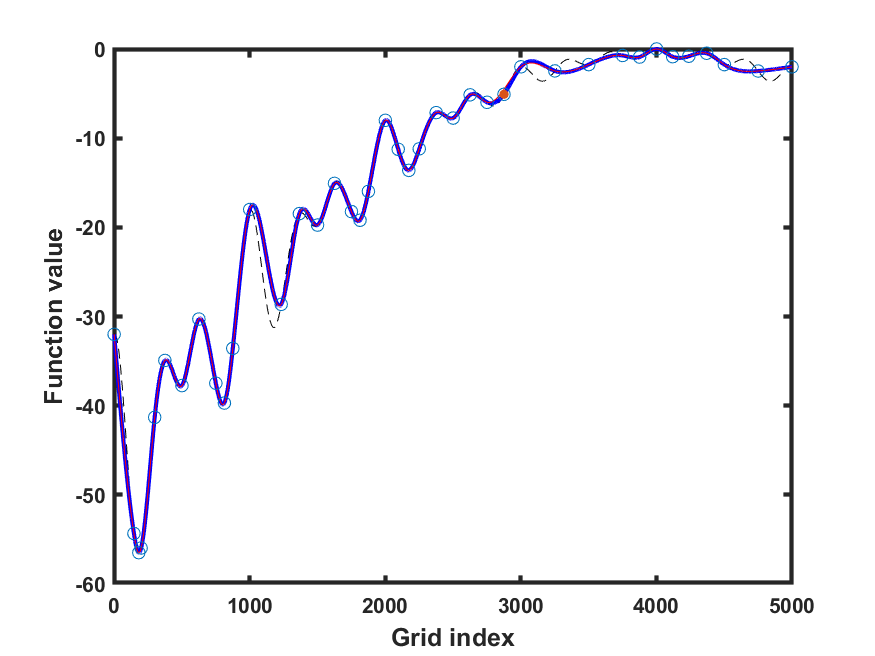}
\caption{LineWalker40}
\end{subfigure}
\newline
\begin{subfigure}[b]{0.270\textwidth}
\includegraphics[width=\textwidth]{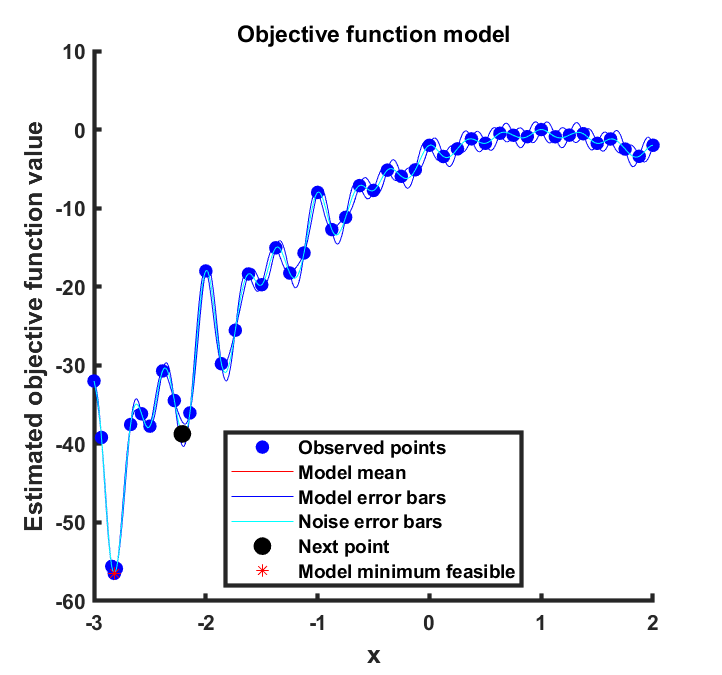}
\caption{bayesopt50}
\end{subfigure}
\begin{subfigure}[b]{0.330\textwidth}
\includegraphics[width=\textwidth]{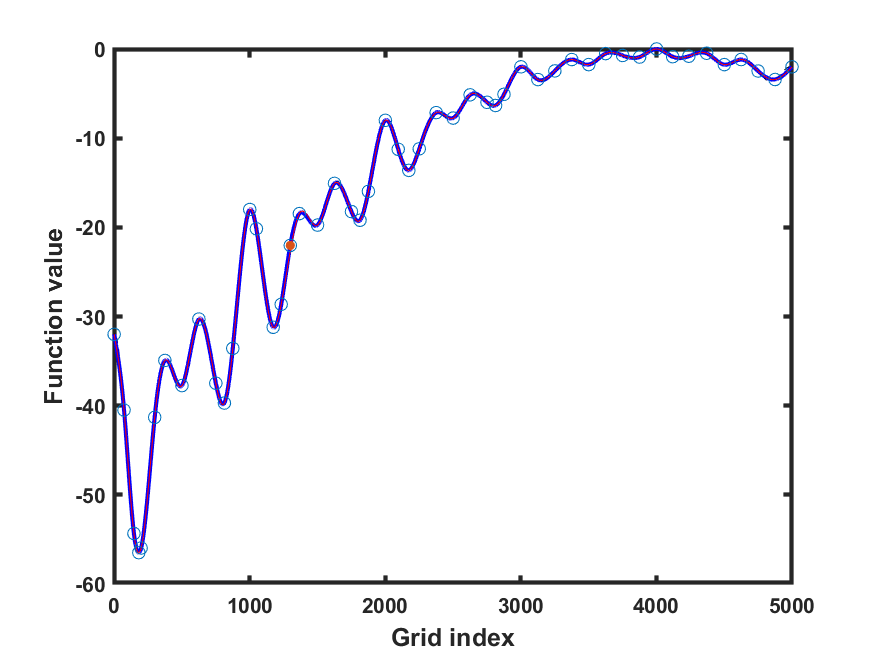}
\caption{LineWalker50}
\end{subfigure}
\newline
\caption{levy13. Left column = \texttt{bayesopt}. Right column = \texttt{LineWalker-full}}
\label{fig:out_levy13_1Dslice}
\end{figure}

\begin{figure} [h!]
\centering
\begin{subfigure}[b]{0.270\textwidth}
\includegraphics[width=\textwidth]{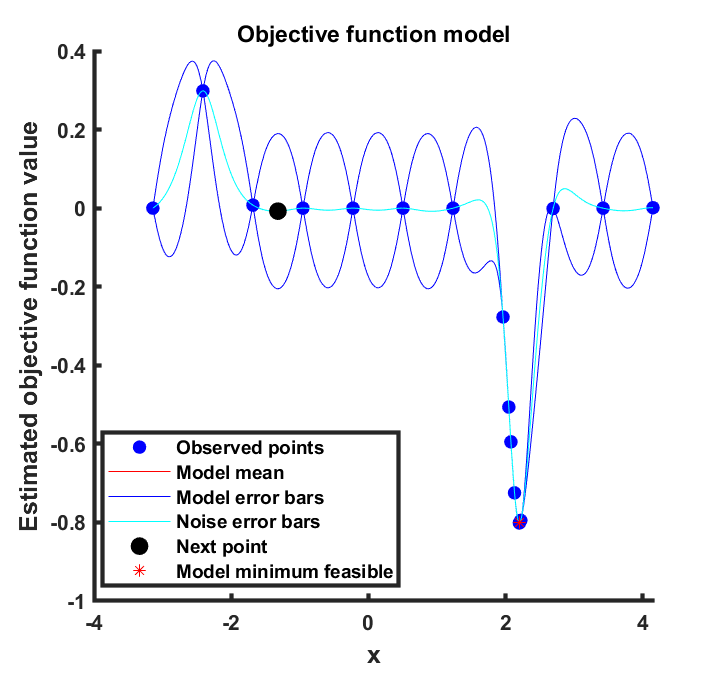}
\caption{bayesopt20}
\end{subfigure}
\begin{subfigure}[b]{0.330\textwidth}
\includegraphics[width=\textwidth]{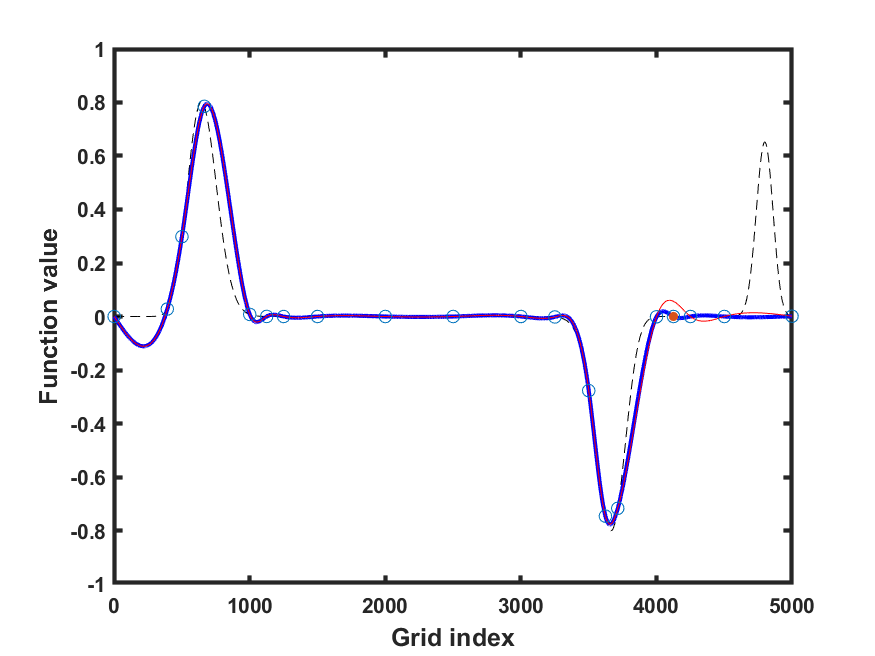}
\caption{LineWalker20}
\end{subfigure}
\newline
\begin{subfigure}[b]{0.270\textwidth}
\includegraphics[width=\textwidth]{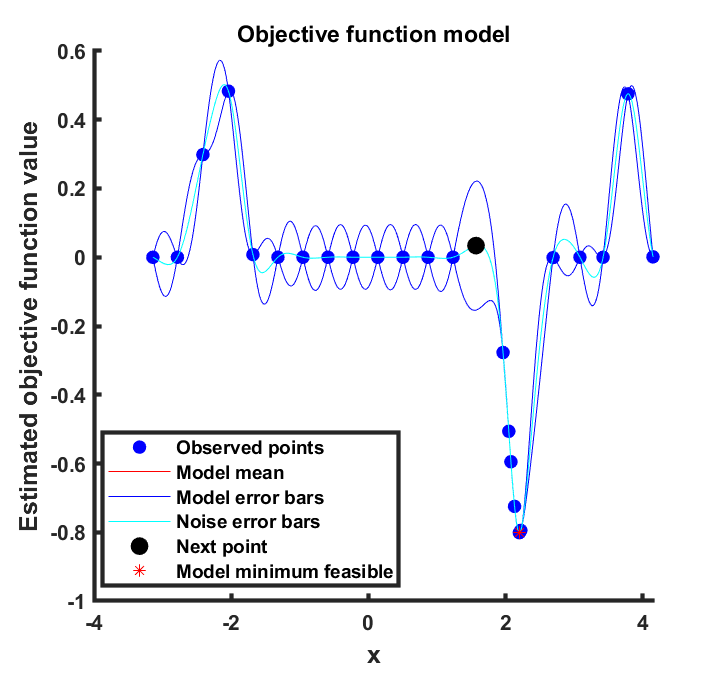}
\caption{bayesopt30}
\end{subfigure}
\begin{subfigure}[b]{0.330\textwidth}
\includegraphics[width=\textwidth]{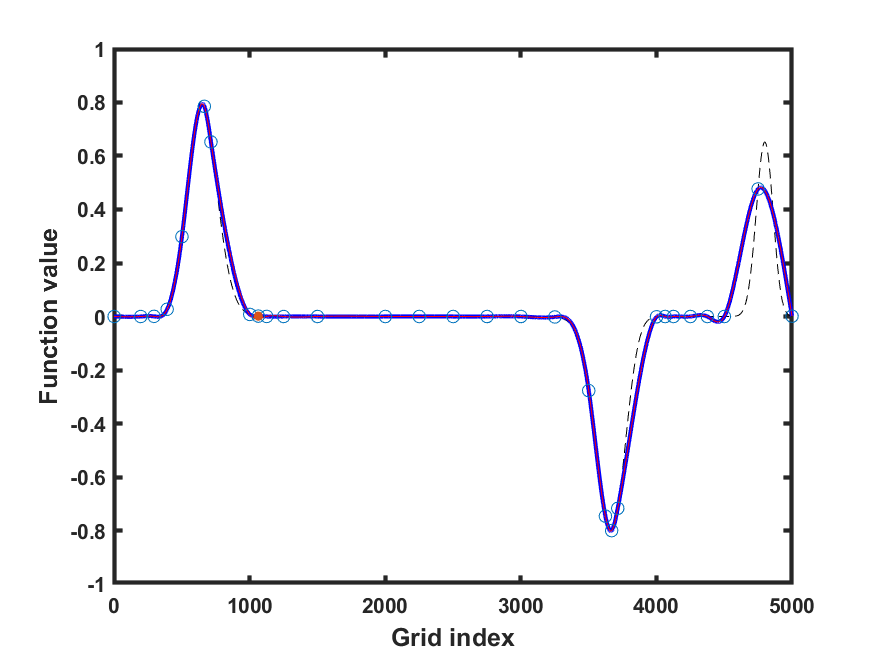}
\caption{LineWalker30}
\end{subfigure}
\newline
\begin{subfigure}[b]{0.270\textwidth}
\includegraphics[width=\textwidth]{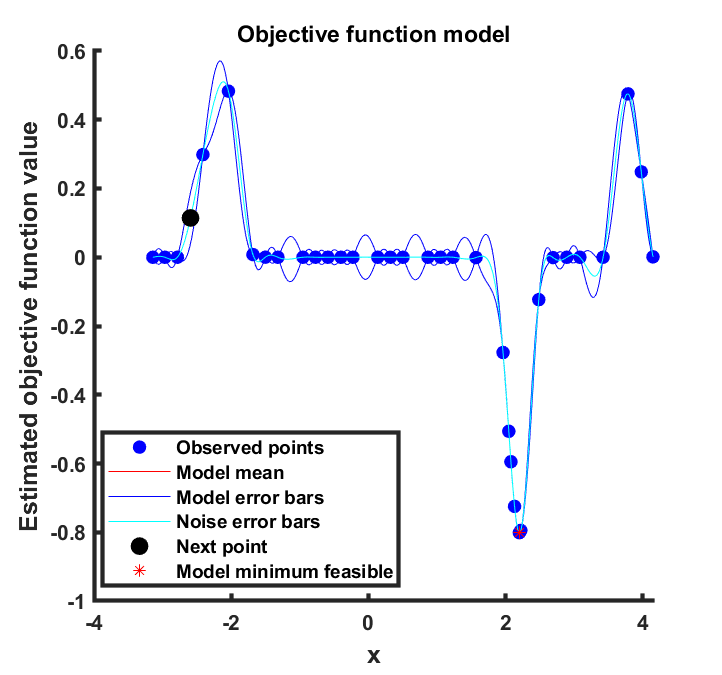}
\caption{bayesopt40}
\end{subfigure}
\begin{subfigure}[b]{0.330\textwidth}
\includegraphics[width=\textwidth]{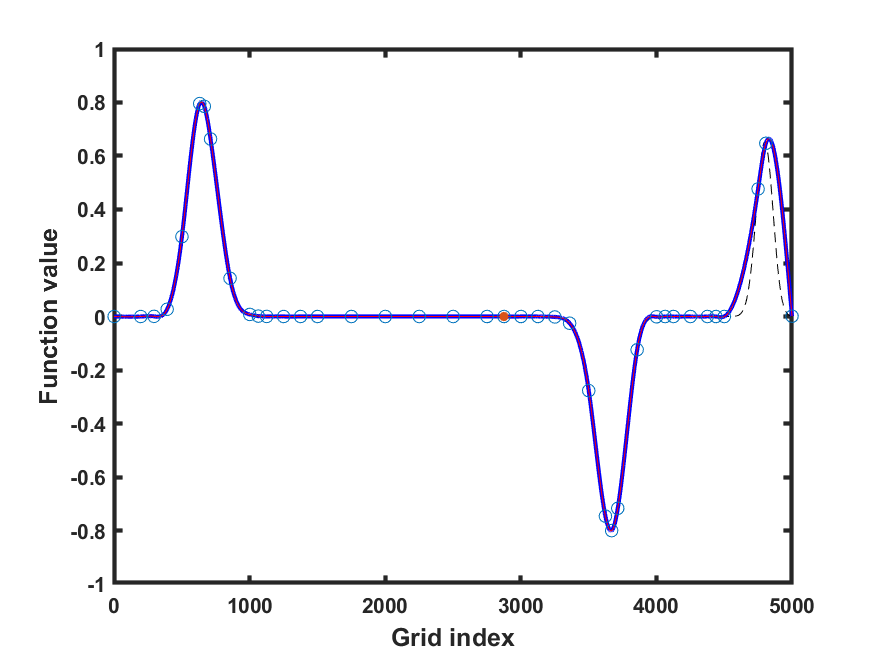}
\caption{LineWalker40}
\end{subfigure}
\newline
\begin{subfigure}[b]{0.270\textwidth}
\includegraphics[width=\textwidth]{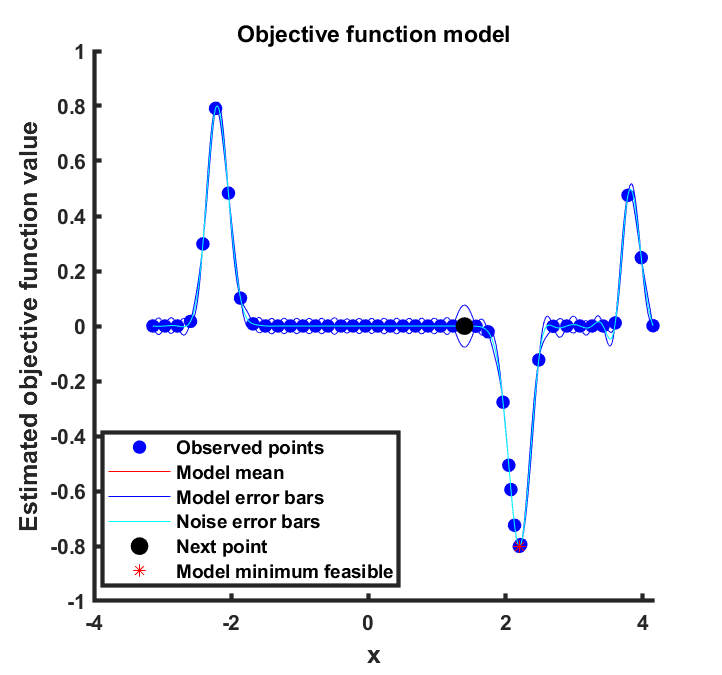}
\caption{bayesopt50}
\end{subfigure}
\begin{subfigure}[b]{0.330\textwidth}
\includegraphics[width=\textwidth]{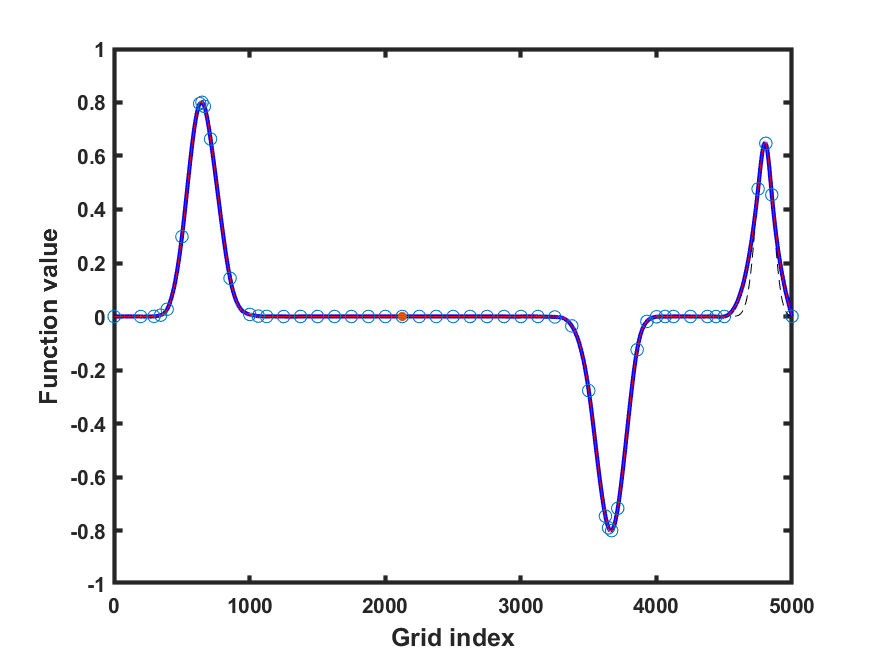}
\caption{LineWalker50}
\end{subfigure}
\newline
\caption{michal. Left column = \texttt{bayesopt}. Right column = \texttt{LineWalker-full}}
\label{fig:out_michal}
\end{figure}

\begin{figure} [h!]
\centering
\begin{subfigure}[b]{0.270\textwidth}
\includegraphics[width=\textwidth]{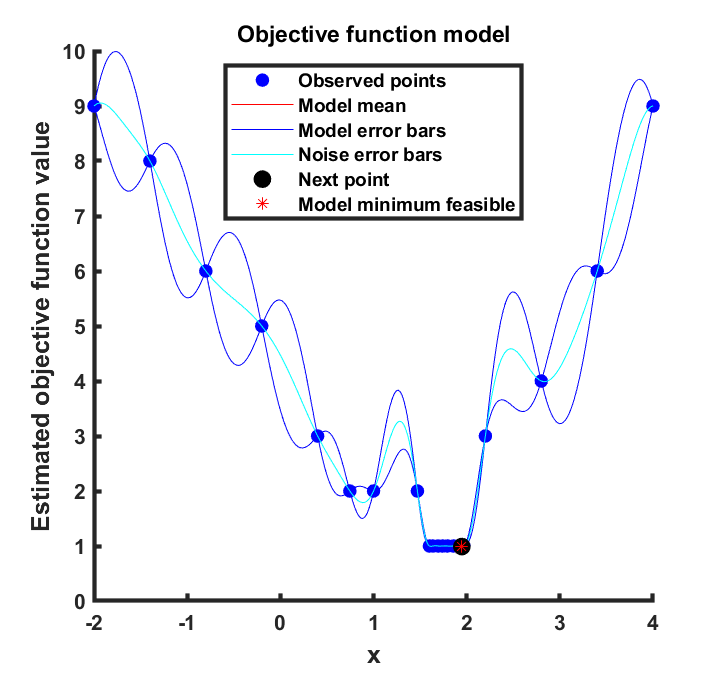}
\caption{bayesopt20}
\end{subfigure}
\begin{subfigure}[b]{0.330\textwidth}
\includegraphics[width=\textwidth]{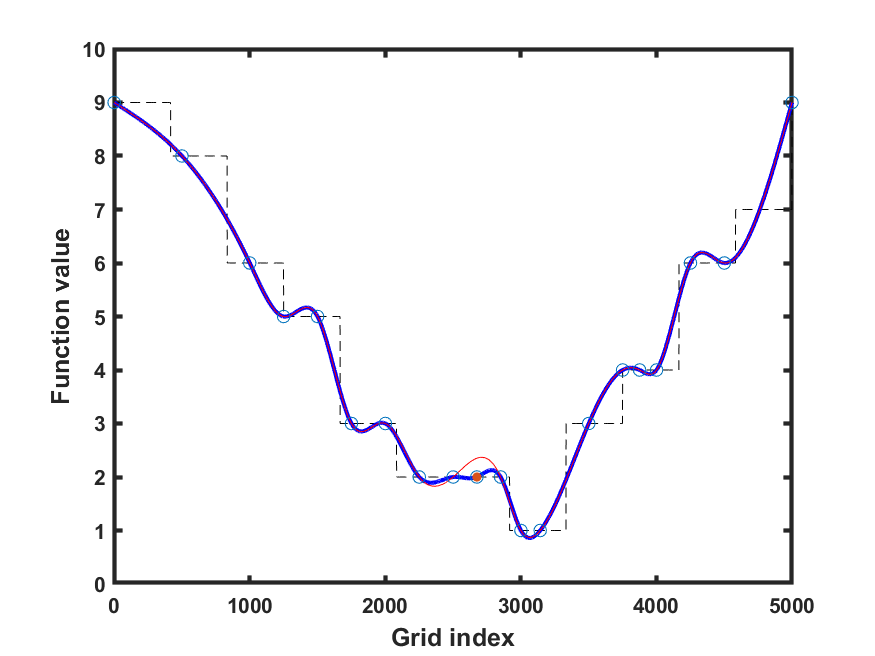}
\caption{LineWalker20}
\end{subfigure}
\newline
\begin{subfigure}[b]{0.270\textwidth}
\includegraphics[width=\textwidth]{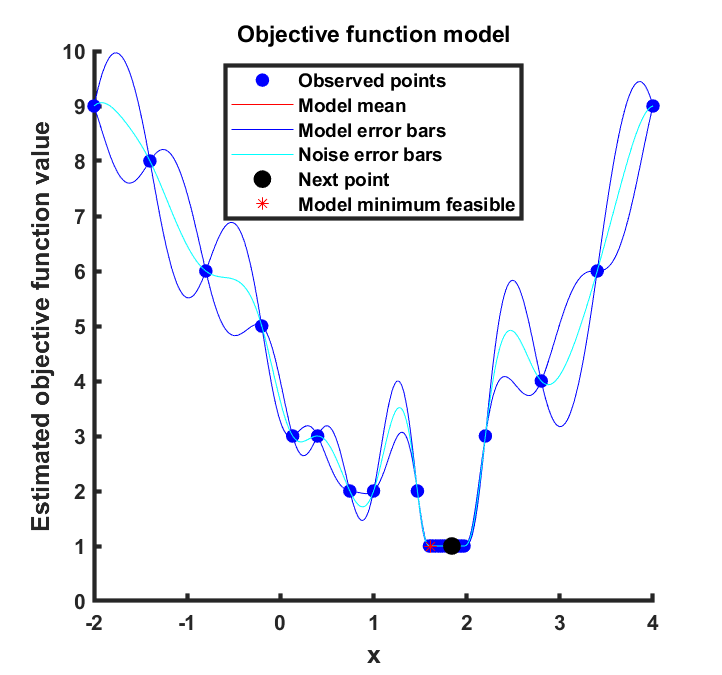}
\caption{bayesopt30}
\end{subfigure}
\begin{subfigure}[b]{0.330\textwidth}
\includegraphics[width=\textwidth]{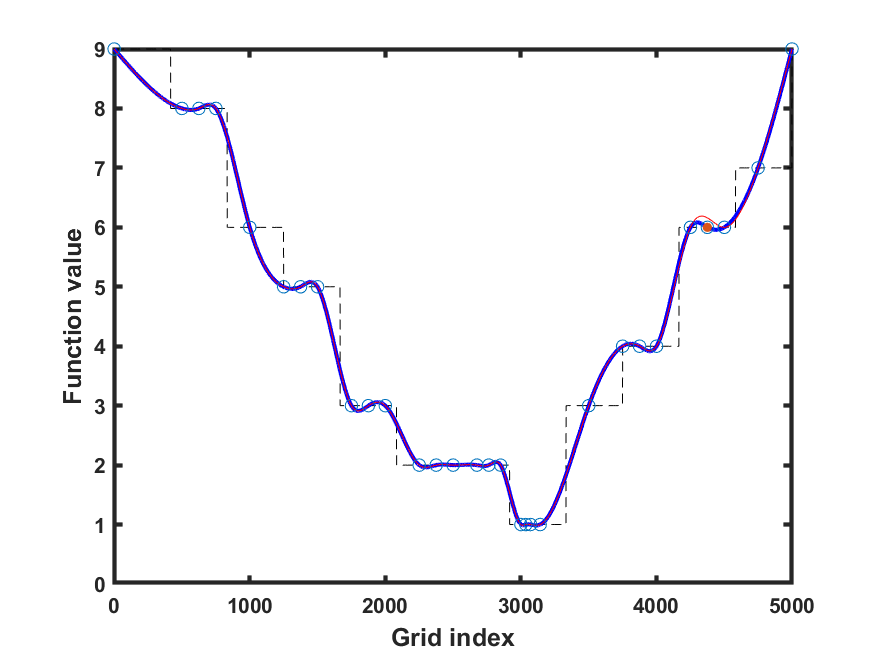}
\caption{LineWalker30}
\end{subfigure}
\newline
\begin{subfigure}[b]{0.270\textwidth}
\includegraphics[width=\textwidth]{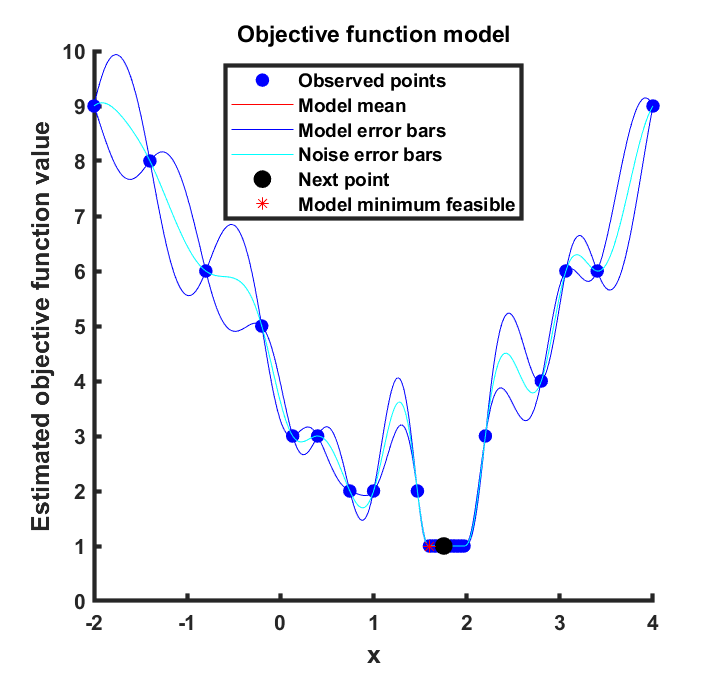}
\caption{bayesopt40}
\end{subfigure}
\begin{subfigure}[b]{0.330\textwidth}
\includegraphics[width=\textwidth]{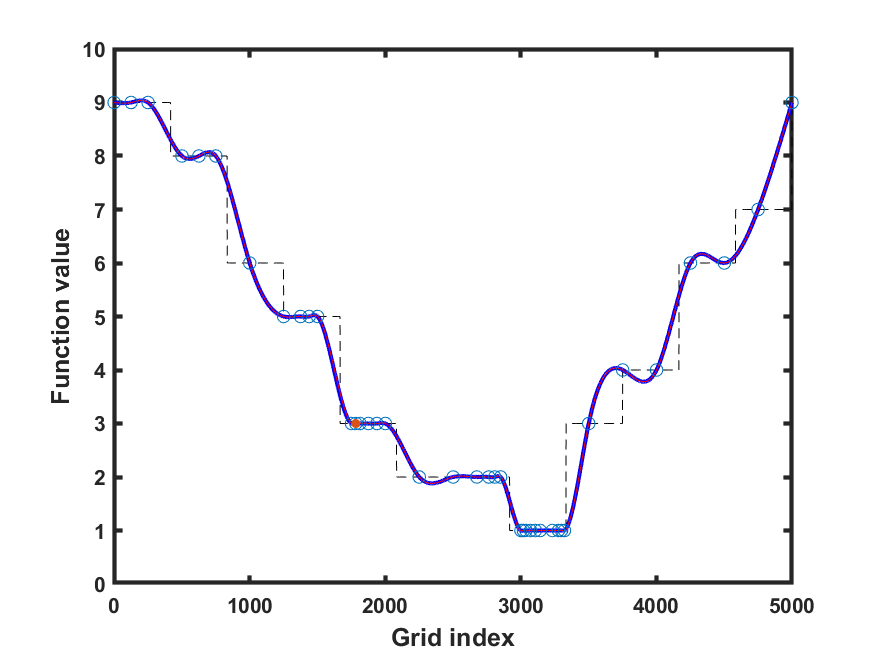}
\caption{LineWalker40}
\end{subfigure}
\newline
\begin{subfigure}[b]{0.270\textwidth}
\includegraphics[width=\textwidth]{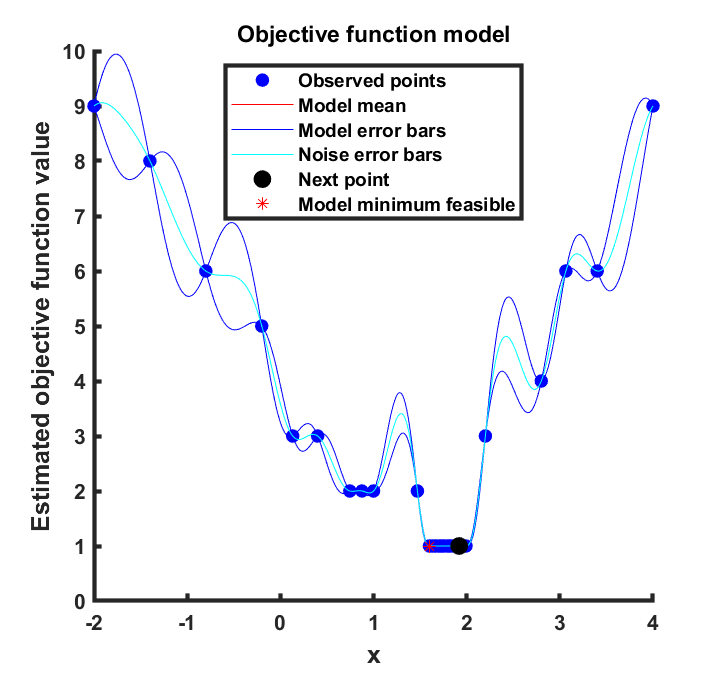}
\caption{bayesopt50}
\end{subfigure}
\begin{subfigure}[b]{0.330\textwidth}
\includegraphics[width=\textwidth]{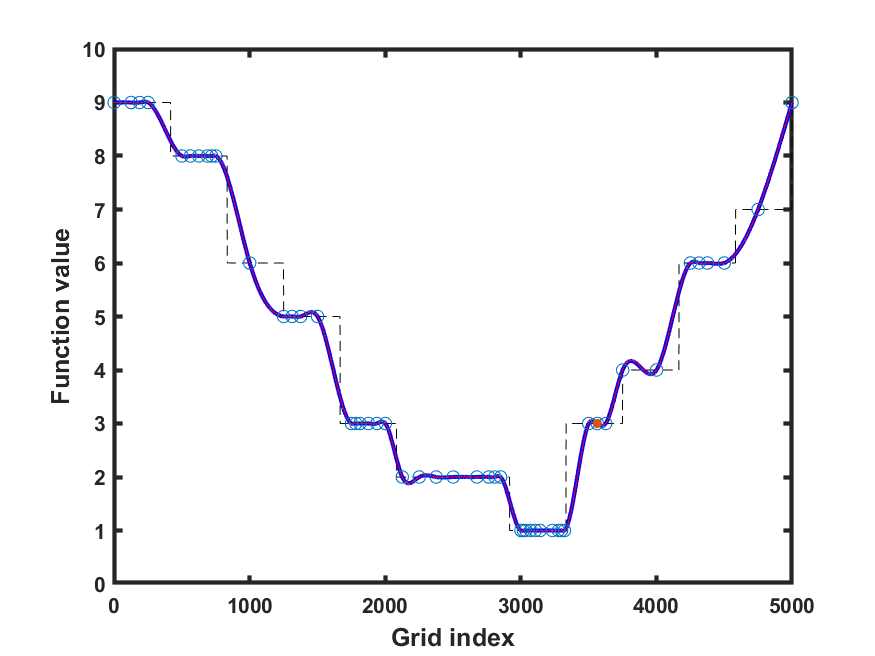}
\caption{LineWalker50}
\end{subfigure}
\newline
\caption{plateau. Left column = \texttt{bayesopt}. Right column = \texttt{LineWalker-full}}
\label{fig:out_plateau_1Dslice}
\end{figure}

\begin{figure} [h!]
\centering
\begin{subfigure}[b]{0.270\textwidth}
\includegraphics[width=\textwidth]{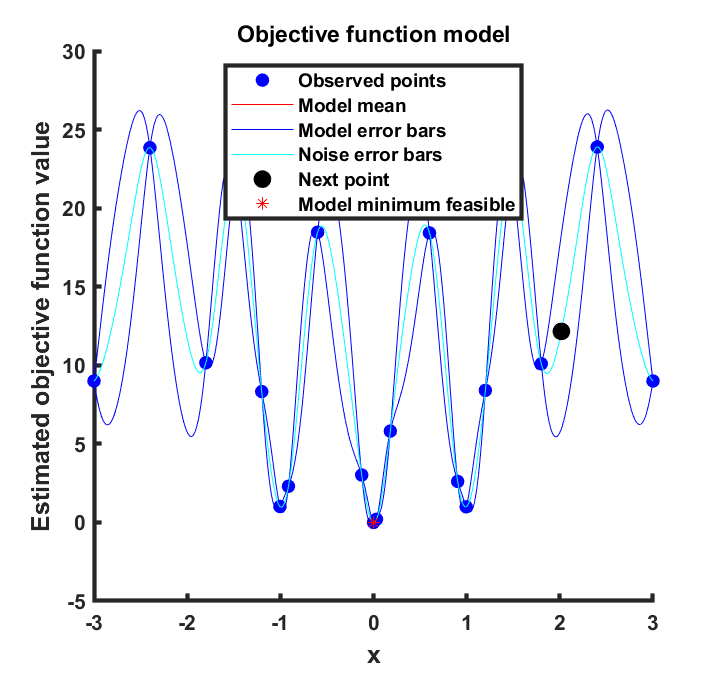}
\caption{bayesopt20}
\end{subfigure}
\begin{subfigure}[b]{0.330\textwidth}
\includegraphics[width=\textwidth]{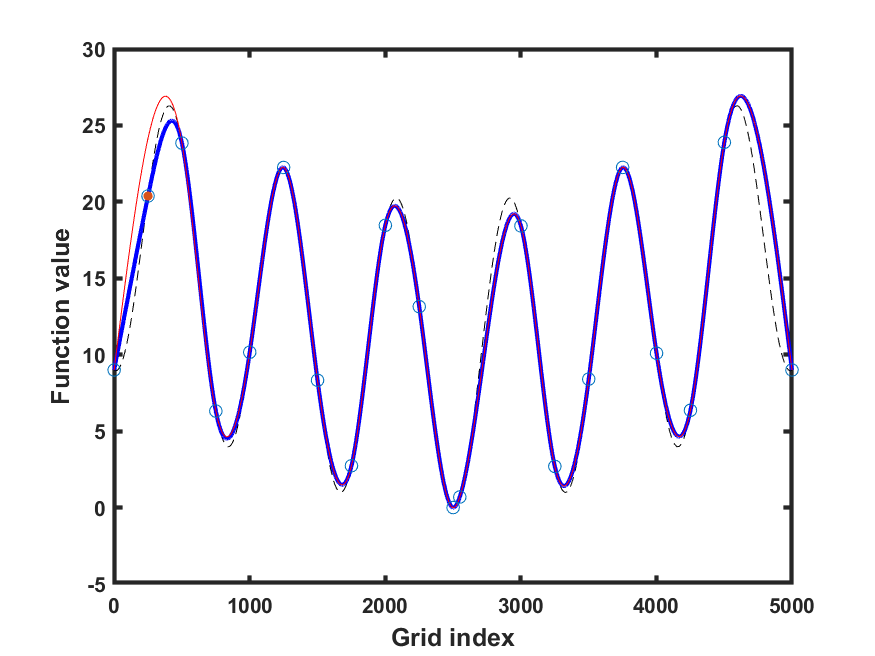}
\caption{LineWalker20}
\end{subfigure}
\newline
\begin{subfigure}[b]{0.270\textwidth}
\includegraphics[width=\textwidth]{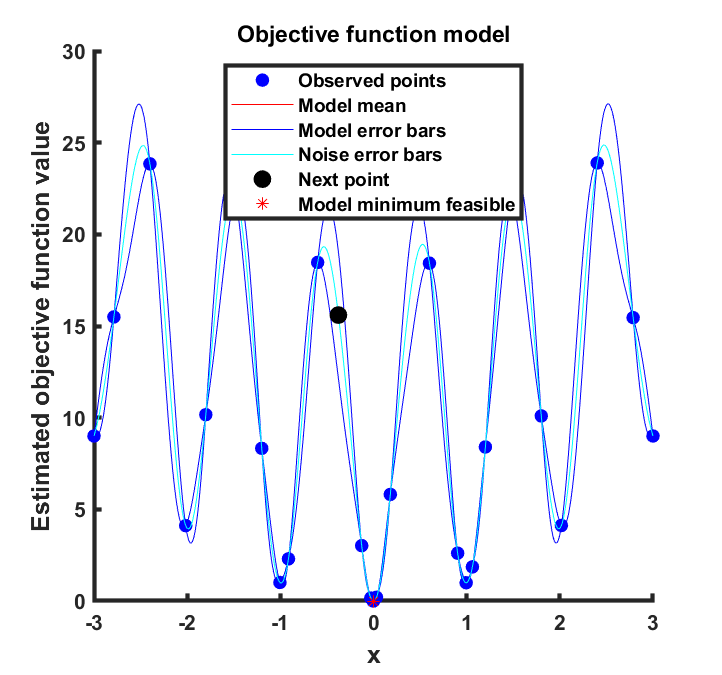}
\caption{bayesopt30}
\end{subfigure}
\begin{subfigure}[b]{0.330\textwidth}
\includegraphics[width=\textwidth]{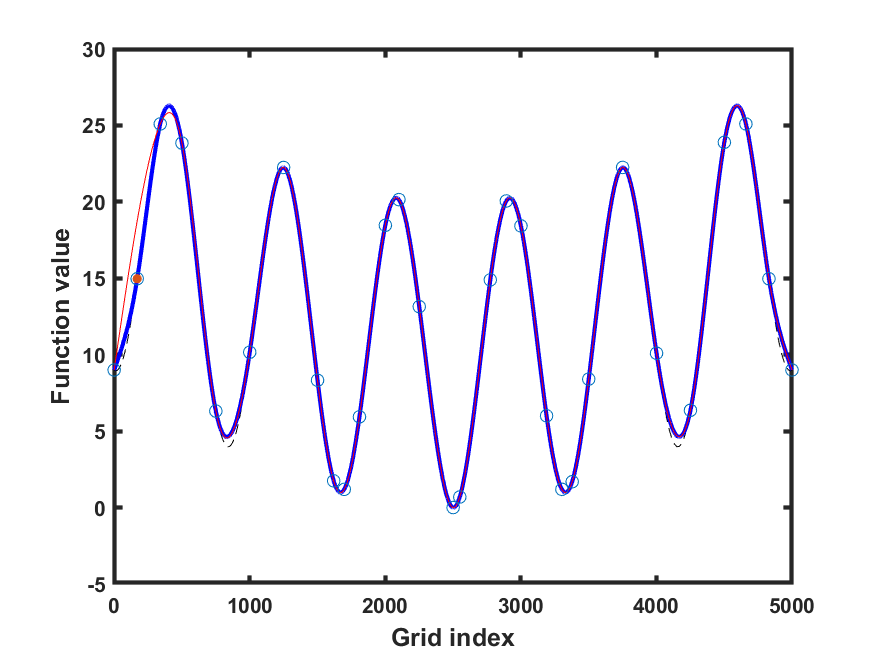}
\caption{LineWalker30}
\end{subfigure}
\newline
\begin{subfigure}[b]{0.270\textwidth}
\includegraphics[width=\textwidth]{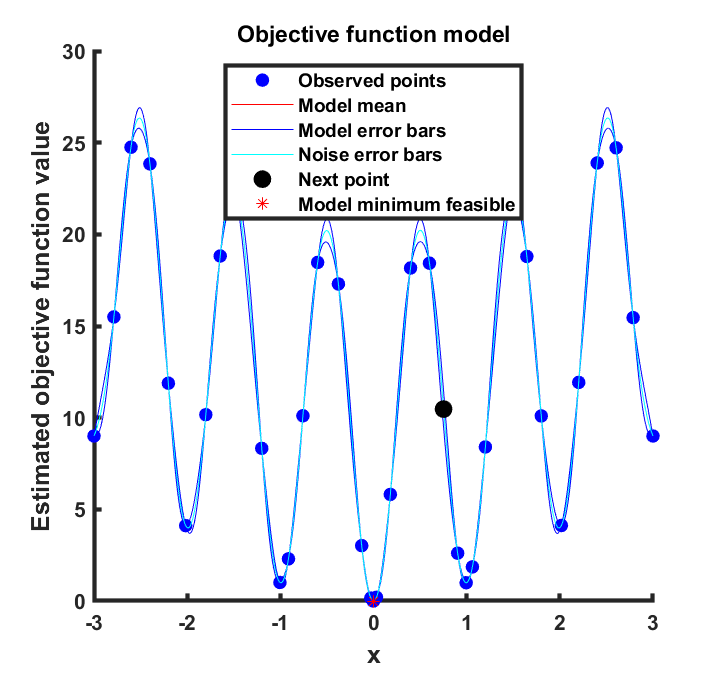}
\caption{bayesopt40}
\end{subfigure}
\begin{subfigure}[b]{0.330\textwidth}
\includegraphics[width=\textwidth]{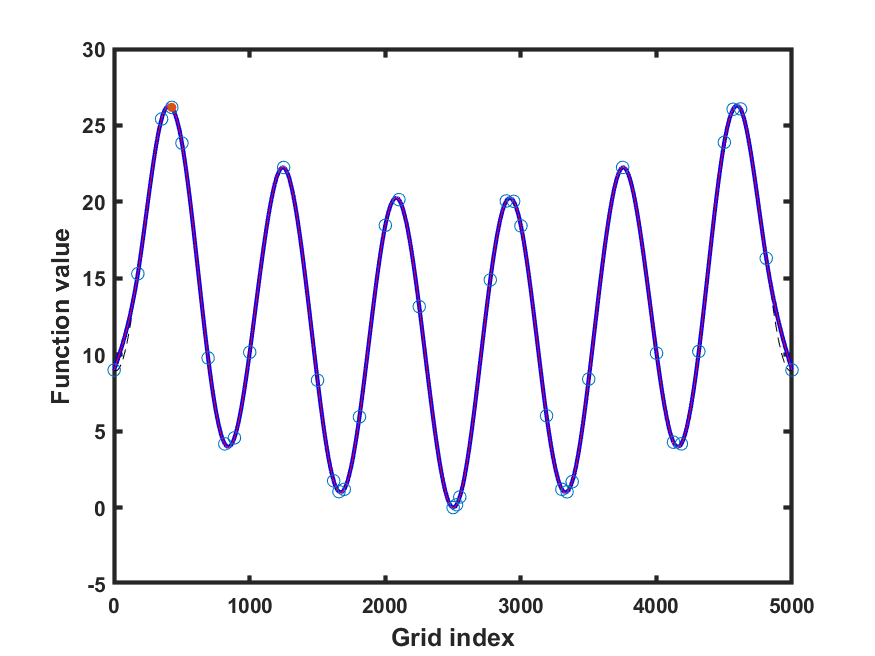}
\caption{LineWalker40}
\end{subfigure}
\newline
\begin{subfigure}[b]{0.270\textwidth}
\includegraphics[width=\textwidth]{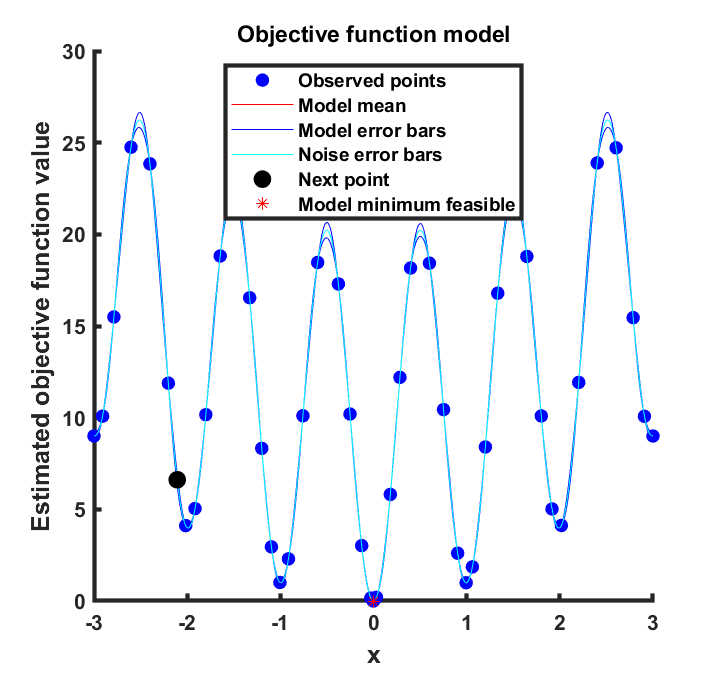}
\caption{bayesopt50}
\end{subfigure}
\begin{subfigure}[b]{0.330\textwidth}
\includegraphics[width=\textwidth]{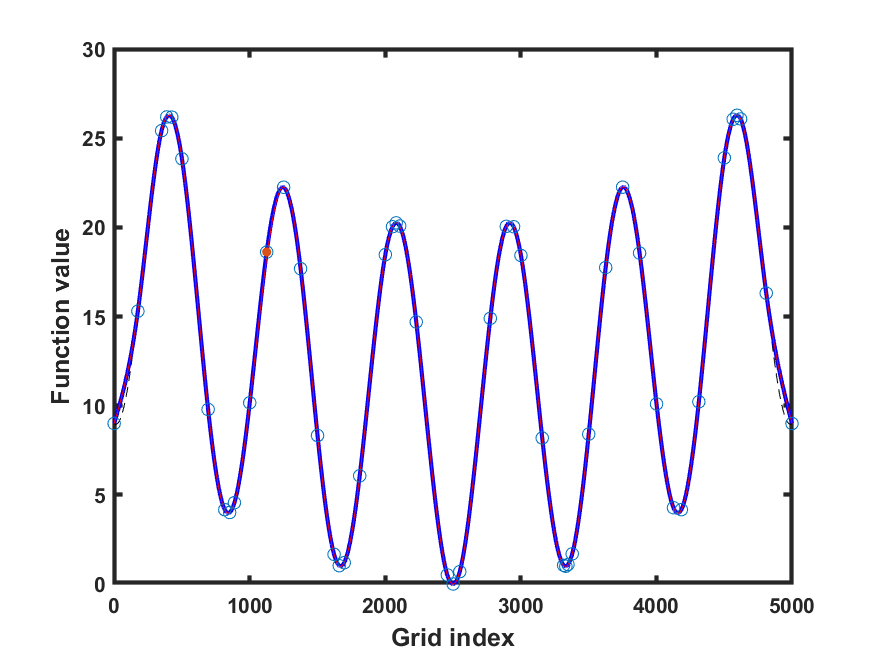}
\caption{LineWalker50}
\end{subfigure}
\newline
\caption{rastr. Left column = \texttt{bayesopt}. Right column = \texttt{LineWalker-full}}
\label{fig:out_rastr}
\end{figure}

\begin{figure}
\centering
\begin{subfigure}[b]{0.270\textwidth}
\includegraphics[width=\textwidth]{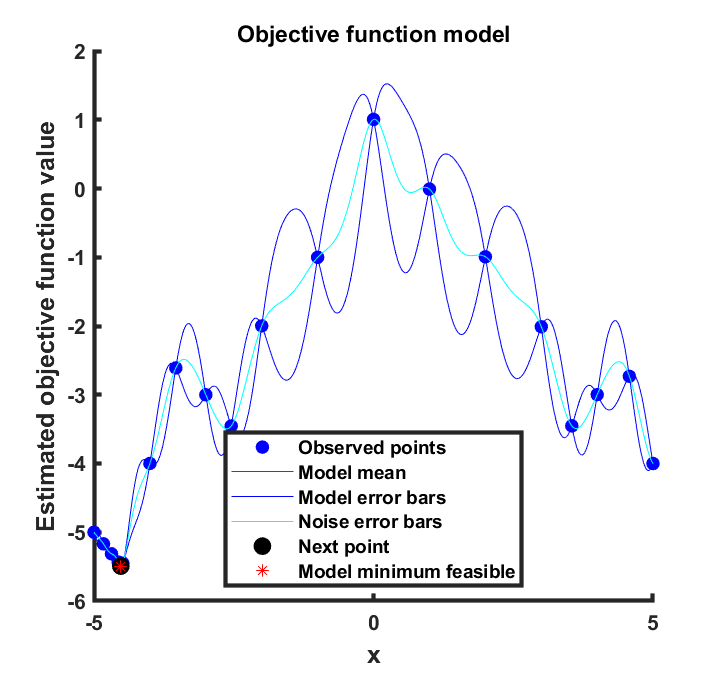}
\caption{bayesopt20}
\end{subfigure}
\begin{subfigure}[b]{0.330\textwidth}
\includegraphics[width=\textwidth]{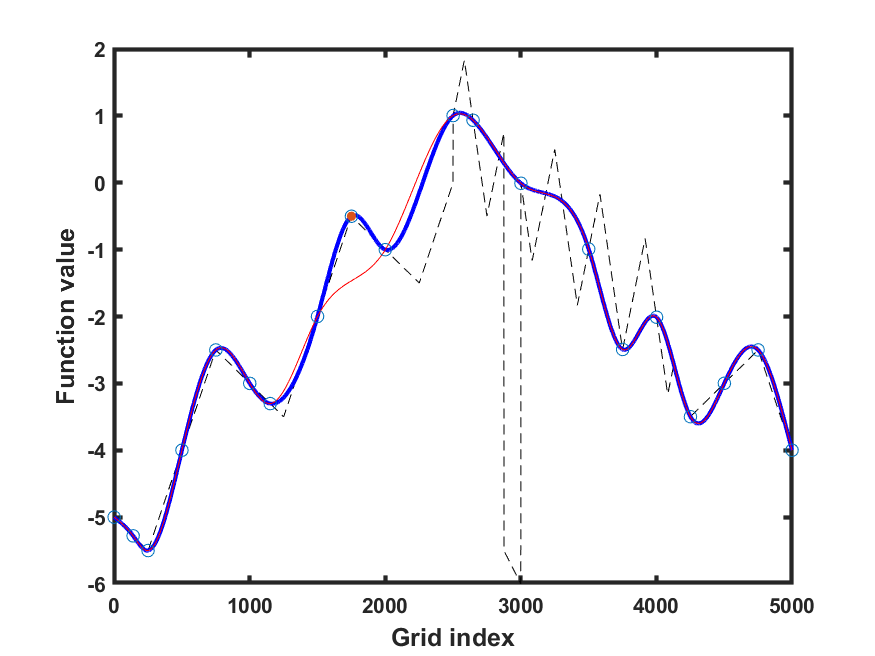}
\caption{LineWalker20}
\end{subfigure}
\newline
\begin{subfigure}[b]{0.270\textwidth}
\includegraphics[width=\textwidth]{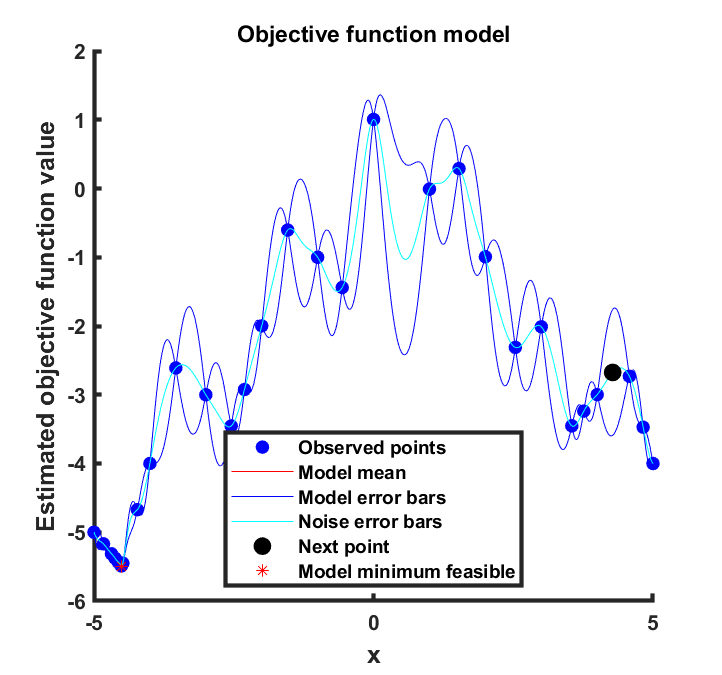}
\caption{bayesopt30}
\end{subfigure}
\begin{subfigure}[b]{0.330\textwidth}
\includegraphics[width=\textwidth]{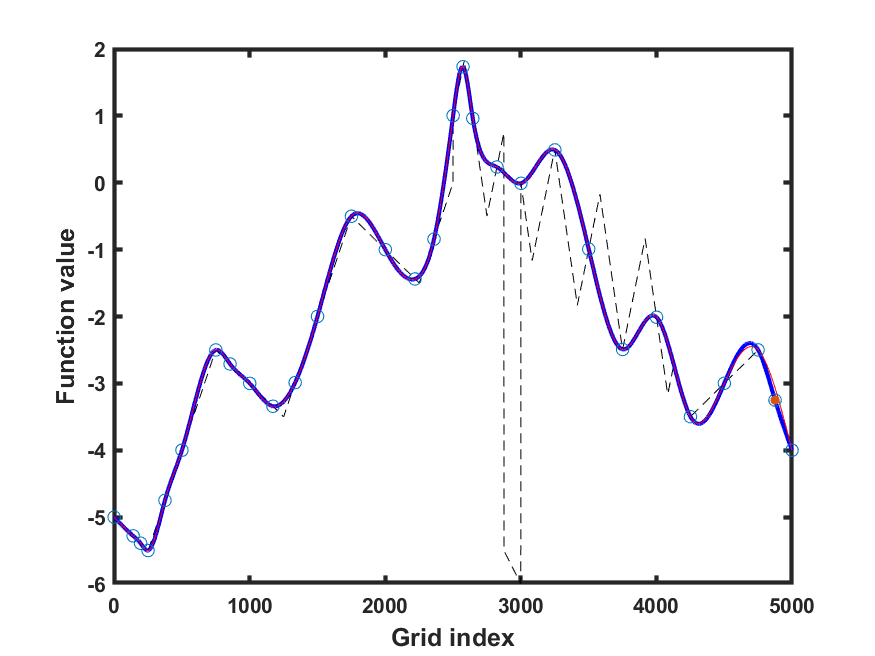}
\caption{LineWalker30}
\end{subfigure}
\newline
\begin{subfigure}[b]{0.270\textwidth}
\includegraphics[width=\textwidth]{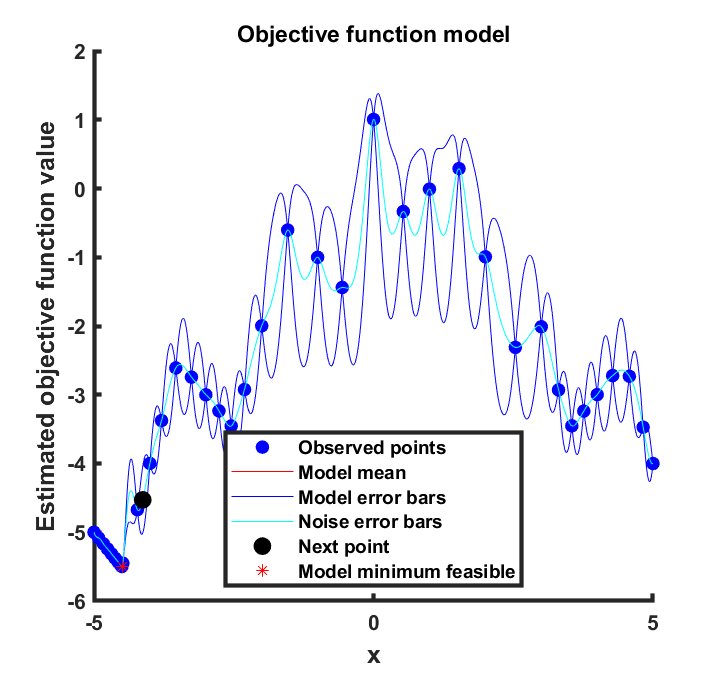}
\caption{bayesopt40}
\end{subfigure}
\begin{subfigure}[b]{0.330\textwidth}
\includegraphics[width=\textwidth]{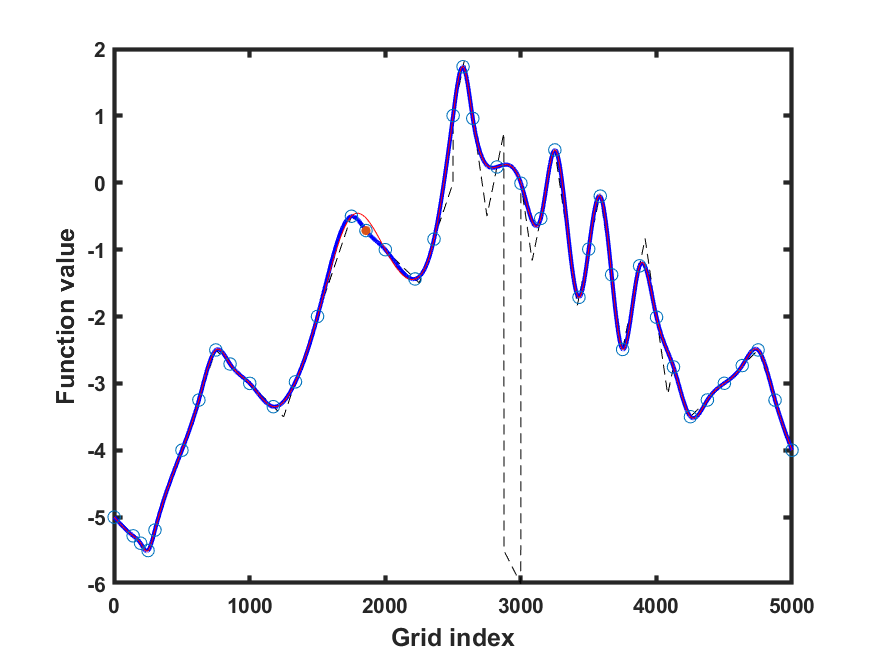}
\caption{LineWalker40}
\end{subfigure}
\newline
\begin{subfigure}[b]{0.270\textwidth}
\includegraphics[width=\textwidth]{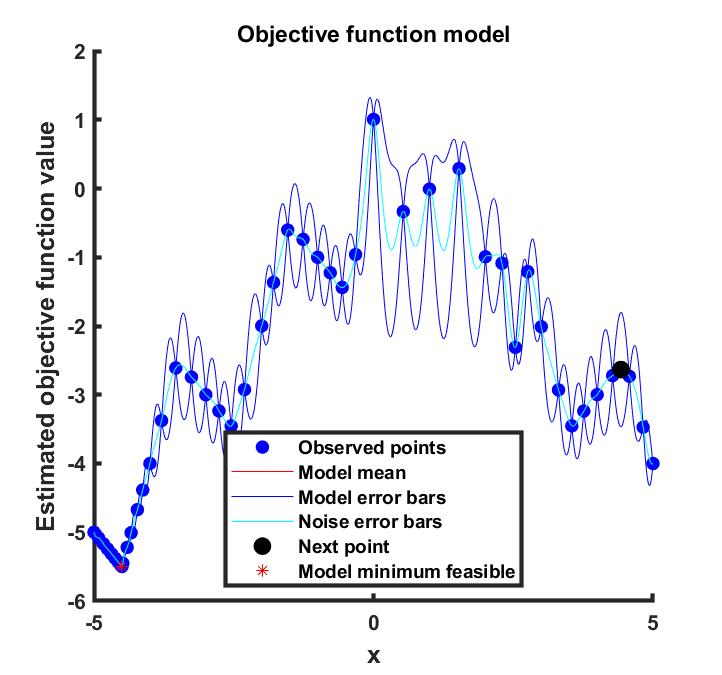}
\caption{bayesopt50}
\end{subfigure}
\begin{subfigure}[b]{0.330\textwidth}
\includegraphics[width=\textwidth]{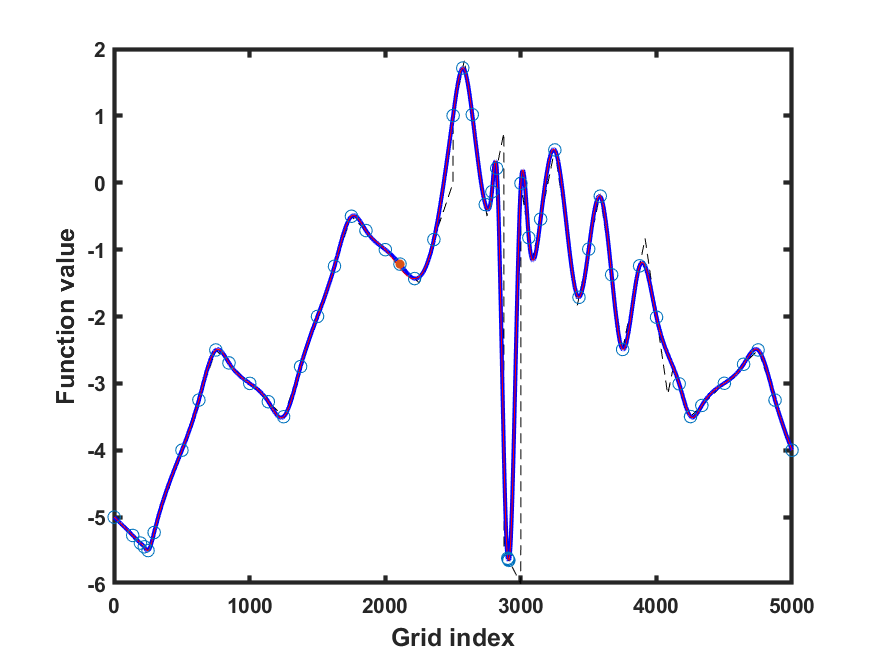}
\caption{LineWalker50}
\end{subfigure}
\newline
\caption{sawtoothD. Left column = \texttt{bayesopt}. Right column = \texttt{LineWalker-full}}
\label{fig:out_sawtoothD}
\end{figure}

\begin{figure}
\centering
\begin{subfigure}[b]{0.270\textwidth}
\includegraphics[width=\textwidth]{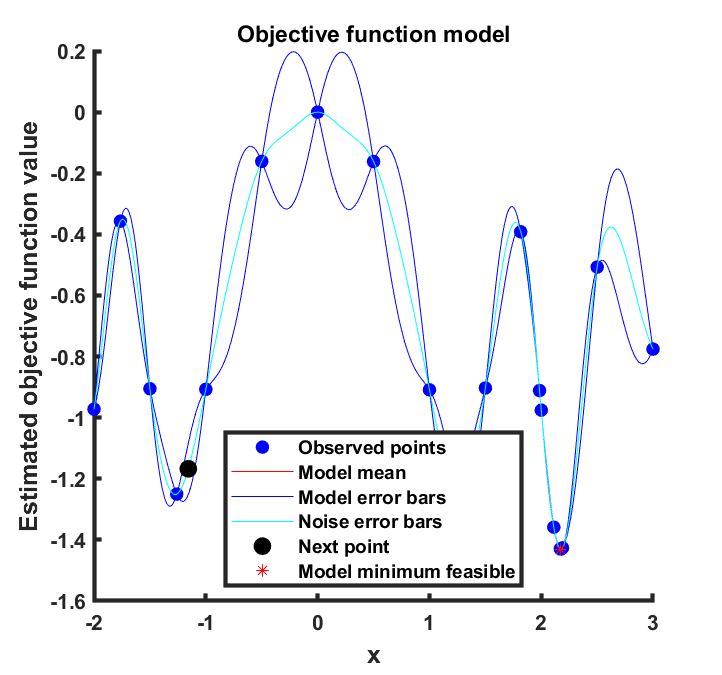}
\caption{bayesopt20}
\end{subfigure}
\begin{subfigure}[b]{0.330\textwidth}
\includegraphics[width=\textwidth]{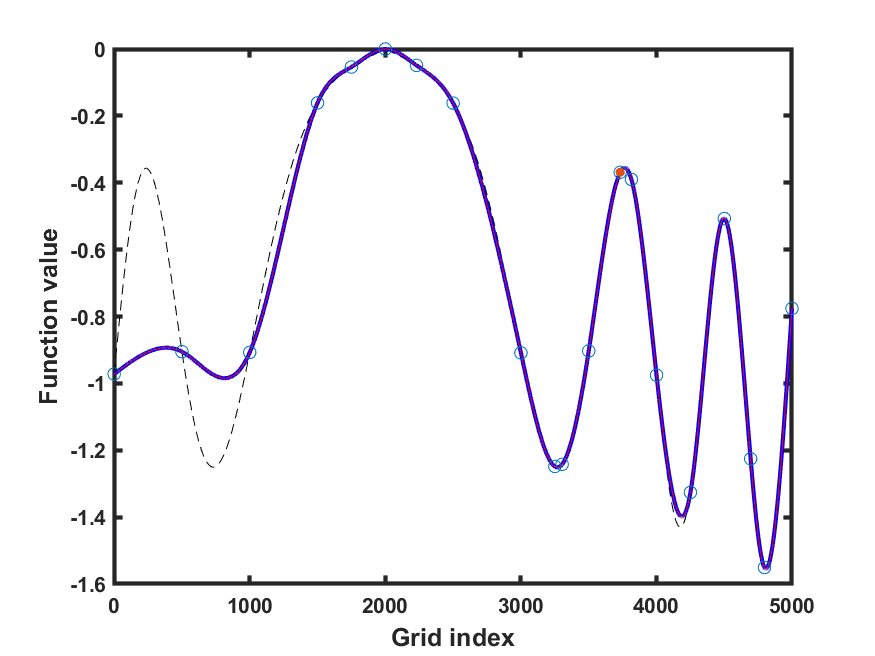}
\caption{LineWalker20}
\end{subfigure}
\newline
\begin{subfigure}[b]{0.270\textwidth}
\includegraphics[width=\textwidth]{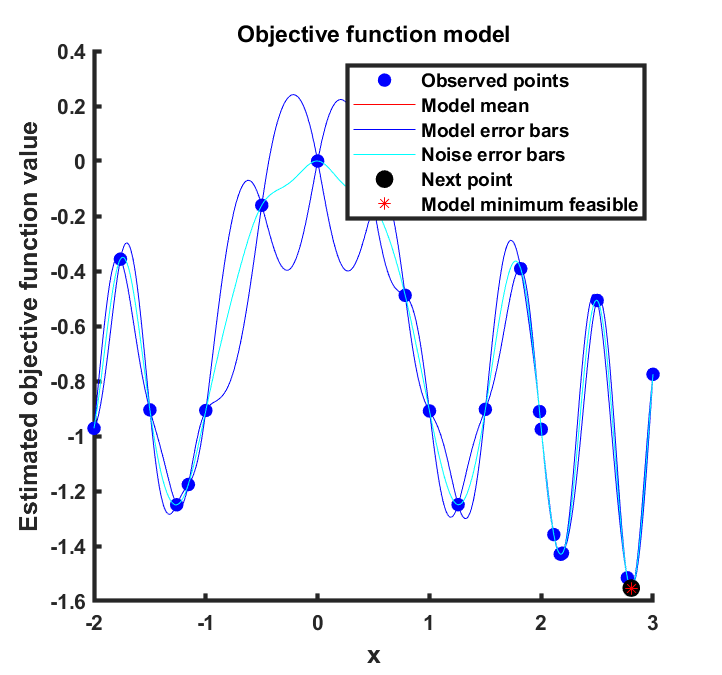}
\caption{bayesopt30}
\end{subfigure}
\begin{subfigure}[b]{0.330\textwidth}
\includegraphics[width=\textwidth]{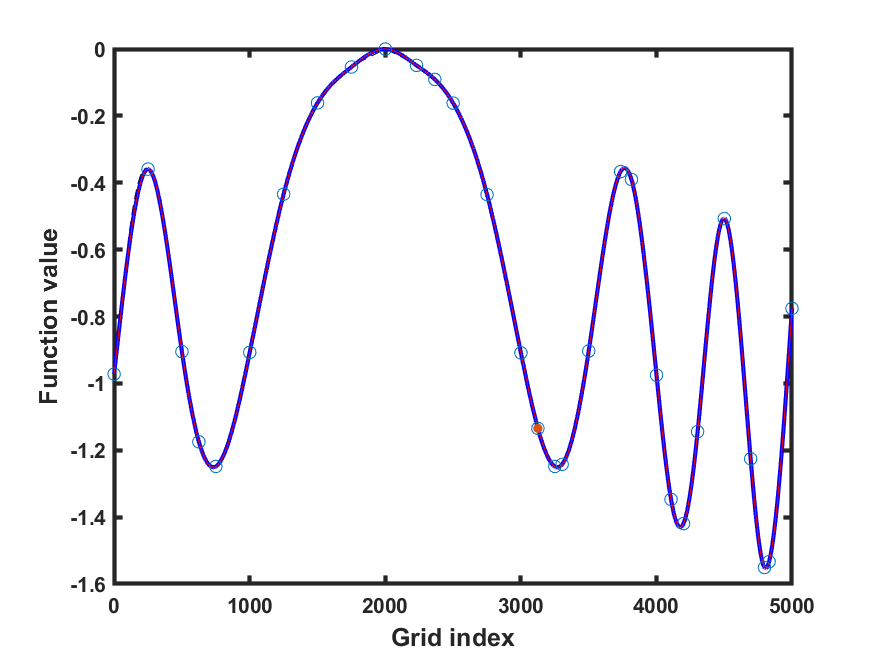}
\caption{LineWalker30}
\end{subfigure}
\newline
\begin{subfigure}[b]{0.270\textwidth}
\includegraphics[width=\textwidth]{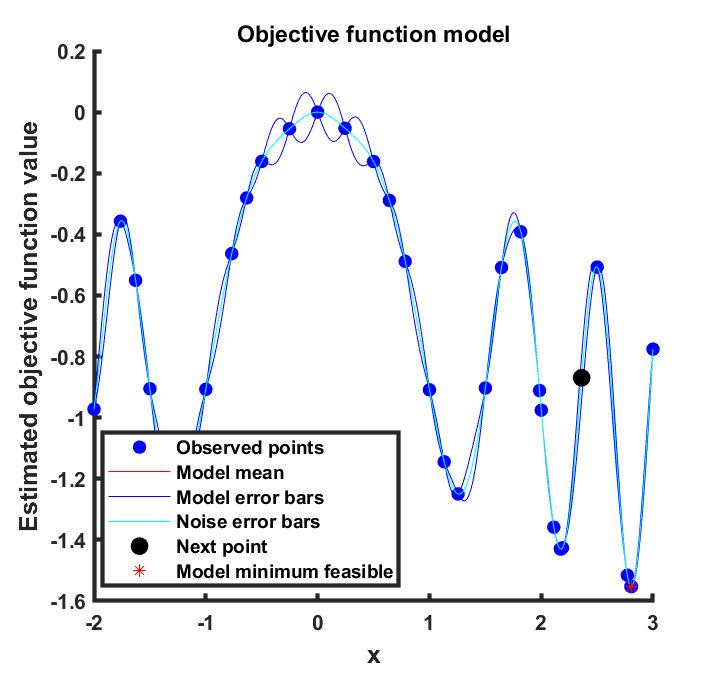}
\caption{bayesopt40}
\end{subfigure}
\begin{subfigure}[b]{0.330\textwidth}
\includegraphics[width=\textwidth]{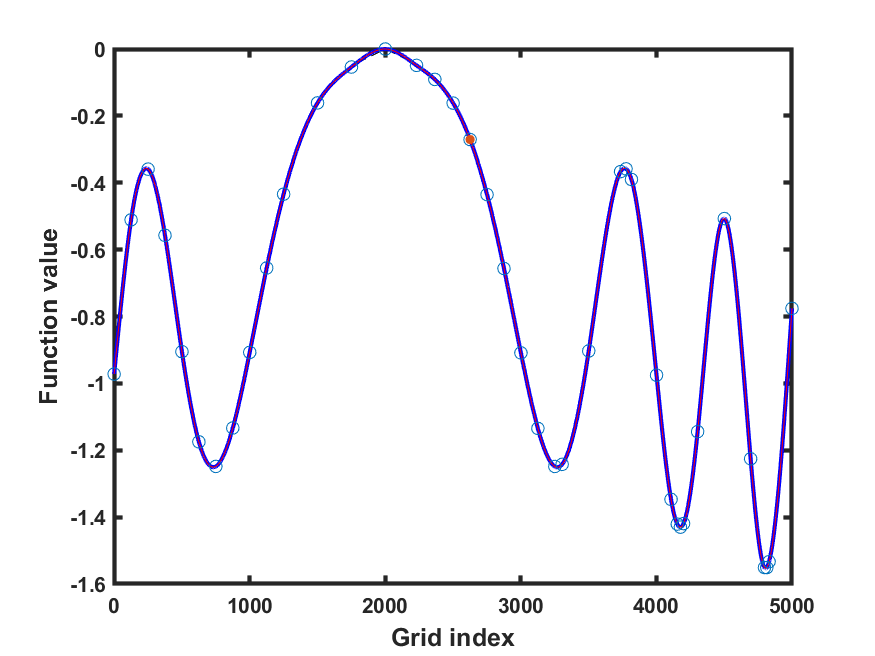}
\caption{LineWalker40}
\end{subfigure}
\newline
\begin{subfigure}[b]{0.270\textwidth}
\includegraphics[width=\textwidth]{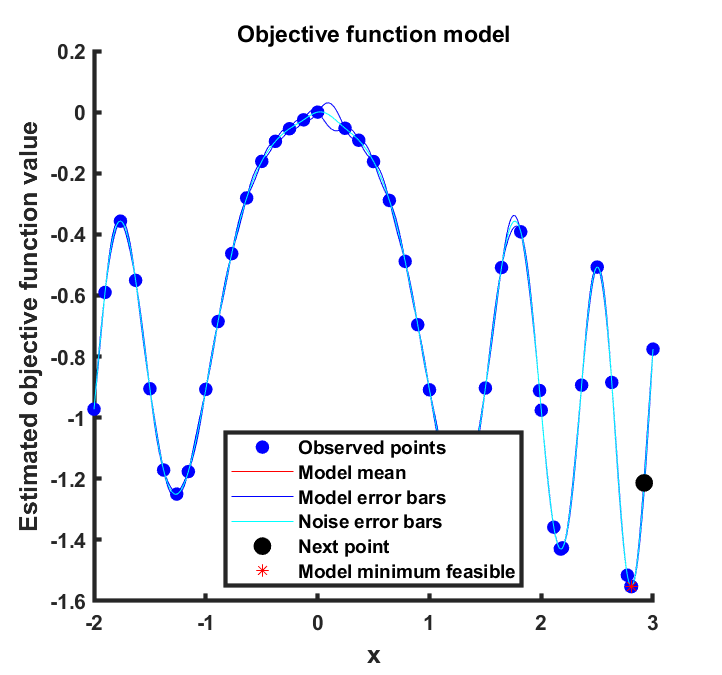}
\caption{bayesopt50}
\end{subfigure}
\begin{subfigure}[b]{0.330\textwidth}
\includegraphics[width=\textwidth]{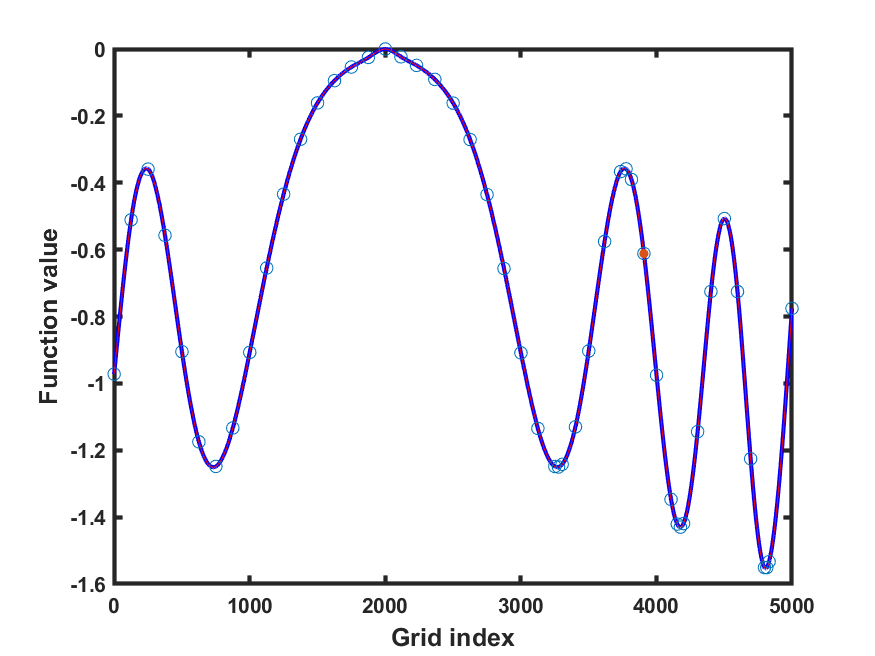}
\caption{LineWalker50}
\end{subfigure}
\newline
\caption{schaffer2A. Left column = \texttt{bayesopt}. Right column = \texttt{LineWalker-full}}
\label{fig:out_schaffer2A_1Dslice}
\end{figure}

\begin{figure}
\centering
\begin{subfigure}[b]{0.270\textwidth}
\includegraphics[width=\textwidth]{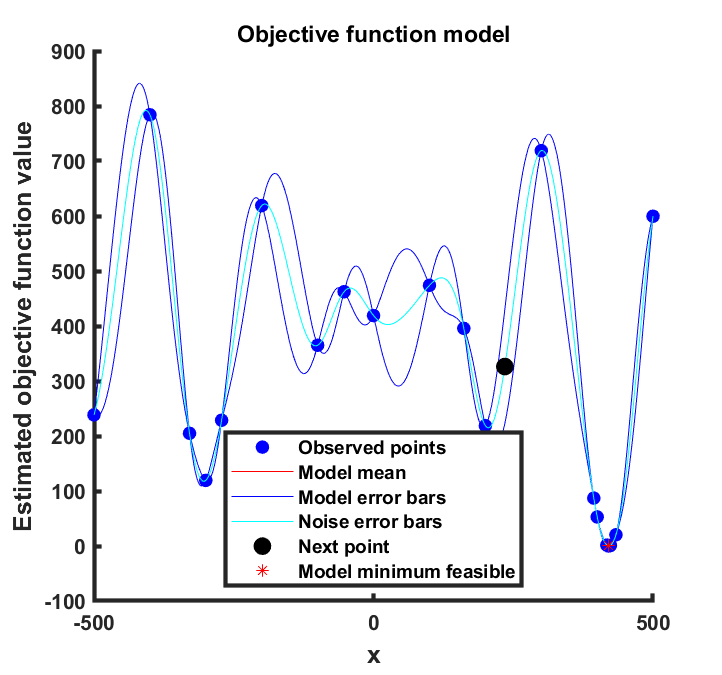}
\caption{bayesopt20}
\end{subfigure}
\begin{subfigure}[b]{0.330\textwidth}
\includegraphics[width=\textwidth]{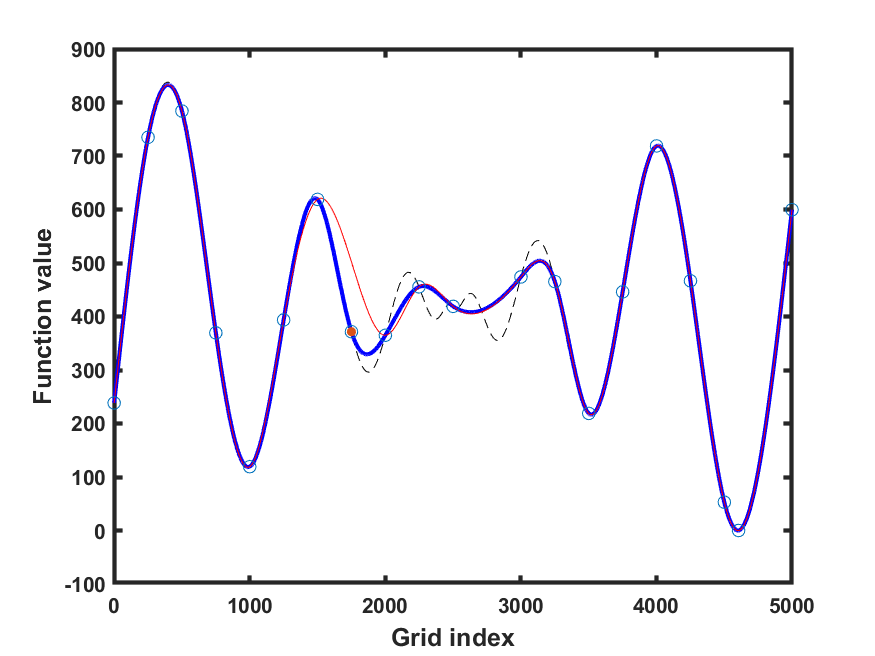}
\caption{LineWalker20}
\end{subfigure}
\newline
\begin{subfigure}[b]{0.270\textwidth}
\includegraphics[width=\textwidth]{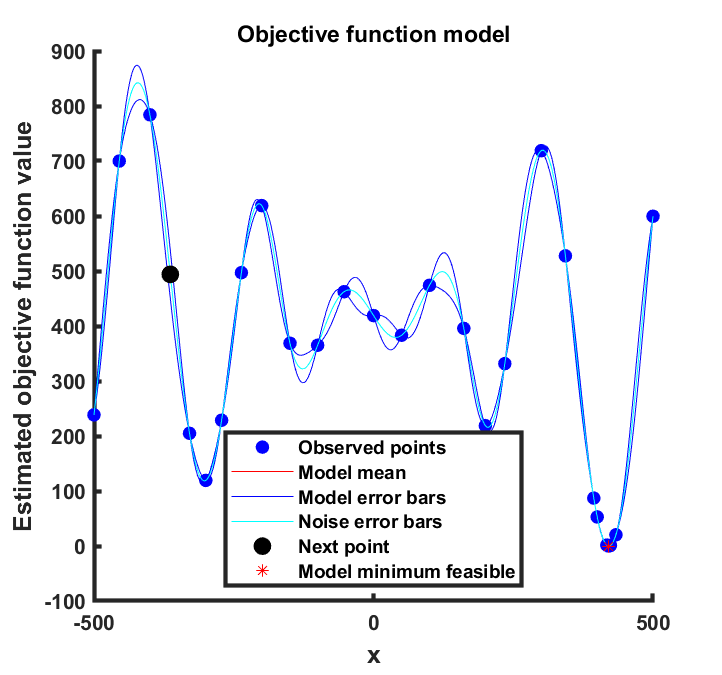}
\caption{bayesopt30}
\end{subfigure}
\begin{subfigure}[b]{0.330\textwidth}
\includegraphics[width=\textwidth]{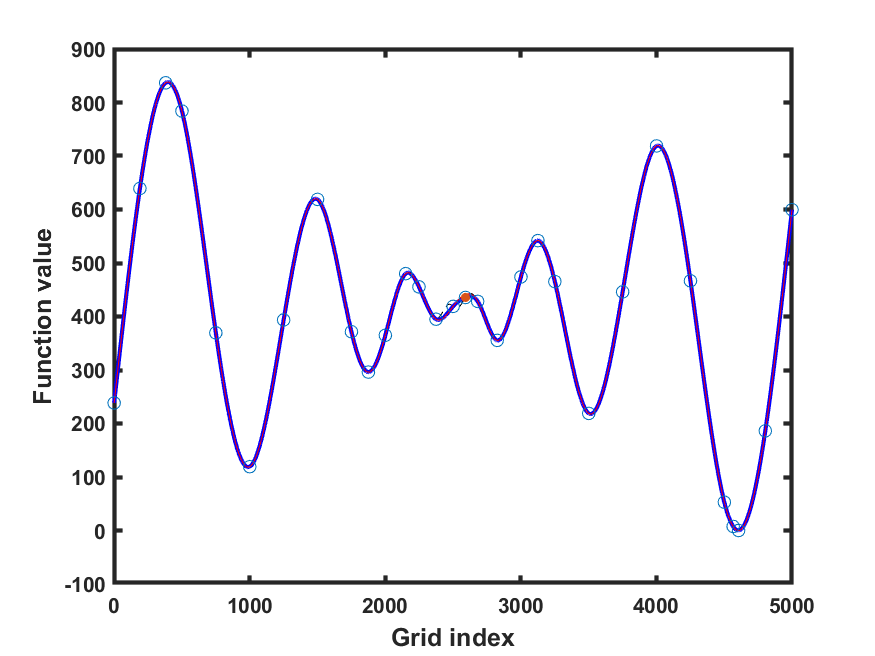}
\caption{LineWalker30}
\end{subfigure}
\newline
\begin{subfigure}[b]{0.270\textwidth}
\includegraphics[width=\textwidth]{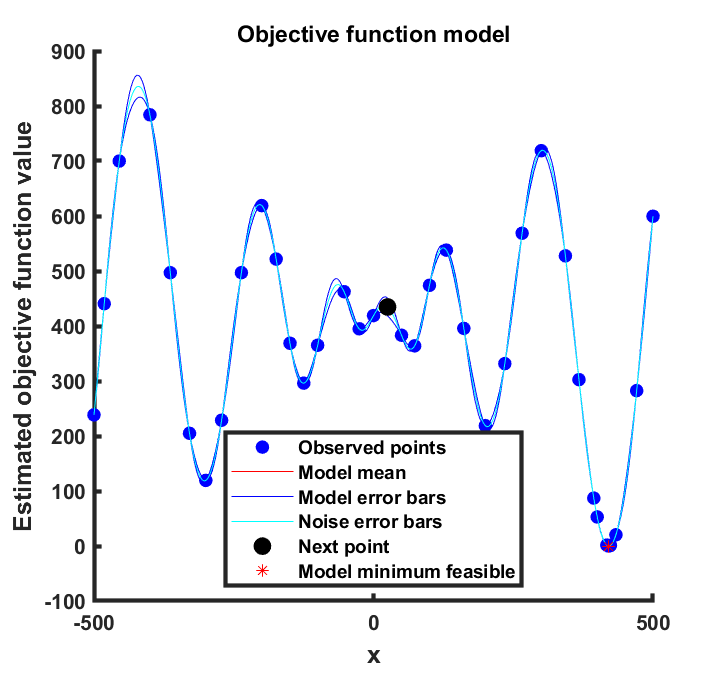}
\caption{bayesopt40}
\end{subfigure}
\begin{subfigure}[b]{0.330\textwidth}
\includegraphics[width=\textwidth]{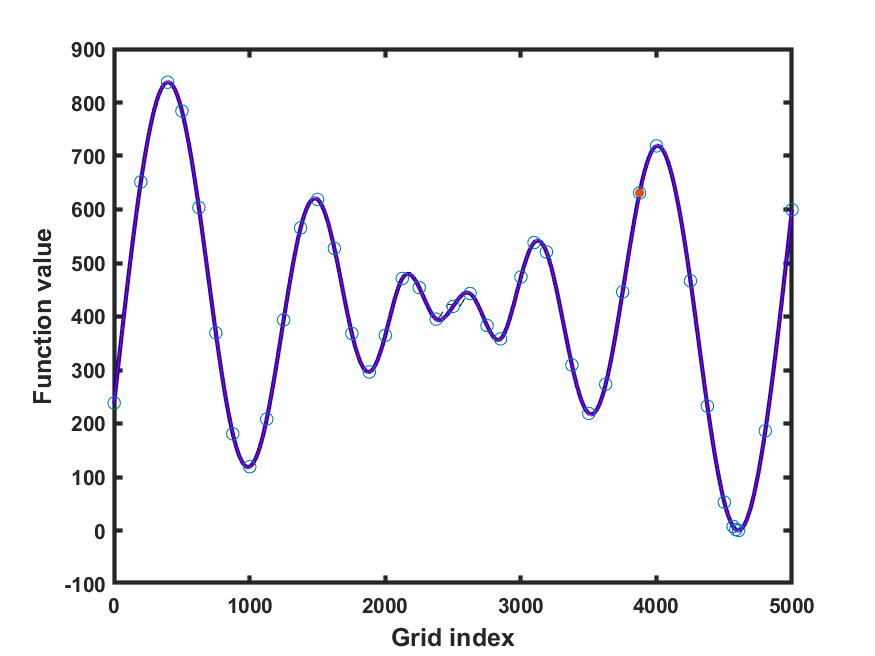}
\caption{LineWalker40}
\end{subfigure}
\newline
\begin{subfigure}[b]{0.270\textwidth}
\includegraphics[width=\textwidth]{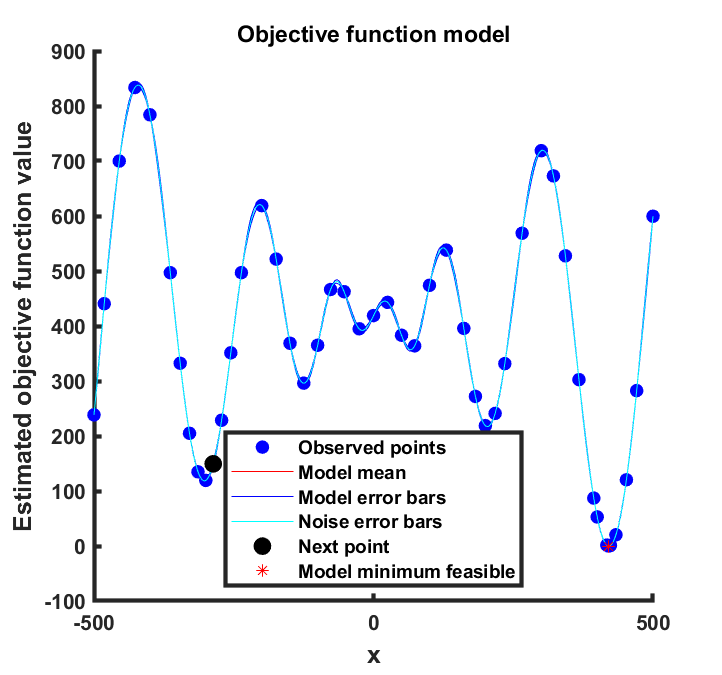}
\caption{bayesopt50}
\end{subfigure}
\begin{subfigure}[b]{0.330\textwidth}
\includegraphics[width=\textwidth]{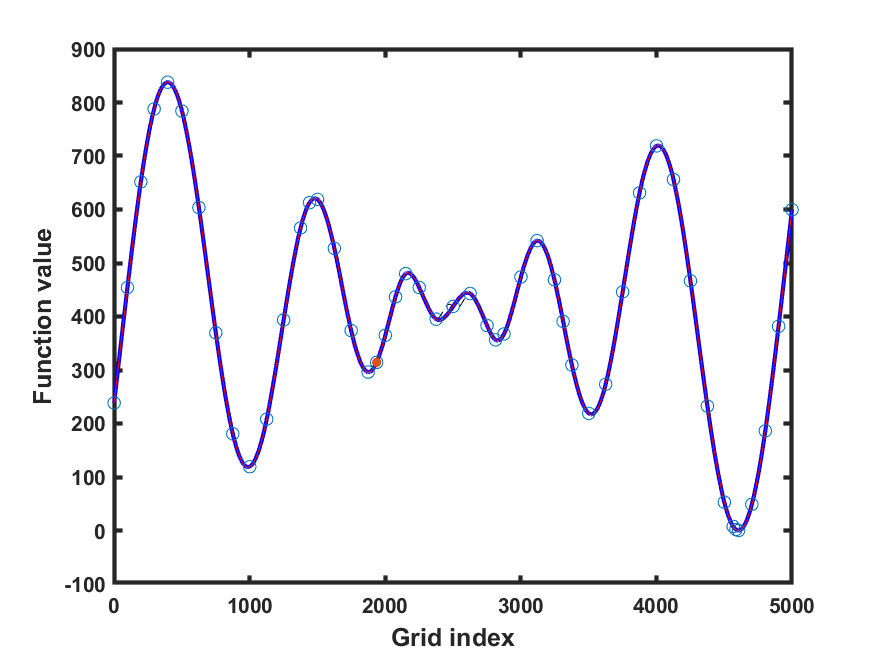}
\caption{LineWalker50}
\end{subfigure}
\newline
\caption{schwef. Left column = \texttt{bayesopt}. Right column = \texttt{LineWalker-full}}
\label{fig:out_schwef}
\end{figure}

\begin{figure}
\centering
\begin{subfigure}[b]{0.270\textwidth}
\includegraphics[width=\textwidth]{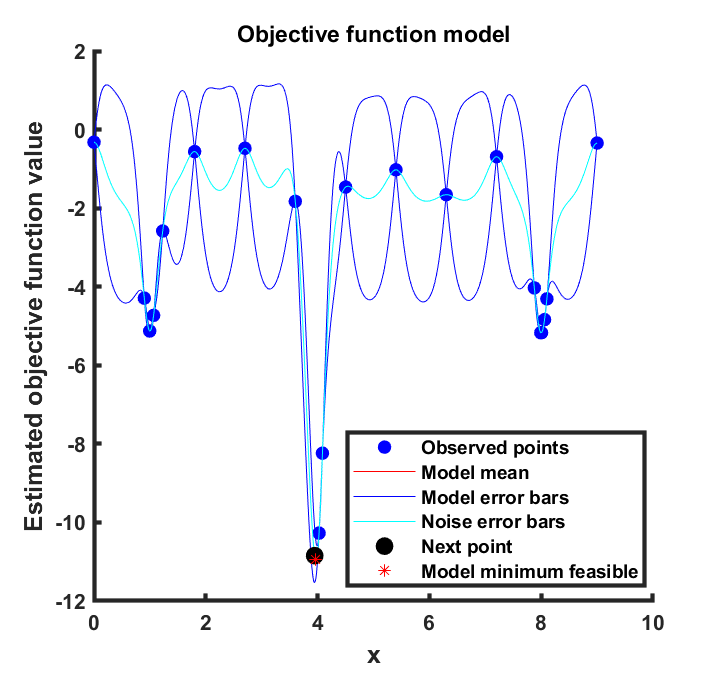}
\caption{bayesopt20}
\end{subfigure}
\begin{subfigure}[b]{0.330\textwidth}
\includegraphics[width=\textwidth]{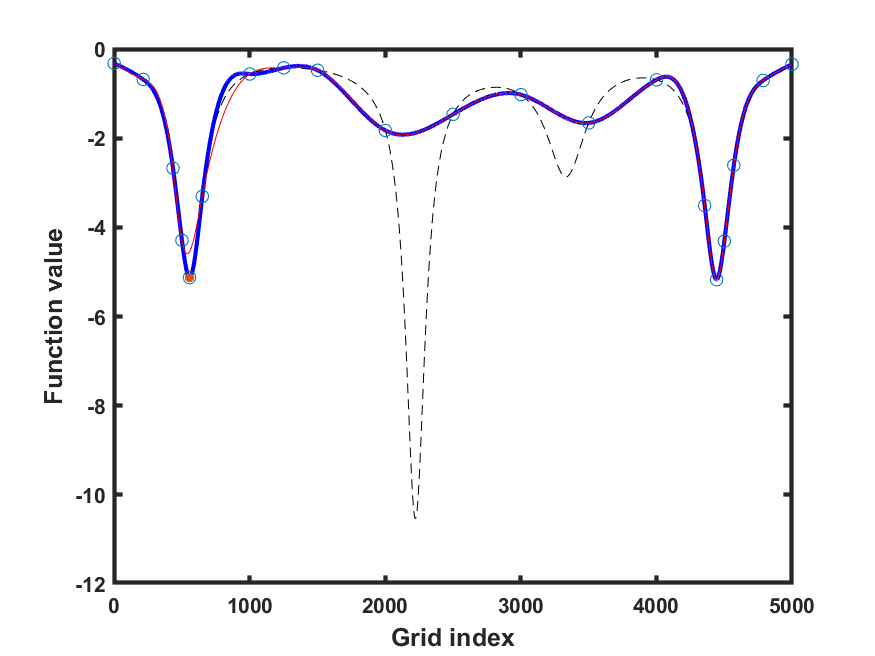}
\caption{LineWalker20}
\end{subfigure}
\newline
\begin{subfigure}[b]{0.270\textwidth}
\includegraphics[width=\textwidth]{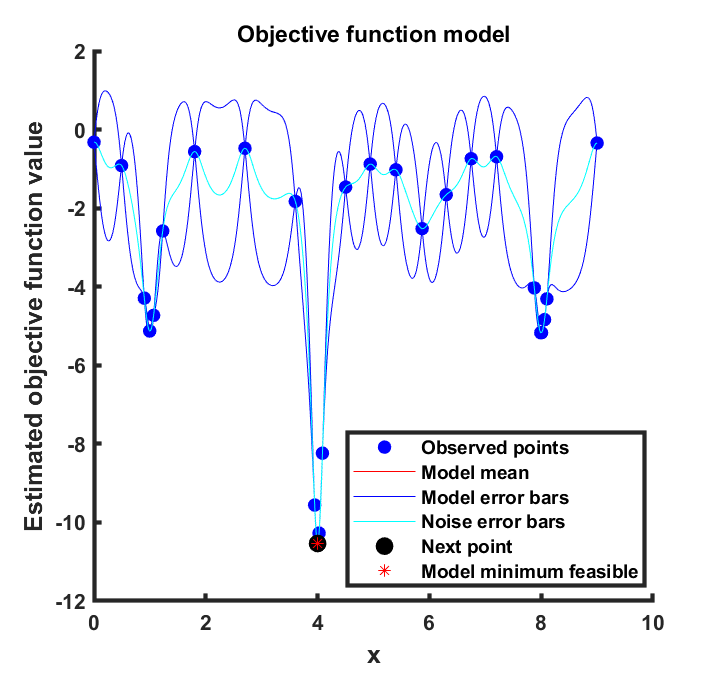}
\caption{bayesopt30}
\end{subfigure}
\begin{subfigure}[b]{0.330\textwidth}
\includegraphics[width=\textwidth]{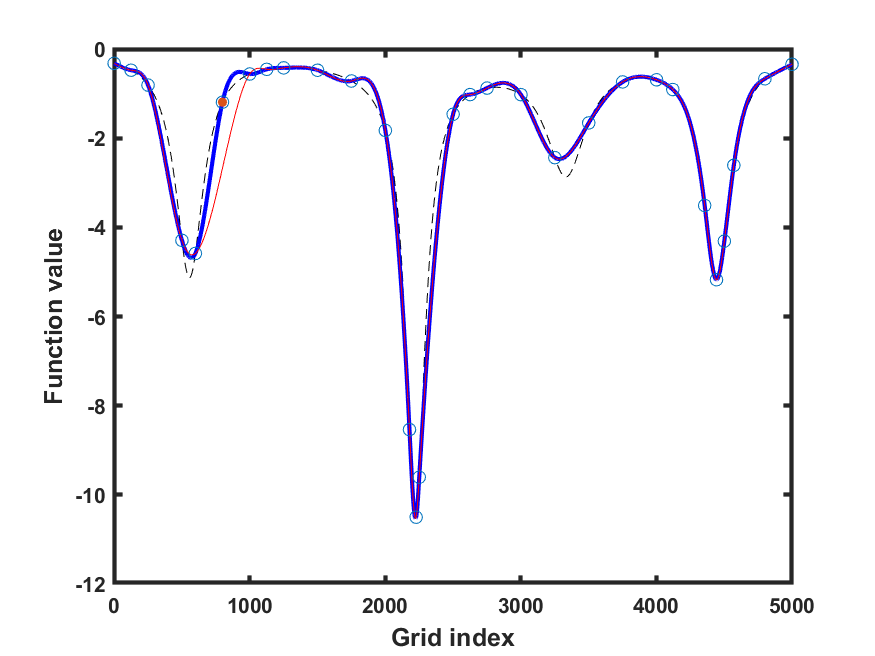}
\caption{LineWalker30}
\end{subfigure}
\newline
\begin{subfigure}[b]{0.270\textwidth}
\includegraphics[width=\textwidth]{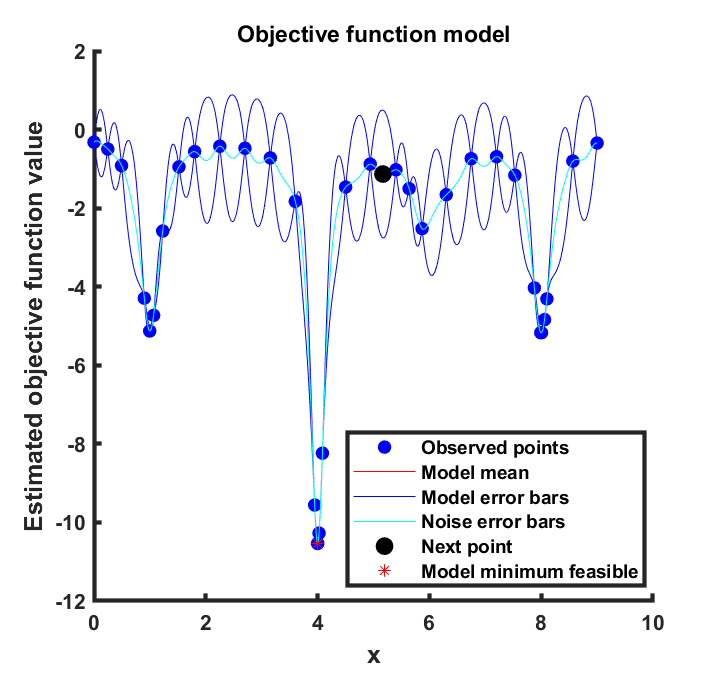}
\caption{bayesopt40}
\end{subfigure}
\begin{subfigure}[b]{0.330\textwidth}
\includegraphics[width=\textwidth]{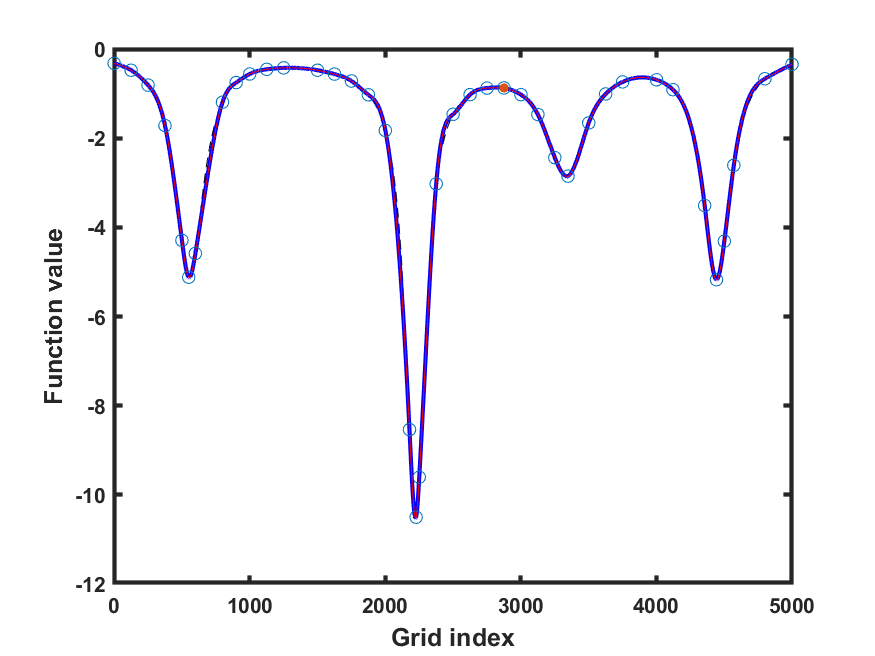}
\caption{LineWalker40}
\end{subfigure}
\newline
\begin{subfigure}[b]{0.270\textwidth}
\includegraphics[width=\textwidth]{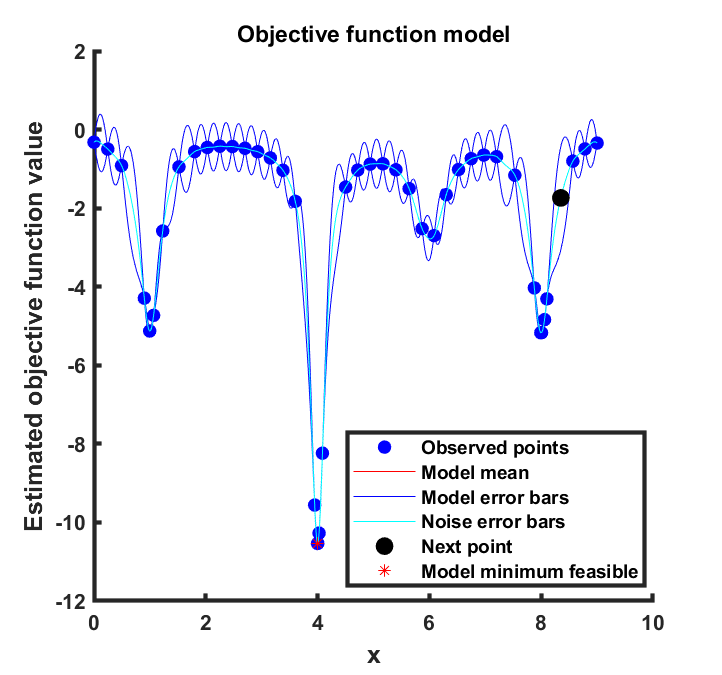}
\caption{bayesopt50}
\end{subfigure}
\begin{subfigure}[b]{0.330\textwidth}
\includegraphics[width=\textwidth]{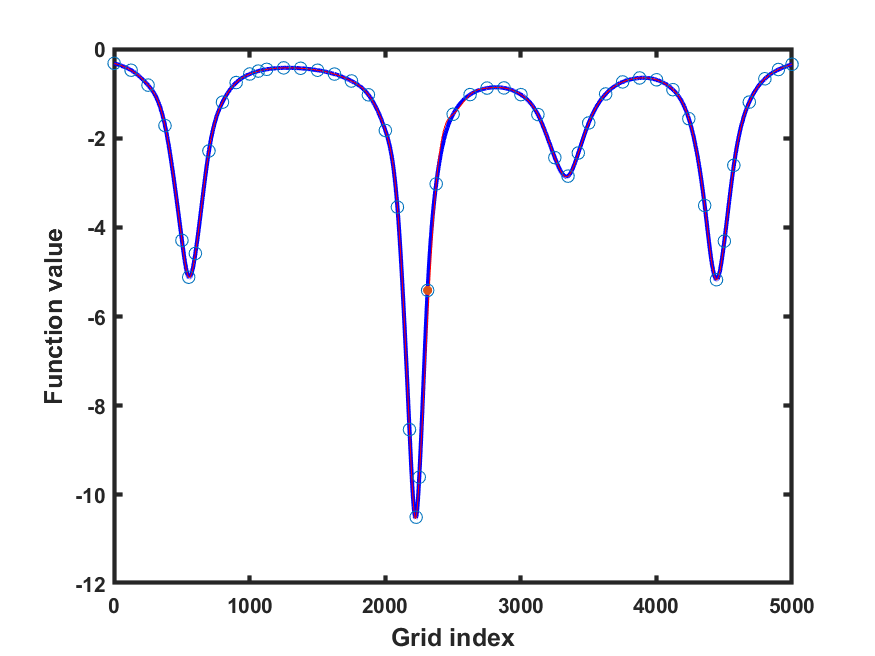}
\caption{LineWalker50}
\end{subfigure}
\newline
\caption{shekel. Left column = \texttt{bayesopt}. Right column = \texttt{LineWalker-full}}
\label{fig:out_shekel_1Dslice}
\end{figure}

\begin{figure}
\centering
\begin{subfigure}[b]{0.270\textwidth}
\includegraphics[width=\textwidth]{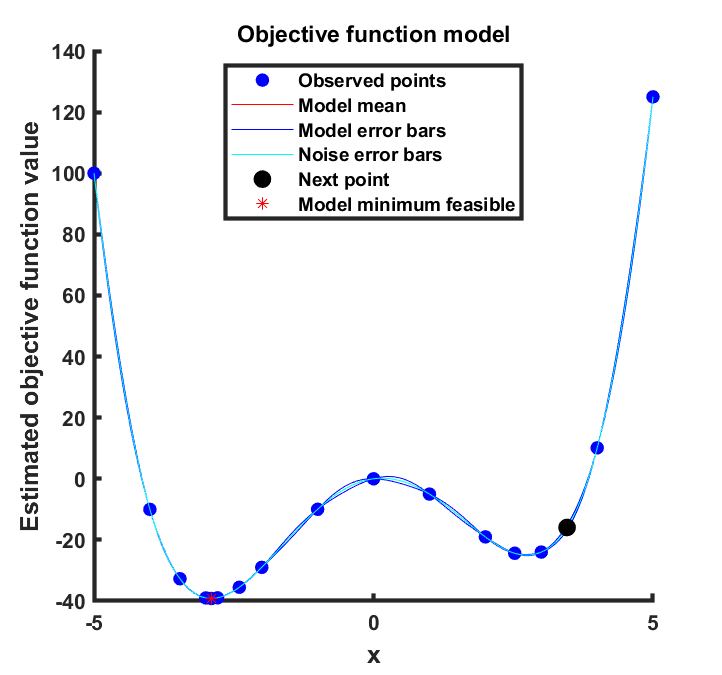}
\caption{bayesopt20}
\end{subfigure}
\begin{subfigure}[b]{0.330\textwidth}
\includegraphics[width=\textwidth]{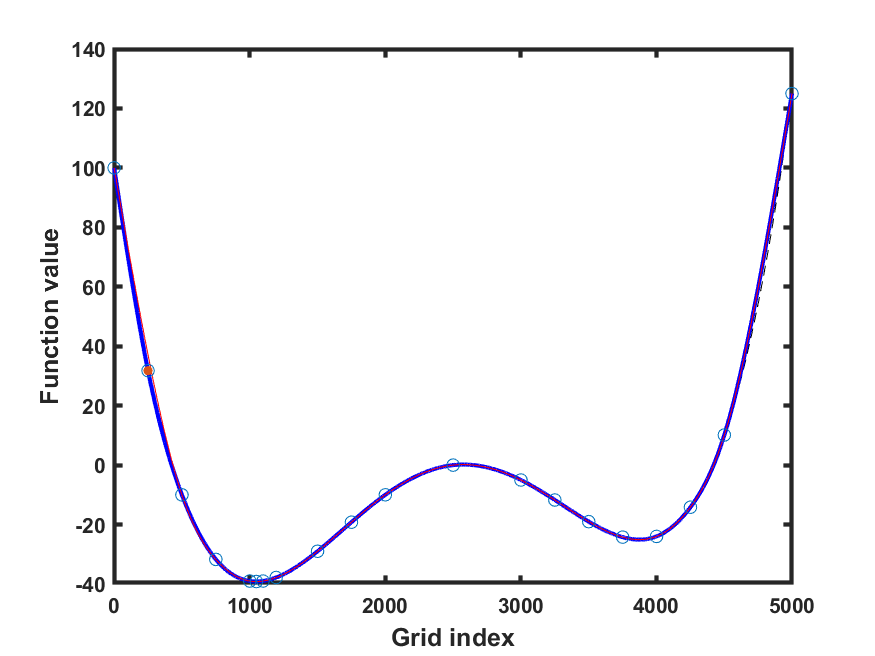}
\caption{LineWalker20}
\end{subfigure}
\newline
\begin{subfigure}[b]{0.270\textwidth}
\includegraphics[width=\textwidth]{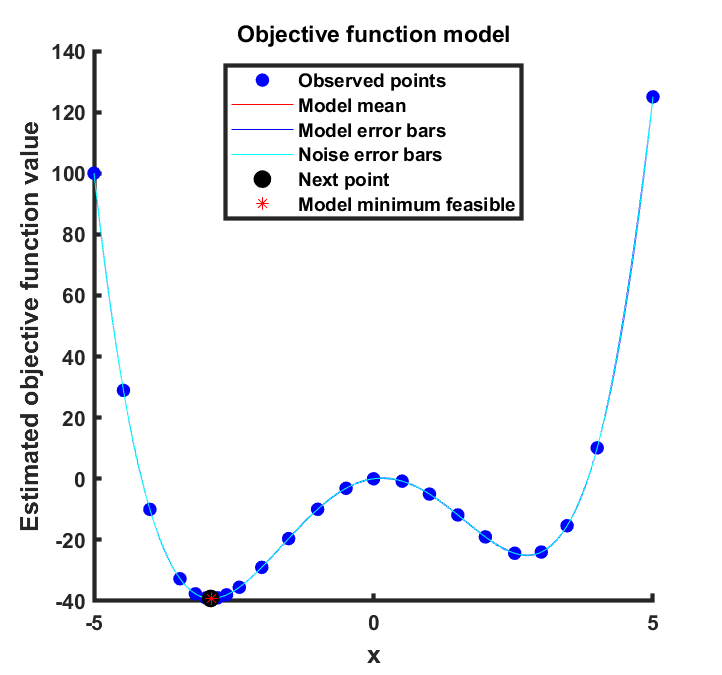}
\caption{bayesopt30}
\end{subfigure}
\begin{subfigure}[b]{0.330\textwidth}
\includegraphics[width=\textwidth]{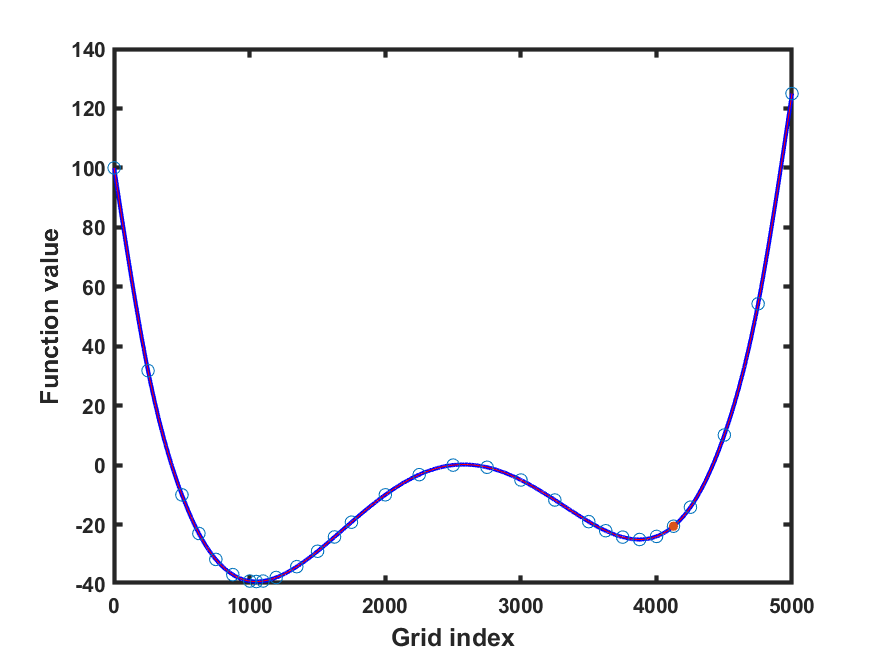}
\caption{LineWalker30}
\end{subfigure}
\newline
\begin{subfigure}[b]{0.270\textwidth}
\includegraphics[width=\textwidth]{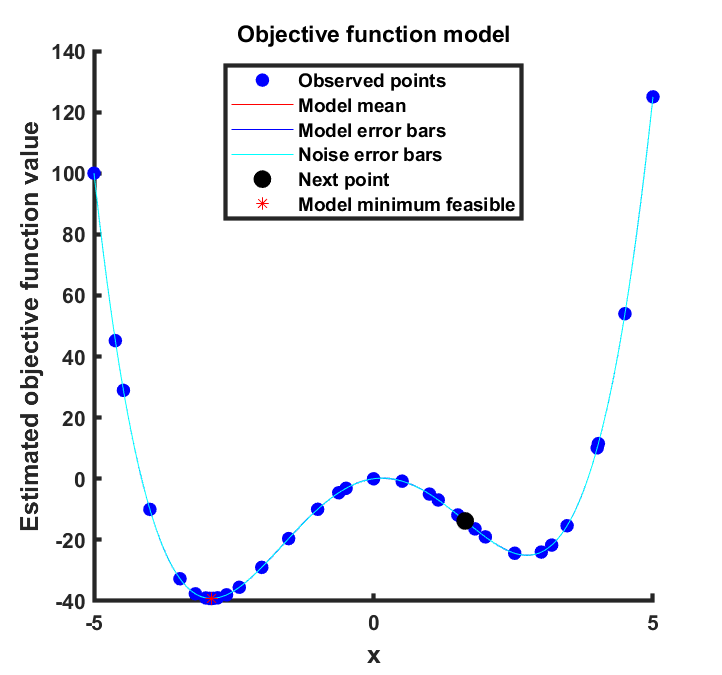}
\caption{bayesopt40}
\end{subfigure}
\begin{subfigure}[b]{0.330\textwidth}
\includegraphics[width=\textwidth]{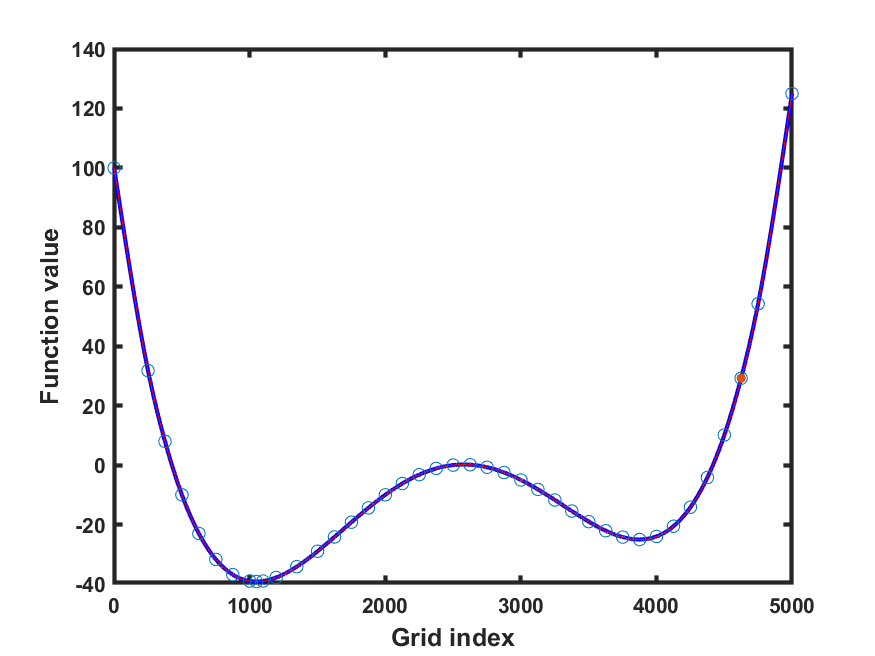}
\caption{LineWalker40}
\end{subfigure}
\newline
\begin{subfigure}[b]{0.270\textwidth}
\includegraphics[width=\textwidth]{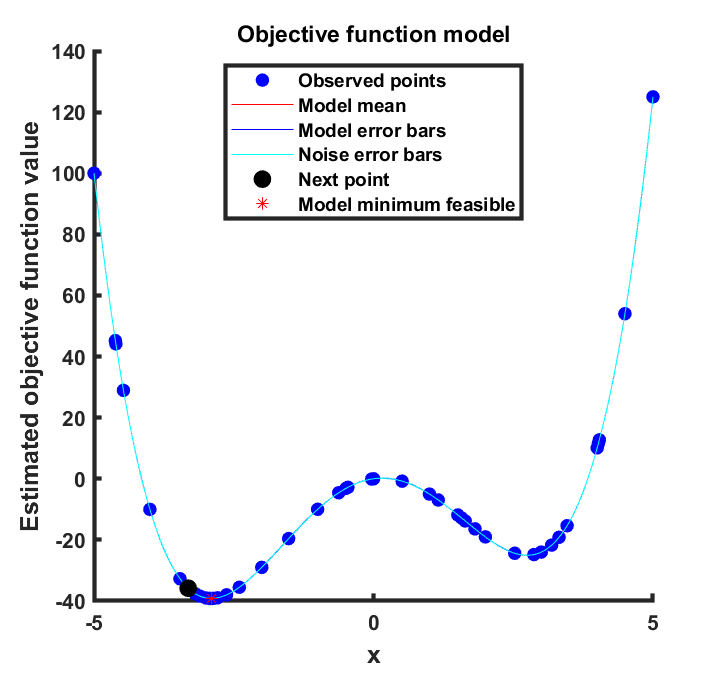}
\caption{bayesopt50}
\end{subfigure}
\begin{subfigure}[b]{0.330\textwidth}
\includegraphics[width=\textwidth]{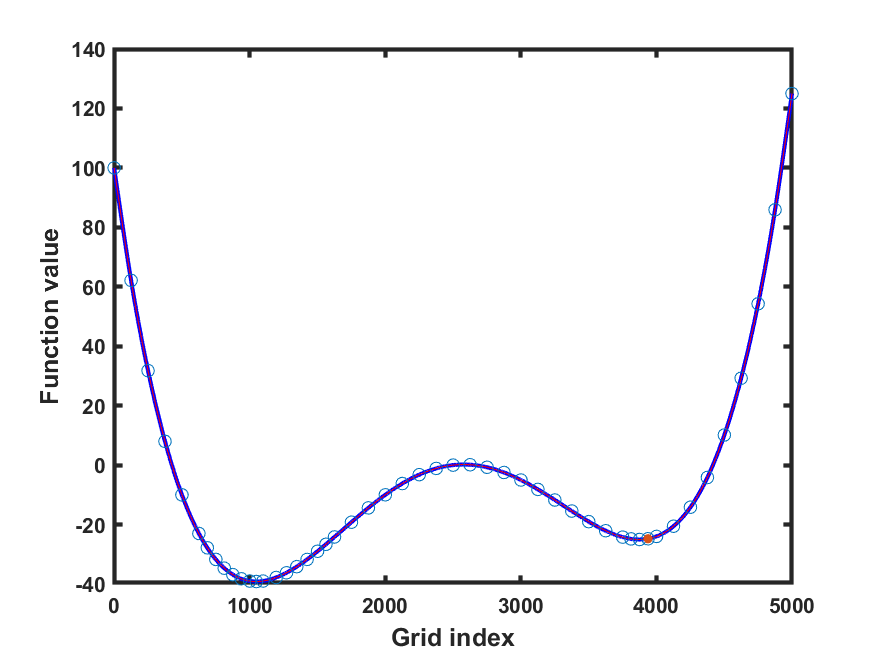}
\caption{LineWalker50}
\end{subfigure}
\newline
\caption{stybtang. Left column = \texttt{bayesopt}. Right column = \texttt{LineWalker-full}}
\label{fig:out_stybtang}
\end{figure}

\begin{figure}
\centering
\begin{subfigure}[b]{0.270\textwidth}
\includegraphics[width=\textwidth]{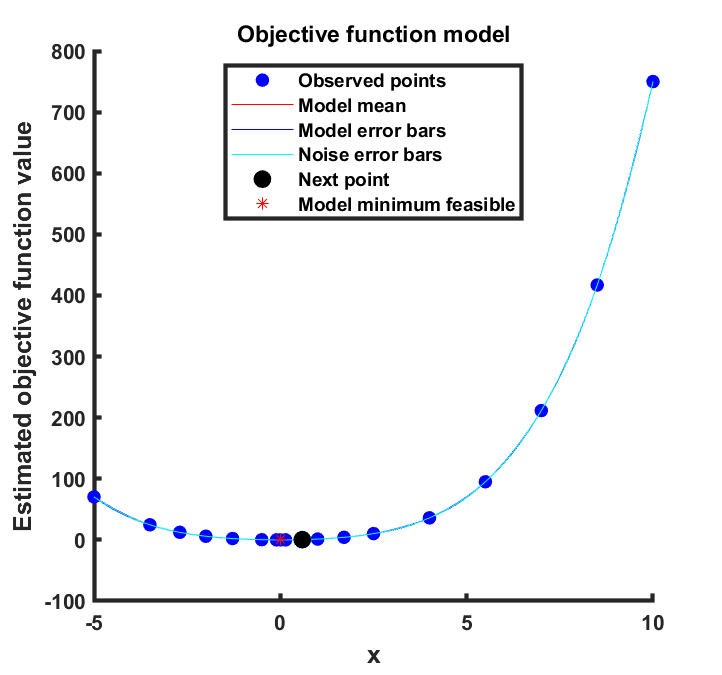}
\caption{bayesopt20}
\end{subfigure}
\begin{subfigure}[b]{0.330\textwidth}
\includegraphics[width=\textwidth]{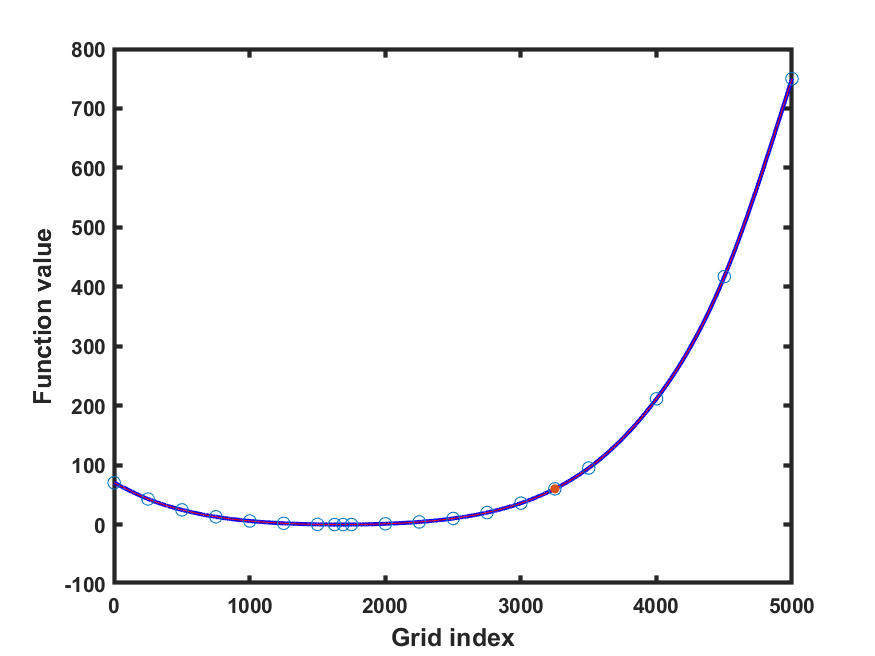}
\caption{LineWalker20}
\end{subfigure}
\newline
\begin{subfigure}[b]{0.270\textwidth}
\includegraphics[width=\textwidth]{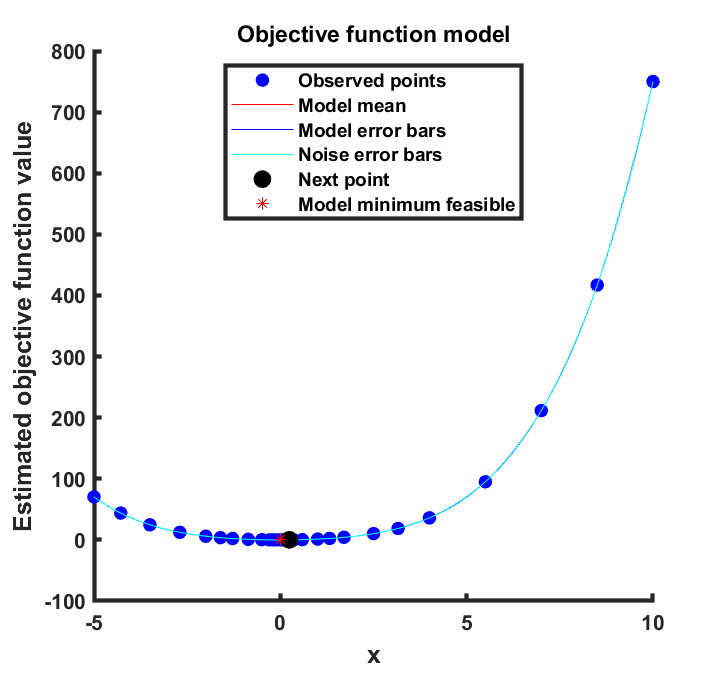}
\caption{bayesopt30}
\end{subfigure}
\begin{subfigure}[b]{0.330\textwidth}
\includegraphics[width=\textwidth]{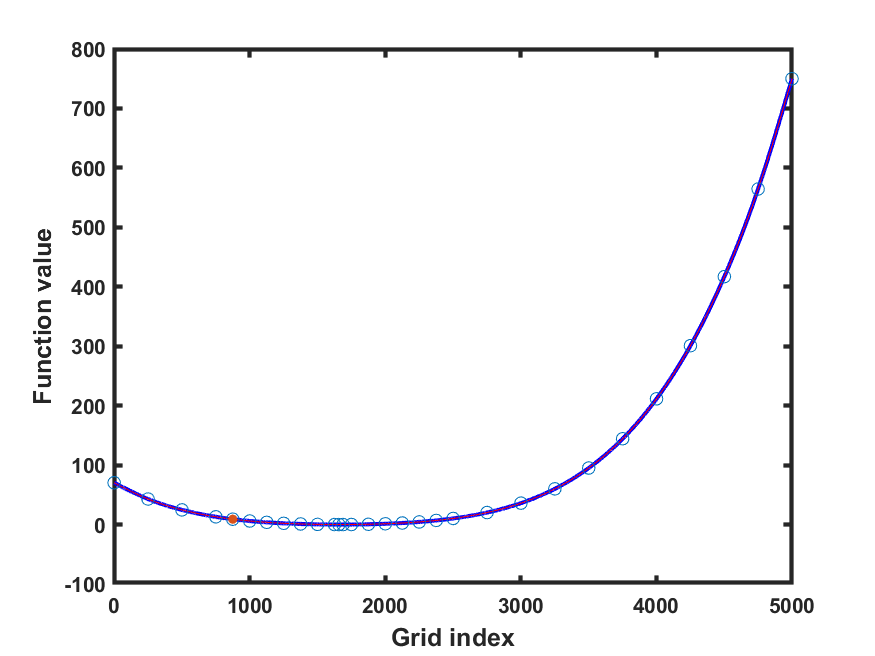}
\caption{LineWalker30}
\end{subfigure}
\newline
\begin{subfigure}[b]{0.270\textwidth}
\includegraphics[width=\textwidth]{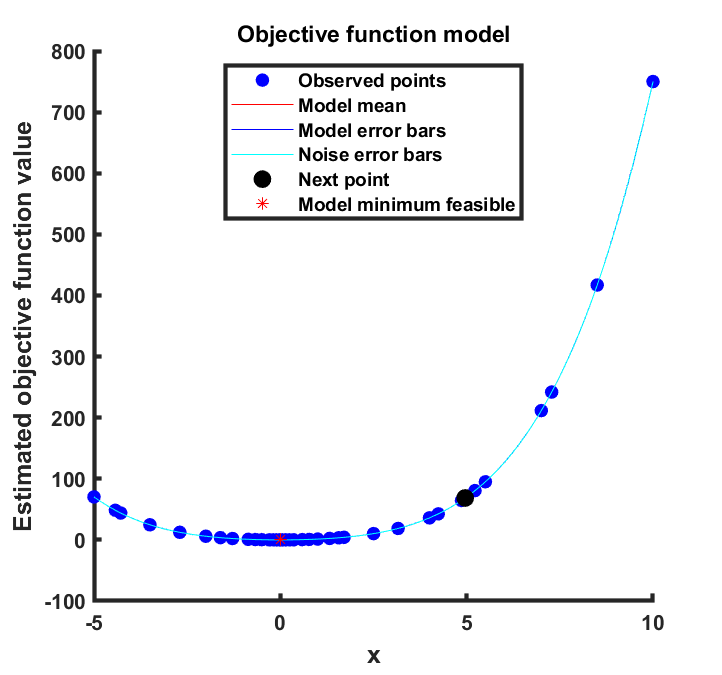}
\caption{bayesopt40}
\end{subfigure}
\begin{subfigure}[b]{0.330\textwidth}
\includegraphics[width=\textwidth]{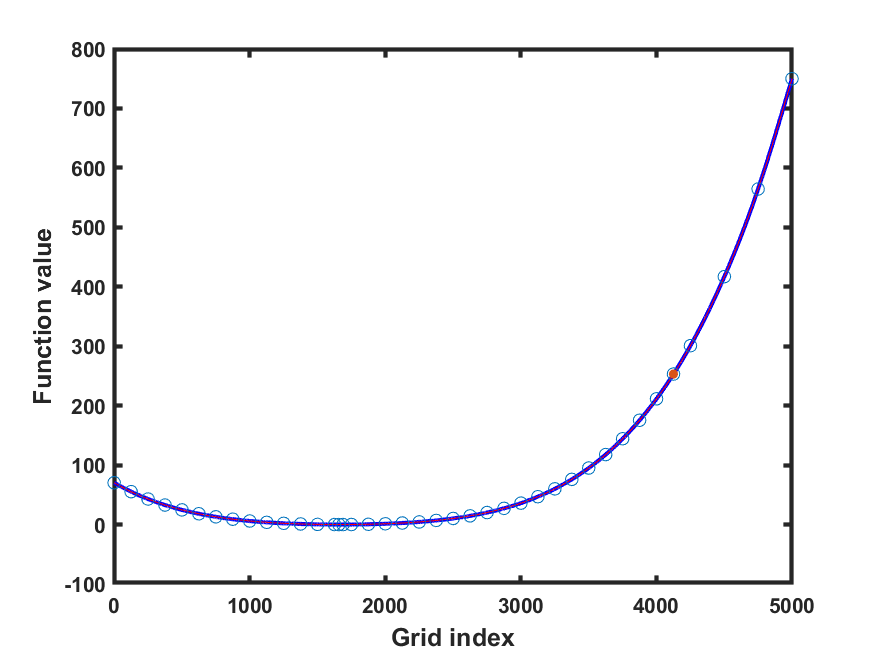}
\caption{LineWalker40}
\end{subfigure}
\newline
\begin{subfigure}[b]{0.270\textwidth}
\includegraphics[width=\textwidth]{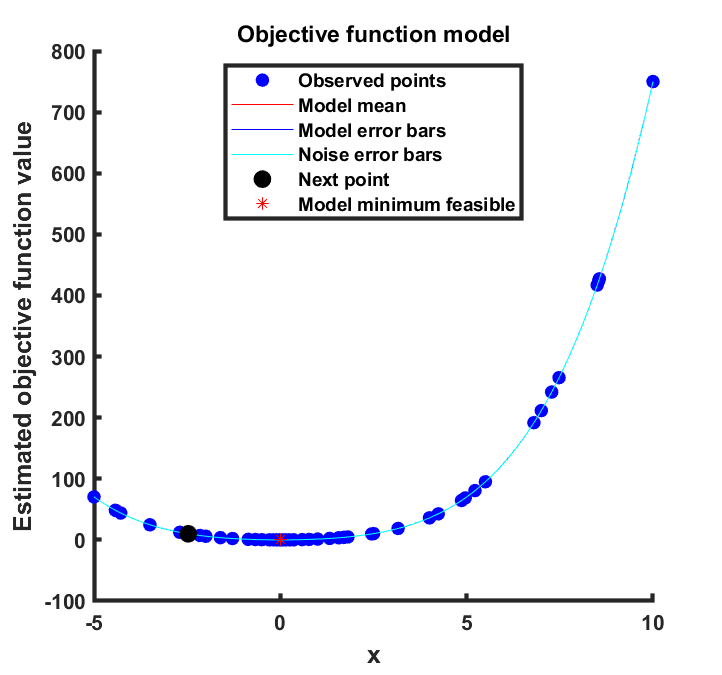}
\caption{bayesopt50}
\end{subfigure}
\begin{subfigure}[b]{0.330\textwidth}
\includegraphics[width=\textwidth]{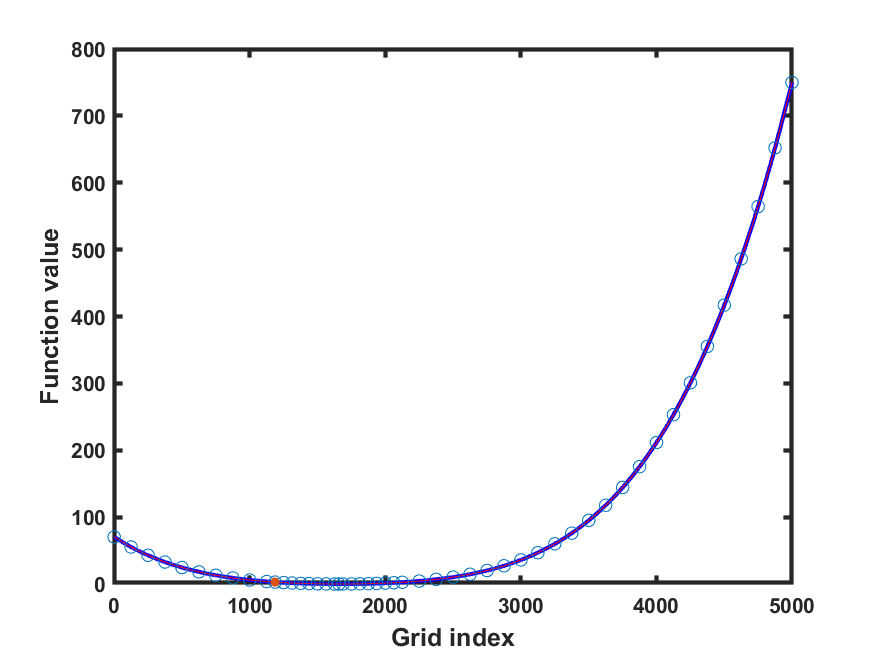}
\caption{LineWalker50}
\end{subfigure}
\newline
\caption{zakharov. Left column = \texttt{bayesopt}. Right column = \texttt{LineWalker-full}}
\label{fig:out_zakharov}
\end{figure}

\end{document}